\documentclass[11pt,oneside,a4paper]{book}

\usepackage{stmaryrd}
\usepackage{amsmath}
\usepackage{verbatim}
\usepackage{amssymb}
\usepackage[curve,matrix,arrow]{xy}
\usepackage{mathrsfs}
\usepackage[style=alphabetic]{biblatex}
\usepackage{amsthm}
\usepackage{graphicx}
\usepackage{paralist}
\usepackage{fancyhdr}

\bibliography{MotThickId}

\def\f{{\lowercase{f}}}
\def\fin{{\lowercase{fin}}}
\def\X{\mathbb{X}}
\def\F{\mathbb{F}}
\def\SH{\mathcal{SH}}
\def\P{\mathbb{P}}
\def\B{\mathbb{B}}
\def\BP{ABP}
\def\G{\mathbb{G}}

\def\C{\mathbb{C}}
\def\R{\mathbb{R}}
\def\Q{\mathbb{Q}}

\def\P{\mathbb{P}}
\def\Z{\mathbb{Z}}
\def\N{\mathbb{N}}
\def\A{\mathbb{A}}

\DeclareMathOperator{\ann}{ann}

\DeclareMathOperator{\cof}{cof}

\DeclareMathOperator{\colim}{colim}

\DeclareMathOperator{\Ext}{Ext}

\DeclareMathOperator{\Fin}{Fin}
\DeclareMathOperator{\FIdl}{FIdl}

\DeclareMathOperator{\h}{H}
\DeclareMathOperator{\Ho}{Ho}
\DeclareMathOperator{\Hom}{Hom}
\DeclareMathOperator{\id}{id}
\DeclareMathOperator{\Idl}{Idl}
\DeclareMathOperator{\im}{im}

\DeclareMathOperator{\Map}{Map}

\DeclareMathOperator{\Mod}{Mod}

\DeclareMathOperator{\Nis}{Nis}

\DeclareMathOperator{\Sing}{Sing}
\DeclareMathOperator{\Sm}{{\bf Sm}}
\DeclareMathOperator{\Spc}{Spc}
\DeclareMathOperator{\Spec}{Spec}
\DeclareMathOperator{\sPre}{{\bf sPre}}
\DeclareMathOperator{\sSet}{{\bf sSet}}
\DeclareMathOperator{\sub}{sub}
\DeclareMathOperator{\supp}{supp}
\DeclareMathOperator{\thickid}{thickid}
\DeclareMathOperator{\Top}{top}
\DeclareMathOperator{\Tor}{Tor}
\DeclareMathOperator{\type}{type}
\def\RMod{R\textup{-}\Mod}

\newcommand{\vsupseteq}{\mathrel{\reflectbox{\rotatebox[origin=c]{270}{$\supseteq$}}}}
\newcommand{\veq}{\mathrel{\reflectbox{\rotatebox[origin=c]{270}{$=$}}}}

\numberwithin{equation}{section}
\newtheorem{theorem}[equation]{Theorem}
\newtheorem{lemma}[equation]{Lemma}
\newtheorem{prop}[equation]{Proposition}
\newtheorem{cor}[equation]{Corollary}
\newtheorem{conj}[equation]{Conjecture}
\theoremstyle{definition}
\newtheorem{defn}[equation]{Definition}

\newtheorem{rk}[equation]{Remark}
\newtheorem{notation}[equation]{Notation}
\newtheorem{eg}[equation]{Example}

\begin{document}

\title{Thick ideals in equivariant and motivic \\ stable homotopy categories 
}
\author{Ruth Joachimi}
\maketitle

\pagestyle{fancy}
\lhead[]{}
\rfoot[]{}
\cfoot[\thepage]{\thepage}

\tableofcontents

\theoremstyle{plain}

\chapter{Introduction}\label{introduction}

In a tensor triangulated category, a thick ideal is a full subcategory which is closed under exact triangles and retracts and under tensoring with arbitrary elements of the category. The classification of thick ideals in the stable homotopy category of $p$-local finite spectra, $\SH^{fin}_{(p)}$, is given by a famous theorem of Hopkins and Smith, \cite[Theorem 7]{HS}, see Section \ref{classic}. It states that, in $\SH^{fin}_{(p)}$, the thick ideals are given as a chain 
\[\SH^{fin}_{(p)}= \mathcal C_0\supsetneq \mathcal C_1\supsetneq\cdots \supsetneq \mathcal C_n\supsetneq \cdots\supsetneq \mathcal C_\infty =\{0\},\]
and each thick ideal is characterised by the vanishing of a particular Morava K-theory, that is, $\mathcal C_n=\{X\in\SH^{fin}_{(p)}\; |\; K(p,n-1)_\ast(X)=0\}$ for $0<n<\infty$. This theorem is a consequence of the nilpotence theorem \cite[Theorem 3]{HS}, the existence of type-$n$ spectra for any $n\geq 0$ \cite[Theorem 4.8]{Mi} and the fact that $K(n)_\ast(X)=0$ implies $K(n-1)_\ast(X)=0$ for $X\in\SH^{fin}$ \cite[Theorem 2.11]{RavLoc}.

In equivariant stable homotopy theory for a finite group $G$, unpublished work of Strickland \cite{Sch}, see Chapter \ref{equivariant}, contains a partial classification of thick ideals in the category $\SH(G)_f\subset \SH(G)$, which is the full subcategory of compact objects in the $G$-equivariant stable homotopy category. This is a generalisation of the above result, which concerns the special case $\SH(\{1\})_f =\SH^{fin}$. In $\SH(G)_f$, any thick ideal is characterised by the vanishing of particular equivariant Morava K-theories, which are indexed by a prime and a nonnegative integer (as the ordinary Morava K-theories) and, additionally, by a conjugacy class of subgroups of $G$. The set of thick ideals can be mapped to a lattice of such multi-indices and Strickland proves lower and upper bounds for a sublattice onto which this map is bijective.

In this thesis, we study thick ideals in $\SH(k)$, $k\subseteq\C$, and related motivic categories, like $(\SH(k)_f)_{(p)}$, the $p$-localisation of the full subcategory of all compact objects, and $\SH(k)^{fin}_{(p)}$, the category of $p$-localised finite cell spectra. We use different approaches, all of which are, in some sense, motivated by the results about thick ideals in $\SH^{fin}$.\\

One approach is to use the comparison functors, 
\[\SH\overset{c_k}\longrightarrow \SH(k)\overset{R_k}\longrightarrow \SH\] for $k\subseteq\C$, and also \[\SH(\Z/2)\overset{c'_k}\longrightarrow\SH(k)\overset{R'_k}\longrightarrow \SH(\Z/2)\] for $k\subseteq\R$. We show that, for $k\subseteq\C$, the preimages $R_k^{-1}(\mathcal C_n)\subseteq (\SH(k)_f)_{(p)}$, $n\geq 0$, form a chain of different thick ideals in $(\SH(k)_f)_{(p)}$ (Theorem \ref{lower bound}). For $k\subseteq\R$, we also show that $(R'_k)^{-1}(\mathcal C)\subseteq (\SH(k)_f)_{(p)}$ are different thick ideals, where $\mathcal C$ runs over all thick ideals in $(\SH(\Z/2)_f)_{(p)}$, as studied in \cite{Sch} and in Chapter \ref{equivariant}. \\

The second approach is to use methods of nilpotence theory. The thick subcategory theorem for $\SH$ is highly related to the nilpotence theorem, which states that the cobordism spectrum $MU$ detects certain kinds of nilpotence. In this context, the Morava K-theories recover the information from $MU$, meaning that also the family $\{K(p,n)\;|\;p \textup{ prime}, n\geq 0 \}$ detects nilpotence. Since the Morava K-theories have a particularly easy structure (they are field theories satisfying the K\"unneth formula, see e.g. \cite[page 176]{Rav}), this can be used to show that any thick ideal in $\SH^{fin}$ can be uniquely described in terms of the vanishing and non-vanishing of Morava K-theories. This is how \cite[Theorem 7]{HS} is proven. Similar nilpotence arguments are used in \cite{Sch} to classify thick ideals in $\SH(G)_f$ for finite groups $G$. Strickland shows that the equivariant Morava K-theories detect nilpotence (Theorem \ref{detect nilp}).

The motivic analog to $MU$ is the algebraic cobordism spectrum $MGL$. But here, the situation is different, as $MGL$ does not detect nilpotence. There is a notion of motivic Morava K-theories, which are more complicated than the topological ones, but they, too, describe a certain family of thick ideals. In Chapter \ref{nilpotence}, we show that they do not discover all thick ideals in $(\SH(k)_f)_{(p)}$, and also that not all thick ideals are of the form $R_k^{-1}(\mathcal C_n)$.\\

The third approach is to find different lifts of topological type-$n$ spectra to the motivic world and to ask whether they generate the same thick ideals or different ones. We consider two explicit different such lifts to $(\SH(\C)_f)_{(p)}$ and show that the motivic Morava K-theories do not distinguish the thick ideals generated by them. The question whether the two lifts generate different thick ideals remains open.\\

In $\SH^{fin}_{(p)}$, the thick ideals are ordered linearly by inclusion, due to the fact that $K(n+1)_\ast (X)= 0$ implies $K(n)_\ast (X)=0$ for $X\in\SH^{fin}$. This raises the question whether this implication also holds in $\SH(k)_f$. For $p>2$, we prove that the analog statement holds for motivic Morava K-theories over $\C$ if $X$ is a finite cellular motivic spectrum, as studied in \cite{DI}. That is, it holds for $X\in\SH(\C)^{fin}\subseteq \SH(\C)_f$. On the way, we prove a couple of interesting facts concerning the motivic versions of $BP$, $K(n)$ and related theories. We prove that the analog of the decomposition of Bousfield classes $\langle E(n)\rangle = \underset{i\leq n}\bigvee\langle K(i)\rangle$ holds in $\SH(\C)$ (for $p>2$), as conjectured by Hornbostel in \cite[Question 2.17]{HornLoc} for arbitrary fields.\\

{\bf Outline.}

In {\bf Chapter \ref{classic}}, we introduce basic notation concerning thick ideals and the stable homotopy category $\SH$. We recall the thick subcategory theorem of Hopkins and Smith (Theorem \ref{HS}). In $\SH^{fin}$, there is no difference between thick subcategories and thick ideals (Lemma \ref{thicksubcat thickideal}).\\

{\bf Chapter \ref{equivariant}} is an account on Strickland's work \cite{Sch}. It contains Strickland's main results on thick ideals in $\SH(G)_f$ and their proofs. The chapter begins with the necessary recollection from equivariant stable homotopy theory, such as compact objects in $\SH(G)$ and geometric fixed point functors. Equivariant Morava K-theories 
\[K(n,H)=G/H_+\wedge \tilde E[\not\geq H]\wedge K(n), \textup{ for } H\subseteq G,\] 
are introduced in Section \ref{Equivariant Morava K-theories}. They are related to the classical Morava K-theories via the geometric fixed point functor (Proposition \ref{K(n) and phi}) and satisfy similar properties, such as the K\"unneth formula (Corollary \ref{Kunneth equivariant}). Section \ref{section lattices} introduces the terminology of lattices and contains the result of Strickland which establishes a general relation between thick ideals and the detection of nilpotence by some family of homology theories (Theorem \ref{nilp-idl}). The equivariant analog of the nilpotence theorem \cite[Theorem 3]{HS} is Theorem \ref{detect nilp}. We give a reformulation of the thick subcategory theorem \cite[Theorem 7]{HS} in a non-$p$-localised way (Theorem \ref{thm Q}) and prove a similar equivariant result, Theorem \ref{tau injective}, which describes an injective lattice homomorphism from the set of thick ideals in $\SH(G)_f$ to the lattice 
\[GQ= \underset{\sub(G)}\prod \{u\in\prod_p \mathcal Q_p\;|\; u_p=1\,\forall\, p \textup{ or } u_p\neq 1 \,\forall \,p\},\] 
where $\mathcal Q_p=\{p^{-n}\; |\; 0\leq n\leq\infty\}$, and gives a lower bound for its image: 

\newtheorem*{tau injective}{Theorem \ref{tau injective}}
\begin{tau injective}
\emph{(Strickland)}\\
The composition
\[\tau: \Idl(\SH(G)_f)\overset{\supp}\longrightarrow \mathcal P(GQ')\overset{\max}\longrightarrow GQ\]
is injective. Its image contains all $u\in GQ$ which satisfy: if $H\subseteq H'$, then $u_{H} \geq u_{H'}$.
\end{tau injective}

Here, $\mathcal P(GQ')$ denotes the power set of the set 
\[GQ'=\{p^{-n}\;|\;p\textup{ prime},\, 0\leq n<\infty\}\times \sub(G).\]
An upper bound is given in Proposition \ref{upper bound}. In Section \ref{Application to Z/2}, we apply Strickland's results to $\SH(\Z/2)_f$, which will be most interesting to us in our study of thick ideals in $\SH(k)_f$, $k\subseteq\R$. Any thick ideal in $(\SH(\Z/2)_f)_{(p)}$ is of the form 
\[\mathcal C_{m,n}=\{X\; |\; \phi^{\{1\}}(X)\in\mathcal C_m \textup{ and }\phi^{\Z/2}(X)\in\mathcal C_n\},\] 
where $m,n\in [0,\infty]$ and $\phi^H:\SH(G)\rightarrow \SH$ is the geometric $H$-fixed point functor (Corollary \ref{Cmn}). But not all $\mathcal C_{m,n}$ are different. Corollary \ref{mn} gives partial information on which ones are.\\

In {\bf Chapter \ref{functors}}, we introduce the comparison functors 
$\SH\overset{c_k}\longrightarrow \SH(k)\overset{R_k}\longrightarrow \SH$ for $k\subseteq\C$, and $\SH(\Z/2)\overset{c'_k}\longrightarrow\SH(k)\overset{R'_k}\longrightarrow \SH(\Z/2)$ for $k\subseteq\R$, which are symmetric monoidal and satisfy $R_k\circ c_k\cong\id$ and $R'_k\circ c'_k\cong \id$, respectively. This is mainly a recollection from various other sources. The same results are independently obtained in \cite[Section 4]{HO}.\\

In {\bf Chapter \ref{ideals}}, we apply our knowledge concerning comparison functors to the study of thick ideals, proving the following theorem for any prime $p$.

\newtheorem*{lower bound}{Theorem \ref{lower bound}}
\begin{lower bound}{\bf (Lower bound on the number of motivic thick ideals)}
\begin{compactenum}[(1)]
\item
If $k\subseteq\C$, the category $(\SH(k)_f)_{(p)}$ contains at least an infinite chain of different thick ideals given by $\overline{R}_k^{-1}(\mathcal C_n)$, $0\leq n\leq\infty$, where $\overline{R}_k$ denotes the $p$-localisation of the restriction of $R_k$ to $\SH(k)_f$ and $\mathcal C_n\subseteq\SH^{fin}_{(p)}$ is as defined in Chapter \ref{introduction}.
\item
If $k\subseteq\R$, then $(\SH(k)_f)_{(p)}$ contains at least a two-dimensional lattice of different thick ideals given by $(\overline{R}'_k)^{-1}(\mathcal C_{m,n})$, for all $(m,n)\in\Gamma_p$ as in Definition \ref{equivariant type}. 
\end{compactenum}\end{lower bound}

One ingredient of this theorem is Proposition \ref{compact1}, where we show that the realisation functors $R_k$ and $R'_k$ preserve compactness. In Section \ref{Motivic thick ideals}, we also prove a couple of additional results on the connection between motivic thick ideals and the comparison functors.\\

{\bf Chapter \ref{idealsCohomo}} begins with an account of homology and cohomology theories in the category of finite motivic cell spectra, $\SH(k)^{fin}$, as studied by Dugger and Isaksen in \cite{DI}. We show that, for a cellular ring spectrum $E$ and a finite cellular spectrum $X$, $E_{\ast\ast}X=0$ is equivalent to $E^{\ast\ast}X=0$ (Proposition \ref{co-homo}). For $k\subseteq \R$, we use a notion of cellular spectra which is more general than the notion from \cite{DI}, see Definition \ref{fin}. This yields another version of Proposition \ref{co-homo} (Corollary \ref{co-homo real}). In Section \ref{thick ideals}, we discuss different ways of defining thick ideals associated with a (ring) spectrum. This is applied to motivic Morava K-theories $AK(n)$ in Section \ref{section morava K}. For example, we show that the thick ideal $\mathcal C_{AK(n)}$ associated with the $n$-th motivic Morava K-theory is contained in $R_k^{-1}(\mathcal C_{n+1})$ (Proposition \ref{motivic model}). We recall the definition and some properties of the motivic Morava K-theories in Section \ref{section AK(n)}. The motivic Atiyah Hirzebruch spectral sequence described in \cite[Example 8.13]{Hoy}, implies that the $n$-th motivic Morava K-theory over the field $\C$ has coefficient ring $H\Z/p_{\ast\ast}\otimes K(n)_\ast$ (Lemma \ref{Ah}), as remarked in \cite{Ya} below Corollary 3.9.\\

In {\bf Chapter \ref{nilpotence}}, we study the thick ideal generated by the cofiber of the motivic Hopf map, $C\eta\cong\P_k^2$, and compare it to the thick ideals $R_k^{-1}(\mathcal C_n)$ and $\mathcal C_{AK(n)}$ for $k\subseteq\C$. We calculate the type of $R_k(C\eta_{(p)})\in\SH^{fin}_{(p)}$, which is $1$, and the equivariant type of $R'_k(C\eta_{(p)})\in (\SH(\Z/2)_f)_{(p)}$, which is $(1,2)$ for $p=2$ and $(1,\infty)$ for odd $p$ (Proposition \ref{RCeta}). In Proposition \ref{counter example}, we show that $C\eta_{(p)}$ generates a thick ideal of $(\SH(k)_f)_{(p)}$ which is neither of the form $R_k^{-1}(\mathcal C_{n+1})$ or $\mathcal C_{AK(n)}$ for any $n\geq 0$, nor is it all of $(\SH(k)_f)_{(p)}$ (at least, if $p=2$ or $k\subseteq\R$). 

\newtheorem*{counter example}{Proposition \ref{counter example}}
\begin{counter example}
For $k\subseteq\C$, let $\thickid( C\eta_{(p)})\subseteq(\SH(k)_f)_{(p)}$ denote the thick ideal generated by the p-localised cofiber of the Hopf map. Then the following hold:
\begin{compactenum}[(1)]
\item $\thickid( C\eta_{(p)})\not\subseteq \mathcal C_{AK(n)} $ for any $ n\geq 0$ and any prime $p$,
\item $\thickid( C\eta_{(p)})\not \subseteq R_k^{-1}(\mathcal C_n)$ for any $ n> 0$ and any prime $p$,
\item $\thickid( C\eta_{(p)})\subsetneq \thickid(S^0_{(p)})=(\SH(k)_f)_{(p)}$ if $k\subseteq \R$ and $p$ is any prime or $k\subseteq \C$ and $p=2$.
\item For any prime $p$, the thick ideals $\thickid( C\eta_{(p)})\cap R_k^{-1}(\mathcal C_n)$ are distinct for different $n\geq 0$ and in particular nonzero if $n<\infty$.
\end{compactenum}
\end{counter example}

This proves that $\SH(k)_f$, $k\subseteq\C$, really has ``more'' thick ideals than its topological counterpart. In Section \ref{prime ideals}, we compare our results to Balmer's work on prime ideals \cite{Balmer}. For the categories $\SH^{fin}$, $\SH(\Z/2)_f$, $\SH(\C)_f$ and $\SH(\R)_f$, we recover the information on prime ideals given in \cite[Section 10]{Balmer} from a different point of view.\\

In {\bf Chapter \ref{motivic type n}}, we study two preimages under $R_\C$ of a type-$n$ spectrum $X_n\in\SH^{fin}_{(p)}$. One of them is $c_\C(X_n)$ and the other one, $\X_n$, is constructed by a motivic version of the construction of $X_n$, as given in \cite[Appendix C]{Rav}. In analogy to Mitchell's result \cite[Theorem 4.8]{Mi}, we prove the following vanishing theorem for motivic Morava K-theory.

\newtheorem*{Mitchell}{Theorem \ref{Mitchell}}
\begin{Mitchell}
{\bf (Vanishing criterion)}\\
Let $s>0$ and $X\in\SH(\C)^{fin}$ be a finite motivic cell spectrum such that $\h^{\ast\ast}(X,\Z/p)$ is free over the exterior algebra $\Lambda_{H\Z/p^{\ast\ast}}(Q_s)$ as a module over the motivic Steenrod algebra. Then $AK(s)_{\ast\ast}X=0$.
\end{Mitchell}

This is proven with the help of the motivic Adams spectral sequence for $Ak(s)\wedge X$, where $Ak(s)$ is the motivic analog of the connective Morava K-theory spectrum. This spectral sequence is studied in Section \ref{Motivic Adams spectral sequence}. In Section \ref{construction}, we construct a spectrum $\X_n$ satisfying the assumption of the theorem, and we show that this spectrum is indeed of motivic type $n$ (Theorem \ref{zero}). In Section \ref{constant}, we show that $c_\C(X_n)$ is also of motivic type $n$, proving that the two thick ideals in $(\SH(\C)_f)_{(p)}$ generated by $c_\C(X_n)$ respectively $\X_n$ cannot be distinguished by the motivic Morava K-theories. However, we do not know whether the ideals are actually equal or not.\\

{\bf Chapter \ref{chapter BC}} is devoted to the study of the Bousfield classes of $AK(n)$ and related motivic spectra. The main goal in writing this chapter was to prove that ${AK(n+1)_{\ast\ast}(X)}=0$ implies $AK(n)_{\ast\ast}(X)=0$, which we show for $X\in\SH(\C)^{fin}$ and $p>2$ in Theorem \ref{AK(n+1)}.

\newtheorem*{AK(n+1)}{Theorem \ref{AK(n+1)}}
\begin{AK(n+1)}
Let $p>2$. If $X\in\SH(\C)^{fin}$ satisfies $AK(n+1)_{\ast\ast}(X)=0$, then it also satisfies $AK(n)_{\ast\ast}(X)=0$.
\end{AK(n+1)}

A lot of results in Chapter \ref{chapter BC} hold more generally. In Section \ref{$v_n$-Torsion}, we prove that $v_n$-torsion in $ABP_{\ast\ast}ABP$ is also $v_{n-1}$-torsion. This holds in any $\SH(k)$, $k\subseteq\C$ (Theorem \ref{torsion}). The proof uses methods similar to the topological version of the statement, \cite[Theorem 0.1]{JY}. Another ingredient is the map of Hopf algebroids $(BP_\ast,BP_\ast BP)\rightarrow (ABP_{\ast\ast},ABP_{\ast\ast}ABP)$, as studied in \cite{MLE}. In Section \ref{Section Wurgler}, we construct certain operations on $AP(n)$ in $\SH(k)$, $k\subseteq\C$ (Theorem \ref{5.1}), similar to the operations on $P(n)$ constructed by W\"urgler in \cite[Theorem 5.1]{W}. These are used to prove the equality of Bousfield classes $\langle AK(n)\rangle=\langle AB(n)\rangle$ in $\SH(\C)$ (Corollary \ref{AK=AB}) with methods similar to those of \cite{JW}. In the proof of Corollary \ref{AK=AB}, we assume $k=\C$ because we make use of the explicitly known coefficient rings $H\Z/p_{\ast\ast}$ and $AK(n)_{\ast\ast}$. The result is used to prove Theorem \ref{decomposition}, which is the following decomposition of Bousfield classes in $\SH(\C)$, as conjectured in \cite[Question 2.17]{HornLoc} for arbitrary fields.

\newtheorem*{decomposition}{Theorem \ref{decomposition}}
\begin{decomposition}
For $p>2$,
\[\langle AE(n)\rangle = \underset{i\leq n}\bigvee\langle AK(i)\rangle \textup{ in } \SH(\C).\]
\end{decomposition}

\vspace{120pt}

{\bf Acknowledgements.}

The current preprint is my PhD thesis. I want to thank my advisor Prof. Jens Hornbostel for introducing me to the exciting topic of motivic homotopy theory and for offering me this project. I am grateful for his steady support during my dissertation and for all the helpful discussions.

I thank Neil Strickland for sharing his preprint on thick ideals in equivariant stable homotopy categories with me and for allowing me to incorporate parts of it as Chapter \ref{equivariant} of this thesis.

I also owe great thanks to my colleagues Marcus Zibrowius and Jeremiah Heller, who carefully read earlier versions of this work and helped me a lot with their comments. Furthermore, I thank Marc Hoyois for answering some of my questions and I thank Kyle Ormsby for pointing out the reference \cite{Balmer} to me. 

This work was financed by the Deutsche Forschungsgemeinschaft, grant no. HO 4729/1-1.

\chapter{Thick ideals in classical stable homotopy theory}\label{classic}

In this chapter, we introduce basic notation concerning thick ideals and the stable homotopy category $\SH$. In $\SH^{fin}$, a thick ideal is the same as a thick subcategory (Lemma \ref{thicksubcat thickideal}). We recall the thick subcategory theorem of Hopkins and Smith in Theorem \ref{HS}.

\begin{defn}
A tensor triangulated category is a triple $(\mathcal T, \wedge, S)$ consisting of a triangulated category $\mathcal T$ and a symmetric monoidal product $\wedge$ on $\mathcal T$ with unit $S$, such that for any $A\in\mathcal T$, $A\wedge -$ preserves exact triangles (see e.g. \cite[Definition 1.1]{BalmerPI}).
\end{defn}

\begin{eg}
The stable homotopy category $(\SH, \wedge, S)$ with $S=S^0=\Sigma^\infty S^0$ is a tensor triangulated category.
\end{eg}

\begin{defn}\label{thick ideal}
Let $(\mathcal T,\wedge, S)$ be a tensor triangulated category. A full triangulated subcategory $\emptyset\neq\mathcal C\subseteq\mathcal T$ is called a
\begin{compactenum}[(1)]
\item thick subcategory if it is closed under retracts.
\item thick ideal if it is a thick subcategory and in addition satisfies:
\[\textup{if }X\in\mathcal T \textup{ and } Y\in \mathcal C \textup{ then } X\wedge Y\in \mathcal C.\]
\end{compactenum}
If $\mathcal X$ is an object or a set of objects, we denote the smallest thick ideal containing $\mathcal X$ by $\thickid(\mathcal X)$ and call it the thick ideal generated by $\mathcal X$. This is well defined because the intersection of thick ideals is again a thick ideal. If $\mathcal X$ is a finite set of objects, $\thickid (\mathcal X)$ is called finitely generated.
\end{defn}

\begin{rk}\label{FIdl}
Any finitely generated thick ideal is generated by a single element, namely the direct sum of all generators.
\end{rk}

\begin{eg}
If $(\mathcal T,\wedge, S)$ is a tensor triangulated category, then \[\thickid(S)=\mathcal T,\] since for any $X\in\mathcal T$, $X\cong X\wedge S$. 

More generally, if $Z$ is in the Picard group $Pic(\mathcal T)$, i.e., if there exists a $Z'$ such that $Z'\wedge Z\cong S$, then $\thickid (Z)=\mathcal T$. The Picard group of the stable homotopy category $\SH$ consists precisely of the spheres $\Sigma^n S^0$, $n\in\Z$ (see e.g. \cite{HMS}), the Picard group of the equivariant stable homotopy category $\SH(G)$ is described in \cite{FLM} and examples for elements in the Picard groups of motivic stable homotopy categories are given in \cite{HuPic}.
\end{eg}

\begin{defn}
The category $\SH^{fin}$ is the smallest full subcategory of $\SH$ that contains all finite desuspensions of suspension spectra of finite CW complexes and is closed under isomorphisms.
\end{defn}

\begin{rk}\label{SH fin}
$\SH^{fin}$ is a tensor triangulated subcategory of $\SH$. It can equivalently be defined as the smallest thick subcategory of $\SH$ that contains $S^0$, or as the full subcategory of compact objects in $\SH$ (see, e.g. \cite[Theorem II.7.4]{SymSpec}).
\end{rk}

\begin{lemma}\label{thicksubcat thickideal}
In $\SH^{fin}$, any thick subcategory is already a thick ideal.
\end{lemma}

\begin{proof}
Let $X$ be an element of the thick subcategory $\mathcal C\subseteq\SH^{fin}$ and let $Y\in\SH^{fin}$. By the definition of $\SH^{fin}$, there is a finite sequence of spectra $\{Y^k\}_{0\leq k\leq n}$ such that $Y^0=S^{n_0}$, $Y\cong Y^n$ and each $Y^k$ is the cofiber of some map $S^{n_k}\rightarrow Y^{k-1}$, $n_k\in\Z$. Any thick subcategory is closed under suspensions and desuspensions because $\Sigma^{\pm 1}X$ lies in an exact triangle with $X\overset{1}\rightarrow X$. Hence, $X\wedge Y^0\in\mathcal C$. Assume that $X\wedge Y^k\in\mathcal C$ for some $k$. Then $X\wedge S^{n_{k+1}}\rightarrow X\wedge Y^k\rightarrow X\wedge Y^{k+1}$ is an exact triangle whose first two objects are in $\mathcal C$. Since $\mathcal C$ is a thick subcategory, it follows that $X\wedge Y^{k+1}\in\mathcal C$, too, and inductively, $X\wedge Y^n\in\mathcal C$. Note further that thick subcategories are closed under isomorphisms as these are special cases of retractions. Hence, $X\wedge Y\in\mathcal C$, which proves that $\mathcal C$ is a thick ideal. 
\end{proof}

\begin{defn}
Let $p$ be a prime number. The $p$-local categories $\SH_{(p)}$ and $\SH^{fin}_{(p)}$ are defined as the Bousfield localisations of $\SH$ and $\SH^{fin}$ at the $p$-local Moore spectrum $M\Z_{(p)}$.
\end{defn}

It is a common procedure to study spectra $p$-locally for each prime $p$, i.e. instead of $X\in\SH$ one studies its image $X_{(p)}$ under $\SH\rightarrow\SH_{(p)}$, and then fits the information together. For example, $n$-th Morava K-theory $K(n)$ is defined for any fixed prime $p$, where it satisfies $K(n)_\ast(X)=K(n)_\ast(X_{(p)})$. For the construction and properties of $K(n)$, see e.g. \cite{JW}.\\

Now we are ready to state the thick subcategory theorem of Hopkins and Smith \cite[Theorem 7]{HS}, which was the main motivation for this thesis. It gives a beautiful and complete description of the thick ideals in $\SH^{fin}_{(p)}$ in terms of Morava K-theories. 

\begin{theorem}\emph{(Hopkins, Smith)}\label{HS}

In $\SH^{fin}_{(p)}$, the thick ideals are given as a chain 
\[\SH^{fin}_{(p)}= \mathcal C_0\supsetneq \mathcal C_1\supsetneq\cdots \supsetneq \mathcal C_n\supsetneq \cdots\supsetneq \mathcal C_\infty =\{0\},\]
with $\mathcal C_n=\{X\in\SH^{fin}_{(p)}\; |\; K(n-1)_\ast(X)=0\}$ for $0<n<\infty$. 
\end{theorem}

\begin{defn}\label{type}
A spectrum $X\in\SH^{fin}_{(p)}$ is said to be of type $n$ if $K(n-1)_\ast(X)=0$ and $K(n)_\ast(X)\neq 0$. We write $\type(X)=n$.
\end{defn}

For any fixed prime $p$, the type of a spectrum is well-defined by \cite[Theorem 2.11]{RavLoc}.
The thick subcategory theorem implies that any spectrum $X$ of type $n$ generates $\mathcal C_{n}$ as a thick ideal.

\chapter{Thick ideals in equivariant stable homotopy theory}\label{equivariant}

The contents of this chapter (except for the introductory section and some details) are due to Neil Strickland \cite{Sch}.
We state Strickland's main results on thick ideals in $\SH(G)_f$ and their proofs. We start with the necessary recollection from equivariant stable homotopy theory. Thick ideals in $\SH(G)_f$ are classified by equivariant Morava K-theories, $K(n,H)=G/H_+\wedge \tilde E[\not\geq H]\wedge K(n)$, $H\subseteq G$, which are introduced in Section \ref{Equivariant Morava K-theories}. They are related to the classical Morava K-theories via the geometric fixed point functor (Proposition \ref{K(n) and phi}) and satisfy similar properties, such as the K\"unneth formula (Corollary \ref{Kunneth equivariant}). As in the non-equivariant theory of Hopkins and Smith, the equivariant Morava K-theories detect nilpotence in $\SH(G)_f$ (Theorem \ref{detect nilp}). The general relation between the detection of nilpotence by a family of homology theories and thick ideals is described in Theorem \ref{nilp-idl}. As a corollary, we reformulate the thick subcategory theorem \cite[Theorem 7]{HS} in a non-$p$-localised way (Theorem \ref{thm Q}) and prove a similar equivariant result, Theorem \ref{tau injective}, which describes an injective lattice homomorphism from the set of thick ideals in $\SH(G)_f$ to a particular lattice $G\mathcal Q$ and gives a lower bound for its image. An upper bound is given in Proposition \ref{upper bound}. 

For our study of thick ideals in the motivic stable homotopy categories $\SH(k)_f$, $k\subseteq \R$, we will use the here given knowledge concerning thick ideals in the $\Z/2$-equivariant stable homotopy category. Therefore, the case $G=\Z/2$ is the interesting one for the rest of this thesis and we will summarise all results on thick ideals in $\SH(\Z/2)_f$ in Section \ref{Application to Z/2}. Any thick ideal in $(\SH(\Z/2)_f)_{(p)}$ is of the form $\mathcal C_{m,n}=\{X\; |\; \phi^{\{1\}}(X)\in\mathcal C_m \textup{ and }\phi^{\Z/2}(X)\in\mathcal C_n\}$, where $m,n\in [0,\infty]$ and $\phi^H$ is the geometric fixed point functor (Corollary \ref{Cmn}). But not all $\mathcal C_{m,n}$ are different. Corollary \ref{mn} gives partial information on which ones are.

\section{Equivariant stable homotopy theory}

Let $G$ be a finite group and $\SH(G)$ be the stable homotopy category of genuine $G$-spectra. 
This category has quite a few models. We switch between spectra of $G$-CW complexes and spectra of $G$-simplicial sets, depending on which is more convenient in the concrete situation. A good model for $\SH(G)$ as a tensor triangulated category is the category of orthogonal $G$-spectra, see e.g. \cite{EquivOrth} or \cite{Schwede}. In Section \ref{Z2 model}, we make use of two other models, namely symmetric $G$-spectra and $G\Sigma_G$-spectra. The following definition of finite $G$-spectra, for example, makes sense if we use the model of orthogonal $G$-spectra with $G$-CW complexes as the underlying category of spaces. The definition induces a notion of finiteness for any other model for $\SH(G)$.

\begin{defn}
For $G$ a finite group, let $\SH(G)^{fin}$ be the smallest full subcategory of $\SH(G)$ that contains all finite desuspensions 
of suspension spectra of finite $G$-CW complexes and is closed under isomorphisms. The objects of $\SH(G)^{fin}$ are called finite $G$-CW spectra. We denote the closure of $\SH(G)^{fin}$ under retracts in $\SH(G)$ by $\SH(G)_f$.
\end{defn}

Both $\SH(G)^{fin}$ and $\SH(G)_f$ are tensor triangulated subcategorties of $\SH(G)$ because finite $G$-CW complexes are closed under cofiber sequences and under smash products and because retracts commute with smash products.

\begin{defn}\label{dualizable}
A spectrum $X\in\SH(G)$ is called dualisable, if the canonical map 
\[F(X,S^0)\wedge Y\rightarrow F(X,Y)\]
is an isomorphism for any $Y\in \SH(G)$, where $F(-,-)$ denotes the derived function spectrum and $S^0=S^0_G$ is the shere spectrum in $\SH(G)$ (for possible definitions of $S^0$ and $F(-,-)$, see e.g. \cite[Examples 2.10 and 5.12]{Schwede}).
$DX=F(X,S^0)$ is called the dual of $X$. It satisfies $DDX\cong X$ (\cite[Proposition III.1.3]{LMS}). In the following, we also use the notation $S$ for $S^0_G$, since it is the unit in $\SH(G)$. 

$X \in \SH(G)$ is called compact if $[X,-]_{\SH(G)}=\Hom_{\SH(G)}(X,-)$ preserves arbitrary coproducts.
\end{defn}

\begin{prop}\label{equiv compact}
The subcategory $\SH(G)_f\subseteq \SH(G)$ has the following equivalent descriptions:
\begin{compactenum}[(1)]
\item It is the full subcategory of retracts of finite $G$-CW spectra.
\item It is the full subcategory of dualisable objects.
\item It is the full subcategory of compact objects.
\end{compactenum}
\end{prop}

\begin{proof}
(1) and (2) are equivalent by \cite[Theorem XVI.7.4]{MayEquiv}. 

Furthermore, any dualisable object $X$ is also compact, because 
\[ \left[X,\bigvee Y_i\right]=\left[S,F(X,\bigvee Y_i)\right]=\left[S,DX\wedge\bigvee Y_i\right]=\left[S,\bigvee (DX\wedge Y_i)\right]\]
\[= \bigoplus\left[S,DX\wedge Y_i\right]= \bigoplus\left[S,F(X, Y_i)\right]= \bigoplus\left[X,Y_i\right],\]
where we used that $F(X,-)$ is right adjoint to $X\wedge -$ and that the unit $S$ is compact, see e.g. \cite[Corollary 3.30(i)]{Schwede}.
Since $\pi_n^H X=0$ for all $H\subseteq G$ and $n\in\Z$ implies $X\cong 0$ in $\SH(G)$ by the definition of $\SH(G)$, $\{\Sigma^{n}\Sigma^\infty (G/H)_+\;|\; H\subseteq G, n\in\Z\}$ is a detecting set and, hence, also a generating set by \cite[Lemma 13.1.6]{MS}. That is, the smallest thick subcategory of $\SH(G)$ which is closed under infinite coproducts and contains $\{\Sigma^{n}\Sigma^\infty (G/H)_+\;|\; H\subseteq G, n\in\Z\}$ is $\SH(G)$ itself. By \cite[Corollary II.6.3]{LMS}, $\Sigma^\infty(G/H)_+$ is dualisable and, thus, compact. Hence, $\{\Sigma^\infty (G/H)_+\;|\; H\subseteq G\}$ is a set of compact generators for $\SH(G)$. 
By general theory due to Neeman \cite{Neeman}, see e.g. \cite[Theorem 13.1.14]{MS}, the full subcategory of compact objects in $\SH(G)$ is the thick subcategory generated by this set (i.e., the smallest thick subcategory of $\SH(G)$ containing this set). Therefore, (3) is also equivalent to (1) and (2). 
\end{proof}

Let $G$ be a finite group and $H\subseteq G$ a subgroup. There are functors $i:\SH\rightarrow \SH(G)$ and 
$\phi^H:\SH(G)\rightarrow \SH$, where $i$ maps a nonequivariant spectrum to the corresponding $G$-spectrum with trivial $G$-action and $\phi^H$ is the geometric fixed point functor (as defined in \cite[Definition 9.7]{LMS}, \cite[Definition 4.3]{EquivOrth} or \cite[Section 7.3]{Schwede}) concatenated with the forgetful functor from $\SH(W(H))$ to $\SH$, where $W(H)$ denotes the Weyl group of $H\subseteq G$. We will need the following properties \cite[Proposition 12.1 and Theorem 12.4]{Sch}.

\begin{prop}\label{phi}
The geometric fixed point functor $\phi^H$ has the following properties:
\begin{compactenum}[(1)]
\item In $\SH$, $\phi^H(\Sigma^\infty X)= \Sigma^\infty X^H$ for any suspension spectrum $\Sigma^\infty X\in\SH(G)$.
\item In $\SH$, $\phi^H(X\wedge Y)=\phi^H(X)\wedge \phi^H(Y)$ for any $X,Y\in\SH(G)$.
\item In $\SH$, $\phi^H(i(X))=X$ for any $X\in\SH$.
\item If $\phi^H(X)=0$ in $\SH$ for all $H\subseteq G$, then $X=0$ in $\SH(G)$.
\end{compactenum}
\end{prop}

\begin{proof}
A proof of (1) can be found in \cite[Corollary II.9.9]{LMS}, \cite[Corollary 4.6]{EquivOrth}, or in \cite[Example 7.7]{Schwede}. (2) follows from \cite[Theorem II.9.8(ii)]{LMS} and \cite[Proposition II.9.12(ii)]{LMS}. (3) follows directly from the definition of $\phi^H$, since $H$ acts trivially on $i(X)$. (4) is proven in \cite[Theorem 7.12]{Schwede} and in \cite[Theorem 12.4]{Sch}.
\end{proof}

\section{Equivariant Morava K-theories}\label{Equivariant Morava K-theories}

For $H\subseteq G$, $\tilde E[\ngeq H]$ denotes a $G$-space which satisfies:
\[\tilde E[\ngeq H]^{K}\simeq \begin{cases} 0 \textup{ \, if } K\ngeq_G H\\ S^0 \textup{ \, if } K\geq_G H\end{cases},\]
where $K\geq_G H$ means that $K$ contains a subgroup conjugate to $H$.

The existance of such a space $\tilde E[\ngeq H]$ follows from the theory of classifying spaces for families (see e.g. \cite[Section II.2]{LMS}), if one takes $\mathcal F$ as the family of all subgroups of $G$ for which $H$ is not subconjugate and then defines $\tilde E[\ngeq H]$ by the cofiber sequence
\[E\mathcal F_+\rightarrow S^0\rightarrow \tilde E[\ngeq H],\]
as in \cite[Notations 4.14]{EquivOrth}.

Fix a prime number $p$. Strickland \cite[Definition 16.2]{Sch} defines Morava K-theory spectra in $\SH(G)$, $G$ a finite group, by 
\[K(n,H)=G/H_+\wedge \tilde E[\ngeq H]\wedge K(n)\] 
for any subgroup $H\subseteq G$. He notes that, as a localisation of $S^0$, $\tilde E[\ngeq H]$ is a commutative ring spectrum, which together with the ring structure of $K(n)$ and the diagonal map on $G/H$ induces a ring structure on $K(n,H)$, which is commutative for $p>2$. We will only be interested in $H$ up to conjugacy, because if $H$ and $H'$ are conjugate, then $K(n,H)\cong K(n,H')$.

The following proposition serves as motivation for this definition of equivariant Morava K-theories \cite[Remark 16.4 ff]{Sch}.

\begin{prop}\label{K(n) and phi}
\[K(n,H)_\ast(X)=K(n)_\ast(\phi^H(X))\]
and
\[K(n,H)^\ast(X)=K(n)^\ast(\phi^H(X)).\]
\end{prop}

\begin{proof}
Here, we need the following formula for the geometric fixed point spectrum \cite[Theorem II.9.8(ii)]{LMS}:
\[\phi^H(X)\cong (\tilde E[\ngeq H]\wedge X)^H,\]
where $(-)^H$ is the spectrification of the levelwise fixed point functor. In the following, we abbreviate $\tilde E[\ngeq H]$ by $\tilde E$. The first equation follows from
\[K(n,H)_\ast(X)= \pi_\ast\left( (X\wedge G/H_+\wedge \tilde{E}\wedge K(n))^G\right) \]\[
=  \pi_\ast\left( (X\wedge \tilde{E}\wedge K(n))^H\right) =  \pi_\ast\left( \phi^H(X\wedge K(n))\right) \]\[
= \pi_\ast\left( \phi^H(X)\wedge K(n)\right)=K(n)_\ast(\phi^H(X)).\]
For the second equation, we use the fact that $G/H_+$ is self-dual \cite[Corollary II.6.3]{LMS}, hence
\[K(n,H)^\ast(X)= [X,F(G/H_+,\tilde{E}\wedge K(n))]^G_\ast = [X,\tilde{E}\wedge K(n)]^H_\ast.\]
We claim that this is isomorphic to $[\phi^H X,\phi^H K(n)]_\ast = K(n)^\ast(\phi^H X)$. To prove the claim, first note that because $\tilde{E}\wedge -$ is a Bousfield localisation functor, 
\[[X,\tilde{E}\wedge K(n)]^H_\ast = [\tilde{E}\wedge X,\tilde{E}\wedge K(n)]^H_\ast.\] 
From here, the $H$-fixed points yield a map 
\[\alpha:[\tilde{E}\wedge X,\tilde{E}\wedge K(n)]^H_\ast \rightarrow [(\tilde{E}\wedge X)^H,(\tilde{E}\wedge K(n))^H]_\ast\] 
and the latter group is isomorphic to $[\phi^H X,\phi^H K(n)]_\ast$.
Assume that $X^H$ is an orbit $H/K_+$ for $K\subseteq H$. If $K\neq H$, $\alpha:0\rightarrow 0$ is an isomorphism. If $K=H$, $[\tilde{E}\wedge X,\tilde{E}\wedge K(n)]^H_\ast=[S^0_G,\tilde{E}\wedge K(n)]^H_\ast =[S^0,\phi^H (K(n))]_\ast =[\phi^H X,\phi^H (K(n))]_\ast$. That is, $\alpha$ is an isomorphism on all orbit types and it follows that $\alpha$ is an isomorphism for any $X\in\SH(G)$.
\end{proof}

From this and the properties of nonequivariant Morava K-theories (see e.g. \cite[Section 1]{HS}), it follows immediately that $K(n,H)$ has coefficients like $K(n)$ and satisfies the K\"unneth formula \cite[Section 16]{Sch}. 

\begin{cor}\label{Kunneth equivariant}
The equivariant Morava K-theories satisfy the following properties:
\begin{compactenum}[(1)]
\item $K(n,H)_\ast S^0_G = K(n)_\ast S^0 =\F_p[v_n^{\pm 1}]$ for any $n>0$ and any $H\subseteq G$.
\item $K(n,H)_\ast(X\wedge Y)\cong K(n,H)_\ast (X)\otimes_{K(n)_\ast} K(n,H)_\ast(Y)$ for any $X,Y\in\SH(G)$.
\item If $X$ is dualisable, i.e., if $X\in\SH(G)_f$, then 
\[K(n,H)_\ast(DX)\cong \Hom_{K(n)_\ast}(K(n,H)_\ast(X),K(n)_\ast).\]
\end{compactenum}
\end{cor}

Furthermore, Strickland shows \cite[Proposition 16.6]{Sch}:

\begin{prop}
If $p\neq p'$ or $n\neq n'$ or $H\neq_G H'$ (i.e., not conjugate in $G$), then 
\[K(p,n,H)\wedge K(p',n',H')= 0.\]
\end{prop}

\begin{proof}
In the cases $p\neq p'$ and $n\neq n'$, this follows from $K(p,n)\wedge K(p',n')= 0$ (see \cite[Theorem 2.1(i)]{RavLoc}),
 as these appear as smash factors in $K(p,n,H)\wedge K(p',n',H')$. Therefore, assume $p= p'$, $n= n'$ and $H\neq_G H'$. Now it suffices to show $G/H_+\wedge \tilde E[\not\geq H]\wedge G/H'_+\wedge \tilde E[\not\geq H']= 0$, which is easily checked on the level of $K$-fixed points for all $K\subseteq G$.
\end{proof}

\section{Nilpotence and lattices of thick ideals}\label{section lattices}

For a convenient description of the collection of thick ideals in the equivariant homotopy category of dualisable spectra, $SH(G)_f$, \cite{Sch} uses the language of lattices.

\begin{defn}\label{lattice}
A lattice is a partially ordered set $A$ for which any finite subset $F\subseteq A$ has a greatest lower bound (called meet) $\bigwedge F$ and a smallest upper bound (called join) $\bigvee F$. The largest element in $A$ is $\bigwedge\emptyset$, which we denote by $1$ and the smallest element is $0=\bigvee\emptyset$. A lattice homomorphism is an order preserving map $f:A\rightarrow B$ which also preserves all joins and meets.
\end{defn}

\begin{eg}\label{Idl}
\begin{compactenum}[(1)]
\item The collection of thick ideals $\mathcal C$ in a tensor triangulated category $\mathcal T$, partially ordered by inclusion, is a lattice. Meets are just intersections, whereas the join of a finite collection of thick ideals is the smallest thick ideal which contains all objects of the different ideals. We denote this lattice by $\Idl(\mathcal T)$.
\item The power set of any set is a lattice, partially ordered by inclusion, meets given by intersections and joins by unions.
\end{compactenum}
\end{eg}

We introduce a new notation, due to Strickland, which will be useful in the rest of this chapter.

\begin{notation}
For a prime $p$ and a nonnegative integer $n$, let $K(p^{-n})$ denote the $n$-th Morava K-theory spectrum at the prime $p$ (which above was denoted by $K(p,n)$ or just $K(n)$). 
\end{notation}

One advantage of this notation is that there is only one name for the zeroth Morava K-theory spectrum (which is independent of $p$): $K(1)=H\Q$.

\begin{defn}\label{Q}
Let \[\mathcal Q_p=\{p^{-n}\; |\; 0\leq n\leq\infty\}\]
and
\[\mathcal Q=\{u\in\prod_p \mathcal Q_p\;|\; u_p=1\,\forall\, p \textup{ or } u_p\neq 1 \,\forall \,p\}.\]
\end{defn}

The sets $\mathcal Q_p$ and $\mathcal Q$ are lattices, with the usual ordering of rational numbers and the componentwise partial ordering of products. We immediately see that Theorem \ref{HS} can be reformulated as follows.

\begin{cor}
Let $\Idl(\SH^{fin}_{(p)})$ be as in Example \ref{Idl}(1). The map
\[\tau_p: \Idl(\SH^{fin}_{(p)})\longrightarrow \mathcal Q_p,\]
\[\tau_p(\mathcal C)= \max\{p^{-n}\; |\;\type(X)=n \textup{ for some } X\in\mathcal C\},\]
is a lattice isomorphism.
\end{cor}

We will see in Theorem \ref{thm Q} how to merge the information for different $p$ to a classification of finitely generated thick ideals of $\SH^{fin}$ using the lattice $\mathcal Q$. But before we are able to do so, we need some more theory on thick ideals and lattice homomorphisms.

\begin{defn}
For $X\in\SH(G)_f$, let $\ann(X)$ denote the fibre of the unit map $S\rightarrow F(X,X)$ and define 
\[\mathcal A_{\ann(X)}=\{A\;|\; \ann(X)^{\wedge n}\wedge A\rightarrow A \textup{ is null for some } n>0\}.\]
The map here is the $n$-fold smash product of the map $\ann(X)\rightarrow S$ from the cofiber sequence, smashed with $A\overset{1}\rightarrow A$.
\end{defn}

The following is \cite[Proposition 15.6]{Sch}.

\begin{prop}\label{ann}
The smallest thick ideal containing $X$ is 
\[\thickid(X)=\mathcal A_{\ann(X)}\]
and $\thickid(X)\subseteq\thickid(Y)$ if and only if the map $\ann(Y)^{\wedge n}\rightarrow S$ factors through $\ann(X)\rightarrow S$ for some $n>0$.
\end{prop}

\begin{proof}
We first show that $\thickid(X)\subseteq\mathcal A_{\ann(X)}$. Since $X$ is a module over $F(X,X)=DX\wedge X$, it is a retract of $DX\wedge X\wedge X$. It follows that $\ann(X)\wedge X$ is the fiber of a map $X\rightarrow DX\wedge X\wedge X$ which splits, so $\ann(X)\wedge X\rightarrow X$ is zero and hence $X\in \mathcal A_{\ann(X)}$. It is easy to see that $\mathcal A_{\ann(X)}$ is closed under exact triangles and retracts, as well as under smashing with arbitrary objects. Hence, $\mathcal A_{\ann(X)}$ is a thick ideal containing $X$, which proves $\thickid(X)\subseteq\mathcal A_{\ann(X)}$.

Now assume $A\in \mathcal A_{\ann(X)}$. We need to show $A\in\thickid(X)$. Consider the cofiber sequence
\[\ann(X)^{\wedge n}\wedge A\rightarrow A\rightarrow S/(\ann(X)^{\wedge n})\wedge A.\]
By the assumption, we can choose $n$ such that the first map is zero. Then it follows that $A$ is a retract of $S/(\ann(X)^{\wedge n})\wedge A $. Therefore, it suffices to show $S/(\ann(X)^{\wedge n})\in\thickid(X)$. By the definition of $\ann(X)$, $S/\ann(X)=F(X,X)=DX\wedge X\in\thickid(X)$. There is a cofiber sequence
\[\ann(X)\wedge S/(\ann(X)^{\wedge j})\rightarrow S/(\ann(X)^{\wedge j+1})\rightarrow S/(\ann(X)^{\wedge j}),\]
which implies inductively that $S/(\ann(X)^{\wedge n})\in\thickid(X)$. 

The existence of this cofiber sequence follows from Verdier's axiom for triangulated categories, also known as octahedral axiom (see e.g. \cite[Proposition 1.4.6]{TriCat}). Applied to the three cofiber sequences $I\wedge J\rightarrow I\wedge S\rightarrow I\wedge S/J$, as well as $I\wedge J\rightarrow S\rightarrow S/(I\wedge J)$ and $I\rightarrow S\rightarrow S/I$, the axiom yields a cofiber sequence $I\wedge S/J\rightarrow S/(I\wedge J)\rightarrow S/I$.

For the second claim, assume that $\thickid(X)\subseteq\thickid(Y)$, which is equivalent to $X\in\thickid (Y)=\mathcal A_{\ann(Y)}$. Let $n>0$ be such that $\ann(Y)^{\wedge n}\wedge X\rightarrow X$ is the zero map. Consider the two cofiber sequences
\[\xymatrix{\ar @{} [dr] |{}
\ann(Y)^{\wedge n} \ar[r] & S \ar[r]\ar@{=}[d] & S/(\ann(Y)^{\wedge n})\\
\ann(X) \ar[r] & S \ar[r] & F(X,X).}\]
The smash product of the upper sequence with $X$ is 
\[\ann(Y)^{\wedge n}\wedge X\overset{0}\rightarrow X\rightarrow S/(\ann(Y)^{\wedge n})\wedge X,\] 
so there is a retraction $S/(\ann(Y)^{\wedge n})\wedge X\rightarrow X$, which then induces a morphism ${S/(\ann(Y)^{\wedge n})\rightarrow F(X,X)}$ making the diagram commutative. From the axioms for triangulated categories it follows that we can fill in the required map $\ann(Y)^{\wedge n}\rightarrow \ann(X)$, as claimed.

On the other hand, if $\ann(Y)^{\wedge n}\rightarrow S$ factors through $\ann(X)\rightarrow S$ and $\ann(X)^{\wedge m}\wedge A\rightarrow A$ is zero then also $\ann(Y)^{\wedge (nm)}\wedge A\rightarrow A$ is zero and it follows $\mathcal A_{\ann(X)}\subseteq\mathcal A_{\ann(Y)}$.
\end{proof}

\begin{prop}\label{intersection}
Let $\mathcal I_0$ and $\mathcal J_0$ be collections of objects in $\SH(G)_f$ and let $\mathcal I$ and $\mathcal J$ be the thick ideals which they generate. Then
\[\mathcal I\cap \mathcal J =\thickid\left(\{Y\wedge Z\;|\; Y\in \mathcal I_0,\, Z\in \mathcal J_0\}\right).\]
\end{prop}

\begin{proof}
This is \cite[Proposition 15.8]{Sch}. Let $\mathcal K =\thickid\left(\{Y\wedge Z\;|\; Y\in \mathcal I_0,\, Z\in \mathcal J_0\}\right)$. The intersection $\mathcal I\cap \mathcal J $ is a thick ideal which contains $Y\wedge Z$ for all $Y\in \mathcal I_0$ and $Z\in \mathcal J_0$. Therefore, $\mathcal K\subseteq \mathcal I\cap \mathcal J$. Now, let
\[\mathcal I'=\{Y\in\mathcal I\;|\; Y\wedge Z \in\mathcal K\;\forall\, Z\in \mathcal J_0\},\]
\[\mathcal J'=\{Z\in\mathcal J\;|\; Y\wedge Z \in\mathcal K\;\forall\, Y\in \mathcal I\}.\]
It is easy to check that $\mathcal I'$ is a thick subideal of $\mathcal I$ which contains $\mathcal I_0$. Hence, $\mathcal I'=\mathcal I$. It follows that the thick subideal $\mathcal J'\subseteq \mathcal J$ contains $\mathcal J_0$, so $\mathcal J'=\mathcal J$. That is, $Y\wedge Z\in\mathcal K$ for all $Y\in\mathcal I$, $Z\in\mathcal J$. 

Now, let $X\in\mathcal I\cap \mathcal J$. Since $X$ is an $F(X,X)$-module, $X$ is a retract of $DX\wedge X\wedge X$. Consider $DX\wedge X$ as an object of $\mathcal I$ and the other $X$ as an object of $\mathcal J$. It follows $DX\wedge X\wedge X\in\mathcal K$ and hence $X\in\mathcal K$.
\end{proof}

\begin{rk}
Note that Proposition \ref{intersection} holds in any tensor triangulated category in which internal hom objects exist and all objects are dualisable. The following definition, theorem and corollary can also be formulated in such a general setting, given a suitable notion of homology theories.
\end{rk}

\begin{defn}
Let $\{E_i\; |\; i\in I\}$ be a family of ring spectra in $\SH(G)$. For $X\in\SH(G)_f$, define 
\[\supp(X)=\{i\in I\;|\; (E_i)_\ast(X)\neq 0\}.\]
If $\mathcal C\subseteq \SH(G)_f$ is a subcategory, let
\[\supp(\mathcal C)=\bigcup_{X\in\mathcal C}\supp(X).\]
For a map $f:X\rightarrow Y$, we also define
\[\supp(f)=\{i\in I\;|\; (E_i)_\ast(f)\neq 0\}.\]
\end{defn}

\begin{rk}\label{supp thickid}
If $(E_i)_\ast(X)=0$, then $(E_i)_\ast(Y)=0$ for any $Y\in\thickid(X)$ by the following arguments. If $A\rightarrow B\rightarrow C$ is a cofiber sequence and the $(E_i)_\ast$-homology of two of the three objects is zero, then, by the long exact $(E_i)_\ast$-sequence, $(E_i)_\ast(-)$ of the third object is zero, too. If $(E_i)_\ast(A)=\pi_\ast(E_i\wedge A)=0$, then also $(E_i)_\ast(A\wedge B)=\pi_\ast(E_i\wedge A\wedge B)=0$. And, finally, if $(E_i)_\ast(A)=0$ and $B$ is a retract of $A$, then $(E_i)_\ast(B)\rightarrow 0\rightarrow (E_i)_\ast(B)$ is the identity map, hence, $(E_i)_\ast(B)=0$.

This implies  
\[\supp(\thickid(X))=\supp(X).\]
\end{rk}

The following theorem is one of the central results in \cite{Sch}, where it is Theorem 15.14.

\begin{theorem}\label{nilp-idl}\emph{(Strickland)}
Assume $\{E_i\; |\; i\in I\}$ is a family of ring spectra in $\SH(G)$ satisfying the following properties:
\begin{compactenum}[(1)]
\item If $f:X\rightarrow Y$, with $X,Y\in\SH(G)_f$, and $(E_i)_\ast(f)=0$ for all $i\in I$, then there exists $n>0$ such that $f^{\wedge n}=0$.
\item For any $X, Y\in\SH(G)_f$ and any $i\in I$, $(E_i)_\ast(X\wedge Y)\cong (E_i)_\ast(X)\otimes_{(E_i)_\ast}(E_i)_\ast(Y)$.
\item For any $i\in I$, $(E_i)_\ast=(E_i)_\ast(S^0)$ is concentrated in even degrees and any nonzero homogeneous element in $(E_i)_\ast$ is invertible.
\end{compactenum}
Then, for any $X,Y\in\SH(G)_f$, $\thickid (X)\subseteq \thickid(Y)$ if and only if $\supp(X)\subseteq\supp(Y)$.
\end{theorem}

In other words, a family $\{E_i\; |\; i\in I\}$ of spectra detecting nilpotence (see Definition \ref{def detect nilp}) and satisfying some additional properties can be used to distinguish any two different thick ideals with the help of the support functor $\supp(-)$.

\begin{proof}
Consider the cofiber sequence $\ann(Y)\overset{v}\rightarrow S\overset{u}\rightarrow F(Y,Y)$. We first show that $\supp(v)=I\setminus \supp(Y)$. Consider the long exact sequence
\[\cdots\rightarrow(E_i)_\ast(F(Y,Y))\rightarrow (E_i)_\ast(\ann(Y))\overset{(E_i)_\ast(v)}\rightarrow (E_i)_\ast(S)\overset{(E_i)_\ast(u)}\rightarrow (E_i)_\ast(F(Y,Y))\rightarrow\cdots.\]
If $i\in I\setminus \supp(Y)$, then $(E_i)_\ast(F(Y,Y))\cong (E_i)_\ast(DY)\otimes_{(E_i)_\ast}(E_i)_\ast(Y)=0$ and $(E_i)_\ast(v)$ is an isomorphism. Hence, $i\in\supp(v)$. If, on the other hand, $(E_i)_\ast(v)\neq 0$, it already has to be surjective because $(E_i)_\ast(\ann(Y))$ is an $(E_i)_\ast$-vector space (by property (3)). It follows that $(E_i)_\ast(u)=0$. But $u$ is the unit map of $F(Y,Y)$, so this implies $(E_i)_\ast(F(Y,Y))=0$. As $Y$ is a retract of $F(Y,Y)\wedge Y$, it follows that $(E_i)_\ast(Y)=0$. This proves $\supp(v)=I\setminus \supp(Y)$.

Now let $X,Y\in\SH(G)_f$ and $\supp(X)\subseteq\supp(Y)$. With $v$ as above, we have $\supp(v)=I\setminus \supp(Y)\subseteq I\setminus\supp(X)$. Hence, 
\[\supp\left( \ann(Y)\overset{v}\rightarrow S\rightarrow F(X,X) \right)=\emptyset.\]
By property (1), this map is smash nilpotent, so there is some $m>0$ such that
\[\ann(Y)^{\wedge m}\rightarrow S\rightarrow F(X,X)^{\wedge m}\]
is the zero map. Concatenation defines a map $F(X,X)^{\wedge m}\rightarrow F(X,X)$, over which the unit map $S\rightarrow F(X,X)$ factors, so we get a diagram in which the lower row is a cofiber sequence, the composition of the two upper maps is zero and the square commutes:
\[\xymatrix{\ar @{} [dr] |{}
\ann(Y)^{\wedge m} \ar[r]\ar@{-->}[d] & S \ar[r]\ar@{=}[d] & F(X,X)^{\wedge m}\ar[d]\\
\ann(X) \ar[r] & S \ar[r] & F(X,X),}\]
It follows that the map $\ann(Y)^{\wedge m}\rightarrow S$ factors over $\ann(X)$. By Proposition \ref{ann}, this is equivalent to $\thickid(X)\subseteq\thickid(Y)$.

For the other direction, assume $\thickid(X)\subseteq\thickid(Y)$. Then by Remark \ref{supp thickid}, 
\[\supp(X)=\supp(\thickid(X))\subseteq\supp(\thickid(Y))=\supp(Y).\]
\end{proof}

\begin{defn}\label{def detect nilp}
We say that a family $\{E_i\; |\; i\in I\}$ detects nilpotence, if for any $f:X\rightarrow Y$ in $\SH(G)_f$, $\supp(f)=\emptyset$ implies $f^{\wedge n}=0$ for some $n>0$.
\end{defn}

\begin{cor}\label{nilp-idl2}\emph{(Strickland)}
Under the assumptions of the above theorem, the map from the collection of thick ideals in $\SH(G)_f$ to the set of subsets of $I$,
\[\Idl(\SH(G)_f)\longrightarrow \mathcal P(I),\]
\[\mathcal C\mapsto \supp(\mathcal C),\]
is a lattice homomorphism (see Definition \ref{lattice}). It is injective on the collection of finitely generated thick ideals, $\FIdl(\SH(G)_f)$ (see Definition \ref{thick ideal}).
\end{cor}

\begin{proof}
This is \cite[Corollary 15.15]{Sch}. It is clear from the definition of $\supp(\mathcal C)$, that $\supp(-)$ is order preserving. The map preserves meets by Proposition \ref{intersection} and by the K\"unneth formula for $E_i$. As $\supp(\mathcal C)$ is the support of any set of generators for $\mathcal C$ and the join of thick ideals $\mathcal C_i$ is generated by the collection of generators of the individual thick ideals, it is also clear that $\supp(-)$ preserves joins. Recall that any finitely generated thick ideal is already generated by a single element (Remark \ref{FIdl}). By the above theorem, $\thickid(X)=\thickid(Y)$ if and only if $\supp(X)=\supp(Y)$, which is the same as $\supp(\thickid(X))=\supp(\thickid(Y))$. This proves the injectivity on $\FIdl$.
\end{proof}

\section{Thick ideals and equivariant Morava K-theories}

\begin{defn}
For a finite group $G$, let $\sub(G)$ denote the set of equivalence classes of conjugate subgroups of $G$. Let
\[\mathcal Q'= \{p^{-n}\;|\;p \textup{ prime}, 0\leq n<\infty\}\]
\[\textup{and }\;GQ'=\mathcal Q'\times \sub(G).\]
\end{defn}

The following theorem shows that the family of equivariant Morava K-theories $\{K(p^{-n},H)\;|\;(p^{-n},H)\in GQ'\}$ (see Section \ref{Equivariant Morava K-theories}) detects nilpotence, as required in the assumptions of Theorem \ref{nilp-idl} and Corollary \ref{nilp-idl2}. As in \cite[Theorem 1]{DHS}, there are different kinds of nilpotence, which are all detected by the Morava K-theories. Although we mainly work with smash nilpotence, the theorem, which is \cite[Theorem 16.7]{Sch}, considers all three definitions.

\begin{theorem}\label{detect nilp}\emph{(Strickland)}
\begin{compactenum}[(1)]
\item Let $R\in\SH(G)$ be a ring spectrum. Then $\alpha\in\pi_\ast^G R$ is nilpotent if and only if for all $v\in GQ'$, $K(v)_\ast(\alpha)$ is nilpotent as an element of $K(v)_\ast R$.
\item A self-map $f:\Sigma^k W\rightarrow W$, $W\in\SH(G)_f$, is nilpotent if and only if for all $v\in GQ'$, $K(v)_\ast(f)$ is nilpotent.
\item A map $f:W\rightarrow X$, with $W\in\SH(G)_f$ and $X\in \SH(G)$, is smash nilpotent if and only if for all $v\in GQ'$, $K(v)_\ast(f)$ is nilpotent.
\end{compactenum}
\end{theorem}

\begin{proof}
\begin{compactenum}[(1)]
\item Let $\alpha:S_G^d\rightarrow R$ be such that $K(v)_\ast(\alpha)$ is nilpotent for all $v\in GQ'$. For any $H\subseteq G$, $\phi^H(S_G^d)$ is a non-equivariant sphere, so $\phi^H\alpha\in\pi_\ast(\phi^H R)$. By Proposition \ref{K(n) and phi}, $K(u)_\ast(\phi^H \alpha)=K(u,H)_\ast(\alpha)$, so this is nilpotent for all $u\in\mathcal Q'$. Furthermore, $\phi^H R$ is a ring spectrum and, by \cite[Theorem 3(i)]{HS}, it follows that $\phi^H(\alpha)$ is nilpotent, so $(\phi^H R)[\phi^H\alpha^{-1}]=0$ (for the definition of this mapping telescope, see \cite[page 212]{DHS}). By \cite[Remark 7.15]{Schwede}, $\phi^H$ preserves telescopes, hence, $\phi^H (R[\alpha^{-1}])\cong (\phi^H R)[\phi^H\alpha^{-1}]=0$. This holds for all $H$, which by Proposition \ref{phi}(4) implies $R[\alpha^{-1}]=0$, and, hence, $\pi_\ast^G R[\alpha^{-1}]=0$. Thus, $\alpha$ is nilpotent.
\item The adjoint of $f$ is an element $\alpha\in\pi^G_d F(W,W)$, and $K(v)_\ast(\alpha)$ is nilpotent for all $v\in GQ'$. So, the claim follows from (1).
\item Let $R=\bigvee_{k\geq 0} F(W,X)^{\wedge k}\in\SH(G)$ be the free associative ring spectrum generated by $F(W,X)$. The map $f$ is adjoint to $\alpha\in \pi^G_d F(W,X)\subset\pi^G_d R$ such that $K(v)_\ast(\alpha)$ is nilpotent for all $v\in GQ'$, and the claim follows from (1).
\end{compactenum}
\end{proof}

\begin{defn}
Let $GQ=\underset{\sub(G)}\prod\mathcal Q$, where $\mathcal Q$ is as in Definition \ref{Q} and let
\[\max:\mathcal P(GQ')\longrightarrow GQ, \;I\mapsto \left\{ (H,p)\mapsto \max\{p^{-n}\;|\; (H,p^{-n})\in I\}\right\} ,\]
with the convention $\max(\emptyset)=0$. Let 
\[\tau=\max\circ\supp:\Idl(\SH(G)_f)\longrightarrow GQ.\] 
\end{defn}

\begin{cor}\label{injective lattice hom}
The functor 
\[\Idl(\SH(G)_f)\longrightarrow \mathcal P(GQ'),\]
\[\mathcal C\mapsto\supp(\mathcal C)= \left\{ v\in GQ'\;|\; K(v)_\ast(X)\neq 0 \textup{ for some }X\in\mathcal C\right\} ,\] 
is a lattice homomorphism. Furthermore, it is injective.
\end{cor}

\begin{proof}
This is \cite[Corollary 16.8]{Sch}. It follows from Corollary \ref{nilp-idl2} applied to the family of equivariant Morava K-theories. The assumptions are satisfied by Theorem \ref{detect nilp}(3) and Proposition \ref{Kunneth equivariant}. 

Corollary \ref{nilp-idl2} states that the restriction of $\supp$ to finitely generated thick ideals $\FIdl(\SH(G)_f)$ is injective.
To show that $\supp$ is injective on arbitrary thick ideals, it suffices to show that $\tau$ is injective. Note that $\tau$ is injective on $\FIdl(\SH(G)_f)$ because, by \cite[Theorem 2.11]{RavLoc}, $K(p^{-n})_\ast(\phi^H X)=0$ implies $K(p^{-m})_\ast(\phi^H X)=0$ for all $m\leq n$. 

Now let $\tau(\mathcal I)=\tau(\mathcal J)$ for some thick ideals $\mathcal I$ and $\mathcal J$ in $\SH(G)_f$. Thus, $\tau(\mathcal I)_{H,p}=\tau(\mathcal J)_{H,p}$ for all $H\subseteq G$ and all primes $p$. Assume $X\in\mathcal I$. Then $\tau(X)_{H,p}\leq \tau(\mathcal I)_{H,p}=\tau(\mathcal J)_{H,p}$. Let $I=\{(H,p)\;|\;\tau(X)_{H,p}\neq 0,1\}$. We claim that $I$ is finite. Since $G$ is finite, there are only finitely many possibilities for $H$. If $(H,p)\in I$, then $K(p^{-0})_\ast(\phi^H X)=\h_\ast(\phi^H X,\Q)=0$. Since $\phi^H X$ is a finite spectrum, this implies $\h_\ast(\phi^H X,\F_q)=0$ for all but finitely many $q$, so $I$ is finite. For any $u\in I$, there exists by assumption a $Y_u\in\mathcal J$ with $\tau(Y_u)_u\geq\tau(X)_u$. Similarly, let $I'=\{H\;|\;\tau(X)_H=1\}$ and pick $Y_H\in\mathcal J$ with $\tau(Y_H)=1$ for each $H\in I'$. Then, by the injectivity of $\tau$ on finitely generated thick ideals, $X$ is contained in $\thickid(\{Y_u\;|\; u\in I\}\cup\{Y_H\;|\;H\in I'\})$, which is a finitely generated subideal of $\mathcal J$. It follows $\mathcal I\subseteq\mathcal J$ and, similarly, $\mathcal J\subseteq\mathcal I$.
\end{proof}

As promised, we now state a reformulation of the thick subcategory theorem from \cite{HS}, which is no longer in $p$-local form. We do not claim that all the notation from above was necessary to state this. But it is helpful for generalising the result to $\SH(G)_f$ or maybe also other categories.

\begin{theorem}\label{thm Q}
The composition 
\[\tau: \Idl(\SH^{fin})\overset{\supp}\longrightarrow \mathcal P(\mathcal Q')\overset{\max}\longrightarrow \mathcal Q\]
is bijective. Its restriction to the finitely generated thick ideals maps $\FIdl(\SH^{fin})$ bijectively to
\[ \Fin(\mathcal Q)=\{u\in\mathcal Q\;|\; u=1 \textup{ or } u_p=0 \textup{ for almost all } p\}.\]
\end{theorem}

\begin{proof}
This is \cite[Proposition 19.14]{Sch}.
The theorem states that $\supp$ maps injectively onto its image, which is isomorphic to $\mathcal Q$. The map $\supp$ is injective by the nonequivariant version of Corollary \ref{injective lattice hom}. The image of $\supp$ maps injectively to $\mathcal Q$ because $K(p^{-n})_\ast X=0$ implies $K(p^{-m})_\ast X=0$ for all $m\leq n$ by \cite[Theorem 2.11]{RavLoc}. We show that $\tau$ is surjective: By \cite[Theorem B(b)]{Mi} or \cite[Section C.3]{Rav}, for any number $n>0$ there is a spectrum $X_{n}\in\SH^{fin}_{(p)}$ of type $n$. Thus, the $p$-localisation of $\tau$ is surjective onto $\mathcal Q_{p}$ (see Definition \ref{Q}). Let $u\in\mathcal Q$. If $u=1$, then $\mathcal C=\thickid(S^0)$ satisfies $\tau(\mathcal C)=u$. Assume $u\neq 1$, so $u_p=p^{-n_p}$ with $0< n_p\leq \infty$. Let $X_{n_p}$ be a $p$-local spectrum of type $n_p$ if $0<n_p<\infty$ and $X_{n_p}=0$ if $n_p=\infty$. Then $\mathcal C=\thickid(X_{n_2},X_{n_3},X_{n_5},\cdots)$ satisfies $\tau(\mathcal C)=u$. Hence, $\tau$ is surjective. If $u=1$ or $u_p=0$ for almost all $p$, then $\mathcal C$ as above is finitely generated. As in the proof of Corollary \ref{injective lattice hom}, all finitely generated thick ideals are mapped to $\Fin(\mathcal Q)$.
\end{proof}

This theorem identifies the thick ideals with a sublattice of $\mathcal P(\mathcal Q')$ that is bijective to $\mathcal Q$. Similarly, thick ideals in the equivariant category $\SH(G)_f$ are mapped injectively into $GQ$. Unlike its nonequivariant version, the equivariant version of $\tau$ is not surjective. However, we can state the following:

\begin{theorem}\label{tau injective}\emph{(Strickland)}
The composition
\[\tau: \Idl(\SH(G)_f)\overset{\supp}\longrightarrow \mathcal P(GQ')\overset{\max}\longrightarrow GQ\]
is injective. Its image contains all $u\in GQ$ which satisfy: If $H\subseteq H'$ for some $H,H'\in\sub(G)$, then $u_{H} \geq u_{H'}$.
\end{theorem}

\begin{proof}
We already know that $\supp$ is an injective lattice homomorphism. The injectivity of $\tau$ follows as in the nonequivariant case, by considering $H$-fixed points for each $H\in\sub(G)$ separately.

Now let $u\in GQ$ be such that $H'\subseteq H$ implies $u_{H'} \geq u_H$. Let $X_{u_H}$ be a finite spectrum in $\SH$ with $\tau(X_{u_H})=u_H$. Recall that any nonequivariant spectrum maps to a $G$-spectrum with trivial $G$-action through the functor $i:\SH\rightarrow \SH(G)$. Let $Y_H=i(X_{u_H})\wedge G/H_+$. This is a finite $G$-spectrum satisfying $\phi^{H'}Y_H=X_{u_H}\wedge G/H_+$ if $H'$ is subconjugate to $H$ and $\phi^{H'}Y_H=0$ if $H'$ is not subconjugate to $H$. Hence, $\tau_{H'}(Y_H)=u_H$ if $H'\subseteq H$ and $\tau_{H'}(Y_H)=0$ if $H'\not\subseteq H$. Let $Y=\bigvee_{H\in\sub(G)} Y_H$. Then $\tau_H(Y)=\max\{u_{H'}\;|\;  H\subseteq H'\}=u_H$ by the assumption.
\end{proof}

This theorem gives a lower bound on the set of thick ideals in $\SH(G)_f$. The whole set $GQ$ is an upper bound. Strickland was able to show that $\tau(\Idl(\SH(G)_f))$ lies in a certain proper subset of $GQ$ (if $G$ is nontrivial). For a cyclic group $G=\Z/p$, his result is the following \cite[Proposition 16.9]{Sch}:

\begin{prop}\label{upper bound}\emph{(Strickland)}
If $X\in\SH(\Z/p)_f$ and $K(p^{-(n+1)})_\ast \phi^{\{1\}}X=0$, then $K(p^{-n})_\ast \phi^{\Z/p}X=0$. (Note that the same prime $p$ appears in two different roles.)
\end{prop}

\begin{proof}
Again, the proof is taken from \cite{Sch}. 
$X$ is of the form $X=\Sigma^{\infty-V} T$, where $T$ is a retract of a finite $\Z/p$-CW complex and $V$ is a representation of $\Z/p$. Then, by Proposition \ref{phi}, $\phi^{\Z/p}X= \Sigma^{\infty-\dim(V^{\Z/p})}T^{\Z/p}$. The assumption is equivalent to $K(p^{-(n+1)})_\ast T=0$ and the claim is equivalent to $K(p^{-n})_\ast T^{\Z/p}=0$. By \cite[Theorem 2.11]{RavLoc}, we can formulate this in terms of Johnson-Wilson theories, namely: We know $E(n+1)_\ast T=0$ and have to show $E(n)_\ast T^{\Z/p}=0$. The proof uses the Greenlees-May theory of Tate spectra \cite{Tate}. For any $G$-spectrum $Y$, let $t_G Y=F(EG_+,Y)\wedge \tilde{E}G$ and $P_G Y=(t_G Y)^G$. They have the following properties:
\begin{compactenum}[(1)]
\item The functors $t_G$ and $P_G$ preserve exact triangles.
\item We have $t_G(X\wedge Y)=X\wedge t_G Y$ for finite $G$-spectra $X$, and $P_G(X\wedge Y)=X\wedge P_G Y$ for finite spectra $X$ with trivial $G$-action.
\item If $Y$ is a free $G$-spectrum, then $t_G Y=0$ and so $P_G Y=0$.
\item If $Y$ is nonequivariantly contractible, then $t_G Y=0$ and so $P_G Y=0$.
\item If $p$ divides the order of $G$, then the spectrum $P_G E(n+1)$ has Bousfield class $\langle P_G E(n+1)\rangle = \langle E(n)\rangle$ \cite[Theorem 1.1]{Tate2}.
\end{compactenum}
As $E(n+1)_\ast T=0$, we see that the $\Z/p$-spectrum $E(n+1)\wedge T$ is nonequivariantly contractible, so $t_G(E(n+1)\wedge T)=0$ by (4). We also have $t_G(E(n+1)\wedge T/T^{\Z/p})=0$ by (3), so (1) implies $t_G(E(n+1)\wedge T^{\Z/p})=0$. Hence, $(P_G E(n+1))\wedge T^{\Z/p}=P_G (E(n+1)\wedge T^{\Z/p})=0$ and (5) gives $E(n)\wedge T^{\Z/p}=0$, as required.
\end{proof}

This result on cyclic groups can be used to derive some restriction on the types of $G$-spectra that can occur for a general group $G$. This is done by applying the result to quotients of subgroups of $G$, see \cite[Corollary 16.10]{Sch}. We do not state this result here, because it needs additional notation and because we will only be interested in the case $G=\Z/2$ later on.

\section{Thick ideals in $\SH(\Z/2)_f$}\label{Application to Z/2}

We apply the results of this chapter to $G=\Z/2$.

Recall that, for $G$ a finite group, the thick ideals in $(\SH(G)_f)_{(p)}$ are characterised by their equivariant types, i.e., by the vanishing or non-vanishing of the different equivariant Morava K-theories.

\begin{cor}\label{Cmn}
By the injectivity of $\tau$ (Theorem \ref{tau injective}), any thick ideal in the category $(\SH(\Z/2)_f)_{(p)}$ is of the form 
\[\mathcal C_{m,n}=\{X\; |\; \phi^{\{1\}}(X)\in\mathcal C_m \textup{ and }\phi^{\Z/2}(X)\in\mathcal C_n\},\]
where $m,n\in [0,\infty]$. By Proposition \ref{K(n) and phi}, $X\in\mathcal C_{m,n}$ is equivalent to $K(m-1,\{1\})_\ast(X)=0$ and $K(n-1,\Z/2)_\ast(X)=0$ if $0<m,n<\infty$. 
\end{cor}

\begin{defn}\label{equivariant type}
We say that $X\in(\SH(\Z/2)_f)_{(p)}$ has type $(m,n)$, $0\leq m,n\leq \infty$, if $X\in\mathcal C_{m,n}\setminus (\mathcal C_{m+1,n}\bigcup \mathcal C_{m,n+1})$.

Let $\Gamma_p\subseteq (\Z_{\geq 0}\cup \{\infty\})\times (\Z_{\geq 0}\cup\{\infty\})$ be the sublattice of all $(m,n)$ such that a $p$-local spectrum of type $(m,n)$ exists.
\end{defn}

\begin{cor}\label{mn}
\begin{compactenum}[(1)]
\item Any type-$(m,n)$ spectrum generates $\mathcal C_{m,n}$ as a thick ideal (by Theorem \ref{tau injective}).
\item Not all pairs $(m,n)$ occur as the type of some spectrum. For example, if $p=2$, $m$ cannot be greater than $n+1$ (by Proposition \ref{upper bound}).
\item If $m\leq n$, then a type-$(m,n)$ spectrum exists, namely $X_{m,n}=(X_m\wedge \Z/2_+)\vee X_n$ with $X_k\in\SH^{fin}_{(p)}$ a type-$k$ spectrum and with $\Z/2$ acting nontrivially only on $\Z/2$ (by Theorem \ref{tau injective}). Thus, $m\leq n$ implies $(m,n)\in\Gamma_p$.
\end{compactenum}
\end{cor}

These results of Strickland give upper and lower bounds for the set of thick ideals in $(\SH(\Z/2)_f)_{(p)}$ (which is bijective to $\Gamma_p$).

\begin{rk}\label{fin and f}
The results on thick ideals in $\SH(G)_f$ presented in this chapter also hold in the category of finite $G$-CW spectra, $\SH(G)^{fin}\subseteq \SH(G)_f$, as we did not use the closure under retracts in any argument.
\end{rk}

\chapter{Comparison functors}\label{functors}

\begin{notation}
For $k$ a field, we write $\SH(k)$ for the motivic stable homotopy category over $k$. The standard spheres are denoted by $S^{p,q}=S_s^{\wedge (p-q)}\wedge \G_m^{\wedge q}$ and their $\P^1$-suspension spectra are also denoted by $S^{p,q}$ if no confusion can arise. $\SH(k)$ is a tensor triangulated category, whose unit is the sphere spectrum $S=\Sigma^\infty_{\P^1} S^{0,0}$, which we also denote by $S^0$. For any $E\in\SH(k)$, $\pi_{p,q}(E)=[S^{p,q},E]_{\SH(k)}$ denotes the set of maps from $S^{p,q}$ to $E$ in $\SH(k)$. For $E$ and $X$ in $\SH(k)$, let $E_{p,q}(X)=\pi_{p,q}(X\wedge E)$ and $E^{p,q}(X)=[X,S^{p,q}\wedge E]$.\\
\end{notation}

For our study of thick ideals in motivic categories $\SH(k)$, $k\subseteq\C$, we want to use the given knowledge on thick ideals in the classical stable homotopy category and the $\Z/2$-equivariant stable homotopy category. For this purpose, we will make use of the functors $\SH\overset{c_k}\rightarrow \SH(k)\overset{R_k}\rightarrow \SH$ for $k\subseteq\C$ and $\SH(\Z/2)\overset{c'_k}\rightarrow \SH(k)\overset{R'_k}\rightarrow \SH(\Z/2)$ for $k\subseteq\R$.
The functors $c_k, R_\C$ and $R'_\R$ appear in various places in the literature but none of the sources conveniently covers all of the constructions. Most details on $R_\C$ can be found in \cite[Appendix]{PPR} and the unstable functor $R'_\R$ is studied in \cite[Section 5]{Du}.
Other important references are \cite[Section 3.4]{VMC}, \cite[Section 3.3]{MV} and \cite{Ayo}. The stable functors $R'_k$ and $c'_k$ are constructed and studied by Heller and Ormsby in \cite[Section 4]{HO}, which was written independently and at the same time as this chapter. Since our approach is slightly different, we give another, mostly self-contained construction of all these functors.

We start with the construction of $R_\C$ and $R'_\R$.

\section{Symmetric $\C P^1$-spectra}

Objects of the motivic stable homotopy category $\SH(\C)$ are spectra with respect to suspension by $\P_\C^1\cong S^{2,1}$. The corresponding analytic space $\P_\C^1(\C)$ is $\C P^1$. We want that the topological realisation $R(X)$ of a spectrum $X\in\SH(\C)$ is a spectrum again. Therefore, we work with $\C P^1$-spectra. 
The category of symmetric $\C P^1$-spectra, $Sp^\Sigma_{\C P^1}$, is described in \cite[Theorem A.44]{PPR} and is a model for the stable homotopy category. A symmetric $\C P^1$-spectrum is defined in the same way as a usual symmetric spectrum with $S^1$ replaced by $\C P^1\cong S^2$. The stable model structure is constructed analogously as for symmetric $S^1$-spectra in \cite{HSS}. Hence, the following results also hold for symmetric $\C P^1$-spectra \cite[Lemmas 3.4.5, 3.4.12, 3.4.13]{HSS}.

\begin{prop}\label{CP1 MS}
\begin{enumerate}
\item The stable trivial fibrations are the level trivial fibrations.
\item A map of $\C P^1$-spectra $f:X\rightarrow Y$ is a stable fibration if and only if it is a level fibration and 
\[\xymatrix{\ar @{} [dr] |{}
X_n \ar[d] \ar[r]^{\tilde\sigma} & \Omega X_{n+1}\ar[d] \\
Y_n  \ar[r]^{\tilde\sigma} & \Omega Y_{n+1} }\]
is homotopy cartesian for all $n$, where the horizontal maps are the adjoints of the structure maps.
\item The fibrant objects are $\Omega$-spectra, i.e., the adjoints of their structure maps are weak equivalences.
\end{enumerate}
\end{prop}

\section{$\Z/2$-equivariant symmetric spectra}\label{Z2 model}

For the $\Z/2$-equivariant stable homotopy category we use the model constructed in \cite{Ma}. We now recall Mandell's definitions and results. Let $G$ be a finite group. We only need the case $G=\Z/2$.

\begin{defn}
Let $S(G)$ denote the based simplicial $G$-set obtained by smashing together copies of the simplicial circle $S^1$ indexed on the elements of $G$, where the $G$-action permutes the smash factors according to the multiplication in $G$.
\end{defn}

For example, $S(\{1\})=S^1$ and $S(\Z/2)\cong S^2$, where $\Z/2$ acts via an orientation reversing map of degree one.

\begin{defn}
A symmetric $G$-spectrum consists of a based $G\times \Sigma_n$-simplicial set $T(n)$ for each $n\in\N$ and structure maps $T(n)\wedge S(G)\rightarrow T(n+1)$ such that the $m$-th iterated structure maps are $G\times \Sigma_n\times \Sigma_m$-equivariant. Morphisms of spectra are defined in the usual way. We denote the category of symmetric $G$-spectra by $Sp^\Sigma(G)$.
\end{defn}

Mandell replaces $Sp^\Sigma(G)$ by the isomorphic category $Sp(G\Sigma_G)$ of $G\Sigma_G$-spectra, which is defined as follows.

\begin{defn}
Let $\Sigma$ be the category whose objects are nonnegative integers $\underbar n=\{1,\cdots,n\}$ and whose morphisms are bijections. Let $\Sigma_G$ be the category $\Sigma$ together with the diagram $S$ indexed on $\Sigma$ taking $\underbar n$ to the $n$-fold smash product of $S(G)$ and with arrows permuting these smash factors.

A $G\Sigma_G$-spectrum $T$ is a functor from $\Sigma$ to based simplicial $G$-sets together with natural transformations $\sigma_{\underbar n,\underbar m}:T(\underbar n)\wedge S(\underbar m)\rightarrow T(\underbar {n+m})$ satisfying certain associativity and unitality conditions \cite[Definition 1.3]{Ma}. A morphism of $G\Sigma_G$-spectra is a natural transformation commuting with the structure maps.
\end{defn}

Mandell defines $\Omega$-spectra in $Sp(G\Sigma_G)$ and constructs a projective level model structure with level equivalences, level fibrations and projective cofibrations. The resulting homotopy category is denoted by $\Ho^l$.

\begin{defn}
A morphism of $G\Sigma_G$-spectra $f:T\rightarrow U$ is called a stable equivalence if it induces bijections $\Ho^l(U,E)\rightarrow \Ho^l(T,E)$ for all $\Omega$-spectra $E$.
\end{defn}
 
The following theorem is \cite[Theorem 4.1]{Ma}.
 
\begin{theorem}\label{Mandell}
The category $Sp(G\Sigma_G)$ has a symmetric monoidal model structure with stable equivalences as weak equivalences and projective cofibrations as cofibrations. A morphism $f:T\rightarrow U$ is a fibration if and only if it is a level fibration and 
\[\xymatrix{\ar @{} [dr] |{}
T(\underbar m) \ar[d] \ar[rr]^{\tilde\sigma_{\tiny{\underbar m,\underbar n}}} && \Omega_{\underbar n}T(\underbar{m+n})\ar[d] \\
U(\underbar m)  \ar[rr]^{\tilde\sigma_{\tiny{\underbar m,\underbar n}}} && \Omega_{\underbar n}U(\underbar{m+n}) ,}\]
is homotopy cartesian for all $\underbar m, \underbar n\in\Sigma$, where the horizontal maps are adjoint to the structure maps. 
\end{theorem}

Using this model structure, Mandell shows \cite[Theorem 2]{Ma}:

\begin{theorem}
The category $Sp^\Sigma(G)$ is Quillen equivalent to the stable $G$-equivariant category indexed on a complete $G$-universe. 
\end{theorem}

\section{Complex and real topological realisation functors}\label{realization}

The aim of this section is to recall the construction of the stable topological realisation functors 
\[R=R_\C:\SH(\C)\rightarrow \SH,\]
\[R'=R'_\R:\SH(\R)\rightarrow \SH(\Z/2).\]
Various unstable and stable versions of these functors were constructed in \cite[Section 3.4]{VMC}, \cite[Section 3.3]{MV}, \cite{Ayo}, \cite{Du}, \cite[Appendix]{PPR} and \cite[Section 4]{HO}.

For $k\subseteq\C$, let $\Sm/k$ be the category of smooth schemes of finite type over $k$ and let $\sPre(\Sm /k)$ be the category of simplicial presheaves on the Nisnevich site $\Sm/k$, see e.g. \cite{MV} or \cite[Appendix B]{J}. 

We begin with the definition of $R_\C$. 
In the next section we will define $R_k$ and $R_k'$ also for subfields $k$ of $\C$ and $\R$ respectively.

The functor 
\[R:\sPre(\Sm/\C)\rightarrow \sSet\]
is defined in the following way: Any simplicial presheaf $A$ can be written as 
\[\underset{X\times \Delta^n \rightarrow A}\colim (X\times \Delta^n)\stackrel\cong\rightarrow A,\]
where the colimit is taken over the over-category of $A$, in which $X$ runs over representable presheaves, see \cite[Section A.4]{PPR}. We set 
\[R(A)= \underset{X\times \Delta^n\rightarrow A}\colim (X(\C)\times \Delta^n) \in \sSet.\]
By $X(\C)$ we actually mean the simplicial set $\Sing(X(\C)^{an})$, where $X(\C)^{an}$ denotes the set of complex points of $X$ with the analytic topology.\\

Now let $k = \R$. Let $\sSet(\Z/2)$ denote the category of $\Z/2$-simplicial sets.
The functor 
\[R':\sPre(\Sm/\R)\rightarrow \sSet(\Z/2)\]
is still defined by
\[R'(A)= \underset{X\times \Delta^n\rightarrow A}\colim (X(\C)\times \Delta^n),\]
but now $\Z/2$ acts on $X(\C)$ by precomposing with conjugation. This induces an action of $\Z/2$ on $R'(A)$, see e.g. \cite[Section 5]{Du}.\\

If $A$ is pointed, then so are $R(A)$ and $R'(A)$ respectively.\\

We equip $\sPre(\Sm/k)$ with the projective model structure defined in \cite[Section 5.1]{Du}, where this category is denoted by $\Spc '(k)_{\Nis} $. The model structure for $\sSet$ can be found in \cite[Section 3.2]{Hov} and the equivariant model structure on $\sSet(\Z/2)$ is, for example, described in \cite[Example 4.2]{Gui}:
\begin{itemize}
\item Weak equivalences are maps $f$ that induce weak equivalences $f^H$ on the fixed point sets for all $H\subseteq \Z/2$.
\item The collection $\{(\Z/2)/H\times \partial \Delta^n\rightarrow (\Z/2)/H\times \Delta^n\; |\; H\subseteq \Z/2\}$ is a set of generating cofibrations.
\item The collection $\{(\Z/2)/H\times \Lambda_i^n\rightarrow (\Z/2)/H\times \Delta^n\; |\; H\subseteq \Z/2\}$ is a set of generating trivial cofibrations.
\end{itemize}

\begin{theorem}
The functors $R$ and $R'$ and their pointed versions are strict symmetric monoidal left Quillen functors.
\end{theorem}

\begin{proof}
This is \cite[Theorems 5.2, 5.5]{Du} and \cite[Theorem A.23]{PPR}. 
\end{proof}

Now we define the functor 
\[\Sing:\sSet\rightarrow\sPre(\Sm/\C),\]
as in \cite[Theorem A.23]{PPR}.
It maps a simplicial set $Z$ to the simplicial presheaf sending $X\in\Sm/\C$ to the simplicial set which in degree $n$ is the set of maps $\sSet(X(\C)\times \Delta^n,Z)$, that is, $\Sing(Z)(X)$ is defined as an internal hom object in $\sSet$.

\begin{prop}\label{adj1}
The functor $\Sing$ is right adjoint to $R$.
\end{prop}

\begin{proof}
We have to find a bijection 
\[\Phi:\sSet(R(X),Y)\rightarrow \sPre(\Sm/\C)(X,\Sing(Y))\]
for any $X\in \sPre(\Sm/\C)$ and $Y\in\sSet$. The general case will follow by passage to colimits after we have shown this for the case that $X$ is representable. So let $X=\Sm/\C(-\times\Delta^\bullet,X')$ with $X'\in\Sm/\C$. Let $f:R(X)\rightarrow Y$ be an element of the left hand side, which now is $\sSet(X'(\C)\times\Delta^\bullet,Y)$. We have to define a natural transformation 
\[\Phi(f):X=\Sm/\C(-\times\Delta^\bullet,X')\rightarrow \sSet(-(\C)\times\Delta^\bullet,Y)=\Sing(Y).\] 
We do this by the following composition:
\[\Phi(f):\Sm/\C(-\times\Delta^\bullet,X')\stackrel{R}\longrightarrow \sSet(-(\C)\times\Delta^\bullet,X'(\C))\stackrel{f_\ast}\longrightarrow \sSet(-(\C)\times\Delta^\bullet,Y).\]
The map $\Phi$ is obviously injective. It is also surjective because any morphism from $\Sm/\C(-\times\Delta^\bullet,X')$ to $\sSet(-(\C)\times\Delta^\bullet,Y)$ factors through the realisation functor.
\end{proof}

The $\Z/2$-version of this functor,
\[\Sing':\sSet(\Z/2)\rightarrow\sPre(\Sm/\R),\]
maps an equivariant simplicial set $Z$ to the simplicial presheaf sending $X\in\Sm/\R$ to the internal hom $\sSet(\Z/2)(X(\C),Z)$. 

\begin{prop}\label{adj2}
The functor $\Sing'$ is right adjoint to $R'$.
\end{prop}

\begin{proof}
The proof is the same as in the complex case, except that we have to replace $\Sm/\C$ by $\Sm/\R$ and $\sSet$ by $\sSet(\Z/2)$.
\end{proof}

We want to define stable versions of the functors $R, R'$ and $\Sing, \Sing'$. Therefore, we consider the category of symmetric ${\P^1}$-spectra on $\sPre(\Sm/k)$, denoted by $Sp^\Sigma_{\P^1}(k)$. The stable model structure on $Sp^\Sigma_{\P^1}(k)$ is constructed in the same way as in \cite{J}, except that we start with our different definitions of fibrations and cofibrations on $\sPre(\Sm/k)$. This construction is also described in \cite[Section A.5]{PPR}. We will only need the following information about this model structure.

Let $J$ be a level fibrant replacement functor and let $Q$ be the stabilisation functor defined in \cite[Remark 2.4]{J}. A map $f$ in $Sp^\Sigma_{\P^1}(k)$ is called a stable equivalence if $QJ(f)$ is a level equivalence.

\begin{prop}\label{QJ}
A map of symmetric ${\P^1}$-spectra, $f:X\rightarrow Y$, is a stable fibration if and only if it is a level fibration and 
\[\xymatrix{\ar @{} [dr] |{}
X_n \ar[d] \ar[r] &  QJX_{n}\ar[d] \\
Y_n  \ar[r] &  QJY_{n} }\]
is homotopy cartesian \cite[Lemma 2.7]{J}.

A stable trivial fibration is the same as a levelwise trivial fibration \cite[Theorem 2.9]{J}.  
\end{prop}

Since $R(\P^1_\C)\cong \C P^1$ and $R'(\P^1_\R)\cong R'(S_s^1\wedge \G_m)\cong S^1_+\wedge S^1_-\cong S(\Z/2)$ (the 2-sphere with orientation reversing $\Z/2$-action) in the homotopy categories, 
we can define 
\[R:Sp^\Sigma_{\P^1}(\C)\rightarrow Sp^\Sigma_{\C P^1}\]
and
\[R':Sp^\Sigma_{\P^1}(\R)\rightarrow Sp^\Sigma(\Z/2)\]
levelwise.\\

We can also extend $\Sing$ and $\Sing'$ to the categories of spectra, as follows.

Over $\C$, the simplicial presheaf ${\P^1}$ is equivalent to the simplicial presheaf sending $Z\in\Sm/\C$ to $\Sm/\C(Z\times\Delta^n,\P^1_\C)$. realisation defines a map from $\Sm/\C(Z\times\Delta^n,\P^1_\C)$ to $\sSet(Z(\C)\times\Delta^n,\C P^1)$. These can be assembled into a map of simplicial presheaves ${\P^1}\rightarrow \Sing(\C P^1)$. For $X\in Sp^\Sigma_{\C P^1}$ we can, hence, define structure maps 
\[\Sing(X_n)\wedge {\P^1}\rightarrow \Sing(X_n)\wedge \Sing(\C P^1)\cong\Sing(X_n\wedge\C P^1)\stackrel{\Sing(\sigma_n)}\longrightarrow\Sing(X_{n+1}),\]
so that we get a spectrum $\Sing(X)\in Sp^\Sigma_{\P^1}(\C)$ defined levelwise.\\

Over $\R$, the same argument holds if we consider $\Z/2$-equivariant maps of simplicial sets. We get a spectrum $\Sing'(X)\in Sp^\Sigma_{\P^1}(\R)$ defined levelwise.

\begin{cor}
The functors $(R,\Sing)$ and $(R',\Sing')$ form adjoint pairs between the categories of symmetric spectra.
\end{cor}

\begin{theorem}{\bf (The stable functors $R$, $R'$)}

The pairs $(R,\Sing)$ and $(R',\Sing')$ are Quillen adjunctions on the spectrum level and $R$, $R'$ are strict symmetric monoidal. 
\end{theorem}

\begin{proof}
The case of $R$ is covered in \cite[Theorem A.45]{PPR} and the claim for $R'$ is proven in \cite[Proposition 4.8]{HO}. For completeness, we reprove the theorem in our own words.

To show that $R$ and $R'$ are Quillen functors, we only have to prove that their right adjoints preserve stable fibrations between stably fibrant objects and stable trivial fibrations \cite[Corollary A.2]{smod}. In all model structures we are considering here, the stable trivial fibrations are the levelwise trivial fibrations. From the unstable version of this theorem it follows therefore that $\Sing$ and $\Sing'$ preserve stable trivial fibrations.

Stable fibrations are, in all of these model structures, levelwise fibrations with some additional homotopy pullback properties and stably fibrant objects are always $\Omega$-spectra. By Propositions \ref{adj1} and \ref{adj2}, the unstable functors $\Sing$ and $\Sing'$ are right Quillen functors. It follows that the levelwise-defined functors $\Sing$ and $\Sing'$ preserve $\Omega$-spectra and level fibrations. Let $f:X\rightarrow Y$ be a stable fibration between $\Omega$-spectra in $Sp^\Sigma_{\C P^1}$ with the model structure from Proposition \ref{CP1 MS} (or in $Sp^\Sigma(\Z/2)$ with the model structure from Theorem \ref{Mandell}). We have to show that
\[\xymatrix{\ar @{} [dr] |{}
\Sing(X)_n \ar[d] \ar[r] &  QJ\Sing(X)_{n}\ar[d] \\
\Sing(Y)_n  \ar[r] &  QJ\Sing(Y)_{n} }\]
is homotopy cartesian for all $n$.
Since $X$ and $Y$ are in particular level fibrant and $\Sing$ preserves level fibrations, $J\Sing(X)\simeq \Sing(X)$ and similarly for $Y$. Since $Q$ is defined using only the adjoint stucture maps, $Q\Sing(X)\simeq\Sing(X)$ and $Q\Sing(Y)\simeq\Sing(Y)$ for the $\Omega$-spectra $\Sing(X)$ and $\Sing(Y)$. It follows that the above square is in particular homotopy cartesian.

The functors $R$ and $R'$ are strict symmetric monoidal, since this holds unstably and the product of symmetric spectra is defined in the same way in all the categories considered here.
\end{proof}

\section{Realisation functors for other fields}

For $k\subseteq K$ a subfield, the canonical map $\Spec K\rightarrow \Spec k$ induces a couple of base change functors between the corresponding motivic homotopy categories. These are studied in \cite[Section 3.1]{MV} and also in \cite[Section 2]{HuBC}. For the stable version, see \cite[Section A.7]{PPR}. For a more general approach, see also \cite{AyoFun}.

Let $f:k\hookrightarrow K$ be the inclusion of the subfield.  
On the level of unpointed schemes, $f^\ast$ is given by
\[f^\ast:\Sm/k\rightarrow \Sm/K, \; f^\ast(X)=X\times_{\Spec k}\Spec K.\]
It induces a functor 
\[f^\ast:\sPre(\Sm/k)\rightarrow \sPre(\Sm/K),\]
which has a right adjoint $f_\ast$. By \cite[Proposition A.47]{PPR}, this adjunction induces a strict symmetric monoidal Quillen adjunction on the level of symmetric spectra, where $f^\ast$ is given by $f^\ast(E)_n=f^\ast(E_n)$.

One can therefore define realisation functors
\[R_k: \SH(k)\overset{f^\ast}\rightarrow\SH(\C)\overset{R}\rightarrow\SH \textup{ for } k\overset{f}\hookrightarrow \C,\]
\[R'_k: \SH(k)\overset{f^\ast}\rightarrow\SH(\R)\overset{R'}\rightarrow\SH(\Z/2) \textup{ for } k\overset{f}\hookrightarrow \R,\]
which are strict symmetric monoidal.

\section{Constant presheaf functors}

The following construction of the constant presheaf functors $c_k:\SH\rightarrow \SH(k)$ for $k\subseteq\C$ and $c'_k:\SH(\Z/2)\rightarrow \SH(k)$ for $k\subseteq \R$ is close to the one given in \cite[Section 4]{HO}.

Let $k\subseteq \C$. 
For $X\in \sSet$ we define $c_k(X)\in\sPre(\Sm/k)$ by $c_k(X)(Z)=X$ for all $Z\in\Sm/k$. Using $c_k(S^1)=S_s^1$, we extend the functor $c_k:\sSet\rightarrow \sPre(\Sm/k)$ levelwise to symmetric $S^1$-spectra and get: 
\[c_k:Sp^\Sigma_{S^1}\rightarrow Sp_{S_s^1}^\Sigma(k).\]
We postcompose this functor with the $\P^1$-suspension functor, yielding a functor to the category of symmetric $(S_s^1,\P^1)$-bispectra, $c_k:Sp^\Sigma_{S^1}\rightarrow Sp_{S_s^1,\P^1}^\Sigma(k)$. The homotopy category of $Sp_{S_s^1,\P^1}^\Sigma(k)$ is equivalent to $\SH(k)$ by \cite[Theorem 9.1]{SpecSymSpec}.

\begin{theorem}\label{Rc} {\bf (The stable functor $c_k$)}

This induces a functor $c_k:\mathcal{SH}\rightarrow\mathcal{SH}(k)$, which is strict symmetric monoidal. It is right inverse to $R_k$ and hence faithful. Furthermore, by \cite[Theorem 1]{L}, $c_k$ is full if $k$ is algebraically closed.
\end{theorem} 

\begin{proof}

We first show that the unstable functor, $c_k:\sSet\rightarrow \sPre(\Sm/k)$, is a left Quillen functor. The generating cofibrations of $\sSet$ are the maps $\partial\Delta^n\hookrightarrow\Delta^n$. The functor $c_k$ takes these maps to the same maps considered as morphisms of constant simplicial presheaves. These are examples of generating cofibrations in the model structure for $\sPre(\Sm/k)$, as described in \cite[Section A.3]{PPR}. The same applies to the generating trivial cofibrations $\Lambda^n_r\hookrightarrow\Delta^n$. The functor $c_k$ preserves colimits by definition, hence it is a left Quillen functor. We denote its right adjoint by $r_0$. It satisfies $r_0(S_s^1)=S^1$.

Now we show that $c_k:Sp^\Sigma_{S^1}\rightarrow Sp_{S_s^1}^\Sigma(k)$ is left Quillen, where the model structure on $Sp_{S_s^1}^\Sigma(k)$ is described in \cite[Section 4.5]{J} and satisfies the analogue of Proposition \ref{QJ}.
The right adjoint, $r$, to $c_k$ is defined by levelwise application of $r_0$. Since $r_0$, as a right Quillen functor, preserves fibrations and trivial fibrations, $r$ preserves level fibrations and level trivial fibrations. Stable trivial fibrations are the same as level trivial fibrations, hence these are preserved by $r$. We have to show that $r$ also preserves stable fibrations between stably fibrant objects. Let $f:X\rightarrow Y$ be a stable fibration in $Sp_{S_s^1}^\Sigma(k)$ with $X$ and $Y$ level fibrant $\Omega$-spectra. We have to show that $r(f)$ is a level fibration---which we already know---and that the squares
\[\xymatrix{\ar @{} [dr] |{}
r(X)_n \ar[d] \ar[r] &  \Omega r(X)_{n+1}\ar[d] \\
r(Y)_n  \ar[r] &  \Omega r(Y)_{n+1} }\]
are homotopy pullbacks. This is trivial because $r$ preserves $\Omega$-spectra (it is defined levelwise and commutes with desuspension), so $r(X)$ and $r(Y)$ are $\Omega$-spectra.
This proves that $c_k:Sp^\Sigma_{S^1}\rightarrow Sp_{S_s^1}^\Sigma(k)$ is a left Quillen functor.
It is symmetric monoidal by its pointset definition and by the definition of products of symmetric spectra. 

The $\P^1$-suspension functor $Sp_{S_s^1}^\Sigma(k)\rightarrow Sp_{S_s^1,\P^1}^\Sigma(k)$ is also a symmetric monoidal left Quillen functor if the category of symmetric $\P^1$-spectra over $Sp_{S_s^1}^\Sigma(k)$ is endowed with the stable model structure (see \cite[Theorems 5.1 and 9.1]{SpecSymSpec}). 
It follows that both functors induce a functor on the respective stable homotopy categories, and the concatenation of the induced functors, $c_k:\SH\rightarrow \SH(k)$, is also strict symmetric monoidal by \cite[Theorem 8.11]{SpecSymSpec}.

To show that $c_k$ is right inverse to $R_k$, first note that, for $f:k\hookrightarrow \C$, we have $f^\ast(\Spec ( k))=\Spec (\C)$, which implies $f^\ast\circ c_k=c_\C$. So, by definition of $R_k$, $R_k\circ c_k =R_\C\circ f^\ast\circ c_k=R_\C\circ c_\C$, and it suffices to consider the case $k=\C$. Unstably, for $A\in\sSet$, 
\[(R\circ c) (A)=\underset{X\times \Delta^\bullet\rightarrow c A}\colim (X(\C)\times \Delta^\bullet)=\underset{\Delta^\bullet \rightarrow A}\colim \Delta^\bullet =A.\]
On the level of spectra, we have used different models for constructing $R$ and $c$. We therefore have to check that the following diagram is commutative, where, by definition, the composition of the upper maps induces $c:\SH\rightarrow \SH(\C)$ and the lower map induces $R:\SH(\C)\rightarrow \SH$.
\[\xymatrix{\ar @{} [dr] |{}
Sp^\Sigma_{S^1} \ar[d]^{\sim}_{\Sigma^\infty_{\C P^1}} \ar[r]^{c} &  Sp^\Sigma_{S_s^1}(\C)\ar[r]^{\Sigma^\infty_{\P^1}} & Sp^\Sigma_{S_s^1,\P^1}(\C) \ar[dll]^{\Sigma R}\\
Sp^\Sigma_{S^1,\C P^1}   &  & \\
Sp^\Sigma_{\C P^1} \ar[u]_{\sim}^{\Sigma^\infty_{S^1}} & & Sp^\Sigma_{\P^1}(\C) \ar[uu]^{\sim}_{\Sigma^\infty_{S_s^1}}\ar[ll]_{R}}\]
Since $R(S_s^1)=S^1$, the functor $R:Sp^\Sigma_{\P^1}(\C)\rightarrow Sp^\Sigma_{\C P^1}$ induces a functor $\Sigma R$ on $S_s^1$-spectra on $Sp^\Sigma_{\P^1}(\C)$ by the levelwise definition. This induced functor is drawn as a diagonal in the above diagram and it makes the lower subdiagram commutative by definition. Thus, it suffices to check that the upper diagram is commutative. Let $X=\{X_n\}_n\in Sp^\Sigma_{S^1}$. $X$ is mapped to $\{cX_n\}_n$ in $Sp^\Sigma_{S_s^1}(\C)$ and to $\{\P^m\wedge cX_n\}_{m,n}$ in $Sp^\Sigma_{S_s^1,\P^1}(\C)$, which realises to $\{(\C P^1)^m\wedge X_n\}_{m,n}$ in $Sp^\Sigma_{S^1,\C P^1}$. This is the same as the image of $X$ under the vertical map $Sp^\Sigma_{S^1}\rightarrow Sp^\Sigma_{S^1,\C P^1}$, which completes the proof that $R\circ c=\id$ on $\SH$.
\end{proof}

For subfields $k\subseteq\R$, we want to define functors $c'_k:\SH(\Z/2)\rightarrow \SH(k)$ which are right inverse to $R'_k$. For a better understanding, we first consider $k=\R$ and then generalise.

To define a functor $c':\mathcal{SH}(\Z/2)\rightarrow \mathcal{SH}(\R)$ which is right inverse to $R'$, we first construct $R':\sSet(\Z/2)\rightarrow \sPre(\Sm/\R)$. Observe that $R'(\Spec\R)=\ast$ is the one-point set and $R'(\Spec \C)=\Z/2$ is the two-point set with non-trivial $\Z/2$-action. We let $c'(\ast)=\Spec\R$ and $c'(\Z/2)=\Spec\C$ and extend this to $\Z/2$-sets $M$ by 
\[c'(M)=\left(\underset {M^{\Z/2}}\coprod \Spec\R\right)\coprod \left(\underset{(M\setminus M^{\Z/2})/(\Z/2)}\coprod \Spec\C\right).\]
This can be done functorially, just note that $-1:\Z/2\rightarrow\Z/2$ has to be mapped to the morphism $\Spec\C\rightarrow\Spec\C$ induced by complex conjugation.
Furthermore, $c'$ extends to simplicial $\Z/2$-sets by $c'(M\times\Delta^n)=c'(M)\times\Delta^n$.
This defines the unstable, basepoint preserving functor $c':\sSet(\Z/2)\rightarrow \sPre(\Sm/\R)$. 
We extend this to a functor of spectra by postcomposing the levelwise defined functor
\[c':Sp^\Sigma_{S(\Z/2)}\rightarrow Sp_{c'(S(\Z/2))}^\Sigma(\R)\]
with the suspension spectrum functor $\Sigma^\infty_{\P^1}:Sp^\Sigma_{c'(S(\Z/2))}(\R)\rightarrow Sp^\Sigma_{c'(S(\Z/2)),\P^1}(\R)$. Note that $c'(S(\Z/2))\cong c'(S^1_+)\wedge c'(S^1_-)\cong S_s^1\wedge c'(S^1_-)$ and $c'(S^1_-)=F_{\C/\R}(S^V)$ in the notation of \cite{HuBC} (with $V$ the sign representation), where it is shown that this is invertible in $\SH(\R)$ \cite[Theorem 3.5]{HuBC}. Thus, by \cite[Theorem 9.1]{SpecSymSpec}, $\Sigma^\infty_{c'(S(\Z/2))}:Sp^\Sigma_{\P^1}(\R)\rightarrow Sp^\Sigma_{c'(S(\Z/2)),\P^1}(\R)$ is a Quillen equivalence.

\begin{theorem}{\bf (The stable functor $c'$)}\\
This induces a functor $c':\SH(\Z/2)\rightarrow \SH(\R)$, which is strict symmetric monoidal and right inverse to $R'$. In particular, $c'$ is faithful. 
\end{theorem}

\begin{proof}
As in the previous proof, we start by considering the functor $c':\sSet(\Z/2)\rightarrow \sPre(\Sm/\R)$. It preserves colimits. The generating cofibrations of $\sSet(\Z/2)$ are the maps $(\Z/2)/H\times \partial\Delta^n\rightarrow (\Z/2)/H\times\Delta^n$, where $H\subseteq\Z/2$ is a subgroup. The images of these maps under $c'$ can be written as pushout products:
\[c'(\Z/2\times\partial\Delta^n\rightarrow \Z/2\times\Delta^n)=(\emptyset\rightarrow\Spec\C)\boxempty(\partial\Delta^n\rightarrow\Delta^n)\]
\[c'(\partial\Delta^n\rightarrow \Delta^n)=(\emptyset\rightarrow\Spec\R)\boxempty(\partial\Delta^n\rightarrow\Delta^n).\]
These are examples of generating cofibrations for $\sPre(\Sm/\R)$ as described in \cite[Section A.3]{PPR}. The same argument holds for the generating trivial cofibrations $(\Z/2)/H\times \Lambda^n_r\rightarrow (\Z/2)/H\times\Delta^n$. The passage to the spectrum level works similarly as in the previous proof. The induced functor $c':\SH(\Z/2)\rightarrow \SH(\R)$ is symmetric monoidal by the same arguments as before.

By its definition, $c'$ is right inverse to $R'$ on the level of simplicial $\Z/2$-sets. On the level of stable homotopy categories, $R'\circ c'=\id$ follows from the commutativity of the following diagram, similarly as in the previous proposition.
 \[\xymatrix{\ar @{} [dr] |{}
Sp^\Sigma_{S(\Z/2)} \ar[d]^{\sim}_{\Sigma^\infty_{S(\Z/2)}}\ar[r]^{c'} &  Sp^\Sigma_{c'(S(\Z/2))}(\R)\ar[r]^{\Sigma^\infty_{\P^1}} & Sp^\Sigma_{c'(S(\Z/2)),\P^1}(\R) \ar[dll]^{\Sigma R'}\\
Sp^\Sigma_{S(\Z/2),S(\Z/2)} & & \\
Sp^\Sigma_{S(\Z/2)} \ar[u]_{\sim}^{\Sigma^\infty_{S(\Z/2)}} & & Sp^\Sigma_{\P^1}(\R) \ar[uu]^{\sim}_{\Sigma^\infty_{c'(S(\Z/2))}}\ar[ll]_{R'}}.\]
\end{proof}

Now let $k\subseteq\R$.
Then $R'_k(\Spec k)=\ast$ and $R'_k(\Spec (k[i]))=\Z/2$. Therefore, we let $c'_k(\ast)=\Spec k$ and $c'_k(\Z/2)=\Spec(k[i])$ and, for a $\Z/2$-set $M$, 
\[c'_k(M)=\left(\underset {M^{\Z/2}}\coprod \Spec k\right)\coprod \left(\underset{(M\setminus M^{\Z/2})/(\Z/2)}\coprod \Spec(k[i])\right).\]
For functoriality, note that $-1:\Z/2\rightarrow\Z/2$ has to be mapped to $\Spec(k[i])\rightarrow\Spec(k[i])$ induced by complex conjugation.
As before, $c_k'$ extends to $c'_k:\sSet(\Z/2)\rightarrow \sPre(\Sm/k)$ and then to 
\[c'_k:Sp^\Sigma_{S(\Z/2)}\rightarrow Sp_{c'_k(S(\Z/2))}^\Sigma(k)\rightarrow Sp_{c'_k(S(\Z/2)),\P^1}^\Sigma(k),\]
where the first functor is defined levelwise by $c'_k$ and the second one is the $\P_k^1$-suspension spectrum functor.
Here, $c'_k(S(\Z/2))\cong S_s^1\wedge c'_k(S^1_-)$ and $c'_k(S^1_-)=F_{k[i]/k}(S^V)$ in the notation of \cite{HuBC}, which is invertible in $\SH(k)$ by \cite[Theorem 3.5]{HuBC}.

\begin{theorem}{\bf (The stable functor $c'_k$)}\\
This induces a functor $c'_k:\SH(\Z/2)\rightarrow \SH(k)$ which is strict symmetric monoidal and right inverse to $R'_k$. 
\end{theorem}

\begin{proof}
The main claim follows exactly as in the case $k=\R$ considered above. It is also implied by \cite[Theorem 4.6]{HO}. It remains to prove $R'_k\circ c'_k\cong \id$. Again, this follows from $f^\ast\circ c'_k\cong c'$ (for $k\overset{f}\hookrightarrow\R$) and $R'\circ c'\cong \id$, where $f^\ast\circ c'_k\cong c'$ holds because $f^\ast(\Spec k)\cong\Spec\R$ and $f^\ast(\Spec(k[i]))\cong \Spec\R\times_{\Spec k}\Spec(k[i])\cong\Spec\C$.
\end{proof}

\begin{rk}
In \cite[Theorem 1.1]{HO}, Heller and Ormsby prove that if $k$ is a real closed field, then $c_k'$ is full after $p$-completion.
\end{rk}

\begin{rk}\label{S1 spectra}
With similar methods as above, one can show that the functors $\sSet\overset{c_k}\rightarrow \sPre(\Sm/k)\overset{R_k}\rightarrow \sSet$ induce functors 
\[\SH\overset{c_k}\rightarrow \SH_{S^1_s}(k)\overset{R_k}\rightarrow\SH,\]
where $\SH_{S_s^1}(k)$ is the stable motivic homotopy category in which $S_s^1$ got inverted but $\G_m$ did not. 

For $k\subseteq\R$, the definition of the stable functor $c_k'$ relied on the invertibility of $F_{\C/\R}(S^V)$ in $\SH(\R)=\SH_{\P^1}(\R)$, as shown in \cite[Theorem 3.5]{HuBC}. One can show that the functors $\sSet(\Z/2)\overset{c'_k}\rightarrow \sPre(\Sm/k)\overset{R'_k}\rightarrow \sSet(\Z/2)$ induce functors 
\[ \SH_{S^1}(\Z/2)\overset{c'_k}\rightarrow \SH_{S^1_s}(k)\overset{R'_k}\rightarrow \SH_{S^1}(\Z/2),\]
where $\SH_{S^1}(\Z/2)$ is the naive equivariant stable homotopy category, in which only the sphere with trivial action got inverted. The functor $R':\SH_{S_s^1}(\R)\rightarrow \SH_{S^1}(\Z/2)$ sends $F_{\C/\R}(S^V)$ to $S^V=S(\Z/2)$, which is not invertible in $\SH_{S^1}(\Z/2)$. Therefore, $F_{\C/\R}(S^V)$ cannot be invertible in $\SH_{S_s^1}(\R)$. This shows that it is not possible to extend $c'_k:\sSet(\Z/2)\rightarrow \sPre(\Sm/k)$ to a functor from $\SH(\Z/2)$ to $\SH_{S_s^1}(k)$.
\end{rk}

\chapter{Thick ideals discovered by comparison functors}\label{ideals}

The aim of this chapter is to draw conclusions concerning thick subcategories and thick ideals in $\SH(k)$, $k\subseteq \C$, and in finite, local versions of this category using the functors from the previous chapter. In the next chapter, we will study thick ideals that are described by motivic Morava K-theories.

\section{Consequences of the properties of $R_k$, $R'_k$, $c_k$ and $c'_k$}\label{consequences}

\begin{prop}\label{f}
\begin{compactenum}[(1)]
\item If $\mathcal C\subseteq \mathcal{SH}$ is a thick subcategory or a thick ideal, then, for $k\subseteq\C$, $R_k^{-1}(\mathcal C)\subseteq\mathcal{SH}(k)$ is a thick subcategory or a thick ideal, respectively. 
\item If $\mathcal C\subseteq \mathcal{SH}(\Z/2)$ is a thick subcategory or a thick ideal, then, for $k\subseteq\R$, $(R'_k)^{-1}(\mathcal C)\subseteq\mathcal{SH}(k)$ is a thick subcategory or a thick ideal, respectively. 
\item If $f:k\hookrightarrow K$ and $\mathcal C\subseteq \SH(K)$ is a thick subcategory or a thick ideal, then $(f^\ast)^{-1}(\mathcal C)\subseteq\mathcal{SH}(k)$ is a thick subcategory or a thick ideal, respectively. 
\end{compactenum}
\end{prop}

\begin{proof}
Any functor preserves retracts, hence $R_k^{-1}(\mathcal C)$ is closed under retracts whenever $\mathcal C$ is. We also have $R_k(S_s^1)=S^1$ and $R_k$ preserves cofibers because it is a left adjoint, hence it preserves exact triangles. Therefore, $R_k^{-1}(\mathcal C)$ is closed under triangles whenever $\mathcal C$ is. Hence, $R_k^{-1}$ preserves thick subcategories. Since $R_k$ is symmetric monoidal (see Chapter \ref{functors}), $X\in R_k^{-1}(\mathcal C)$ and $Y\in\SH(k)$ implies that $R_k(X\wedge Y)\cong R_k(X)\wedge R_k(Y)$ is in $\mathcal C$, if $\mathcal C$ is a thick ideal. Thus, $X\wedge Y\in R_k^{-1}(\mathcal C)$. That is, $R_k^{-1}$ preserves thick ideals, too.

The proofs for $(R'_k)^{-1}$ and $(f^\ast)^{-1}$ are the same.
\end{proof}

\begin{prop}\label{c}
For $k\subseteq\C$, $c_k^{-1}$ preserves thick subcategories and thick ideals. 
Similarly, for $k\subseteq\R$, $(c'_k)^{-1}$ preserves thick subcategories and thick ideals.
\end{prop}

\begin{proof}
Since $S_s^1=c_k(S^1)$ and $c_k$ preserves mapping cones, $c_k$ preserves exact triangles. It also preserves retracts and is strict symmetric monoidal, hence, $c_k^{-1}$ preserves thick subcategories and thick ideals.
\end{proof}

\section{Finite motivic spectra}

The thick subcategory theorem of \cite{HS} concerns the category of finite spectra, $\SH^{fin}$, as defined in Chapter \ref{introduction}. The functors $R_k$ and $c_k$ can therefore only help us to understand subcategories of $\SH(k)$ which are at most as big as $R_k^{-1}(\SH^{fin})$. There are multiple equivalent possibilities to define $\SH^{fin}$, using the notions of finite CW-spectra, dualisable objects or compact objects. These notions are not equivalent in the motivic setting, therefore we obtain more than one possible category of ``finite'' objects in $\SH(k)$.

We will now discuss the various versions of finiteness. Let $k$ be any field.

\begin{defn}\label{fin}
\begin{compactenum}[(1)]
\item The category $\SH(k)^{fin}$ of finite cellular motivic spectra over a field $k$ is the smallest full subcategory of $\SH(k)$ that contains the spheres $S^{p,q}$ for all $p,q\in\Z$ and is closed under exact triangles \cite[Definition 8.1]{DI}. 

\item For $k\subseteq \R$, let $\SH(k)^{fin+}$ be the smallest full subcategory of $\SH(k)$ that contains $S^{p,q}\wedge (\Spec k[i])^{\wedge m}_+$ for all $p,q\in\Z$, $m\geq 0$ and is closed under exact triangles.
\item The closures of $\SH(k)^{fin}$ and $\SH(k)^{fin+}$ under colimits are denoted by $\SH(k)^{cell}$ and $\SH(k)^{cell+}$. Their objects are called cellular, see \cite[Definitions 2.1 and 2.10]{DI}.
\end{compactenum}
\end{defn}

\begin{rk}\label{constant in finite}
Note that these categories are closed under $\wedge $ because so are their sets of generators and because $\wedge$ preserves exact triangles and colimits, as it is a left adjoint.

With this definition, $\SH(k)^{fin}$ is the smallest tensor triangulated full subcategory of $\SH(k)$ that contains $c_k(\SH^{fin})$ and is closed under $-\wedge \G_m^{\pm 1}$, and $\SH(k)^{fin,+}$ is the smallest tensor triangulated full subcategory that contains $c'_k(\SH(\Z/2)^{fin})$ and is closed under $-\wedge \G_m^{\pm 1}$.
\end{rk}

The following results can mostly be found in \cite[Section 4]{MLE}. 

\begin{defn}
\begin{compactenum}[(1)]
\item Let $\mathcal D\subseteq\SH(k)$ be the collection of all (strongly) dualisable objects. That is, all spectra $X$ such that the canonical map $F(X,S)\wedge Y\rightarrow F(X,Y)$ is an isomorphism for all $Y\in \SH(k)$, where $F(-,-)$ denotes the derived internal hom in $\SH(k)$ and $S=S^{0,0}$ is the sphere spectrum. $F(X,S)$ is called the dual of $X$ and is also denoted by $DX$ (compare Definition \ref{dualizable}).
\item A motivic spectrum $F\in\SH(k)$ is called compact if $\Hom_{\SH(k)}(F,-)$ preserves arbitrary sums. Let $\SH(k)_f\subseteq\SH(k)$ denote the full subcategory of compact objects.
\end{compactenum}
\end{defn}

\begin{rk}\label{spec compact}
\begin{compactenum}[(1)]
\item $\SH(k)_f$ is a thick subcategory of $\SH(k)$ \cite[Section 4]{MLE}.
\item Any dualisable object is also compact, as shown in the proof of Proposition \ref{equiv compact}.
\item The smash product of two dualisable objects $X$ and $Y$ is again dualisable, because $ F(X\wedge Y,S)\wedge Z \cong F(X,F(Y,S))\wedge Z\cong F(X,S)\wedge F(Y,S)\wedge Z\cong F(X,S)\wedge F(Y,Z) \cong F(X,F(Y,Z)) \cong F(X\wedge Y, Z) $. 

Similarly, compact objects are closed under $\wedge$.
\item By \cite[Cor. 2.14 and Thm. 4.1]{HuBC}, $F(\Spec (k[i])_+,E)\cong \Spec(k[i])_+\wedge E$ in $\SH(k)$, $k\subseteq\R$. That is, $\Sigma^\infty\Spec(k[i])_+$ is self-dual and in particular compact.
\end{compactenum}
\end{rk}

\begin{defn}
For $\mathcal R$ a collection of objects in $\SH(k)_f$, let $\SH(k)_{\mathcal R,f}\subseteq\SH(k)_f$ be the smallest thick subcategory containing $\mathcal R$.

Let $\mathcal T_k=\{S^{p,q}\; |\; p,q\in\Z\}$ be the collection of all motivic spheres in $\SH(k)$ and let $\mathcal T_k^+=\{S^{p,q}\wedge (\Spec k[i])^{\wedge m}_+\; |\; p,q\in\Z,m\geq 0\}$ if $k\subseteq \R$. These are sets of compact objects in $\SH(k)$ by Remark \ref{spec compact}, parts (2) and (4). 
\end{defn}

Comparing the definitions, we see that $\SH(k)_{\mathcal T_k,f}$ is the closure of $\SH(k)^{fin}$ under retracts and $\SH(k)_{\mathcal T_k^+,f}$ is the closure of $\SH(k)^{fin+}$ under retracts. 

\begin{prop}
$\SH(k)_{\mathcal D,f}\subseteq\SH(k)_f$ is the full subcategory of dualisable objects of $\SH(k)$.
\end{prop}

\begin{proof}
This is  \cite[Lemma 4.2]{MLE}. By Remark \ref{spec compact}(2), $\mathcal D$ is a collection of compact objects. Furthermore, the full subcategory spanned by $\mathcal D$ is already a thick subcategory.
\end{proof}

Since $\SH(k)_{\mathcal D,f}\subseteq\SH(k)_f$ is a thick subcategory and $\SH(k)_f\subseteq\SH(k)$ is a thick subcategory, it follows that the strongly dualisable objects form a thick subcategory of $\SH(k)$. Note also that all the categories mentioned above are closed under $\wedge$.

\begin{prop}\label{prop dualizable}
We have $\SH(k)^{fin}\subseteq\SH(k)_{\mathcal T_k,f}\subseteq \SH(k)_{\mathcal D,f}\subseteq \SH(k)_f$ and for $k\subseteq \R$, $\SH(k)^{fin+}\subseteq\SH(k)_{\mathcal T_k^+,f}\subseteq \SH(k)_{\mathcal D,f}\subseteq \SH(k)_f$. In particular, all objects in $\SH(k)^{fin}$ and $\SH(k)^{fin+}$ are strongly dualisable. Furthermore, $\SH(k)^{fin}$ and $\SH(k)^{fin+}$ are closed under taking duals. 
\end{prop}

\begin{proof}
The first line and the case $\SH(k)^{fin}$ are proven in \cite[Section 4]{MLE}. For $k\subseteq\R$, the only additional input is the self-duality of $\Sigma^\infty\Spec(k[i])_+$.
\end{proof}

\begin{rk}
\cite[Remark 8.2]{MLE} gives an example for an object in $\SH(S)$ that is compact but not dualisable, where $S$ is the spectrum of a discrete valuation ring.
\end{rk}

A stronger result holds if $k$ is a field of characteristic $0$. It is also proven in \cite{RiouSWDuality}.

\begin{prop}
Let $k$ be of characteristic $0$. Then $\SH(k)_{\mathcal D,f} = \SH(k)_f$ is the thick subcategory of $\SH(k)$ generated by $\{\Sigma^{2n,n}\Sigma^\infty U_+\; |\; U\in\Sm/k, n\in\Z\}$. Hence, any object of $\SH(k)$ is dualisable if and only if it is compact. 
\end{prop}

\begin{proof}
By \cite[Theorem 9.2]{DI}, $\{\Sigma^{2n,n}\Sigma^\infty U_+\; |\; U\in\Sm/k,\, n\in\Z\}$ is a set of compact generators for $\SH(k)$, which means two things: First, these objects are compact and second, the only full triangulated subcategory of $\SH(k)$ containing this set and being closed under infinite direct sums is $\SH(k)$ itself. Since schemes are locally affine, also $\{\Sigma^{2n,n}\Sigma^\infty U_+\; |\; U\in\Sm/k \textup{ quasi-projective}\}$ is a set of compact generators. 
General theory \cite[Theorem 13.1.14]{MS} implies that $\SH(k)_f$ is the thick subcategory of $\SH(k)$ generated by $\{\Sigma^{2n,n}\Sigma^\infty U_+\; |\; U\in\Sm/k \textup{ quasi-projective}\}$. By \cite[Theorem 4.9]{RO}, $\Sigma^\infty U_+$ is dualisable for any such $U$. Since dualisability is preserved by exact triangles and retracts, the thick subcategory generated by $\{\Sigma^{2n,n}\Sigma^\infty U_+\; |\; U\in\Sm/k\textup{ quasi-projective}\}$ is contained in $\SH(k)_{\mathcal D, f}$. Thus, $\SH(k)_f=\SH(k)_{\mathcal D,f}$.
\end{proof}

For $k\subseteq\C$ ($k\subseteq\R$) all these categories are furthermore included in the preimage of compact topological spectra under $R_k$ ($R_k'$):

\begin{prop}\label{compact1}\label{compact2}
For $k\subseteq\C$, 
\[\SH(k)_f\subseteq R_k^{-1}(\SH^{fin})\] and for $k\subseteq\R$, \[\SH(k)_f\subseteq R_k'^{-1}(\SH(\Z/2)_f).\]
\end{prop}

\begin{proof}
We have to show that $R_k$ and $R_k'$ preserve compact objects.
Let $f:k\hookrightarrow \C$. Then $f^\ast$ restricts to a functor between the categories of compact objects: $f^\ast:\SH(k)_f\rightarrow \SH(\C)_f$, because the base change functor preserves smooth schemes and so $f^\ast$ sends a compact generator $\Sigma^{(2n,n)}\Sigma^\infty U_+$, $U\in\Sm/k$, $n\in\Z$, of $\SH(k)$ to a compact generator of $\SH(\C)$. Hence, for the first claim it suffices to prove that $R=R_\C$ preserves compact objects. Similarly, for the second claim it suffices to show that $R'=R_\R'$ preserves compact objects.

Let $X=\Sigma^{(2n,n)}\Sigma^\infty U_+$ be a compact generator of $\SH(\C)$. As in the proof of the previous proposition, we can assume that $U$ is a smooth quasi-projective scheme. By Jouanolou's trick \cite[Lemma 1.5]{Jouanolou}, there exists an affine vector bundle torsor over $U$. This is a vector bundle $E\rightarrow U$, together with a torsor $p:V\rightarrow U$ on $E$ with $V$ affine. By the definition of a torsor, $V$ is locally isomorphic to $E$. This implies that for some $m$, $U\times \A^m$ is locally isomorphic to an affine smooth scheme $V$. In particular, $U$ is $\A^1$-equivalent to $V$, and so $X\cong \Sigma^{(2n,n)}\Sigma^\infty V_+$ in $\SH(\C)$. Now, since $V$ is smooth and affine, its complex realisation has the homotopy type of a finite CW complex by \cite[Example 3.1.9]{Lazarsfeld}. Hence, $R(X)\in\SH$ is isomorphic to $\Sigma^n\Sigma^\infty Y_+$ for some finite CW complex $Y$ and, so, $R(X)$ is compact by Remark \ref{SH fin}. 

The proof that for any smooth and affine variety $V$, $V(\C)\subseteq \C^r$ has the homotopy type of a finite CW-complex, can be summarised as follows \cite[Example 3.1.9]{Lazarsfeld}: $V(\C)$ is a complex submanifold of $\C^r$ without boundary (because $V$ is smooth, see e.g. \cite[Section 3.1.2]{AlGeo}), which is closed as a subset of $\C^r$ (because the zero locus of any polynomial is closed). For almost any $c\in\C^r$, the squared distance function $\phi_c:V(\C)\rightarrow \R$, $\phi_c(x)=||x-c||^2$, has only non-degenerate critical points \cite[Theorem 6.6]{Morse}. Furthermore, $\phi_c$ being real algebraic implies that it has only finitely many critical points, as in \cite[Example 3.1.9]{Lazarsfeld}. Using Morse theory \cite[Theorem 3.5]{Morse}, it follows that $V(\C)$ has the homotopy type of a CW complex with one cell of dimension $n$ for each critical point of $\phi_c$ of index $n$.

For $X$ a compact generator of $\SH(\R)$, we also have $X\cong \Sigma^{(2n,n)}\Sigma^\infty V_+$ and now $V$ is a smooth and affine real variety. Its realisation $R'(V)=V(\C)\subseteq\C^r$ is still a complex manifold, which is closed, but with the property that, for any $x\in V(\C)$, also its complex conjugate $\bar x$ lies in $V(\C)$. This was used for the definition of the $\Z/2$-action on $V(\C)$: $\rho(x)=\bar x$ for $\rho\in\Z/2$ the generator. As before, we can choose $c\in\C^r$ such that $\phi_c$ has finitely many critical points, which are all non-degenerate. If $c\in\R^r$, $\phi_c$ is $\Z/2$-invariant and the claim follows from equivariant Morse theory: by the proof of \cite[Theorem 2.2]{Mayer}, any invariant Morse function can be turned into a special invariant Morse function and, by \cite[Theorem 3.3]{Mayer}, a manifold with a special invariant Morse function is equivariantly homotopy equivalent to an equivariant CW complex with one equivariant cell for each critical orbit. Thus, $V(\C)$ is equivalent to a finite $\Z/2$-CW complex, whose suspension spectrum is compact in $\SH(\Z/2)$ by Proposition \ref{equiv compact}. For more information on equivariant Morse theory in English language, we refer the reader to \cite{Wasserman}.

Now, assume $c\in\C^r\setminus\R^r$. We would like to take $x\mapsto\min(\phi_c(x),\phi_c(\bar{x}))$ as a Morse function. It is continuous and $\Z/2$-invariant but it is not differentiable for $x\in\R^r$. Outside of $\R^r$, the critical points of $\min(\phi_c(x),\phi_c(\bar{x}))$ are a subset of the critical points of $\phi_c$ and their complex conjugates. The idea is to proceed in two steps: first, to take care of the real part $V(\C)\cap \R^r=V(\R)=V(\C)^{\Z/2}$ and, second, to use $\min(\phi_c(x),\phi_c(\bar{x}))$ as a Morse function away from $\R^r$. 

$V(\R)\subseteq\R^r$ is a manifold, which is closed and without boundary and for which there exists $d\in\R^r$ such that $\phi_d:V(\R)\rightarrow \R$ is a Morse function with finitely many critical points (analogously to $V(\C)\subseteq\C^r$). Let $B$ be an open ball around $0\in\C^r$ which contains a ball $D(d)$ around $d$ containing all critical points of $\phi_d$ and a ball $D(c)$ around $c$ containing all critical points of $\phi_c$ and their conjugates. $V(\R)$ can be contracted inside $\R^r$ to $V(\R)\cap \overline{B}$, following the orthogonal trajectories of the hypersurfaces on which $\phi_d$ is constant, until $\overline{B}$ is reached (that is, we glue an infinite sequence of the diffeomorphisms $M^a\cong M^b$ constructed in the proof of \cite[Theorem 3.1]{Morse} to one homotopy equivalence). We extend this homotopy equivalence to a narrow open neighborhood $U\subset\C^r$ of $\R^r\setminus (B\cap\R^r)$ so that $\C^r\simeq\C^r\setminus U$ by a $\Z/2$-equivariant homotopy equivalence which does not add critical points to $\min(\phi_c(x),\phi_c(\bar{x}))$ outside of $B$. Using this homotopy equivalence, we can assume that $V(\C)\setminus B$ does not contain any points near the real subspace $\R^r$. Now we use $\phi_c$ to contract $V(\C)$ to $V(\C)\cap\bar{B}$ with an equivariant homotopy equivalence: any $x\in V(\C)\setminus B$ with $\phi_c(x)<\phi_c(\bar{x})$ gets moved into $B$ along the line of steepest descent of $\phi_c$ inside $V(\C)$. If, however, $\phi_c(\bar{x})<\phi_c({x})$, $x$ gets moved along the line of steepest descent of $\phi_{\bar{c}}$. Note that $\phi_{\bar{c}}(x)=\phi_c(\bar{x})$ and that the homotopy equivalence is well defined because we removed $\R^r$ before and because all critical points of $\phi_c$ and $\phi_{\bar{c}}$ are inside $B$. Hence, this defines a $\Z/2$-equivariant homotopy equivalence between $V(\C)$ and $V(\C)\cap\bar{B}$.

Thus, $V(\C)$ is $\Z/2$-homotopy equivalent to a $\Z/2$-manifold (with boundary) which is closed and bounded and, hence, compact in $\C^r$. Using equivariant Morse theory, this compact manifold is $\Z/2$-homotopy equivalent to a finite $\Z/2$-CW complex (again, see \cite[Theorems 2.2, 3.3]{Mayer} or \cite{Wasserman}). Note that, although most of the literature only studies Morse theory for manifolds without boundary, it still covers this case for the following reason: Since $D(c)\subset B$, the Morse function $\phi_c$ evaluated on the boundary of $V(\C)\cap \overline{B}$ takes a higher value than on any critical point. Therefore, $V(\C)\cap \overline{B}$ appears at a finite step in the Morse theoretical construction of the CW complex associated with $\phi_c$, i.e., $V(\C)\cap \overline{B}\simeq V(\C)^{m}=\phi_c^{-1}\left([0,m]\right)$ for some $m\in\R$, which is a finite $\Z/2$-CW complex. 
\end{proof}

\section{Motivic thick ideals}\label{Motivic thick ideals}

Let $k\subseteq \C$ and $p$ be any prime. The following theorem identifies important families of thick ideals in $(\SH(k)_f)_{(p)}$. It is the main result in this chapter.

\begin{theorem}\label{lower bound} {\bf (Lower bound on the number of motivic thick ideals)}
\begin{compactenum}[(1)]
\item
The category $(\SH(k)_f)_{(p)}$ contains at least an infinite chain of different thick ideals, given by $\overline{R}_k^{-1}(\mathcal C_n)$, $0\leq n\leq\infty$, where $\overline{R}_k$ denotes the $p$-localisation of the restriction of $R_k$ to $\SH(k)_f$ and $\mathcal C_n\subseteq\SH^{fin}_{(p)}$ is as defined in Chapter \ref{introduction}.
\item
If $k\subseteq\R$, then $(\SH(k)_f)_{(p)}$ contains at least a two-dimensional lattice of different thick ideals, given by $(\overline{R}'_k)^{-1}(\mathcal C_{m,n})$, for all $(m,n)\in\Gamma_p$ as in Definition \ref{equivariant type}. 
\end{compactenum}
\end{theorem}

\begin{proof}
We prove the first part, the second is proven similarly. By Remark \ref{constant in finite}, Proposition \ref{prop dualizable} and Proposition \ref{compact1}, $c_k(\SH^{fin})\subseteq \SH(k)_f\subseteq R_k^{-1}(\SH^{fin})$, hence $c_k$ and $R_k$ restrict to functors  $\SH^{fin}\overset{c_k}\longrightarrow \SH(k)_f \overset{\overline{R}_k}\longrightarrow \SH^{fin}$.
Since the motivic $p$-local Moore spectrum (the homotopy colimit of a diagram of sphere spectra whose arrows are multiplications by integers prime to $p$) is the image of the topological $p$-local Moore spectrum under $c_k$, we get induced functors between the localised categories,
\[\SH^{fin}_{(p)}\overset{c_k}\longrightarrow (\SH(k)_f)_{(p)}\overset{\overline{R}_k}\longrightarrow \SH^{fin}_{(p)},\] 
which still have the properties that $c_k$ and $\overline{R}_k$ preserve exact triangles and smash products and that $\overline{R}_k\circ c_k=\id$, as proven in Theorem \ref{Rc}. Similarly to Proposition \ref{f}, it follows that $\overline{R}_k^{-1}(\mathcal C_n)$ is a thick ideal. Now let $n<m$. Then $\mathcal C_m\subset \mathcal C_n$ in $\SH^{fin}_{(p)}$ and there is some $X\in\mathcal C_n\setminus\mathcal C_m$. It follows that $\overline{R}_k^{-1}(\mathcal C_m)\subseteq \overline{R}_k^{-1}(\mathcal C_n)$ and that $c_k(X)\in \overline{R}_k^{-1}(\mathcal C_n)\setminus \overline{R}_k^{-1}(\mathcal C_m)$. Thus, $\overline{R}_k^{-1}(\mathcal C_n)$, $0\leq n\leq\infty$, form a chain of pairwise different thick ideals.

For the second part, one needs to consider the functors 
\[(\SH(\Z/2)_f)_{(p)}\overset{c_k'}\longrightarrow (\SH(k)_f)_{(p)}\overset{\overline{R}_k'}\longrightarrow (\SH(\Z/2)_f)_{(p)}.\]
\end{proof}

In the following, we will omit the overline and use the notation $R_k$, $R'_k$ for the ($p$-localised) restricted functors, too.

\begin{rk}
In the above theorem, $\SH(k)_f$ can be replaced by any other tensor triangulated full subcategory $\mathcal D\subseteq\SH(k)$ satisfying 
\[c_k(\SH^{fin})\subseteq \mathcal D\subseteq R_k^{-1}(\SH^{fin}),\;\textup{or } \]
\[c'_k(\SH(\Z/2)^{fin+})\subseteq \mathcal D\subseteq (R'_k)^{-1}(\SH(\Z/2)_f)\;\textup{respectively}.\]
In particular, the theorem applies to any 
\[\mathcal D\in\{\SH(k)^{fin},\;\SH(k)^{fin+},\; \SH(k)_{\mathcal T_k,f},\; \SH(k)_{\mathcal T^+_k,f},\; R^{-1}_k(\SH^{fin}),\; R'^{-1}_k(\SH(\Z/2)_f)\}\]
for the following reasons. 
Recall from Remark \ref{fin and f} that Strickland's characterisation of equivariant thick ideals works in $\SH(G)^{fin}$ as well as in its closure under retracts, $\SH(G)_f$. Therefore, we can here take $\SH(k)^{fin(+)}$ as well as its closure under retracts, $\SH(k)_{\mathcal T_k^{(+)},f}$. Note also that all categories mentioned here are closed under $\wedge$ because they are generated as thick or triangulated subcategories by classes of objects closed under $\wedge$: $\SH(k)^{fin(+)}$ and $\SH(k)_{\mathcal T^{(+)}_k,f}$ are generated by $\{S^{p,q}(\wedge \Spec(k[i])^{\wedge m}_+)\}$ and $\SH(k)_f$ is generated by smooth schemes, which are also closed under smash product. $R_k^{-1}(\SH^{fin})$ and $(R'_k)^{-1}(\SH(\Z/2)_f)$ are closed under $\wedge$ because $\wedge$ commutes with $R_k$, $R'_k$.
\end{rk}

\begin{rk}
The construction of the functors $R_k$ and $c_k$ does not depend on the fact that $\P^1$ is invertible in $\SH(k)$. They can also be constructed for the category $\SH_{S_s^1}(k)$ in which only $S_s^1=S^{1,0}$ got inverted and $\G_m$ did not. Therefore, part (1) of the theorem also holds for $(\SH_{S_s^1}(k)_f)_{(p)}$. The construction of $c'_k$, however, needed the invertibility of $\P^1$ (see Remark \ref{S1 spectra}). Thus, part (2) cannot be applied to $(\SH_{S_s^1}(k)_f)_{(p)}$.
\end{rk}

\begin{defn}
For a full subcategory $\mathcal C$ of a tensor triangulated category $\mathcal T$, let $\Delta(\mathcal C)$ denote the smallest thick subcategory of $\mathcal T$ that contains $\mathcal C$ and recall that $\thickid(\mathcal C)$ denotes the smallest thick ideal that contains $\mathcal C$.
\end{defn}

We state two more observations about thick ideals in the category $(\SH(k)_f)_{(p)}$. They also hold in any of the above categories $\mathcal D$.

\begin{prop}
Let $X_n\in \SH^{fin}_{(p)}$ be any spectrum of type $n$, i.e. $X_n\in\mathcal C_{n}\setminus\mathcal C_{n+1}$ (see Definition \ref{type}) and let $X_{m,n}\in\SH(\Z/2)^{fin}_{(p)}$ be any spectrum of type $(m,n)$ (see Definition \ref{equivariant type}). Then  
\[\thickid(c_k X_n)=\thickid(c_k \mathcal C_{n})\;\textup{ \emph{in} } (\SH(k)_f)_{(p)}\textup{ \emph{if} }k\subseteq\C\textup{ \emph{and} }\]
\[\thickid(c'_k X_{m,n})=\thickid(c'_k \mathcal C_{m,n}) \;\textup{ \emph{in} } (\SH(k)_f)_{(p)}\textup{ \emph{if} }k\subseteq\R.\]
\end{prop}

\begin{proof}
It is clear that $\thickid(c_k X_n)\subseteq\thickid(c_k \mathcal C_{n})$. Since $c_k^{-1}$ preserves thick ideals by Proposition \ref{c}, $c_k^{-1}(\thickid(c_k X_n))$ is a thick ideal containing $X_n$. Since $\mathcal C_{n}$ is the smallest thick ideal containing $X_n$, we have $\mathcal C_{n}\subseteq c_k^{-1}(\thickid(c X_n))$ and hence $c_k \mathcal C_{n}\subseteq \thickid(c_k X_n)$, which implies the first claim. The same proof shows the second claim.
\end{proof}

\begin{prop}
If $X,Y\in (\SH(k)_f)_{(p)}$ with $\type(R_k X)\neq\type(R_k Y)$ (see Definition \ref{type}) or, if $k\subseteq\R$, $\type(R'_k X)\neq\type(R'_k Y)$ (see Definition \ref{equivariant type}), then \[\thickid(X)\neq \thickid (Y).\] 
\end{prop}

\begin{proof}
Let $\type(R_k X)=n> \type(R_k Y)$. Then $\thickid (X)\subseteq R_k^{-1}(\mathcal C_{n})$ but $Y\notin R_k^{-1}(\mathcal C_{n})$. The case of $R'_k$ is similar. 
\end{proof}

The next proposition gives a description of thick ideals in the categories of finite cellular spectra, $\SH(k)^{fin}$, $k\subseteq\C$, and $\SH(k)^{fin+}$, $k\subseteq\R$, from Definition \ref{fin}. In this case, a thick ideal is a thick subcategory that is closed under $-\wedge \G_m^{\pm 1}$ and under $-\wedge\Spec(k[i])_+$ if $k\in\R$.

\begin{prop}\label{closed under Gm}
Let $\mathcal C\subseteq\SH(k)^{fin}$ be a subcategory, $k\subseteq\C$. Then 
\[\thickid(\mathcal C)=\Delta\left(\underset{n\in\Z}\bigcup \mathcal C\wedge \G_m^{\wedge n}\right).\] 
For a subcategory $\mathcal C\subseteq\SH(k)^{fin+}$, $k\subseteq\R$, we have 
\[\thickid(\mathcal C)=\Delta\left(\underset{n\in\Z}\bigcup (\mathcal C\wedge \G_m^{\wedge n}) \cup \underset{n\in\Z}\bigcup (\mathcal C\wedge \G_m^{\wedge n}\wedge\Spec(k[i])_+)\right).\]
\end{prop}

\begin{proof}
We prove the second claim, the proof of the first claim is slightly shorter. Let $k\subseteq\R$. 
The subcategory $\thickid(\mathcal C)$ contains 
\[\overline{\mathcal C}=\Delta\left(\underset{n\in\Z}\bigcup (\mathcal C\wedge \G_m^n) \cup \underset{n\in\Z}\bigcup (\mathcal C\wedge \G_m^n\wedge\Spec(k[i])_+)\right)\] 
because $\thickid(\mathcal C)$ is closed under $-\wedge\G_m^n$ and $-\wedge\Spec(k[i])_+$ and is a thick subcategory.
We have to show that the triangulated subcategory $\overline{\mathcal C}$ is already a thick ideal. Let $\mathcal D\subseteq \SH(k)^{fin+}$ be the full subcategory consisting of all objects $X$ such that $\overline{\mathcal C}$ is closed under $-\wedge X$. We have to show $\mathcal D=\SH(k)^{fin+}$. First note that $\mathcal D$ contains all spheres $S^{p,q}$ because, as a triangulated subcategory, $\overline{\mathcal C}$ is closed under $-\wedge S^{p,0}$ and because we added $S^{q,q}$. 
If $X\rightarrow Y\rightarrow Z$ is a triangle with two objects in $\mathcal D$, then $A\wedge X\rightarrow A\wedge Y\rightarrow A\wedge Z$ is a triangle with two objects in $\overline{\mathcal C}$ for any $A\in\overline{\mathcal C}$, hence the third object is also in $\overline{\mathcal C}$. It follows that $\mathcal D$ is closed under exact triangles. Furthermore, $\mathcal D$ is closed under $\wedge$ by definition. It remains to show that $\mathcal D$ contains $\Sigma^\infty_T\Spec(k[i])_+$. By the equivalence 
\[\Spec(k[i])_+\wedge \Spec (k[i])_+\cong \Spec(k[i])_+\vee\Spec(k[i])_+\cong \cof\left(\Spec (k[i])_+\overset{0}\rightarrow \Spec (k[i])_+\right),\] 
any triangulated subcategory of $\SH(k)$ containing the spectrum $\Sigma^\infty_T\Spec(k[i])_+$ also contains $\Sigma^\infty_T\Spec(k[i])_+^{\wedge l}$, $l\geq 1$. It follows that $\overline{\mathcal C}$ contains $\mathcal C\wedge\G_m^n\wedge\Spec(k[i])_+^{\wedge l}$, $l\geq 0$. Hence, $\mathcal D$ contains $\Sigma^\infty_T\Spec(k[i])_+$ and is thus equal to $\SH(k)^{fin+}$.
\end{proof}

For $k\subseteq\C$, we have so far identified the following thick ideals in $(\SH(k)_f)_{(p)}$:
\[\xymatrix@1{
R_k^{-1}(\mathcal C_{0})\ar @{} [d] |-*[@]{{\vsupseteq}} \ar @{} [r] |-*[@]{\supset} &  R_k^{-1}(\mathcal C_{1})\ar @{} [d] |-*[@]{{\vsupseteq}} \ar @{} [r] |-*[@]{\supset} &  R_k^{-1}(\mathcal C_{2})\ar @{}[d] |-*[@]{{\vsupseteq}} \ar @{} [r] |-*[@]{\supset}  &\cdots\\
\thickid(c_k\mathcal C_{0}) \ar @{} [r] |-*[@]{\supset} &  \thickid(c_k\mathcal C_{1}) \ar @{} [r] |-*[@]{\supset} & \thickid(c_k\mathcal C_{2}) \ar @{} [r] |-*[@]{\supset} &\cdots}\]

For $k\subseteq\R$, the picture has another dimension and depends on the classification of thick ideals in the equivariant category as described in Chapter \ref{equivariant}. There is at least one spot where the inclusion from the lower row into the upper row is actually an equality.

\begin{prop}
\[R_k^{-1}(\mathcal C_{0})=\thickid(c_k(\mathcal C_{0}))=(\SH(k)_f)_{(p)}\;\textup{ and}\]
\[(R'_k)^{-1}(\mathcal C_{0,0})=\thickid(c'_k(\mathcal C_{0,0}))=(\SH(k)_f)_{(p)}.\]
\end{prop}

\begin{proof}
This is because all these subcategories contain the sphere spectrum and are closed under smashing with arbitrary elements of $(\SH(k)_f)_{(p)}$.
\end{proof}

\chapter{Thick ideals associated with cohomology theories}\label{idealsCohomo}

\section{Equivalence of homology and cohomology theories}

We will now concentrate on the categories of finite cellular spectra $\SH(k)^{fin}$, $\SH(k)^{fin+}$, and on (co)homology theories represented by objects in $\SH(k)^{cell}$ or $\SH(k)^{cell+}$ respectively, because these satisfy some useful additional properties. A couple of these are proven in \cite{DI}. 

A consequence of the results in \cite{DI} is the following proposition, which states that cellular homology theories and cohomology theories for $\SH(k)^{fin}$ are exchangeable. We will state the analogous result for $\SH(k)^{fin+}$ below, in Corollary \ref{co-homo real}.

\begin{prop}\label{co-homo}
Let $E\in\SH(k)^{cell}$ (see Definition \ref{fin}) be a ring spectrum and $X\in\SH(k)^{fin}$. Then $E_{\ast\ast}(X)=0$ if and only if $E^{\ast\ast}(X)=0$.
\end{prop}

\begin{proof}
In the universal coefficient spectral sequence (see \cite[Proposition 7.7]{DI} or Proposition \ref{ucss}),
\[E_2=\Ext_{E_{\ast\ast}}^{a,b,c}(M_{\ast\ast},N_{\ast\ast})\Rightarrow \pi_{-a-b,-c}F_E(M,N),\]
we set $M=X\wedge E$ and $N=E$:
\[E_2=\Ext_{E_{\ast\ast}}^{a,b,c}(E_{\ast\ast}(X),E_{\ast\ast})\Rightarrow \pi_{-a-b,-c}F_E(X\wedge E,E).\]
Since $F_E(X\wedge E,E)=F(X,E)$, the spectral sequence converges conditionally to $E^{\ast\ast}(X)$. If $E_{\ast\ast}(X)=0$, the spectral sequence collapses and thus converges strongly to $E^{\ast\ast}(X)$, which, hence, is $0$. For the other direction, we set $M=F(X,E)$ and $N=E$. Note that this $M$ is cellular because $X$ is dualisable and its dual is again a finite cell spectrum by Proposition \ref{prop dualizable}. So we get
\[E_2=\Ext_{E_{\ast\ast}}^{a,b,c}(E^{\ast\ast}(X),E_{\ast\ast})\Rightarrow \pi_{-a-b,-c}F_E(F(X,E),E).\]
Now $F_E(F(X,E),E)=F_E(D(X)\wedge E,E)=F(D(X),E)=X\wedge E$, so the sequence converges conditionally to $E_{\ast\ast}(X)$. Hence, $E^{\ast\ast}(X)=0$ implies $E_{\ast\ast}(X)=0$.
\end{proof}

One important result of Dugger and Isaksen, \cite[Proposition 7.1]{DI}, states that for cellular objects $E\in\SH(k)^{cell}$, $\pi_{\ast\ast}E=0$ implies $E\cong 0$. We will show that an adjusted statement holds for $E\in\SH(k)^{cell+}$, $k\subseteq\R$. 

We know from equivariant stable homotopy theory that a generalisation of homotopy groups is needed to obtain the corresponding result in $\SH(G)$. In $\SH(\Z/2)$, for example, we have the equivariant homotopy groups 
\[\pi_n^{\Z/2}(X)=[S^n,X]_{\Z/2}=[S^n,X^{\Z/2}]_{\{1\}}\textup{ and }\pi_n^{\{1\}}(X)=[S^n\wedge \Z/2_+,X]_{\Z/2}=[S^n,X]_{\{1\}}.\]
We can similarly define homotopy groups in $\SH(k)$, $k\subseteq\R$ such that $R'_k$ maps $\pi_{p,q}^{k}(X)$ to $\pi_p^{\Z/2}(R'_k X)$ and $\pi_{p,q}^{k[i]}(X)$ to $\pi_p^{\{1\}}(R'_k X)$.

\begin{defn}\label{astast}
For $k\subseteq\R$ and $X\in\SH(k)$, let
\[\pi_{p,q}^{k}(X)=[S^{p,q},X]_{\SH(k)}\textup{ and }\pi_{p,q}^{k[i]}(X)=[S^{p,q}\wedge \Spec(k[i])_+,X]_{\SH(k)}.\]
We write $\pi^+_{\ast\ast}(X)=0$ if $\pi_{\ast\ast}^K(X)=0$ for both $K=k$ and $K=k[i]$. Furthermore, $E^+_{\ast\ast}(X)=0$ will mean $\pi^+_{\ast\ast}(E\wedge X)=0$ and $(E^+)^{\ast\ast}(X)=0$ will mean $\pi^+_{\ast\ast}(F(X,E))=0$.
\end{defn}

The same arguments as in \cite[Proposition 7.1]{DI} now imply the following.

\begin{prop}\label{7.1}
If $X\in\SH(k)^{cell+}$, $k\subseteq\R$ and $\pi^+_{\ast\ast}(X)=0$ then $X\cong 0$.
\end{prop}

\begin{proof}
Assuming $X$ is cofibrant and fibrant, one considers the class $\mathcal D$ of all $Y$ such that $\Map(Y^{cof},X)$ is contractible. By assumption, $\mathcal D$ contains $S^{p,q}$ and $S^{p,q}\wedge \Spec(k[i])_+$. Furthermore, $\mathcal D$ is closed under isomorphisms, (de-)suspensions and homotopy colimits. As in the proof of Proposition \ref{closed under Gm}, it follows that it also contains $S^{p,q}\wedge (\Spec(k[i])_+)^{\wedge m}$, $m\geq 0$. As $\SH(k)^{cell+}\subseteq\SH(k)$ is generated by $\{S^{p,q}\wedge(\Spec(k[i])_+)^{\wedge m}\; |\; $ $p,q\in\Z,m\geq 0\}$ under isomorphisms and homotopy colimits, it follows that $\SH(k)^{cell+}\subseteq \mathcal D$, in particular $X\in\mathcal D$. Hence, $\Map(X,X)$ is contractible, which implies $X\cong 0$. 
\end{proof}

The following corollary is a $k\subseteq\R$-version of Proposition \ref{co-homo}.

\begin{cor}\label{co-homo real}
Let $E\in\SH(k)^{cell+}$ be a ring spectrum and $X\in\SH(k)^{fin+}$. Then $E^+_{\ast\ast}(X)=0$ if and only if $(E^+)^{\ast\ast}(X)=0$.
\end{cor}

\begin{proof}
The spectral sequence from \cite[Proposition 7.7]{DI} is derived similarly to the universal coefficient spectral sequence in \cite[Section IV.5]{EKMM}. The crucial point is the convergence, where cellularity is needed to apply \cite[Proposition 7.1]{DI}. The equivariant version of this spectral sequence (a spectral sequence of Mackey functors) is proven in \cite{LM}. Using \cite{LM} and Proposition \ref{7.1}, one can derive the universal coefficient spectral sequence for our motivic homotopy groups over $k\subseteq\R$. The new version of Proposition \ref{co-homo} can then be deduced from this spectral sequence as before.
\end{proof}

\section{Thick ideals}\label{thick ideals}

In $\SH^{fin}_{(p)}$, any thick ideal can be described by some cohomology theory (namely, $n$-th Morava K-theory). Conversely, any cohomology theory defines a thick ideal. We are therefore interested in thick ideals of $\SH(k)_f$, as well as $\SH(k)^{fin}$, $\SH(k)^{fin+}$, described by cohomology theories as follows.

\begin{lemma}\label{Lemma C_E}
Let $k$ be any field and let $\mathcal T\subseteq\SH(k)$ be a tensor triangulated subcategory of $\SH(k)$.
For any $E\in\SH(k)$, the full subcategory of $\mathcal T$ given by
\[\mathcal C_E=\{X\in\mathcal T\; |\; X\wedge E\cong 0\}\]
is a thick ideal of $\mathcal T$.
\end{lemma}

\begin{proof}
Since $-\wedge E$ preserves exact triangles, $\mathcal C_E$ is closed under these. If $Y$ is a retract of $X$ and $X\wedge E=0$, then $Y\wedge E\rightarrow 0\rightarrow Y\wedge E$ is the identity on $Y\wedge E$, hence $Y\wedge E\cong 0$. Let $X\in\mathcal C_E$ and $Y\in \mathcal T$. Then $ X\wedge Y\wedge E\cong 0$, hence $X\wedge Y\in\mathcal C_E$.
\end{proof}

\begin{prop}\label{C_E}
Let $\mathcal C_E$ be as defined in the above lemma.

If $k\subseteq\C$, $E\in\SH(k)^{cell}$ and $\mathcal T=\SH(k)^{fin}$ or $\mathcal T=\SH(k)^{cell}$, then
\[\mathcal C_E= \{X\in\mathcal T\; |\; E_{\ast\ast}(X)=0\}.\]
If, furthermore, $E$ is a ring spectrum and $\mathcal T=\SH(k)^{fin}$, then also
\[\mathcal C_E=\{X\in\mathcal T\; |\; E^{\ast\ast}(X)=0\}. \]
If $k\subseteq\R$, $E\in\SH(k)^{cell+}$ and $\mathcal T=\SH(k)^{fin+}$ or $\mathcal T=\SH(k)^{cell+}$, then
\[\mathcal C_E=\{X\in\mathcal T\; |\; (E^+)_{\ast\ast}(X)=0\}.\]
If, furthermore, $E$ is a ring spectrum and $\mathcal T=\SH(k)^{fin+}$, then also
\[\mathcal C_E=\{X\in\mathcal T\; |\; (E^+)^{\ast\ast}(X)=0\}.\]
The same descriptions apply if $\mathcal T$ is the $p$-localisation of any of these categories for some prime $p$.
\end{prop}

\begin{proof}
For $E$ and $X$ as in the first claim, we have $E_{\ast\ast}(X)=0 \Leftrightarrow E\wedge X\cong 0$ by \cite[Proposition 7.1]{DI} and for $X$ finite, $E$ a ring spectrum,  
$E^{\ast\ast}(X)=0\Leftrightarrow E_{\ast\ast}(X)=0 $ by Proposition \ref{co-homo}. 
For $k\subseteq\R$, the same arguments hold 
by Proposition \ref{7.1} and Corollary \ref{co-homo real}. 
\end{proof}

\section{Morava K-theories in $\SH(k)$}\label{section morava K}

Topology suggests that thick ideals described by Morava K-theories are particularly interesting. Therefore, we want to study motivic Morava K-theories and their properties.

\subsection{Construction and properties of $AK(n)$}\label{section AK(n)}

Motivic Morava K-theories were introduced in \cite{Bo}. The rough idea is as follows. One starts with $MGL$, the motivic analogue of $MU$. Since $MU_\ast$ is a subring of $MGL_{\ast\ast}$, elements in $MU_\ast$ can be used to define maps on $MGL$ and to construct motivic spectra which are analogous to certain other topological spectra constructed from $MU$.

\begin{defn}\label{AK(n)}
Let $k\subseteq\C$. 
Let $MGL$ be the algebraic cobordism spectrum as constructed in \cite[Section 3.5]{VMC}, see also \cite[Section 6.3]{VoevA1} or \cite[Section 6.5]{PY}. In \cite[Theorem 10]{Bo}, elements $a_i\in MGL_{2i,i}$, $i\geq 1$, are defined, whose images under $R_k$ are $a_i^{\Top}\in MU_{2i}$. If $E$ is an $MGL$-module, then $E/a_i$ and $a_i^{-1}E$ can be defined as in \cite[Definition 2.10]{HornLoc}.

Note that the functor $MGL\wedge -$ from $\SH(k)$ to the category of $MGL$-modules in $\SH(k)$ has $F(MGL,-)$ as right adjoint and has the forgetful functor as left adjoint. As in \cite[Lemma II.1.3]{EKMM} and \cite[page 99]{Bo}, it follows that the category of $MGL$-modules is complete and cocomplete. Thus, $E/a_i$ and $a_i^{-1}E$ are again $MGL$-modules (see Lemma \ref{module map} for the proof that the action of $a_i$ on $E$ is a map of $MGL$-modules).

The motivic Brown-Peterson spectrum for a fixed prime $p$ was first constructed in \cite[Section 5]{Vez}. By \cite[Remark 6.20]{Hoy}, it is equivalent as an $MGL$-module to $MGL_{(p)}/I$, where $I$ is the image under $MU_\ast\rightarrow MGL_{\ast\ast}$ of a regular sequence in the Lazard ring $L$ that generates the vanishing ideal for $p$-typical formal group laws. We take this as the definition of the motivic Brown Peterson spectrum $ABP$ at the prime $p$. That is, \[ABP=MGL_{(p)}/(a_i\; |\; i\neq p^j-1).\]

With $v_0=p$ and $v_i=a_{p^i-1}$ for $i\geq 1$, we define:
\[AP(n)=ABP/(v_0,\cdots,v_{n-1}) \textup{ and } AB(n)=v_n^{-1}AP(n),\] 
\[Ak(n)= ABP/(v_0,\cdots,v_{n-1},v_{n+1},v_{n+2},\cdots) \textup{ and } AK(n)=v_n^{-1}Ak(n),\] 
\[AE(n)=v_n^{-1} ABP/(v_{n+1},v_{n+2},\cdots).\]
\end{defn}

We will make use of all these spectra in Chapter \ref{chapter BC}. For now, we are mostly interested in $AK(n)$.

\begin{lemma}
Let $Ah$ be one of the motivic spectra mentioned in the above definition, e.g. $Ah=AK(n)$ with $h=K(n)$. Then $R_k(Ah)=h$.
\end{lemma}

\begin{proof}
This follows from $R_k(MGL)=MU$ \cite[Section 3.5]{VMC} and $R_k(v_i)=v_i^{\Top}$ because $R_k$ preserves colimits. 
\end{proof}

\begin{rk}\label{AMU module}
\begin{compactenum}[(1)]
\item All spectra mentioned in the above definition are cellular: Theorem 6.4 in \cite{DI} shows that $MGL$ is cellular. Since the spectra above are constructed from $MGL$ by taking homotopy cofibers and homotopy colimits, they are cellular. The mod-$p$ version of the main result in \cite{Hoy} states that $H\Z/p$ is the cofiber of $v_n:\Sigma^{2(p^n-1),p^n-1}Ak(n)\rightarrow Ak(n)$. Hence, $H\Z/p$ is cellular, too.
\item $ABP$ is a homotopy commutative ring spectrum by construction \cite[Definition 5.3]{Vez} and the orientation of $MGL$ induces an orientation on $ABP$ by \cite[Theorem 1.1]{PPRUniv}. The ring structure of $ABP$ induces an $ABP$-module structure on the quotients of $ABP$ defined above.
\item Since $MGL$ is a ring spectrum, $ABP$ is also an $MGL_{(p)}$-module spectrum and so is
$ABP/I$ for $I\subseteq MU_\ast$ a regular ideal.
\item As remarked at the end of the introduction in \cite{MLE}, $ABP$ is Landweber exact in the sense of \cite{MLE} and their results for $MGL$ also hold for $ABP$.
\item $MGL_{2\ast,\ast}$ and $\h_{\ast\ast}MGL$ (with coefficients in $\Z$ or $\Z/p$) are known (see e.g. \cite[Theorem 5]{Bo} and \cite[Corollary 6.9]{Hoy}). From this, one can compute $Ah_{2\ast,\ast}$ and $\h_{\ast\ast}Ah$ if $Ah$ is one of the above spectra (\cite[Theorem 12]{Bo} and \cite[Lemma 6.10]{Hoy}). More can be said about $Ah_{\ast\ast}$ if $\h_{\ast\ast}(\Spec k)$ is known, see Lemma \ref{Ah}.
\item In general, however, $Ah_{\ast\ast}$ is not known in degrees different from $(2i,i)$. It is also not known whether $Ah$ can be given the structure of a ring spectrum (except for $Ah=ABP$). Another open question is whether $AK(n)$ satisfies the K\"unneth formula like $K(n)$.
\end{compactenum}
\end{rk}

We will make use of motivic Atiyah-Hirzebruch spectral sequences as discovered by Hopkins and Morel and worked out by Hoyois \cite[Example 8.13]{Hoy}. See also \cite[Section 11]{HtpyConivTower}.

\begin{prop}\label{AHSS}
Let $h=MU_{(p)}/I$ and $Ah=MGL_{(p)}/I$ for some regular ideal $I\subseteq MU_\ast$.
For $X\in\SH(k)^{fin}_{(p)}$, there are strongly convergent spectral sequences: 
\[E_2^{p,q,t}=\h^{p+2t,q+t}(X,h_{t})\Rightarrow Ah^{p,q}(X),\]
\[E_2^{p,q,t}=\h^{p+2t,q+t}(X,(v_n^{-1}h)_{t})\Rightarrow (v_n^{-1}Ah)^{p,q}(X).\]
\end{prop}

\begin{rk}\label{remark AHSS}
\begin{compactenum}[(1)]
\item 
These are the spectral sequences associated with the slice filtrations of $Ah$ and $v_n^{-1}Ah$. 
The $n$-th truncation of $MGL$ in the slice filtration is described in the proof of \cite[Theorem 4.6]{Spitz} as the colimit of a certain diagram $D_{\deg\geq n}$, meaning that $f_n MGL$ is constructed from $MGL$ by quotienting out all monomials $a_1^{k_1}\cdots a_m^{k_m}$ with $a_i\in MGL_{2\ast,\ast}$ and $\sum_{i=1}^m ik_i\geq n$. The same construction with $MU$ instead of $MGL$ yields the Postnikov truncation of $MU$ because $MU_\ast=\Z[a^{\Top}_1,a^{\Top}_2,\cdots]$ with $|a^{\Top}_i|=2i$. It follows that, for $k\subseteq\C$, $R_k:\SH(k)\rightarrow \SH$ maps the slice filtration of $MGL$ to the Postnikov filtration of $MU$. By the construction of $Ah$ and $v_n^{-1}Ah$ from $MGL$, this implies that their slice filtrations also realise to the Postnikov filtrations of $h$ and $v_n^{-1}h$. For the associated spectral sequences, this means that $R_k$ maps the above spectral sequences to the analogous topological Atiyah-Hirzebruch spectral sequences.
\item
In \cite[Example 8.13]{Hoy}, the convergence is stated for $X\in\Sm/S$ and follows from $Ah$, $v_n^{-1}Ah$ being convergent with respect to $[\Sigma^{0,q}\Sigma^\infty X_+,-]$, $X\in\Sm/S$, in the sense of \cite[Section 8.5]{Hoy}. But this implies convergence with respect to $[X,-]$ for all finite cell spectra $X$ by motivic cellular induction. Hence, the sequences converge for all finite cell spectra $X$. Note that \cite[Example 8.13]{Hoy} shows the convergence for Landweber exact spectra, but his proof holds for quotients of $MGL$ as well.
\end{compactenum}
\end{rk}

In Chapters \ref{motivic type n} and \ref{chapter BC}, we will often assume $k=\C$ to be able to prove more results than for general $k$. The reason is that the coefficients of $H\Z/p$ are particularly simple in this case.

\begin{lemma}\label{H(C)}
For $k=\C$,
\[\h^{\ast\ast}(\Spec(\C),\Z/p)\cong\F_p[\tau]\]
with $\deg(\tau)=(0,1)$.
\end{lemma}

\begin{proof}
This is \cite[Equation (74)]{Voe10}. 
\end{proof}

This, and the above spectral sequence, can be used to calculate the coefficients of theories like $Ak(n)$.

\begin{lemma}\label{Ah}
Let $k=\C$ and $h=MU_{(p)}/I$ with $(p)\subseteq I\subseteq MU_\ast$ a regular ideal. Let $Ah=MGL_{(p)}/I$. Then 
\[Ah_{\ast\ast}\cong \h_{\ast\ast}(\Spec\C,\F_p)\otimes_{\F_p} h_\ast\cong h_\ast[\tau].\]
\end{lemma}

\begin{proof}
This is remarked in \cite{Ya}, below Corollary 3.9. The reason is the following: For $X=\Spec \C$, the motivic Atiyah-Hirzebruch spectral sequence is
\[\h^{p+2t,q+t}(\Spec\C,\F_p)\otimes h_t\Rightarrow Ah^{p,q}(\Spec\C).\]
By Lemma \ref{H(C)}, $\h^{\ast\ast}(\Spec\C,\F_p)\cong \F_p[\tau]$ with $\deg(\tau)=(0,1)$. Thus, for fixed $t$, $E_2^{p,q,t}=\h^{p+2t,q+t}(X,h_{t})$ can only be nonzero in the column $p=-2t$, which implies that all differentials vanish.
Therefore, the spectral sequence collapses immediately, proving $Ah_{\ast\ast}\cong \h_{\ast\ast}\otimes h_\ast$. 
\end{proof}

\subsection{Thick ideals and Morava K-theories}\label{section motivic model}

Let $K$ be a cellular spectrum in $\SH(k)$, $k\subseteq\C$, such that $R_k(K)=K(n)$ is the $n$-th Morava K-theory with respect to a fixed prime $p$. We call such a $K$ a motivic model for $K(n)$. 

Let $\mathcal C_K=\{X\in(\SH(k)_f)_{(p)}\; |\; K\wedge X\cong 0\}$, which is a thick ideal in the $p$-localised category $(\SH(k)_f)_{(p)}$ by Lemma \ref{Lemma C_E}. If $K\wedge X\cong 0$, then $R_k(K\wedge X)\cong K(n)\wedge R_k(X)\cong 0$. Consequently, $R_k(X)\in\mathcal C_{n+1}$ for any $X\in\mathcal C_K$. This is content of the folowing proposition.

\begin{prop}\label{motivic model}
For $K$ a motivic model for $K(n)$, we have an inclusion of thick ideals in $(\SH(k)_f)_{(p)}$:
\[\mathcal C_K\subseteq R_k^{-1}(\mathcal C_{n+1}).\]
The same is true in the category $\SH(k)^{fin}_{(p)}$.
\end{prop}

\begin{eg}
The motivic Morava K-theory spectra $AK(n)$, as defined in Definition \ref{AK(n)}, satisfy $R_k(AK(n))=K(n)$ and are cellular. In particular, they satisfy the previous proposition. 

Another possibility for such a spectrum $K$ is the constant Morava K-theory spectrum $c_k (K(n))$. In Sections \ref{vanishing} - \ref{constant} we will have a closer look at the thick ideals $\mathcal C_{AK(n)}$ and $\mathcal C_{c_k K(n)}$ for $k=\C$.
\end{eg}

\begin{rk}\label{infinite}
The spectrum $K(n)$ is not finite. One possibility to see this is by the equivalence \cite[Theorem 2.1.(h) and (i)]{RavLoc}:
\[K(m)\wedge K(n)\cong 0\Leftrightarrow m\neq n\] 
If $K(n)$ were finite, we would have $K(m)_\ast K(n)=0\Rightarrow K(m-1)_\ast K(n)=0$ by \cite[Theorem 2.11]{RavLoc}, which is wrong for $n=m-1$. As a consequence, any spectrum $K$ with $R_k(K)=K(n)$ can also not be finite cellular.
\end{rk}

\begin{rk}
Another corollary of the above equivalence and of \cite[Theorem 2.1.(h)]{RavLoc} is the following: If $K$ and $K'$ are motivic models for $K(n)$, then $R_k(K\wedge K')=K(n)\wedge K(n)$, which has the same Bousfield class as $K(n)$. In particular, any $X\in \mathcal C_{K\wedge K'}$ satisfies $R_k(X)\in \mathcal C_{n+1}\setminus \mathcal C_{n+2}$. Hence, 
\[\mathcal C_{K\wedge K'}\subseteq R_k^{-1}(\mathcal C_{n+1}) \textup{ and } \mathcal C_{K\wedge K'}\nsubseteq R_k^{-1}(\mathcal C_{n+2}).\]
\end{rk}

\chapter{$\SH(k)_\f$ has more thick ideals than $\SH^{\fin}$}\label{nilpotence}

\section{The motivic Hopf map}

In this chapter, we study the cofiber of the motivic Hopf map, which generates a thick ideal that is not of the form $R_k^{-1}(\mathcal C_n)$ or $\mathcal C_{AK(n)}$. The first part of this claim was proven by Balmer in \cite[Proposition 10.4]{Balmer}. We will reprove and extend this result, as well as summarise and apply other results of Balmer's work on prime ideals in tensor triangulated categories. 

For $k\subseteq\R$, we have already shown that $\SH(k)_f$ has more thick ideals than $\SH^{fin}$. This chapter is, therefore, in particular interesting for $k=\C$.

To understand why there cannot be a complete analogy between motivic thick ideals and topological thick ideals, recall the reasoning in the topological case: the thick subcategory (i.e., thick ideal) theorem of Hopkins and Smith \cite{HS} is derived from the fact that Morava K-theories detect nilpotence, which follows from the theorem that $MU$ detects nilpotence. This is not the case in the motivic setting.\\

The Hopf map in $\SH(k)$ is given by $\eta:\A^2\setminus\{0\}\rightarrow \P^1$, $(x,y)\mapsto [x:y]$, defining an element $\eta\in\pi_{1,1}(S)$. Unlike the topological Hopf element, $\eta$ is not nilpotent \cite[Theorem 4.7]{MorelA1algtop}. 
The following lemma is due to Morel, see e.g. \cite{MorelPi}, but we could not find a proof in the literature. It shows that the spectrum $MGL$ does not detect the non-nilpotence of $\eta$. 

\begin{lemma}
\[MGL\wedge \eta\cong 0\]
\end{lemma}

\begin{proof}
The unit map $u$ of the ring spectrum $MGL$ factors through $\Sigma^{-2,-1}C\eta$, as in the proof of \cite[Theorem 3.8]{Hoy}. Consider the diagram:
\[\xymatrix{\ar @{} [dr] |{}
MGL\wedge S^{1,1} \ar[d]^{\cong} \ar[rrr]^{1\wedge\Sigma^{-2,-1}\eta} &&& MGL\wedge S^{0,0}\ar[d]^{\cong}\ar[r]\ar[rd]_{1\wedge u} & MGL\wedge \Sigma^{-2,-1}C\eta \ar[d]\\
MGL\wedge S^{1,1}  \ar[rrr] &&& MGL & MGL\wedge MGL \ar[l]^{m}  }\]
The upper row is a cofiber sequence and $u,m$ are the structure maps of the ring spectrum $MGL$. The diagram is commutative. It follows that $1\wedge\Sigma^{-2,-1}\eta$ factors through its own cofiber, hence it must be zero. Suspending by $S^{2,1}$, we get $MGL\wedge\eta\cong 0$.
\end{proof}

It seems likely that also $AK(n)\wedge \eta\cong 0$, but this does not immediately follow from the above lemma. For our interests, it will suffice to know that $AK(n)\wedge C\eta \not\cong 0$, which we prove differently.

\begin{rk}
By \cite[Lemma 6.2.1]{MorelIntro}, $C\eta\cong \P^2$.
\end{rk}

\begin{prop}\label{RCeta}
For any $k\subseteq\C$ and any prime $p$, 
\[\thickid (R_k( C\eta_{(p)}))=\mathcal C_{0}=\SH^{fin}_{(p)}.\]
For $k\subseteq\R$, 
\[\thickid (R'_k( C\eta_{(p)}))=\begin{cases} \mathcal C_{0,1}\subset (\SH(\Z/2)_f)_{(2)} & \mbox{ if } p=2 \\ \mathcal C_{0,\infty}\subset(\SH(\Z/2)_f)_{(p)} & \mbox{ if } p\neq 2,  \end{cases}\]
where $\mathcal C_n$ is the thick ideal defined in Theorem \ref{HS} and $\mathcal C_{m,n}$ is the thick ideal defined in Corollary \ref{Cmn}.
\end{prop}

\begin{proof}
Since, by definition, $R_k=R_\C\circ f^\ast$ for $f:k\hookrightarrow\C$, $R'_k=R_\R\circ f^\ast$ for $f:k\hookrightarrow \R$, and $f^\ast(\eta_k)=\eta_K$ for $f:k\hookrightarrow K$, it suffices to prove the claims for $k=\C$ and $k=\R$.

For $k=\C$, we have $R(C\eta)=R(\P^2)=\C P^2$. Since $K(0)_\ast(\C P^2_{(p)})=\h_\ast(\C P^2_{(p)},\Q)\neq 0$, $R(C\eta_{(p)})$ has type $0$. By \cite[Theorem 7]{HS}, any spectrum of type $0$ generates $\SH^{fin}_{(p)}$. So, $\thickid (R( C\eta_{(p)}))=\SH^{fin}_{(p)}$.

For $k=\R$, $R'(C\eta)^{\{1\}}=C\eta(\C)=\C P^2$ as before. Therefore, if $R'( C\eta_{(p)})\in\mathcal C_{m,n}$, then $m=0$. 

Furthermore, $R'(\eta)^{\Z/2}=\eta(\R)$ is the quotient map $\R^2\setminus\{0\}\rightarrow \R P^1$, which is isomorphic to $2:S^1\rightarrow S^1$. It follows that $R'(C\eta)^{\Z/2}=S/2$.

If $p\neq 2$, $S\overset{2}\rightarrow S$ is an isomorphism in $\SH^{fin}_{(p)}$ and, hence, $(S/2)_{(p)}=0$, which has type $\infty$. This proves $\thickid (R'( C\eta_{(p)}))=\mathcal C_{0,\infty}$ if $p\neq 2$. 

Now let $p=2$. Since $S\overset{2}\rightarrow S$ is an isomorphism rationally, we have $K(0)_\ast(S/2)=H_\ast(S/2,\Q)=0$. On the other hand, $K(1)$ is a direct summand of mod $2$ topological K-theory, so $K(1)_\ast(S\overset{2}\rightarrow S)=0$ and $K(1)_\ast(S/2)\neq 0$, as in the proof of \cite[Proposition 9.4]{Balmer}. Hence, the type of $S/2\in\SH^{fin}_{(2)}$ is $1$ and, therefore, $\thickid (R'_k((C\eta)_{(2)}))=\mathcal C_{0,1}$.
\end{proof}

In terms of the motivic thick ideals $\thickid( C\eta_{(p)})\subseteq(\SH(k)_f)_{(p)}$, the above proposition states the following: For $k\subseteq\C$, $\thickid( C\eta_{(p)})\subseteq R_k^{-1}(\mathcal C_{0})$ (which is a trivial statement) and $\thickid( C\eta_{(p)})\not\subseteq R_k^{-1}(\mathcal C_{n})$ for $n> 0$. For $k\subseteq\R$ and $p\neq 2$, $\thickid( C\eta_{(p)}) \subseteq (R'_k)^{-1}(\mathcal C_{0,\infty})$ and $\thickid( C\eta_{(p)}) \not\subseteq (R'_k)^{-1}(\mathcal C_{m,n})$ for any $m> 0$, $n$ arbitrary. If $p=2$, then $\thickid((C\eta)_{(2)}) \subseteq (R'_k)^{-1}(\mathcal C_{0,1})$ and $\thickid((C\eta)_{(2)}) \not\subseteq (R'_k)^{-1}(\mathcal C_{m,n})$ for any $m> 0$ or $n>1$.

The next proposition contains two more results on $\thickid( C\eta_{(p)})\subseteq(\SH(k)_f)_{(p)}$. 

\begin{prop}\label{counter example}
For $k\subseteq\C$, let $\thickid( C\eta_{(p)})\subseteq(\SH(k)_f)_{(p)}$ denote the thick ideal generated by the $p$-localised cofiber of the Hopf map. Then the following hold:
\begin{compactenum}[(1)]
\item $\thickid( C\eta_{(p)})\not\subseteq \mathcal C_{AK(n)} $ for any $ n\geq 0$ and any prime $p$,
\item $\thickid( C\eta_{(p)})\not \subseteq R_k^{-1}(\mathcal C_n)$ for any $ n> 0$ and any prime $p$,
\item $\thickid( C\eta_{(p)})\subsetneq \thickid(S^0_{(p)})=(\SH(k)_f)_{(p)}$ if $k\subseteq \R$ and $p$ is any prime or $k\subseteq \C$ and $p=2$.
\item For any prime $p$, the thick ideals $\thickid( C\eta_{(p)})\cap R_k^{-1}(\mathcal C_n)$ are distinct for different $n\geq 0$ and in particular nonzero if $n<\infty$.
\end{compactenum}
\end{prop}

\begin{proof}
\begin{compactenum}[(1)]
\item 
By Proposition \ref{motivic model}, $\mathcal C_{AK(n)}\subseteq R_k^{-1}(\mathcal C_{n+1})$. Assuming $\thickid( C\eta_{(p)})\subseteq \mathcal C_{AK(n)}$ therefore implies $\thickid( C\eta_{(p)})\subseteq R_k^{-1}(\mathcal C_{n+1})$. Hence, (1) will follow from (2).
\item 
This can either be derived from the previous proposition or can be seen by the following argument.

$R_k(\eta)$ is the topological Hopf map, which is nilpotent. Hence, for $n\geq 1$, $R_k(\eta)_\ast$ is not surjective in the sequence
\[\cdots\rightarrow K(n-1)_{\ast}(S^{3})\overset{R_k(\eta)_\ast}\rightarrow K(n-1)_{\ast}(S^{2})\rightarrow K(n-1)_{\ast}(C(R_k\eta))\rightarrow\cdots.\]
It follows that $K(n-1)_{\ast}(C(R_k\eta))\neq 0$ and, since $C(R_k\eta)\cong R_k(C\eta)$, $ C\eta_{(p)}\not\in R_k^{-1}(\mathcal C_n)$.
\item 
Note that $C\eta\in\thickid(S^0)=\SH(k)_f$, so $\thickid( C\eta_{(p)})\subseteq\thickid( S^0_{(p)})$ for any $p$. We have to show that $S^0_{(p)}\not\in\thickid(C\eta_{(p)})$. We consider the sheaf cohomology theory $\h^\ast(-,\underline{K}_\ast^{MW}[\eta^{-1}])$, on which $\eta$ induces an isomorphism as in the proof of \cite[Theorem 4.7]{MorelA1algtop}, where Morel concludes that $\eta$ cannot be nilpotent. Localising at $p$, we get that $\eta_{(p)}$ induces an isomorphism on $\h^\ast(S^0_{(p)},\underline{K}_\ast^{MW}[\eta^{-1}])$. 
Consider $\h^0(S^0_{(p)},\underline{K}_\ast^{MW}[\eta^{-1}])=K_\ast^{MW}(k)[\eta^{-1}]_{(p)}$. As in the proof of \cite[Theorem 4.7]{MorelA1algtop}, this is $K_0^W(k)[\eta,\eta^{-1}]_{(p)}$, and $K_0^W(k)$ is isomorphic to the Witt ring $W(k)$ by \cite[Remark 4.2]{MorelA1algtop}. We have $W(\C)=\Z/2$ and $W(\R)=\Z$, see e.g. \cite[p. 34]{Lam}.

Hence, for $k=\C$, $\h^0(S^0_{(p)},\underline{K}_\ast^{MW}[\eta^{-1}])=(\Z/2[\eta,\eta^{-1}])_{(p)}$, which is nonzero if $p=2$. And, for $k=\R$, $\h^0(S^0_{(p)},\underline{K}_\ast^{MW}[\eta^{-1}])=(\Z[\eta,\eta^{-1}])_{(p)}\neq 0$ for any $p$. Therefore, the above mentioned isomorphism induced by $\eta_{(p)}$ is not the zero map, and, as in \cite[Theorem 4.7]{MorelA1algtop}, it follows that $\eta_{(p)}$ is not nilpotent in these cases.

Now we apply \cite[Theorem 2.15]{Balmer}, which states that, for a map $f:X\rightarrow Y$ between invertible objects $X$ and $Y$ in a tensor triangulated category $\mathcal T$, the thick ideal generated by $Cf$ is equal to the full subcategory consisting of all objects $A$ such that $f^{\wedge n}\wedge A=0$ for some $n\geq 1$ (this is similar to Proposition \ref{ann}). A corollary of this theorem is that $\thickid(Cf)=\mathcal T$ if and only if $f$ is nilpotent. 

It follows that, since $\eta_{(p)}$ is not nilpotent, $S^0_{(p)}\not\in \thickid(C\eta_{(p)})$ and $\thickid( C\eta_{(p)})\neq\thickid(S^0_{(p)})$ in the cases $k=\C$ and $p=2$ or $k=\R$ and $p$ any prime.

Now, let $f:k\hookrightarrow \C$. By Proposition \ref{f}, $(f^{\ast})^{-1}$ preserves thick ideals. Since $f^{\ast}(C\eta_k)=C\eta_\C$, $\thickid((C\eta_k)_{(2)})\subseteq (f^{\ast})^{-1}(\thickid((C\eta_\C)_{(2)}))$. As $f^\ast((S^0_k)_{(2)})=(S^0_\C)_{(2)}\not\in\thickid((C\eta_\C)_{(2)})$, it follows $(S^0_k)_{(2)}\not\in (f^{\ast})^{-1}(\thickid((C\eta_\C)_{(2)}))$ and, hence, $(S^0_k)_{(2)}\not\in\thickid((C\eta_k)_{(2)})$. 

The same argument holds for arbitrary primes $p$ if $f:k\hookrightarrow \R$. Alternatively, the statement for $k\subseteq\R$ can be derived from Proposition \ref{RCeta} and Theorem \ref{lower bound}: 
\[\thickid(C\eta_{(p)})\subseteq (R'_k)^{-1}(\thickid(R'_k(C\eta_{(p)})))\neq (R'_k)^{-1}(\mathcal C_{0,0})=\thickid(S^0_{(p)}).\]
\item
Let $X\in R_k^{-1}(\mathcal C_n)$ be such that $R_k(X)$ is of type $n$, i.e., $R_k(X)\in\mathcal C_n\setminus \mathcal C_{n+1}$. From the proof of (2), we know that $R_k(C\eta_{(p)})$ is of type $0$. From the K{\"u}nneth formulas for $K(n)$ and $K(n+1)$, it follows that $R_k(C\eta_{(p)}\wedge X)\cong R_k(C\eta_{(p)})\wedge R_k(X)$ is of type $n$. Therefore, \[C\eta_{(p)}\wedge X\in \left(\thickid(C\eta_{(p)})\cap R_k^{-1}(\mathcal C_n)\right)\setminus \left(\thickid(C\eta_{(p)})\cap R_k^{-1}(\mathcal C_{n+1})\right).\]
\end{compactenum}
\end{proof}

\begin{cor}
For $k\subseteq\C$, $\thickid((C\eta)_{(2)})$ is neither of the form $\mathcal C_{AK(n)}$ for any $n\geq 0$ (by Proposition \ref{counter example}(1)) nor of the form $R_k^{-1}(\mathcal C_n)$ for any $n\geq 0$ (by Proposition \ref{counter example}(2) and (3)).
\end{cor}

For $k\subseteq\R$, the statements analogous to (2) and (4) of Proposition \ref{counter example} read as follows.

\begin{prop}
Let $\eta$ denote the Hopf map in $\SH(k)$, $k\subseteq\R$. In $(\SH(k)_f)_{(p)}$ the following inequalities hold.
\begin{compactenum}[(1)]
\item $\thickid(C\eta_{(p)})\not\subseteq (R'_k)^{-1}(\mathcal C_{m,n})$ for any $m> 0$ ($p,n$ arbitrary) or, if $p=2$, $m=0$ and $n> 1$.
\item Let $(m,n)$ and $(m',n')$ be pairs of integers $\geq 0$ such that $\Z/2$-equivariant spectra of types $(m,n)$ and $(m',n')$ exist (see Chapter \ref{equivariant}). If $p$ is any prime and $m\neq m'$ or $p=2$ and $\max(n,1)\neq\max(n',1)$ then 
\[\thickid(C\eta_{(p)})\cap (R'_k)^{-1}(\mathcal C_{m,n})\neq \thickid(C\eta_{(p)})\cap (R'_k)^{-1}(\mathcal C_{m',n'}).\]
Otherwise, the two thick ideals are equal.
\end{compactenum}
\end{prop}

\begin{proof}
(1) is a reformulation of the second part of Proposition \ref{RCeta}. For (2), take $X=c'_k(X_{m,n})$ or any other spectrum whose realisation is a spectrum of type $(m,n)$. Then $R'_k(C\eta_{(p)}\wedge X)\cong R'_k(C\eta_{(p)})\wedge R'_k(X)$ is the smash product of a spectrum of type $(0,\infty)$ if $p$ is odd ($(0,1)$ if $p=2$) with a spectrum of type $(m,n)$. By the equivariant K{\"u}nneth formula, Corollary \ref{Kunneth equivariant}, it follows that $R'_k(C\eta_{(p)}\wedge X)$ has type $(m,\infty)$ if $p$ is odd and $(m, \max(n,1))$ if $p=2$. Consequently, $\thickid(C\eta_{(p)})\cap (R'_k)^{-1}(\mathcal C_{m,n})$ is equal to $\thickid(C\eta_{(p)})\cap(R'_k)^{-1}(\mathcal C_{m,\infty})$ if $p$ is odd and to $\thickid(C\eta_{(p)})\cap(R'_k)^{-1}(\mathcal C_{m,\max(n,1)})$ if $p=2$. The intersection is not contained in $(R'_k)^{-1}(\mathcal C_{m',n'})$ for any $m'>m$ or, if $p=2$, for any $n'>\max(n,1)$.
\end{proof}

Related to the previous propositions are the following conjectures.
\begin{conj}\label{conjectures}
\begin{compactenum}[(1)]
\item $R_k^{-1}(\mathcal C_n)\not\subseteq \thickid(C\eta_{(p)})$ for all $0\leq n < \infty$.
\item $\mathcal C_{AK(n)}\not\subseteq \thickid(C\eta_{(p)})$ for all $n\geq 0$.
\item The thick ideals $\thickid(C\eta_{(p)})\cap \mathcal C_{AK(n)}$ are distinct for different $n\geq 0$ and in particular nonzero if $n<\infty$.
\item For $k\subseteq\R$, $\thickid(C\eta_{(p)})\subsetneq (R'_k)^{-1}(\mathcal C_{0,\infty})$ for odd primes $p$ and, for $p=2$, $\thickid(C\eta_{(p)})\subsetneq (R'_k)^{-1}(\mathcal C_{0,1})$.
\item For $k\subseteq\R$, $(R'_k)^{-1}(\mathcal C_{m,n})\not\subseteq \thickid (C\eta_{(p)})$ for all $0\leq m,n < \infty$.
\end{compactenum}
\end{conj}

\begin{rk}
A possible approach for proving (1) or (2) might be to choose a motivic spectrum $X$ whose realisation is of type $n$ (e.g. $c_k(X_{n})$) or a spectrum which is of motivic type $n$ (e.g. as constructed in Chapter \ref{motivic type n}), respectively. If one can show that $\h^\ast(X,\underline{K}_\ast^{MW}[\eta^{-1}])\neq 0$, this proves (1) respectively (2).

An idea to prove (3) is to choose a spectrum $X$ of motivic type $n$ and to show that $C\eta\wedge X$ is again of motivic type $n$ for this particular choice. 

For (4), $p$ odd, one has to find a spectrum $X\in (R'_k)^{-1}(\mathcal C_{0,\infty})$, which is not contained in $\thickid(C\eta_{(p)})$. $X=(\Sigma^\infty \Spec\C_+)_{(p)}$ satisfies the first condition and it might be possible to show $\h^\ast(X,\underline{K}_\ast^{MW}[\eta^{-1}])\neq 0$, which would imply the second. Since $\mathcal C_{0,\infty}\subseteq\mathcal C_{0,1}$, the same $X$ could be used for the case $p=2$.
\end{rk}

\begin{rk}
If and only if the first conjecture is true, the ideals in Proposition \ref{counter example}(4) are different from the ideals $R_k^{-1}(\mathcal C_n)$ for all $n\geq 0$.
\end{rk}

\section{Prime ideals}\label{prime ideals}

Let us recall some results of \cite{Balmer}, where Balmer studies prime ideals in tensor triangulated categories. 

\begin{defn} 
A prime ideal is a proper thick ideal $\mathcal C$ with the additional property that $X\wedge Y\in\mathcal C$ implies $X\in\mathcal C$ or $Y\in\mathcal C$. The set of prime ideals of a tensor triangulated category $\mathcal T$ is denoted by $\Spc(\mathcal T)$. 
\end{defn}

In \cite{Balmer}, the endomorphism ring of the unit object of $\mathcal T$ is denoted by $R_{\mathcal T}$. To avoid confusion with the realisation functors, we use the notation $\pi_0^{\mathcal T}$ instead of $R_{\mathcal T}$. In \cite[Corollary 5.6]{Balmer}, Balmer defines a functor
\[\rho_{\mathcal T}:\Spc(\mathcal T)\rightarrow \Spec(\pi_0^{\mathcal T}),\]
\[\rho_{\mathcal T}(\mathcal P)=\{f\in\pi_0^{\mathcal T}\; |\; C(f)\not\in\mathcal P\},\]
which he proves to be surjective if $\mathcal T$ is connective \cite[Theorem 7.13]{Balmer}, where connectivity means that $\Hom_{\mathcal T}(S,\Sigma^i S)=0$ for all $i>0$ and $S$ the unit object of $\mathcal T$. Furthermore, this functor is natural for tensor triangulated functors $f:\mathcal T\rightarrow \mathcal T'$ by \cite[Theorem 5.3(c) and Corollary 5.6(b)]{Balmer}.

\subsection{Prime ideals in the topological categories $\SH^{fin}$ and $\SH(\Z/2)_f$}

Applied to $\SH^{fin}_{(p)}$, this yields the following \cite[Proposition 9.4]{Balmer}: 

All proper thick ideals $\mathcal C_n\subset\SH^{fin}_{(p)}$, $0<n\leq\infty$ are prime ideals, as can be seen either from the K\"unneth formula for $K(n)$ or by the linear ordering of the $\mathcal C_n$. The functor
\[\rho_{\SH^{fin}_{(p)}}:\Spc(\SH^{fin}_{(p)})\rightarrow \Spec(\Z_{(p)})\]
has the following values:
\[\rho_{\SH^{fin}_{(p)}}(\mathcal C_n)=\begin{cases} p\Z_{(p)} &\mbox{if } n>1, \\ 0 &\mbox{if } n=1.\end{cases}\]
Note that $\mathcal C_0$ is not proper, so it is not in $\Spc(\SH_{(p)}^{fin})$. Note also that \cite{Balmer} uses a different indexing convention, in which the indices are shifted by one.

The prime ideals in the non-localised category $\SH^{fin}$ are given as follows.

\begin{prop}
A subcategory $\mathcal C\subset\SH^{fin}$ is a prime ideal if and only if there exist a prime $p$ and a number $1\leq n\leq \infty$ such that
\[\mathcal C=\mathcal C_{p,n}=\{X\in\SH^{fin}\;|\; K(p,n-1)_\ast X=0\}\]
if $n<\infty$, or
\[\mathcal C=\mathcal C_{p,\infty}=\{X\in\SH^{fin}\;|\; X_{(p)}=0\}\]
if $n=\infty$.
Here, $\mathcal C_{p,1}=\mathcal C_{q,1}=\mathcal C_1$ for any primes $p$ and $q$. Except for $n=1$, the $\mathcal C_{p,n}$ are pairwise different. 

The functor $\rho_{\SH^{fin}}:\Spc(\SH^{fin})\rightarrow \Spec(\Z)$ maps $\mathcal C_{p,n}$ to $p\Z$ for any $n>1$ and it maps $\mathcal C_1$ to $0$.
\end{prop}

\begin{proof}
This is \cite[Corollary 9.5]{Balmer}. We give another proof of the first statement. By Theorem \ref{thm Q}, any thick ideal of $\SH^{fin}$ is an intersection 
\[\mathcal C=\underset{p\in P}\bigcap \{X\in\SH^{fin}\;|\; X_{(p)}\in\mathcal C_{n_p}\subseteq \SH^{fin}_{(p)} \}\]
for some set of primes $P$ and some numbers $1\leq n_p\leq \infty$. If $|P|=0$, then $\mathcal C=\SH^{fin}$, which is not proper, and, thus, is no prime ideal. Assume $|P|\geq 1$. By Proposition \ref{intersection}, the intersection of two thick ideals $\mathcal I$ and $\mathcal J$ satisfying $\mathcal I\not\subseteq\mathcal J$ and $\mathcal J\not\subseteq\mathcal I$ is never a prime ideal, because it contains all $X\wedge Y$ with $X\in\mathcal I$ and $Y\in\mathcal J$. It follows that, if $\mathcal C$ is a prime ideal, then $|P|=1$. This proves that any prime ideal of $\SH^{fin}$ is of the form $\mathcal C=\{X\in\SH^{fin}\;|\; X_{(p)}\in\mathcal C_{n_p} \}$ for some prime $p$ and some $1\leq n_p\leq\infty$. 

On the other hand, the K\"unneth formula implies that $K(p,n)_\ast (X\wedge Y)$ can only be zero if $K(p,n)_\ast X=0$ or $K(p,n)_\ast Y=0$, proving that $\mathcal C_{p,n}$ as above is indeed a prime ideal. Note that, for $n=\infty$, we have $X\in\mathcal C_{p,n}$ if and only if $K(m)_\ast(X)=0$ for all $m\geq 0$, so $\mathcal C_{p,\infty}$ is also a prime ideal by the K\"unneth formula (using that $K(m+1)_\ast(X)=0$ implies $K(m)_\ast(X)=0$). 
\end{proof}

Now we turn to the $\Z/2$-equivariant category, $\mathcal T=\SH(\Z/2)_f$. By \cite[Corollary 1]{Segal}, the map 
\[\pi_0^{\mathcal T}\rightarrow A(\Z/2),\; [f]\mapsto (\deg(f^{\{1\}}),\deg(f^{\Z/2}))\] 
is an isomorphism to the Burnside ring $A(\Z/2)\cong\Z\oplus\Z$.
Hence, \cite[Theorem 7.13]{Balmer} yields surjective functors
\[\rho_{\SH(\Z/2)_{f}}:\Spc(\SH(\Z/2)_{f})\rightarrow \Spec(\Z\oplus\Z)\]
and
\[\rho_{(\SH(\Z/2)_{f})_{(p)}}:\Spc((\SH(\Z/2)_{f})_{(p)})\rightarrow \Spec(\Z_{(p)}\oplus\Z_{(p)}).\]

\begin{prop}\label{equiv prime ideals}
A thick ideal $\mathcal C_{m,n}\subseteq (\SH(\Z/2)_{f})_{(p)}$ is a prime ideal if and only if $\mathcal C_{m,n}=\mathcal C_{m,0}$, $0<m\leq\infty$, or $\mathcal C_{m,n}=\mathcal C_{0,n}$, $0<n\leq \infty$.
\end{prop}

\begin{proof}
From Section \ref{Application to Z/2}, we know that any thick ideal in $(\SH(\Z/2)_{f})_{(p)}$ is of the form 
\[\mathcal C_{m,n}=\{X\; |\; \phi^{\{1\}}(X)\in\mathcal C_m \textup{ and }\phi^{\Z/2}(X)\in\mathcal C_n\}=\mathcal C_{m,0}\cap\mathcal C_{0,n}.\]
By Proposition \ref{intersection}, this is equal to $\{X\wedge Y\; |\; X\in \mathcal C_{m,0}, Y\in\mathcal C_{0,n}\}$. It follows that $\mathcal C_{m,n}$ being a prime ideal implies $\mathcal C_{m,n}=\mathcal C_{m,0}$ or $\mathcal C_{m,n}=\mathcal C_{0,n}$. On the other hand, if either $m$ or $n$ is $0$ then $\mathcal C_{m,n}$ is a prime ideal by the K\"unneth formula for $K(m,\{1\})$ or $K(n,\Z/2)$. Note that $\mathcal C_{0,0}=(\SH(\Z/2)_f)_{(p)}$ is not a prime ideal since it is no proper subcategory. 
\end{proof}

\begin{rk}
As a consequence of Strickland's results \cite{Sch}, $\mathcal C_{m,n}=\mathcal C_{0,n}$ implies $m=0$ but $\mathcal C_{m,n}=\mathcal C_{m,0}$ does not always imply that $n=0$. This follows from Corollary \ref{mn}: for $m\leq n$, a type $(m,n)$-spectrum $X_{m,n}$ always exists, and $X_{0,n}\in\mathcal C_{0,n}\setminus\mathcal C_{m,n}$ for all $m>0$. On the other hand, Proposition \ref{upper bound} implies that, for $p=2$, $\mathcal C_{m,0}=\mathcal C_{m,n}$ for any $0\leq n\leq m-1$.
\end{rk}

Recall from Definition \ref{equivariant type} that the thick ideals in $(\SH(\Z/2)_f)_{(p)}$ are in bijection with a lattice $\Gamma_p\subseteq(\Z_{\geq 0}\cup\{\infty\})\times(\Z_{\geq 0}\cup\{\infty\})$ containing all $(m,n)$ such that a spectrum of type $(m,n)$ exists. In the case just considered, $(m,0)$ would not be in $\Gamma_2$ (at least for $m>1$). In terms of the bijection to $\Gamma_p$, the above proposition instead reads as follows:

\begin{cor}
For any $(m,n)\in\Gamma_p$, the thick ideal $\mathcal C_{m,n}$ is a prime ideal if and only if one of the following conditions holds: 
\begin{compactenum}[(1)]
\item $(m,n)=(0,n)$, with $0<n\leq\infty$,
\item $0<m\leq\infty$, $0\leq n\leq\infty$ and $(m,n-1)\not\in\Gamma_p$.
\end{compactenum}
\end{cor}

For evaluating the functor $\rho_{(\SH(\Z/2)_{f})_{(p)}}$, we use the first description of the prime ideals in $(\SH(\Z/2)_f)_{(p)}$, given by Proposition \ref{equiv prime ideals}.

\begin{prop}
The functor
\[\rho_{(\SH(\Z/2)_{f})_{(p)}}:\Spc((\SH(\Z/2)_{f})_{(p)})\rightarrow \Spec(\Z_{(p)}\oplus\Z_{(p)})\]
has the following values:
\[\rho_{(\SH(\Z/2)_{f})_{(p)}}(\mathcal C_{m,0})=\begin{cases} 0\oplus\Z_{(p)} &\mbox{if } m=1 \\ p\Z_{(p)}\oplus\Z_{(p)} &\mbox{if } m> 1,\end{cases}\]
\[\rho_{(\SH(\Z/2)_{f})_{(p)}}(\mathcal C_{0,n})=\begin{cases} \Z_{(p)}\oplus 0&\mbox{if } n=1 \\ \Z_{(p)}\oplus p\Z_{(p)}&\mbox{if } n> 1.\end{cases}\]
\end{prop}

\begin{proof}
By \cite[Corollary 5.6(b)]{Balmer}, $\rho_{\mathcal T}$ is natural in $\mathcal T$.
Since $X\in\mathcal C_{m,0}$ is only a condition on $\phi^{\{1\}}X$ and $X\in\mathcal C_{0,n}$ is only a condition on $\phi^{\Z/2}X$, the fixed point functors can be used to derive this result from the nonequivariant case. Note that the induced functor between the endomorphism rings, 
\[\pi_0(\phi^H): \pi_0^{(\SH(\Z/2)_{f})_{(p)}}\rightarrow \pi_0^{\SH^{fin}_{(p)}}\] 
is the projection onto the first summand of $\Z\oplus\Z$ if $H=\{1\}$ and onto the second summand of $\Z\oplus\Z$ if $H=\Z/2$ \cite[Corollary 1]{Segal}.
\end{proof}

The generalisation to the non-localised category $\SH(\Z/2)_f$ works similarly as in the non-equivariant case. In particular, any prime ideal in $\SH(\Z/2)_f$ is the preimage under the $p$-localisation functor of a prime ideal in $(\SH(\Z/2)_f)_{(p)}$ for some prime $p$.\\

Summarising the result, we can say that all information on thick ideals in $\SH(\Z/2)_f$ given in \cite{Balmer}---namely, that they map surjectively to $\Spec(\Z\oplus\Z)$---is recovered in the results of \cite{Sch} as presented in Chapter \ref{equivariant}. Furthermore, \cite{Sch} does not only specify a preimage of any element in $\Spec(\Z\oplus\Z)$ but gives a complete list of all possible such preimages.

\subsection{Prime ideals in the motivic category $\SH(k)_f$}

\cite[Section 10]{Balmer} studies $\mathcal T=\SH(F)_f$ for $F$ a perfect field, in which case the map $\rho_{\mathcal T}:\Spc(\mathcal T)\rightarrow \Spec(GW(F))$ is surjective. That is, $\Spec(GW(F))$ gives a lower bound on the thick ideals of $\SH(F)_f$. \\

For $F=\C$, the naturality of $\rho_{\mathcal T}$ yields a commutative diagram:
\[\xymatrix{\ar @{} [dr] |{}
\Spc(\SH(\C)_f) \ar@{>>}[r]^{\rho} & \Spec(\pi_0^{\SH(\C)_f}) \\
\Spc(\SH^{fin})\ar[u]^{\Spc(R)}  \ar@{>>}[r]_{\rho} & \Spec(\pi_0^{\SH^{fin}})\ar[u]_{\Spec(\pi_0(R))}.}\]
The map $\Spc(R)$ takes a prime ideal $\mathcal C$ of $\SH^{fin}$ to its preimage under $R=R_\C$. The endomorphism ring $\pi_0^{\SH^{fin}}$ is isomorphic to $\Z$, generated by the identity $S^0\rightarrow S^0$. By \cite[Corollary 4.11]{MorelA1algtop}, the same holds for $\pi_0^{\SH(\C)_f}\cong\Z$. Since $R(\id:S_\C^0\rightarrow S_\C^0)=(\id:S^0\rightarrow S^0)$, it follows that the map $\Spec(\pi_0(R))$ is isomorphic to the identity map $\Spec(\Z)\rightarrow\Spec(\Z)$. Thus:

\begin{cor}
For any $0<n\leq\infty$, the thick ideal $R^{-1}(\mathcal C_n)\subseteq (\SH(\C)_f)_{(p)}$ is a prime ideal, and 
\[\rho_{(\SH(\C)_f)_{(p)}}(R^{-1}(\mathcal C_n))=\begin{cases} 0 & \textup{ if } n=1, \\ 
p\Z_{(p)}& \textup{ if } n>1.\end{cases}\]
\end{cor}

\begin{rk}
From Proposition \ref{counter example}(2) or \cite[Proposition 10.4]{Balmer}, we know that $\Spc(R)$ is not surjective, since $C\eta$ lies in some prime ideal which is not of the form $R^{-1}(\mathcal C_n)$. \\
\end{rk}

For $F=\R$, there is also a commutative diagram:
\[\xymatrix{\ar @{} [dr] |{}
\Spc(\SH(\R)_f) \ar@{>>}[r]^{\rho} & \Spec(\pi_0^{\SH(\R)_f}) \\
\Spc(\SH(\Z/2)_f)\ar[u]^{\Spc(R')}  \ar@{>>}[r]_{\rho} & \Spec(\pi_0^{\SH(\Z/2)_f})\ar[u]_{\Spec(\pi_0(R'))}.}\]

\begin{lemma}
In the above diagram, the right map $\Spec(\pi_0(R'))$ is an isomorphism.
\end{lemma}

\begin{proof}
The realisation functor $R'$ maps generators of $\pi_0^{\SH(\R)}\cong GW(\R)\cong \Z\oplus\Z$ to generators of $\pi_0^{\SH(\Z/2)}$. This follows from \cite[Section 4]{MorelA1algtop} and is explained on \cite[Slides 15-16 and 22-24]{Arolla}: $\pi_0^{\SH(\R)}$ is generated by $1$ and $\epsilon=-1-\rho_{-1}\eta$ (in $GW(\R)$, $\epsilon$ corresponds to the quadratic form $q(x)=-x^2$). The element $\epsilon\in \pi_0^{\SH(\R)}$ is also represented by the twist map $S^{1,1}\wedge S^{1,1}\rightarrow S^{1,1}\wedge S^{1,1}$ \cite[Slide 22]{Arolla}. Now, $S^{1,1}=\G_m=\A^1\setminus \{0\}$ is mapped by $R'$ to the circle with $\Z/2$ acting by involution, which is also denoted by $S^{1,1}$. Hence, $R'(\epsilon)=\epsilon:S^{1,1}\wedge S^{1,1}\rightarrow S^{1,1}\wedge S^{1,1}$, the twist map in $\SH(\Z/2)$. By \cite[Slide 16]{Arolla}, $1$ and $\epsilon$ are generators for $\pi_0^{\SH(\Z/2)}$. Thus, $\pi_0(R'):\pi_0^{\SH(\R)}\rightarrow \pi_0^{\SH(\Z/2)}$ is an isomorphism.
\end{proof}

\begin{cor}
For any prime ideal $\mathcal C_{m,n}\subseteq (\SH(\Z/2)_f)_{(p)}$, $(R')^{-1}(\mathcal C_{m,n})$ is a prime ideal of $(\SH(\R)_f)_{(p)}$ and $\rho_{(\SH(\R)_f)_{(p)}}((R')^{-1}(\mathcal C_{m,n}))$ can be identified with the same prime ideal of $\Z_{(p)}\oplus\Z_{(p)}$ as $\rho_{(\SH(\Z/2)_f)_{(p)}}(\mathcal C_{m,n})$.
\end{cor}

\begin{rk}
Conjecture \ref{conjectures}(4) would imply that in the above diagram, $(R')^{-1}$ is not surjective, as $\thickid(C\eta_{(p)})$ would have to lie in some prime ideal which is not of the form $(R')^{-1}(\mathcal C_{m,n})$. 
\end{rk}

\begin{rk}
Since the K{\"u}nneth formula might not hold for $AK(n)$, we do not know whether $\mathcal C_{AK(n)}$ or $\mathcal C_{AK(m,n)}$ are prime ideals. \end{rk}

\begin{rk}
If $k\subset \C$ or $k\subset \R$ is a field such that $\Spec(GW(\C))\rightarrow\Spec(GW(k))$ is not surjective or $\Spec(GW(\R))\rightarrow\Spec(GW(k))$ is not surjective respectively then \cite[Corollary 10.1]{Balmer} already implies that there is an infinite family of thick ideals in $\SH(k)_f$ which are not of the form $R_k^{-1}(\mathcal C)$ or $(R'_k)^{-1}(\mathcal C)$ respectively.
\end{rk}

\chapter{Motivic type-n spectra}\label{motivic type n}

\begin{defn}
Let $AK(n)$ be the motivic Morava K-theory spectrum as defined in Definition \ref{AK(n)}. We say that $X\in(\SH(k)_f)_{(p)}$ has motivic type $n$ if $AK(n-1)_{\ast\ast}(X)=0$ and $AK(n)_{\ast\ast}(X)\neq 0$. 
\end{defn}

A priori, the motivic type of $X$ might not be unique, as we do not know whether $\mathcal C_{AK(n-1)}\subseteq \mathcal C_{AK(s)}$ for all $s<n$. 
In Section \ref{section n+1}, we will prove that any $X\in\SH(\C)^{fin}_{(p)}$, $p>2$, has a unique motivic type.
For any prime $p$, the motivic type-$n$ spectra that we are going to consider in this chapter satisfy $AK(s)_{\ast\ast}(X)=0$ for all $s<n$.

\begin{rk}
In the topological category $\SH_{(p)}^{fin}$, the notion of types is equivalent to the notion of thick ideals by the thick subcatgory theorem \cite[Theorem 7]{HS}. In Chapter \ref{equivariant}, we have seen that, in equivariant homotopy theory, a more general notion of types is required. Also in the motivic world, not every thick ideal can be described in the language of types, as defined above. For example, as shown in Chapter \ref{nilpotence}, the motivic Morava K-theories $AK(n)$ do not distinguish between the nonequal thick ideals $\thickid(C\eta_{(2)})$ and $\thickid(S^0_{(2)})$, as both are generated by a spectrum of motivic type $0$. 
\end{rk}

In this chapter, we will often assume $k=\C$, because we will need explicit knowledge of $\h^{\ast\ast}$ and of $AK(n)_{\ast\ast}$ for some of our arguments. The results might hold in greater generality, but this seems to require different methods of proof.

If $X_n\in \SH^{fin}_{(p)}$ has type $n$ (see Definition \ref{type}), then $c(X_n)\in\SH(\C)^{fin}_{(p)}$ has motivic type $n$ (Section \ref{constant}). That is, $c(X_n)\in\mathcal C_{AK(n-1)}$ and $c(X_n)\not\in\mathcal C_{AK(n)}$. 
For any given $n$, we show how to construct a spectrum $\X_n\in (\SH(\C)_f)_{(p)}$ with motivic type $n$ that is not in the image of the functor $c$. The construction will be similar to topological constructions given by \cite[Section 4]{Mi} and \cite[Appendix C]{Rav}. We stick to the approach in \cite{Rav} but we believe that a motivic version of Mitchell's spectrum would give another spectrum of motivic type $n$.

This chapter is organised as follows:

We first discuss some foundations that are needed to apply the motivic Adams spectral sequence and obtain a vanishing result for motivic Morava K-theory (Theorem \ref{Mitchell}). Afterwards, we construct a spectrum satisfying the conditions of the theorem and we prove that it has indeed motivic type $n$ (Section \ref{construction}). Finally, we compare our findings to a constant type-$n$ spectrum, $c(X_n)$ (Section \ref{constant}), realising that motivic Morava K-theory does not distinguish the thick ideal generated by $\X_n$ from the one generated by $cX_n$, meaning that both spectra have motivic type $n$. We partly calculate their types with respect to the cohomology theories $c(K(s))$, as well. However, we do not have the answer to the question whether $\thickid(\X_n)$ equals $\thickid(c(X_n))$.

\section{Universal coefficient and K\"unneth theorems}

This section states some general results, which hold in $\SH(k)$, $k$ any field, and which will be used later on.
In the whole chapter, we use the notation $H=H\Z/p$. Recall from Remark \ref{AMU module}(1) that $H$ is cellular.

Proposition 7.7 in \cite{DI} describes the following universal coefficient spectral sequences:

\begin{prop}\label{ucss}
Let $E\in\SH(k)$ be a motivic ring spectrum, $M$ a right $E$-module and $N$ a left $E$-module. Furthermore, assume that $E,M\in\SH(k)^{cell}$ (see Definition \ref{fin}).
\begin{compactenum}[(1)]
\item There is a strongly convergent spectral sequence
\[E^2=\Tor^{E_{\ast\ast}}_{a,b,c}(M_{\ast\ast},N_{\ast\ast})\Rightarrow \pi_{a+b,c}(M\wedge_E N).\]
\item There is a conditionally convergent spectral sequence 
\[E_2=\Ext_{E_{\ast\ast}}^{a,b,c}(M_{\ast\ast},N_{\ast\ast})\Rightarrow \pi_{-a-b,-c}F_E(M,N),\]
where $F_E(-,-)$ denotes the $E$-function spectrum.
\end{compactenum}
\end{prop}

We apply (2) to the case $E=H=H\Z/p$, $M=A\wedge H$ and $N=H$ for $A$ a cell spectrum and get:
\[\Ext_{\h_{\ast\ast}}(\h_{\ast\ast}A,\h_{\ast\ast})\Rightarrow \pi_{\ast\ast}F_H(A\wedge H,H)=\h^{\ast\ast}A.\]

If $A$ is a finite cell spectrum, we can also apply the spectral sequence to the case $E=H$, $M=F(A,H)$, $N=H$:
\[\Ext_{\h^{\ast\ast}}(\h^{\ast\ast}A,\h^{\ast\ast})\Rightarrow \pi_{\ast\ast}F_H(F(A,H),H)=\pi_{\ast\ast}F(F(A,S^0),H)=\h_{\ast\ast}A.\]
The first equality holds because $F(A,H)=F(A,S^0)\wedge H$ for finite $A$ and the second holds because taking the dual of $A$ twice gives $A$ again (see \cite[Proposition III.1.3]{LMS}).

In the case of vanishing higher $\Ext$-groups, the spectral sequence collapses and we get the following result:

\begin{cor}\label{uct}
If $A\in\SH(k)^{cell}$ is any cell spectrum such that $\h_{\ast\ast}A$ is free over $\h_{\ast\ast}$, then $\h^{\ast\ast}A\cong\Hom_{\h_{\ast\ast}}(\h_{\ast\ast}A,\h_{\ast\ast})$.
If $A$ is a finite cell spectrum with $\h^{\ast\ast}A$ free over $\h^{\ast\ast}$, then $\h_{\ast\ast}A\cong\Hom_{\h^{\ast\ast}}(\h^{\ast\ast}A,\h^{\ast\ast})$.
\end{cor}

The universal coefficient spectral sequence also implies the following K\"unneth theorem \cite[Remark 8.7]{DI}:

\begin{prop}\label{kunneth}
Let $A$ and $B$ be motivic spectra such that $A$ is a finite cell spectrum. If $\h^{\ast\ast}A$ is free over $\h^{\ast\ast}$, then 
\[\h^{\ast\ast}(A)\otimes_{\h^{\ast\ast}}\h^{\ast\ast}(B)\cong \h^{\ast\ast}(A\wedge B).\]
\end{prop}

\section{The motivic Steenrod algebra}

In this section, too, we let $H=H\Z/p$. Let $k\subseteq\C$. 

The motivic mod-$p$ Steenrod algebra, $\mathcal A=\mathcal A^{\ast\ast}$ has first been defined in \cite[Section 11]{Operations} as the algebra of certain bistable natural transformations $\h^{\ast\ast}(-)\rightarrow \h^{\ast\ast}(-)$. By \cite[Theorem 3.49]{Voe10} and \cite[Lemma 5.7]{Hoy}, $\mathcal A$ is the algebra of all such operations and, as an $\h^{\ast\ast}$-module,  
\[\mathcal A\cong \h^{\ast\ast}\otimes_{\F_p}\mathcal A_{\Top},\]
where $\mathcal A_{\Top}$ is the topological mod-$p$ Steenrod algebra. Thus, as $\h^{\ast\ast}$-modules,
\[\mathcal A\cong  \h^{\ast\ast}\otimes_{\F_p} RP\otimes_{\F_p} \Lambda_{\F_p}(Q_0,Q_1,\cdots),\] 
where $RP$ is the $\F_p$-module generated by certain products of reduced powers $P^i:\h^{\ast\ast}(-)\rightarrow\h^{\ast+2i(p-1),\ast+i(p-1)}(-)$, $i\geq 0$, and $\Lambda_{\F_p}(Q_0,Q_1,\cdots)$ denotes the exterior algebra over $\F_p$ generated by $Q_i:\h^{\ast\ast}(-)\rightarrow \h^{\ast+2p^i-1,\ast+p^i-1}$, $i\geq 0$, as defined in \cite[Section 13]{Operations}. See also \cite[Corollaries 3 and 4]{Bo} or \cite[Equation (2.18)]{Ya}.

Borghesi \cite[Theorem 12]{Bo} computes the cohomology of the motivic connective Morava K-theory spectrum, which can be expressed by the same formula as in topology.

\begin{prop}\label{borghesi}
\[\h^{\ast\ast}(Ak(s))=\mathcal A/\mathcal A Q_s=\h^{\ast\ast}\otimes RP\otimes \Lambda_{\F_p}(Q_0,\cdots,Q_{s-1},Q_{s+1},\cdots).\]
\end{prop}

If $X$ is a finite cell spectrum such that $\h^{\ast\ast}(X)$ is free over $\h^{\ast\ast}$, we can apply the K\"unneth theorem to $Ak(s)\wedge X$.

\begin{cor}\label{kunneth Ak(n)}
Let $X\in\SH(k)^{fin}$ with $\h^{\ast\ast}(X)$ free over $\h^{\ast\ast}$. Then \[\h^{\ast\ast}(Ak(s)\wedge X)=\mathcal A/\mathcal A Q_s\otimes_{\h^{\ast\ast}}\h^{\ast\ast}(X).\]
\end{cor}

Writing $\Lambda(Q_s)$ for the exterior algebra over $\h^{\ast\ast}$ generated by $Q_s$, we have $\mathcal A=\mathcal A/\mathcal A Q_s\otimes_{\h^{\ast\ast}}\Lambda(Q_s)$. Any resolution of $\h^{\ast\ast}X$ by projective $\Lambda(Q_s)$-modules $P_i$ yields a resolution $\mathcal A/\mathcal A Q_s\otimes_{\h^{\ast\ast}}P_i$ of $\mathcal A/\mathcal A Q_s\otimes_{\h^{\ast\ast}}\h^{\ast\ast}X$ by projective $\mathcal A$-modules. 
This implies the following change of rings isomorphism.

\begin{cor}\label{change of rings}
For any finite cell spectrum $X$ with $\h^{\ast\ast}(X)$ free over $\h^{\ast\ast}$, we have 
\[\Ext_{\mathcal A}(\h^{\ast\ast}(Ak(s)\wedge X),\h^{\ast\ast})\cong \Ext_{\Lambda(Q_s)}(\h^{\ast\ast}(X),\h^{\ast\ast}).\]
\end{cor}

This isomorphism will later be applied to the motivic Adams spectral sequence for $Ak(s)_{\ast\ast}(X)$.

\section{The motivic Adams spectral sequence}\label{Motivic Adams spectral sequence}

Recall the notation $H=H\Z/p$. Let $k\subseteq\C$.

Our aim in this section is to show the existence and convergence of the following Adams spectral sequence:
\[E_2^{s,t,u}=\Ext_{\mathcal A}^{s,t,u}(\h^{\ast\ast}(Ak(s)\wedge X),\h^{\ast\ast})\Rightarrow Ak(s)_{\ast\ast}(X)\]
for finite cell spectra $X\in\SH(k)^{fin}$ (see Definition \ref{fin}), with $\h^{\ast\ast}X$ free over $\h^{\ast\ast}$.

Motivic Adams spectral sequences were first described in \cite{MorASS}. \cite[Corollary 3]{HKO} shows that over fields of characteristic $0$, the Adams spectral sequence for a motivic cell spectrum $X$ of finite type converges to the homotopy groups of the completion $X_{p,\eta}^{\wedge}$. In \cite{Motad}, calculations are made for the case $p=2$, $X=S$. More details and explanations can be found in \cite{Sven}. The convergence of the Adams spectral sequence for $Ak(s)\wedge X$ will follow from \cite{HKO}. However, we have to start again from the Adams resolution, to get the limit term $Ak(s)_{\ast\ast}(X)$ using arguments from \cite[Section 2.1]{green}, and to get a module structure on the sequence.

\begin{defn}
An Adams resolution $(Y_s,g_s,K_s,f_s)_{s\geq 0}$ for a motivic cell spectrum $Y\in\SH(k)^{cell}$ is a diagram
\[\xymatrix{\ar @{} [dr] |{}
Y \ar @{=}[r] & Y_0 \ar[d]^{f_0} & Y_1  \ar[d]^{f_1} \ar[l]^{g_0} & Y_2 \ar[d]^{f_2}\ar[l]^{g_2} & \cdots \ar[l]^{g_3} \\
 & K_0 & K_1 & K_2   },\]
where each $K_s$ is a wedge of suspensions of $H$, $f_s$ is surjective on motivic mod-$p$ cohomology and $Y_{s+1}$ is the fiber of $f_s$. 
\end{defn}

Such an Adams resolution exists whenever $\h^{\ast\ast}Y$ is a free module over $\h^{\ast\ast}$ of motivically finite type \cite{HKO}, \cite[Section 2.5.4]{Sven}. The finite type condition is defined in \cite[Definition 2.12]{Motad}. For our purposes, it suffices to know that, if the generators of $\h^{\ast\ast}Y$ are located in degrees $(i_\alpha,j_\alpha)_\alpha$ with $i_\alpha$ bounded below and, for each $\alpha$, there are only finitely many $\beta$'s with $i_\alpha=i_\beta$ and $j_\alpha\geq j_\beta$, then $\h^{\ast\ast}Y$ is of motivically finite type. This holds, for example, if $Y\in\SH(k)^{fin}$.

\begin{cor}\label{AdRes}
Let $X$ be a finite cell spectrum such that $\h^{\ast\ast}X$ is free over $\h^{\ast\ast}$ and let $Ak(s)$ be the motivic connective Morava K-theory spectrum (see Definition \ref{AK(n)}). Then $Y=Ak(s)\wedge X$ satisfies the finite type condition as described above.
\end{cor}

\begin{proof}
By Corollary \ref{kunneth Ak(n)}, $\h^{\ast\ast}(Y)\cong \mathcal A / \mathcal A Q_s\otimes_{\h^{\ast\ast}}\h^{\ast\ast}X$.
The motivic Steenrod algebra $\mathcal A$ is of motivically finite type, since, for the bidegrees $(i_\alpha,j_\alpha)$ of its $\h^{\ast\ast}$-generators, $i_\alpha$ and $j_\alpha$ are nonnegative and there are only finitely many generators of a fixed bidegree. The same holds for $\mathcal A/\mathcal A Q_s$. Furthermore, $\h^{\ast\ast}X$ is of motivically finite type, since $X$ is finite. It follows that the tensor product $\mathcal A / \mathcal A Q_s\otimes_{\h^{\ast\ast}}\h^{\ast\ast}X$ is also of motivically finite type. 
\end{proof}

\begin{cor}\label{resolution} {\bf (Convergence of the ASS)}

Let $k=\C$ (or any other field satisfying the assumptions of \cite[Theorem 1]{HKO}) 
and let $X$ be a finite cell spectrum such that $\h^{\ast\ast}(X)$ is a free $\h^{\ast\ast}$-module.
Then the Adams spectral sequence for $Y=Ak(s)\wedge X$, $s>0$, strongly converges to $Ak(s)_{\ast\ast}(X)$.
\end{cor}

\begin{proof}
Since $Y$ is of motivically finite type by the previous corollary, strong convergence follows from \cite[Corollary 3]{HKO}. By \cite[Theorem 1]{HKO}, the spectral sequence converges to the $p$-completion of $\pi_{\ast\ast}Y$. Since $Ak(s)$ is a quotient of $ABP/(p)$ by definition, $Y^{\wedge}_p=Y$, whence the limit term is $Ak(s)_{\ast\ast}(X)$.
\end{proof}

\begin{rk}\label{weaker condition}
In Corollary \ref{resolution}, $X$ does not neccessarily need to be finite. All we need to know is that $Ak(s)\wedge X$ is of motivically finite type.
\end{rk}

In Theorem \ref{vanishing}, it will be crucial that this particular spectral sequence is a spectral sequence of $(MGL_{(p)})_{\ast\ast}$-modules.

Since $Ak(s)$ is an $MGL_{(p)}$-module spectrum by Remark \ref{AMU module}(3), $Ak(s)_{\ast\ast}X$ and $\h^{\ast\ast}(Ak(s)\wedge X)$ are $(MGL_{(p)})_{\ast\ast}$-modules. 
Before we can prove compatibility of the Adams spectral sequence with the module structure, we need to determine $\h_{\ast\ast}Ak(s)$. Borghesi \cite[Theorem 5 and Remark 2.2]{Bo} shows:

\begin{lemma}
\[\h_{\ast\ast}(MGL)\cong \h_{\ast\ast}[m_1,m_2,\cdots]\cong \h_{\ast\ast}\otimes_{\F_p}\h_\ast(MU,\Z/p).\]
\end{lemma}

Since $\pi_{\ast\ast}(H\wedge MGL_{(p)})=\pi_{\ast\ast}(H_{(p)}\wedge MGL)$ and $H_{(p)}=H$, it follows that also $\h_{\ast\ast}(MGL_{(p)})$ is a free $\h_{\ast\ast}$-module with basis elements $m_I=m_{i_1}^{k_1}m_{i_2}^{k_2}\cdots m_{i_t}^{k_t}$. Here, $m_i$ is defined as the Hurewicz image of $a_i\in MGL_{2i,i}$ and $a_i$ is the image of a polynomial generator of $MU_\ast$. The motivic connective Morava K-theory spectra are defined as homotopy colimits of spectra $E_i$, which are defined by successively taking cofibers of maps $\Sigma^{2(i-1),i-1}E_{i-1}\stackrel{a_{i-1}}\rightarrow E_{i-1}$, starting from $MGL_{(p)}$ (see Definition \ref{AK(n)}). Passing from $\h_{\ast\ast}E_{i-1}$ to $\h_{\ast\ast}E_i$, the $\h_{\ast\ast}$-basis changes but the fact that the homology is a free $\h_{\ast\ast}$-module remains true for each $i$.  

\begin{cor}\label{dual}
$\h_{\ast\ast}(Ak(s))$ is a free $\h_{\ast\ast}$-module. Hence, 
\[\h_{\ast\ast}(Ak(s))\cong \Hom_{\h^{\ast\ast}}(\h^{\ast\ast}(Ak(s)),\h^{\ast\ast}).\]
\end{cor}

\begin{proof}
The second statement follows from Corollary \ref{uct}. For the cellularity of $Ak(s)$, see Remark \ref{AMU module}(1).
\end{proof}

\begin{prop}
Let $k$ and $X$ be as in Corollary \ref{resolution}. The Adams spectral sequence for $Ak(s)\wedge X$ is a spectral sequence of $(MGL_{(p)})_{\ast\ast}$-modules.
\end{prop}

\begin{proof}
We show that there is an Adams resolution by $MGL_{(p)}$-modules. The construction of the spectral sequence from the Adams resolution preserves such a module structure in every step and the claim will follow.

We use similar arguments as in Corollary \ref{resolution}, but starting with an Adams resolution $\{X_t,g_t,K_t,f_t\}$ for $X$. 
Then we take the smash product of such a resolution with $Ak(s)$ and show that this yields an Adams resolution for $Y=Ak(s)\wedge X$. Since $Ak(s)$ is an $MGL_{(p)}$-module, this will be a resolution by $MGL_{(p)}$-modules. By definition, $f_t$ is surjective on cohomology, which is equivalent to $g_t$ inducing zero in cohomology. It follows that $Ak(s)\wedge g_t$ induces zero in cohomology, hence, $Ak(s)\wedge f_t$ is surjective on cohomology by the long exact fiber sequence.

It remains to show that $Ak(s)\wedge K_t$ is a wedge of suspensions of $H$. Consider one wedge summand $H$ of $K_t$. 
By Corollary \ref{dual}, $\pi_{\ast\ast}(Ak(s)\wedge H)$ is free over $\pi_{\ast\ast}(H)$. Furthermore, $Ak(s)\wedge H$ is a cellular $H$-module. Therefore, \cite[Lemma 5.3]{Hoy} implies that $Ak(s)\wedge H$ is split, that is, $Ak(s)\wedge H\cong \bigvee \Sigma^{\ast\ast}H$, as we wanted to show.
\end{proof}

\section{Vanishing criterion for motivic Morava K-theory}\label{vanishing}

Now we can prove the following motivic version of the vanishing result \cite[Theorem 4.8]{Mi}.

\begin{theorem}\label{Mitchell}{\bf (Vanishing criterion)}

Let $p$ be any prime, $s>0$ and $k$ be as in Corollary \ref{resolution}. Let $X\in\SH(k)^{fin}$ be a finite motivic cell spectrum such that $\h^{\ast\ast}X$ is free over $\Lambda(Q_s)$ (the exterior algebra over $\h^{\ast\ast}$) as a module over the Steenrod algebra. Then \[AK(s)_{\ast\ast}X=0.\]
\end{theorem}

\begin{proof}
With the preparations made so far, the rest of the proof is exactly as in \cite[Theorem 4.8]{Mi}. Since $AK(s)=v_s^{-1}Ak(s)$, it suffices to show that $Ak(s)_{\ast\ast}X$ is $v_s$-torsion. We apply the change of rings isomorphism, Corollary \ref{change of rings}, to the Adams spectral sequence for $Ak(s)\wedge X$ and get:
\[E_2\cong \Ext_{\Lambda(Q_s)}(\h^{\ast\ast}X,\h^{\ast\ast})\Rightarrow Ak(s)_{\ast\ast}X.\]
By the assumption on $X$, this collapses to $\Hom_{\Lambda(Q_s)}(\h^{\ast\ast}X,\h^{\ast\ast})\cong Ak(s)_{\ast\ast}X$, and, by the previous proposition, this is an isomorphism of $(MGL_{(p)})_{\ast\ast}$-modules. Since $(MGL_{(p)})_{\ast\ast}$ acts trivially on the left hand side of this isomorphism, $v_s$ acts trivially on $Ak(s)_{\ast\ast}X$, too. Hence, $AK(s)_{\ast\ast}X=0$.
\end{proof}

\section{Construction of an $X$ such that $\h^{\ast\ast}X$ is free over $\Lambda(Q_s)$}\label{construction}

In this section, we assume $k=\C$ because we will work explicitly with $\h^{\ast\ast}\cong\F_p[\tau]$, $\deg(\tau)=(0,1)$ (see Lemma \ref{H(C)}). 
The construction of a spectrum $X$ with the properties required in the previous theorem can be done similarly as in \cite[Appendix C]{Rav}. That is, one starts with a so-called weakly type $n$ spectrum ($n>s$) and then uses a particular idempotent to split off a (strongly) type $n$ spectrum of a certain smash power of it.

\subsection{Idempotents for free $\h^{\ast\ast}$-modules}

For $V^{\ast\ast}$ an Adams graded graded abelian group (i.e., $V^{\ast\ast}$ has a sign rule in the first grading but not in the second one, see e.g. \cite[Section 3]{MLE}) which is a free $\h^{\ast\ast}$-module, let $V^+=\underset{p \textup{ even}}\bigoplus V^{p,q}$ and $V^-=\underset{p\textup{ odd}}\bigoplus V^{p,q}$ be the even and odd dimensional parts of $V$. 
That is, commuting with an element of $V^+$ does not change the sign but commuting two elements of $V^-$ does. For a vector space $V^\ast$ over $\F_p$, Ravenel defines a number $k_V$ and an idempotent $e_V\in\Z_{(p)}[\Sigma_{k_V}]$, which only depend on $\dim_{\F_p} V^+$ and $\dim_{\F_p} V^-$ \cite[Appendix C.2]{Rav}. The analogous definition can be formulated using $\dim_{\h^{\ast\ast}} V^+$ and $\dim_{\h^{\ast\ast}} V^-$ for our bigraded $\h^{\ast\ast}$-modules. The symmetric group $\Sigma_{k_V}$ acts on $V^{\otimes k_V}$ by permuting the factors. As $V$ is an $\F_p$-module, this induces an action of $\Z_{(p)}[\Sigma_{k_V}]$ on $V^{\otimes k_V}$. For our purposes, it will not be important to know the precise definitions of $k_V$ and $e_V$. We just need to know that they are defined in such a way that the following analogue of \cite[Theorem C.2.1]{Rav} holds.

\begin{prop}
Let $k_V$ and $e_V\in\Z_{(p)}[\Sigma_{k_V}]$ be the number and idempotent defined in \cite[Appendix C.2]{Rav} and let $W=V^{\otimes k_V}$. Then $e_V W\neq 0$. If $U\subset V$ has $\dim U^+\leq \dim V^+ -1$ or $\dim U^-\leq\dim V^- -(p-1)$, then $e_V U^{\otimes k_V}=0.$ Here, $\dim$ denotes the $\h^{\ast\ast}$-dimension and $\otimes$ denotes the tensor product over $\h^{\ast\ast}$.
\end{prop}

\begin{proof}
The proof of \cite[Theorem C.2.1]{Rav} applies to our setting without changes.
\end{proof}

We will need the following Lemma.

\begin{lemma}\label{free module}
Let $M$ be a module over $\F_p[\tau,Q]/Q^2$ which is free as a module over $\F_p[\tau]$ and free as a module over $\F_p[Q]/Q^2$. Then $M$ is a free $\F_p[\tau,Q]/Q^2$-module.
\end{lemma}

\begin{proof}
Let $\{m_i\}_{i\in I}$ be a basis of $M$ over $\F_p[\tau]$ and $\{n_j\}_{j\in J}$ a basis over $\F_p[Q]/Q^2$. Then $M$ is a free $\F_p$-module with bases $\{\tau^k m_i\}_{i\in I,k\in\N}$ and $\{n_j,Qn_j\}_{j\in J}$. As an $\F_p[Q]/Q^2$-module, $M$ decomposes as $M\cong M'\oplus QM'$ with $M'\cong QM'\cong QM$ as $\F_p$-modules. Hence, the elements $m_i$ can be written as $m_i=a_i+Qb_i$ with $a_i,b_i\in M'$. For any $i\in I$ such that both $a_i$ and $b_i$ are nonzero, we replace the basis element $m_i$ by the two elements $a_i$ and $Qb_i$. Then we still have a set of generators for $M$ over $\F_p[\tau]$, which can be turned into a basis by removing elements. Hence, we can assume that all $m_i$ are of the form $m_i=a_i$ or $m_i=Qb_i$. Let $I'=\{i\in I\; |\; m_i=a_i\in M'\}$. Then $M'\cong \F_p\{\tau^k a_i\}_{i\in I',\tau \in\N}$ as an $\F_p$-module. Hence, $QM'\cong \F_p\{Q\tau^k a_i\}_{i\in I',\tau \in\N}$ and $M\cong \F_p\{\tau^ka_i,Q\tau^ka_i\}$. It follows that $\{a_i\}_{i\in I'}$ is a basis of $M$ as a free $\F_p[\tau,Q]/Q^2$-module.
\end{proof}

Theorem C.2.2 of \cite{Rav} explains how to split off a free module over the exterior $\F_p$-algebra generated by $Q_s$ from a module with nontrivial $Q_s$-action. In our setting, this also works for the exterior $\h^{\ast\ast}$-algebra $\Lambda(Q_s)$.

\begin{prop}
Let $V=U\oplus F$ be a splitting of $\Lambda(Q_s)$-modules which are free over $\h^{\ast\ast}$, and $F\neq 0$ be free over $\Lambda(Q_s)$. Then $e_V V^{\otimes k_V}$ is a free $\Lambda(Q_s)$-module.
\end{prop}

\begin{proof}
We write the tensor product as $V^{\otimes k_V}=U^{\otimes k_V}\oplus F'$, where $F'=\underset{\substack{a+b=k_V \\ b\geq 1}}\bigoplus U^{\otimes a}\otimes F^{\otimes b}$. By the proof of \cite[Theorem C.2.2]{Rav}, $F'$ is free over the Hopf algebra $\F_p[Q_s]/Q_s^2$, which we abbreviate by $E$. Let us give the reason for this statement. We show that if the $E$-module $U$ has basis $\{u_i\}_I$ over $\F_p$ and $F$ has basis $\{f_j\}_J$ over $E$, then $U\otimes_{\F_p}F$ has $\F_p$-basis $\{u_i\otimes f_j, Q_s(u_i\otimes f_j)\}_{I,J}$ and hence is a free module over $E$. The module structure of $U\otimes F$ is defined by $Q_s(u\otimes f)=Q_su\otimes f+u\otimes Q_sf$. Since $\{f_j,Q_sf_j\}_J$ defines an $\F_p$-basis of $F$, an $\F_p$-basis of $U\otimes F$ can be given by $\{u_i\otimes f_j,u_i\otimes Q_sf_j\}$. In the formula $Q_s(u_i\otimes f_j)=Q_su_i\otimes f_j+u_i\otimes Q_sf_j$, $Q_su_i=\Sigma r_k u_k$ can be expressed by the basis elements of $U$, hence $Q_su_i\otimes f_j\in \F_p\{u_k\otimes f_j\}$ and the basis elements $u_i\otimes Q_s f_j$ of $U\otimes F$ can be replaced by $Q_s(u_i\otimes f_j)$. We obtain $U\otimes F=\F_p\{u_i\otimes f_j,Q_s(u_i\otimes f_j)\}$. Inductively, it follows that all mixed summands $U^{\otimes a}\otimes F^{\otimes b}$ in $V^{\otimes k_V}$ are free and hence $F'$ is free over $E$. 

The analogue holds if we consider $U$ and $F$ as free modules over $\F_p[\tau]$ instead of $\F_p$ and use $\otimes_{\F_p[\tau]}$. 
It follows that $F'$ is free over $\Lambda(Q_s)$. The direct summands are invariant under the $\Sigma_{k_V}$-action. Hence, we have a short exact sequence 
\[0\rightarrow e_V U^{\otimes k_V}\rightarrow e_V V^{\otimes k_V}\rightarrow e_V F'\rightarrow 0.\]
Since $\deg(Q_s)=(2p^i-1,p^i-1)$, multiplication by $Q_s$ sends $V^+$ to $V^-$ and vice versa. It follows that $\dim F^+>0$ (and $\dim F^->0$) and, hence, $\dim U^+ < \dim V^+$. By the previous proposition, this implies $e_V U^{\otimes k_V}=0$. It follows that $e_V V^{\otimes k_V}=e_V F'$. We have to show that $e_V F'$ is a free $\Lambda(Q_s)$-module. As a module over the exterior $\F_p$-algebra over $Q_s$, this is a direct summand of a free module over a local ring. Hence, $e_V F'$ is free over $\F_p[Q_s]/(Q_s^2)$. Since $e_V F'$ is also a free $\h^{\ast\ast}$-module, it is free over $\Lambda(Q_s)=\F_p[\tau,Q_s]/Q_s^2$ by Lemma \ref{free module}.
\end{proof}

We can apply this to motivic cohomology in the following way:

\begin{theorem}\label{hY}{\bf (Splitting off free $\Lambda(Q_s)$-modules)}

Let $X\in\SH(\C)^{fin}_{(p)}$ be a p-local finite cell spectrum such that $Q_s$ acts nontrivially on $\h^{\ast\ast}(X)$ as an element of the Steenrod algebra. Assume that $V=\h^{\ast\ast}(X)$ is a free $\h^{\ast\ast}$-module and let $Y=e_V(X^{\wedge k_V})$. Then $\h^{\ast\ast}(Y)$ is free over $\Lambda(Q_s)$.
\end{theorem}

\begin{proof}
This is analogous to a statement in \cite[Theorem C.3.2]{Rav}. Since $Q_s$ acts nontrivially, $\h^{\ast\ast}(X)$ contains a nontrivial summand which is free over $\Lambda(Q_s)$. The previous proposition yields the claim. Note that $\h^{\ast\ast}(e_V X^{\wedge k_V})=e_V\h^{\ast\ast}(X)^{\otimes k_V}$ by the K\"unneth theorem (Proposition \ref{kunneth}) and by the way $\Z_{(p)}[\Sigma_{k_V}]$ acts. The K\"unneth theorem also holds for $p$-local finite spectra because $p$-localisation commutes with $\h^{\ast\ast}(-)$.
\end{proof}

$Y$ is a retract of the p-local finite cell spectrum $X^{\wedge k_V}$, but maybe it is not finite itself. Therefore, we need an additional argument which shows that Theorem \ref{Mitchell} holds for $Y$.

\begin{cor}\label{7.4.5}
For $s>0$ and $Y$ as in Theorem \ref{hY}, $AK(s)_{\ast\ast}(Y)=0$.
\end{cor}

\begin{proof}
We have to show $\h^{\ast\ast}(Ak(s)\wedge Y)\cong \mathcal A/\mathcal A Q_s\otimes_{\h^{\ast\ast}}\h^{\ast\ast}(Y)$. Since $\h^{\ast\ast}Y\cong e_V \h^{\ast\ast}(X)^{\otimes k_V}$, $Y$ is of motivically finite type. Remark \ref{weaker condition} applies and the claim follows as in the proof of Theorem \ref{Mitchell}. The left hand side of the claimed isomorphism can be rewritten as 
\[\h^{\ast\ast}(Ak(s)\wedge e_V X^{\wedge k_V})\cong \h^{\ast\ast}((1\wedge e_V)(Ak(s)\wedge X^{\wedge k_V}))\cong (1\otimes e_V)\h^{\ast\ast}(Ak(s)\wedge X^{\wedge k_V}).\] 
Now we can apply the K\"unneth isomorphism and get $(1\otimes e_V)(\mathcal A/\mathcal A Q_s\otimes_{\h^{\ast\ast}}\h^{\ast\ast}(X)^{\wedge k_V})$. This is isomorphic to $\mathcal A/\mathcal A Q_s\otimes_{\h^{\ast\ast}}e_V\h^{\ast\ast}(X)^{\wedge k_V}\cong \mathcal A/\mathcal A Q_s\otimes_{\h^{\ast\ast}}\h^{\ast\ast}(Y)$, which is the right hand side.
\end{proof}

This result tells us that, given a nontrivial $Q_s$-action on $\h^{\ast\ast}(X)$, $X\in\SH(\C)^{fin}_{(p)}$, we can construct a spectrum $Y$ for which $AK(s)_{\ast\ast}(Y)=0$. So, let's construct such an $X$.

\subsection{A finite cell spectrum with nontrivial $Q_s$-action and trivial $Q_n$-action}

Let $k=\C$. 
We combine ideas of Ravenel \cite{Rav} with computations by Voevodsky \cite{Operations}. 
In \cite[Lemma 6.2.6]{Rav}, the given example of a spectrum with nontrivial $Q_s^{\Top}$-action, $s<n$, and trivial $Q_n^{\Top}$-action on $\h^\ast (X)$ is $X=(B\Z/p)^{2p^n}_2$, that is, the suspension spectrum of the $2p^n$-skeleton of the classifying space $B\Z/p$ modulo its $1$-skeleton. Cutting off higher dimensional cells leads to a trivial $Q_n^{\Top}$-action, which is needed for nontrivial $n$-th Morava K-theory.
In \cite[Section 6]{Operations}, the algebraic analogue to $B\Z/p$ is defined as $B\mu_p=\colim_n \tilde{V}_n/\mu_p$, where $\tilde{V}_n=\A^n\setminus \{0\}$ (see the proof of \cite[Lemma 6.3]{Operations}) and $\mu_p$ acts by multiplication with a $p$-th root of unity in each of the $n$ coordinates. Under $R=R_\C$, this action realises to the $\Z/p$-action on $S^{2n-1}\subset \C^n$ rotating each $\C$ factor by a $p$-th root of unity.

\begin{lemma}
\[R(\tilde{V}_n/\mu_p)\cong S^{2n-1}/(\Z/p)\]
is the $(2n-1)$-skeleton of $B\Z/p$ in the CW-structure having one cell in each dimension.
\end{lemma}

\begin{proof}
$B\Z/p$ is the infinite dimensional lens space, as studied for example in \cite[Example 2.43]{Hatcher}. There, it is explained that the $(2n-1)$-skeleton is precisely the $(2n-1)$-dimensional lens space, which is defined as the orbit space $S^{2n-1}/(\Z/p)$.
\end{proof}

By \cite[Example 2.43, page 146]{Hatcher}, the attaching map of the $2k$-cell of $B\Z/p$ is the quotient map $S^{2k-1}\rightarrow S^{2k-1}/(\Z/p)$. 
We define $V_n\in\SH(\C)$ to be the cofiber of the quotient map of suspension spectra $\tilde{V}_{p^n}\rightarrow \tilde{V}_{p^n}/\mu_p$, so that the following lemma holds.

\begin{lemma}
\[R(V_n)=(B\Z/p)^{2p^n}.\]
\end{lemma}

Let $\B$ be the cofiber of the composite map $\tilde{V}_1/\mu_p\rightarrow \tilde{V}_{p^n}/\mu_p\rightarrow V_n$. Then $\B$ is a finite cell spectrum and satisfies the following corollary.

\begin{cor}
\[R(\B)=(B\Z/p)^{2p^n}_2.\]
\end{cor}

The following is a special case of \cite[Proposition 6.10]{Operations} (as explained on \cite[page 20]{Operations}).

\begin{prop}
$\h^{\ast\ast}(B\mu_p)$ is a free $\h^{\ast\ast}$-module with basis $\{v^i,uv^i\; |\; i\geq 0\}$, where $v\in\h^{2,1}(B\mu_p)$ and $u\in\h^{1,1}(B\mu_p)$.
\end{prop}

From this, it follows for dimensional reasons:

\begin{prop}\label{hB}
The cohomology $\h^{\ast\ast}\B$ is the free $\h^{\ast\ast}$-module with basis $\{v^i\; |\; 1\leq i\leq p^n\}\cup\{uv^i\; |\; 1\leq i\leq p^n-1\}$.
\end{prop}

Furthermore, Voevodsky shows \cite[Lemmas 11.2 and 11.3]{Operations}:

\begin{lemma}
On
\[\h^{\ast\ast}(B\mu_p)/\h^{\ast,>0}\h^{\ast\ast}(B\mu_p)\cong \F_p[u,v]/(u^2=0),\]
$\mathcal A/(\h^{\ast,>0}\mathcal A)$ acts by $\beta(u)=v$, $P^i(u)=0$ for $i>0$, $\beta(v^k)=0$ and $P^i(v^k)= \tbinom{k}{i}v^{k+i(p-1)}$.
\end{lemma}

Over $\C$, the action of $Q_i$ can be defined by $Q_0=\beta$ and $Q_{i+1}=P^{p^i}Q_i-Q_iP^{p^i}$ \cite[Proposition 3.1]{VMC}. From this, we can inductively compute the action of $Q_s$ on $uv^k\in \h^{\ast\ast}(\B)/\h^{\ast,>0}\h^{\ast\ast}(\B)$. We get $Q_{s}(uv^k)=cv^{k+p^s}$ with $c\equiv 1 \mod p$, which is nontrivial for $s<n$ and $k<p^n-p^s$. It follows that $Q_s$ acts nontrivially on $\h^{\ast\ast}(\B)$ for $s<n$.

Now we have all ingredients for the motivic type $n$ spectrum in $\SH(\C)$.

\begin{theorem}\label{zero}{\bf (A spectrum of motivic type $n$)}

For a fixed $n>0$, let $V=\h^{\ast\ast}(\B_{(p)})$ and $X=e_V(\B_{(p)})^{\wedge k_V}$, then $AK(s)_{\ast\ast}(X)=0$ for all $s<n$ and $AK(n)_{\ast\ast}(X)\neq 0$. 
\end{theorem}

\begin{proof}
For $s>0$, $AK(s)_{\ast\ast}(X)=0$ follows from Corollary \ref{7.4.5}, whose assumptions are satisfied by the above considerations. 

For $s=0$, note that $AK(0)=p^{-1}MGL_{(p)}/(a_1,a_2,\cdots)$ by Definition \ref{AK(n)}. The main result of \cite{Hoy} implies $AK(0)\cong p^{-1}(H\Z)_{(p)}$. It follows that 
\[AK(0)_{\ast\ast}X\cong \pi_{\ast\ast}(p^{-1}(H\Z\wedge X)_{(p)}).\]
But $p$ acts trivially on $\tilde{V}_m/\mu_p$, which implies that $X$ is $p$-torsion. Therefore, $p^{-1}X\cong 0$ and $AK(0)_{\ast\ast}X=0$.
It remains to show that $AK(n)_{\ast\ast}(X)\neq 0$. 
This can either be done analogously to \cite[Theorem 4.8]{Mi}, using the motivic Atiyah Hirzebruch spectral sequence from Proposition \ref{AHSS}, or by considering the topological realisation $R_\C(X)$. In Section \ref{type R(X)}, we show that $R(X)$ is of type $n$. It follows that $X\in R^{-1}(\mathcal C_n\setminus \mathcal C_{n+1})$. In particular, $X\not\in R^{-1}(\mathcal C_{n+1})$. By Proposition \ref{motivic model}, $\mathcal C_{AK(n)}\subseteq R^{-1}(\mathcal C_{n+1})$. This proves $X\not\in \mathcal C_{AK(n)}$.
\end{proof}

\begin{rk}
The spectrum $X$ is a retract of the $p$-local finite cell spectrum $\B_{(p)}^{\wedge k_V}$ and it follows by Remark \ref{spec compact}(1) that $X$ is compact, i.e., $X\in(\SH(\C)_f)_{(p)}$.
\end{rk}

\begin{rk}
$X=e_V(\B_{(p)})^{\wedge k_V}$ is an example of a motivic spectrum with vanishing Margolis homology groups $MH_s^{p,q}(X)$ for all $s<n$, $p,q\in\Z$, as defined in \cite[Section 3]{VoevZ/2}: we have shown that $\h^{\ast\ast}(X)$ is free over $\Lambda(Q_s)$, which implies that $\ker Q_s=\im Q_s$ for $Q_s:\h^{\ast\ast}(X)\rightarrow \h^{\ast\ast}(X)$.
\end{rk}

\section{The type of the realisation $R(X)$}\label{type R(X)}

Let $k=\C$ and $B=(B\Z/p)^{2p^n}_2$.
We already know that for $R:\SH(\C)\rightarrow \SH$, $R(\B)=B$. Since $R$ preserves $\wedge$-products, it follows 
\[R(X)=R(e_V (\B_{(p)})^{\wedge k_V})=e_V R((\B_{(p)}))^{\wedge k_V}=e_V (B_{(p)})^{\wedge k_V}.\]
The $\F_p$-vector space $\h^\ast (B_{(p)})$ is generated by similar elements as the $\F_p[\tau]$-vector space $V=\h^{\ast\ast}(\B_{(p)})$ (compare Proposition \ref{hB} with \cite[Lemma 6.2.6]{Rav}). In particular, $\dim_{\F_p} (\h^\ast (B_{(p)}))^+=\dim_{\h^{\ast\ast}} V^+$ and $\dim_{\F_p} (\h^\ast(B_{(p)}))^-=\dim_{\h^{\ast\ast}} V^-$. It follows that $R (X)$ is the type-$n$ spectrum defined in \cite[Theorem C.3.2]{Rav}.

\section{The constant type-n spectrum}\label{constant}

In this section, let $k=\C$ and let $X_n=e_V B^{\wedge k_V}$ be the type-n spectrum defined by Ravenel. We want to determine the motivic type of $c X_n$, where $c:\SH\rightarrow \SH(\C)$ as in Chapter \ref{functors}. $X_n$ is constructed via an idempotent from the finite cell spectrum $B=B\Z/p_2^{2p^n}$ \cite[Lemma 6.2.6]{Rav}. First, we calculate $\h^{\ast\ast}(cB)$.

\begin{prop}\label{HcX}
Let $X=\Sigma^\infty Y$ be the suspension spectrum of a finite CW complex. Then $\h^{\ast\ast}(c X)\cong \h^{\ast} (X)[\tau]$ as $\F_p$-modules, where a generator in degree $i$ from the right hand side maps to bidegree $(i,0)$ on the left hand side.
\end{prop}

\begin{proof}
For any $F\in \SH(\C)$, we have $\h^{\ast\ast}(F)\cong \h^{\ast+n,\ast}(F\wedge S_s^n)$. For $F=S^0$, we get $\h^{\ast\ast}(S_s^n)\cong \h^{\ast-n,\ast}(S^0)\cong \h^{\ast-n,\ast}\cong \h^\ast(S^n)[\tau]$. Now let $(Y^k)_k$ be a CW decomposition of $Y$, that is, $\Sigma^\infty Y^k$ is the cofiber of some $\Sigma^\infty\alpha_k:S^{n_k}\rightarrow\Sigma^\infty Y^{k-1}$. We write $X^k$ for $\Sigma^\infty Y^k$. Since $c$ preserves cofiber sequences, we get a cofiber sequence of suspension spectra $S_s^{n_k}\stackrel{c\alpha}\rightarrow c X^{k-1}\rightarrow c X^{k}$. It induces a long exact sequence 
\[\cdots\rightarrow \h^{\ast-1,\ast}(cX^{k-1})\stackrel{(c\alpha)^\ast}\rightarrow \h^{\ast-1,\ast}(S_s^{n_k})\rightarrow \h^{\ast\ast}(cX^{k})\rightarrow \h^{\ast\ast}(cX^{k-1})\stackrel{(c\alpha)^\ast}\rightarrow \h^{\ast\ast}(S_s^{n_k})\rightarrow\cdots\]
We assume inductively that $\h^{i,j}(cX^{k-1})\cong \h^i(X^{k-1})\{\tau^j\}$. let $x \tau^j\in\h^{i,j}(cX^{k-1})$. Since $R(\tau)=1$ and $R(c\alpha)=\alpha$, we have 
\[R((c\alpha)^\ast(x \tau^j))=\alpha^\ast(R(x \tau^j))=\alpha^\ast(x).\]
The only element in $\h^{i,j}(S_s^{n_k})$ which is mapped to $\alpha^\ast(x)$ by $R$ is $\alpha^\ast(x)\tau^j$. This proves $(c\alpha)^\ast\cong\alpha^\ast[\tau]$.
By the five lemma, it follows that the map $\h^{\ast\ast}(cX^k)\rightarrow \h^{\ast}(X^k)[\tau]$, given by sending $x\in\h^{i,j}(cX^k)$ to $ R(x)\tau^j$, is an isomorphism and, inductively, $\h^{\ast\ast}(c X)\cong \h^{\ast} (X)[\tau]$.
\end{proof}

\begin{cor}\label{hcB}
As $\F_p$-modules, 
\[\h^{\ast\ast}(cB)\cong\h^{\ast}(B)[\tau]\cong \h^{\ast\ast}\{x^k\; |\; 1\leq k\leq p^n\}\cup\{yx^k\; |\; 1\leq k\leq p^n-1\},\]
with $\deg(x)=(2,0)$ and $\deg(y)=(1,0)$.
\end{cor}

\begin{proof}
The cohomology of $B$ is described in the proof of \cite[Lemma 6.2.6]{Rav}. We only have to add the polynomial generator $\tau\in\h^{\ast\ast}$.
\end{proof}

For $X=cY$, $Y$ a finite CW spectrum, Proposition \ref{AHSS} implies the convergence of the motivic Atiyah Hirzebruch spectral sequence, 
\[E^{p,q,t}_2=\h^{p+2t,q+t}(X)\otimes K(n)_t\Rightarrow AK(n)^{p,q}(X).\]
By Remark \ref{remark AHSS}(1), the realisation functor $R$ maps this spectral sequence to the topological Atiyah Hirzebruch spectral sequence:
\[\xymatrix{\ar @{} [dr] |{}
E^{p,q,t}_{r} \ar[d]_{R} \ar [rrr]^{d_{r}} &&& E^{p-2r+1,q-r,t+2r}_{r} \ar[d]^{R}  \\
F^{p,t}_{r} \ar[rrr]_{d'_{r}}  &&& F^{p-2r+1,t+2r}_{r} ,  }\]
where $(E^{p,q,t}_{r},d_{r})$ is the spectral sequence mentioned above and $(F^{p,t}_{r},d'_{r})$ is the Atiyah Hirzebruch spectral sequence (see, e.g., \cite[Theorem A.3.7]{Rav}) 
\[F^{p,t}_2=\h^{p+2t}(R(X))\otimes K(n)_t\Rightarrow K(n)^p(R(X)).\]
Note that there is no differential in between $d'_r$ and $d'_{r+1}$, since $K(n)_\ast$ is concentrated in even degrees.\\

Now let $r$ be small enough, so that $E_{r}$ and $F_{r}$ are still equal to the 2-pages. Using Proposition \ref{HcX}, the above square is
\[\xymatrix{\ar @{} [dr] |{}
\h^{p+2t}(R(X))\cdot \tau^{q+t}\otimes K(n)_t \ar[d]_{R} \ar [rrr]^{d_{r}} &&& \h^{{p+2t}+2r+1}(R(X))\cdot\tau^{{q+t}+r}\otimes K(n)_{t+2r} \ar[d]^{R}  \\
\h^{{p+2t}}(R(X))\otimes K(n)_t \ar[rrr]_{d'_{r}}  &&& \h^{{p+2t}+2r+1}(R(X))\otimes K(n)_{t+2r} . }\]
Since $R(\tau)=1$, 
 the vertical maps are isomorphisms. Hence, for this $X$, $d_r$ is completely determined by $d'_r$ through $d_{r}=d'_{r}\cdot \tau^r$. Taking homology, it follows by induction that $E^{p,q,t}_r=F^{p,t}_r\cdot\tau^q$ for all $r<\infty$. Hence, this also holds for $r=\infty$. Since both spectral sequences are strongly convergent, we get:

\begin{prop}
Let $Y$ be a finite CW spectrum. Then 
\[AK(n)^{\ast\ast}(c(Y))=0 \Leftrightarrow K(n)^\ast(Y)=0.\]
\end{prop}

We have proven:

\begin{theorem}\label{AKcX}{\bf (Constant type-$n$ spectra have motivic type $n$)}

The spectrum $c X_n$, where $X_n$ denotes Ravenel's type-$n$ spectrum (or any other $p$-local finite type-$n$ spectrum), has motivic type $n$.
\end{theorem}

This and the previous sections can be summarised by:

\begin{cor}
Let $X=c X_n$ or $X=\X_n$, where $\X_n=e_V(\B_{(p)})^{\wedge k_V}$. Then 
\[\thickid(X)\subseteq \mathcal C_{AK(n-1)}\subseteq R^{-1}(\mathcal C_{n})\] 
is a chain of thick ideals in $(\SH(\C)_f)_{(p)}$, and $\thickid(X)\nsubseteq R^{-1}(\mathcal C_{n+1})$.
\end{cor}

Recall that we derived the second of the above inclusions in Proposition \ref{motivic model}. 

As $cK(n)$ is another motivic model for Morava K-theory (see Section \ref{section motivic model}), we can likewise ask the question, whether also $\thickid(X)\subseteq \mathcal C_{c K(n-1)}$. Let us first consider $c X_n$. 

\begin{prop}
$c K(s)\wedge c(X_n)\cong 0$ if and only if $s<n$. Hence, 
\[\thickid(c X_n)\subseteq \mathcal C_{c K(n-1)}\subseteq R^{-1}(\mathcal C_{n})\]
in $(\SH(\C)_f)_{(p)}$. Furthermore, $\thickid(c X_n)\not\subseteq R^{-1}(\mathcal C_{n+1})$.
\end{prop}

\begin{proof}
We have $c(K(s)\wedge X_n)\cong 0$ if and only if $X_n$ is in $\mathcal C_{s+1}$, since $c$ is fully faithful by \cite[Theorem 1]{L}. Since $X_n\in\mathcal C_{n}\setminus \mathcal C_{n+1}$, the first claim follows. The second claim holds because $R(cX_n)=X_n$.
\end{proof}

For $\X_n$, we know so far:

\begin{lemma}
If $s\geq n$, then $c K(s)\wedge\X_n\not\cong 0$. 
\end{lemma}

\begin{proof}
Assume $c K(s)\wedge \X_n\cong 0$. Then $R(c K(s)\wedge \X_n)\cong K(s)\wedge X_n\cong 0$ and it follows $s<n$.
\end{proof}

For now, the question whether $c K(s)\wedge\X_n\cong 0$ for $s<n$ remains open. If it is true, then $\thickid(\X_n)\subseteq \mathcal C_{c K(n-1)}\subseteq R^{-1}(\mathcal C_{n})$. If it is not true, then $\thickid(\X_n)\nsubseteq \mathcal C_{c K(n-1)}$ and in particular $\X_n\notin \thickid (c X_n)$ and $\thickid(\X_n)\neq\thickid(cX_n)$.\\

\noindent
{\bf Summary.}

We have constructed two different lifts of a topological type-$n$ spectrum to the motivic category $(\SH(\C)_f)_{(p)}$, one of them is in $c(\SH^{fin}_{(p)})$ and the other one is not. These are candidates for generators of different thick sub-ideals of $R^{-1}(\mathcal C_{n})$ inside $(\SH(\C)_f)_{(p)}$. We proved that both spectra have motivic type $n$, meaning that the thick ideals generated by them cannot be distinguished using motivic Morava K-theories $AK(s)$. It is an open question, whether these ideals are equal or whether they can be distinguished by some other motivic model for Morava K-theory or any other method.

\chapter{Bousfield classes}\label{chapter BC}

So far, we have seen that the thick ideals $R_k^{-1}\mathcal C_n$ form a descending chain and that $\mathcal C_{AK(n-1)}$ is a thick ideal contained in $R_k^{-1}\mathcal C_n$. However, we have not seen that the thick ideals $\mathcal C_{AK(n)}$ form a descending chain themselves. The aim of this chapter is to prove this, at least for $k=\C$ and finite cell spectra. That is, we prove that $AK(n)_{\ast\ast}X=0$ implies $AK(n-1)_{\ast\ast}X=0$ for $X\in\SH(\C)^{fin}$, where $n\geq 1$ and $p>2$ (Theorem \ref{AK(n+1)}). As is done in topology \cite[Theorem 2.11]{RavLoc}, we will work in terms of Bousfield classes (Definition \ref{Bousfield class}). 

We proceed as follows. In the first two sections, $k$ can be any subfield of $\C$. In Section \ref{$v_n$-Torsion}, we show that $v_n$-torsion is also $v_{n-1}$-torsion (Theorem \ref{torsion}). Then we show that some basic results on Bousfield classes also apply to the motivic setting (Lemma \ref{bousfield2}). In Section \ref{section action}, we construct a product on $AP(n)$. Here, we need to know that $AP(n)_{\ast\ast}$ vanishes in certain degrees, which we do if we assume $k=\C$, $p>2$ and $n>0$. We continue by showing that, for $p>2$ and $n>0$, $\langle AK(n)\rangle=\langle AB(n)\rangle$ in $\SH(\C)$ (Corollary \ref{AK=AB}), passing from $AK(n)$ to the cohomology theory $AB(n)=v_n^{-1}AP(n)$, which is slightly easier to understand. On the way, we need to compute a couple of things like $AP(m)^{\ast\ast}AP(n)$, to construct stable operations $AP(n)^{\ast\ast}(-)\rightarrow AP(n)^{\ast\ast}(-)$ (Theorem \ref{5.1}). Here, the assumption $k=\C$ is also helpful, as we make explicit use of the formula $AP(n)_{\ast\ast}\cong P(n)_\ast[\tau]$ (Lemma \ref{Ah}). An application of all these results is Theorem \ref{decomposition}, where, for $p>2$ and $k=\C$, we prove 
\[\langle AE(n)\rangle =\underset{0\leq i\leq n}\bigvee\langle AK(i)\rangle.\] 
For the definitions of $AK(n)$, $AB(n)$, $AP(n)$ and $AE(n)$, see Definition \ref{AK(n)}.

Let's start with a definition of Bousfield classes.

\begin{defn}\label{Bousfield class}
Let $k$ be a field and let $\mathcal T=\SH(k)$, $\SH(k)^{cell}$ or $\SH(k)^{fin}$. For any $E\in\SH(k)$, the class of all spectra $X\in\mathcal T$ satisfying $E_{\ast\ast}X\neq 0$ is denoted by $\langle E\rangle$. We write $\langle E\rangle\leq \langle F\rangle$ if $\langle E\rangle$ is a subclass of $\langle F\rangle$. Meet and join of Bousfield classes are given by $\langle E\rangle \wedge\langle F\rangle =\langle E\wedge F\rangle$ and $\langle E\rangle\vee \langle F\rangle=\langle E\vee F\rangle$ \cite[Definition 1.20]{RavLoc}.
\end{defn}

\begin{rk}\label{802}
In \cite[Definition 1.19]{RavLoc}, $\langle E\rangle$ is defined to be the equivalence class of all $F$ such that, for all $X$, $E_\ast X=0$ if and only if $F_\ast X=0$. Thus, $\langle E\rangle $ is determined by the collection of all $X$ such that $E_\ast X\neq 0$, as in the definition above. 
For $\mathcal T=\SH(k)^{cell}$ or $\SH(k)^{fin}$ and $E\in\SH(k)^{cell}$, $\langle E\rangle=\{X\in\mathcal T\;|\;E\wedge X\not\cong 0\}=\mathcal C_E$ by Proposition \ref{C_E}.
\end{rk}

\section{$v_n$-Torsion}\label{$v_n$-Torsion}

In this section, we work in the category $\SH(k)$, $k\subseteq\C$. We prove the following theorem, refining \cite[Theorem 0.1]{JY}.

\begin{theorem}\label{torsion}
Any $v_n$-torsion element in an $\BP_{\ast\ast}\BP$-comodule is also a $v_{n-1}$-torsion element.
\end{theorem}

Recall that $BP_\ast BP\cong BP_\ast\{t^E_{\Top}\}$, where $E=(e_1,e_2,\cdots)$ runs over all finite sequences of non-negative integers and $\deg(t^E_{\Top})=\sum e_i(2p^i-2)$, as is shown in \cite[Theorem II.16.1(ii)]{Adams}. 

From \cite[Definition 5.3]{Vez}, we know that $\BP$ is a commutative ring spectrum. It follows that $\BP_{\ast\ast}\BP$ is an $\BP_{\ast\ast}$-module. We describe its structure:

\begin{lemma}\label{ABPABP}
\begin{compactenum}[(1)]
\item As a left $\BP_{\ast\ast}$-module,
\[\BP_{\ast\ast}\BP\cong \BP_{\ast\ast}\{t^E\},\] 
where $E$ runs over all finite sequences of non-negative integers, 
\[\deg\left(t^{(e_1,e_2,\cdots)}\right)=\left(\sum_i{e_i(2p^i-2)},\sum_i{e_i(p^i-1)}\right),\]
and $R_k(t^E)=t^E_{\Top}$ 
Consequently, as a right $\BP_{\ast\ast}$-module, 
\[\BP_{\ast\ast}\BP\cong \BP_{\ast\ast}\{c(t^E)\},\] 
where $c:\BP_{\ast\ast}\BP\rightarrow \BP_{\ast\ast}\BP$ 
is the conjugation, induced by the twist map $\BP\wedge \BP\rightarrow \BP\wedge \BP$.
\item As a left $ABP^{\ast\ast}$-module, 
\[\BP^{\ast\ast}\BP\cong \BP^{\ast\ast}[[s^E]],\] 
which is the completion of $\BP^{\ast\ast}\{s^E\}$ under infinite sums. Here, $\deg(s^E)=\deg(t^E)$ and the $s^E$ are the dual basis elements to $t^E$.
\end{compactenum}
\end{lemma}

In particular, $\BP_{\ast\ast}\BP$ is a flat $\BP_{\ast\ast}$-module.

\begin{proof}
By Remark \ref{AMU module}(4), \cite[Proposition 9.1.(i)]{MLE} applies to $\BP$, so we have $\BP_{\ast\ast}\BP\cong \BP_{\ast\ast}\otimes_{BP_\ast}BP_\ast BP$. Since $BP_\ast BP\cong BP_\ast\{t^E_{\Top}\}$, the first claim follows. As $BP_\ast BP$ is projective over $BP_\ast$, \cite[Proposition 9.7(i)]{MLE} implies  $\BP^{\ast\ast}\BP\cong \Hom_{BP_\ast}(BP_\ast BP,\BP_{\ast\ast})$. Since the analogue holds for $BP^\ast BP$, this is the same as $\BP^{\ast\ast}\otimes_{BP^\ast}BP^\ast BP$, which is $\BP^{\ast\ast}[[s^E]]$ by \cite[Lemma 5.12]{JY}.
\end{proof}

For any finite motivic cell spectrum $X$, the morphism 
\[ m_{\ast}:\BP_{\ast\ast}(\BP)\otimes_{\BP_{\ast\ast}}\BP_{\ast\ast}( X)\rightarrow \BP_{\ast\ast}(\BP\wedge X)\] 
induced by the $\BP$-module structure of $\BP\wedge X$ is an isomorphism: this holds for $X=S^0$, since $\BP_{\ast\ast}\BP$ is free over $\BP_{\ast\ast}$ and, hence, for any finite $X$ by cellular induction via the five lemma, see also \cite[Lecture 3, Lemma 1]{AdaLect}. More precisely, one has to check that a cofiber sequence $X\rightarrow Y\rightarrow Z$ induces a long exact sequence on both ends of $ m_{\ast}$. On the right hand side, this is the long exact $\BP_{\ast\ast}$-sequence induced by $\BP\wedge X\rightarrow \BP\wedge Y\rightarrow \BP\wedge Z$ and on the left hand side, we get a long exact $\BP_{\ast\ast}$-sequence tensored over the field $\F_p$ with $\F_p\{t^E\}$, which is still exact.
Actually, the above map is an isomorphism for any motivic spectrum $X$ by \cite[Lemma 5.1(i)]{MLE}.

This can be used to define elementary $\BP$-operations as in \cite[Section 1]{JY}:

\begin{defn}\label{s_E}
Let 
\[\psi_X:\BP_{\ast\ast}X\rightarrow \BP_{\ast\ast}\BP\otimes_{\BP_{\ast\ast}}\BP_{\ast\ast}X\]
be the map induced by $1\wedge i\wedge 1:\BP\wedge S^0\wedge X\rightarrow \BP\wedge \BP\wedge X$ (where $i$ is the unit of the ring spectrum $ABP$) followed by $( m_\ast)^{-1}$. Then the elementary $\BP$-operation $s_E:\BP_{\ast\ast}X\rightarrow \BP_{\ast\ast}X$ is defined by 
\[\psi_X(x)=\sum_{E}{c(t^E)\otimes s_E(x)}.\]
\end{defn}

The $s^E\in \BP^{\ast\ast}\BP$ from the above lemma are special cases of these operations (see \cite[Lemma 5.12]{JY}). The $\BP$-operations satisfy a Cartan formula similar to \cite[Formula (1.7)]{JY}:

\begin{lemma}\label{Cartan}
If $y\in \BP_{\ast\ast}$ and $x\in \BP_{\ast\ast}X$, then 
\[s_E(yx)=\sum_{F+G=E}s_F(y)s_G(x).\]
\end{lemma}

\begin{proof}
We have to show that 
\[\psi_X(yx)=\sum_E\left(c(t^E)\otimes \sum_{F+G=E}s_F(y)s_G(x)\right).\] 
Since $\psi_X$ is a map of $\BP_{\ast\ast}$-modules, $\psi_X(yx)=\psi_{S^0}(y)\psi_X(x)$, which is equal to $\sum_E\sum_{F+G=E}(c(t^F)c(t^G)\otimes s_F(y)s_G(x))$. As $F$ and $G$ are exponent sequences, $t^F t^G=t^E$ and it follows that $\psi_X(yx)=\sum_E(c(t^E)\otimes \sum_{F+G=E}s_F(y)s_G(x))$.
\end{proof}

The next lemma compares the Hopf algebroid structures of $(BP_\ast,BP_\ast BP)$ and $(ABP_{\ast\ast}, ABP_{\ast\ast}ABP)$ and is closely related to \cite[Section 5]{MLE}.

\begin{lemma}\label{hopf alg}
$(\BP_{\ast\ast},\BP_{\ast\ast}\BP)$ is a flat Hopf algebroid, and there is a map of Hopf algebroids $(BP_\ast,BP_\ast BP)\rightarrow (\BP_{\ast\ast},\BP_{\ast\ast}\BP)$ such that the following hold:
\begin{compactenum}[(1)]
\item $BP_\ast\rightarrow \BP_{\ast\ast}$ is the inclusion into $\underset{i}\bigoplus \BP_{(2i,i)}$, mapping $v_i^{\Top}$ to $v_i$.
\item $BP_\ast BP\rightarrow \BP_{\ast\ast}\BP$ is the map $BP_\ast\{t^E_{\Top}\}\rightarrow \BP_{\ast\ast}\{t^E\}$ given by (1) on $BP_\ast$ and mapping $t^E_{\Top}$ to $t^E$.
\item The map $\psi=\psi_{S^0}$ from Definition \ref{s_E} is the coaction map of $\BP_{\ast\ast}$ as a left $(\BP_{\ast\ast},\BP_{\ast\ast}\BP)$-comodule and, similarly, for the map $\psi^{\Top}$ from \cite{JY}. Furthermore, the map of Hopf algebroids preserves the comodule structure in the sense that
\[\xymatrix{\ar @{} [dr] |{}
BP_\ast \ar[d] \ar[rrr]^{\psi^{\Top}} &&& BP_\ast BP\otimes_{BP_\ast}BP_\ast\ar[d] \\
\BP_{\ast\ast}  \ar[rrr]^{\psi} &&& \BP_{\ast\ast} \BP\otimes_{\BP_{\ast\ast}}\BP_{\ast\ast }}\]
commutes.
\end{compactenum}
\end{lemma}

\begin{proof}
In \cite[Corollary 5.2(i)]{MLE}, it is shown that $(E_{\ast\ast},E_{\ast\ast}E)$ is a flat Hopf algebroid whenever $E$ is a cellular ring spectrum and $E_{\ast\ast}E$ is a flat $E_{\ast\ast}$-module. This is the case for $E=\BP$ by Remark \ref{AMU module} and Lemma \ref{ABPABP}(1). Furthermore, an orientation on $E$ induces a map of Hopf algebriods $(MU_\ast,MU_\ast MU)\rightarrow (E_{\ast\ast},E_{\ast\ast}E)$ by \cite[Corollary 6.7]{MLE}, where $MU_\ast\rightarrow E_{\ast\ast}$ is the map classifying the formal group law (FGL) given by the orientation on $E$ and $MU_\ast MU\rightarrow E_{\ast\ast}E$ classifies the strict isomorphism of formal group laws induced on $E\wedge E$ by the left and right units $E\rightarrow E\wedge E$.

If the FGL associated with the orientation of $E$ is p-typical, the map of Hopf algebroids factors through $(BP_{\ast},BP_{\ast}BP)$ because $BP_\ast$ and $BP_\ast BP$ classify p-typical group laws and strict isomorphisms of p-typical group laws, respectively (see \cite[Appendix 2]{green}). Recall that $\BP$ is oriented by $MU_\ast\rightarrow MGL_{\ast\ast}\rightarrow \BP_{\ast\ast}$ and its FGL is p-typical because this factors as $MU_\ast\rightarrow BP_\ast\rightarrow \BP_{\ast\ast}$ by the construction of $\BP$, where the latter map is as claimed in (1). Hence, we get a map of Hopf algebroids $(BP_\ast,BP_\ast BP)\rightarrow (\BP_{\ast\ast},\BP_{\ast\ast}\BP)$ satisfying (1).

Before we pove (2), we will show the analogous statement for $MGL$. By \cite[Corollary 6.7]{MLE}, as above, there is a map of Hopf algebroids 
\[(MU_\ast, MU_\ast MU)\rightarrow (MGL_{\ast\ast},MGL_{\ast\ast}MGL),\] determined by the complex orientation on $MGL$. Let $x$ be the orientation on $MGL\wedge MGL$ induced by the left unit $MGL\rightarrow MGL\wedge MGL$ and $x'$ be the orientation induced by the right unit. By \cite[Lemma 6.4.(ii)]{MLE}, $x'=\underset{i\geq 0}\sum b_i x^{i+1}$. The orientations $x$ and $x'$ correspond to formal group laws $F_L$ and $F_R$ and the formula implies that the $b_i$ are the coefficients of the power series of the strict isomorphism $\varphi$ between $F_L$ and $F_R$ (as in the proof of \cite[Theorem 4.1.11]{green}). The same formula holds for the orientations on $MU\wedge MU$ by \cite[Lemma 4.1.8]{green} and the strict isomorphism between the FGLs $F_L^{\Top}$ and $F_R^{\Top}$ therefore has coefficients $b_i^{\Top}$. By definition, $b_i$ is the image of $b_i^{\Top}$ under $MU_\ast[b_i^{\Top}]\cong MU_\ast MU\rightarrow MGL_{\ast\ast}\otimes_{MU_\ast}MU_\ast MU\cong MGL_{\ast\ast} MGL$, as in \cite[Lemma 6.4.(i)]{MLE}. 

$MU_\ast MU$ classifies strict isomorphisms $F\overset{f}\rightarrow G$ of FGLs in the following way: $MU_\ast MU\cong MU_\ast[b_i^{\Top}]$, where $MU_\ast$ classifies $F$, $f$ is given by a power series in $b_i^{\Top}$ and $G$ is determined by $F$ and $f$. 

In our setting, the map $MU_\ast MU\rightarrow MGL_{\ast\ast}MGL$ is the map corresponding to the strict isomorphism $F_L\overset{\varphi}\rightarrow F_R$. Furthermore, $F_L$ is the FGL associated with the orientation of $MGL_{\ast\ast} MGL$ given by $MU_\ast\rightarrow MGL_{\ast\ast}\rightarrow MGL_{\ast\ast}[b_i]\cong MGL_{\ast\ast}MGL$ (because the isomorphism herein is an isomorphism of left $MGL_{\ast\ast}$-modules), and, similarly, $F_L^{\Top}$
 is the FGL associated with the orientation $MU_\ast\rightarrow MU_\ast[b_i^{\Top}]\cong MU_\ast MU$. This implies that the following square commutes, where the left horizontal maps are the obvious inclusions, and the vertical maps are the maps from \cite[Corollary 6.7]{MLE}.
\[\xymatrix{\ar @{} [dr] |{}
MU_\ast \ar[d] \ar[r] & MU_\ast[b_i^{\Top}]\ar[r]^{\cong}& MU_\ast MU\ar[d] \\
MGL_{\ast\ast}  \ar[r] & MGL_{\ast\ast}[b_i] \ar[r]^{\cong}& MGL_{\ast\ast}MGL.}\]
In terms of group laws, the right hand map sends $\varphi^{\Top}$ to $\varphi$. Since these are the power series described above, $b_i^{\Top}$ is sent to $b_i$. This proves the $MGL$-version of (2).

For the $\BP$-version, one has to show that the following diagram commutes, where the right map is the Hopf algebroid morphism and the left map is as described in (2).
\[\xymatrix{\ar @{} [dr] |{}
BP_\ast[t_i^{\Top}]\ar[d]\ar[r]^{\cong}& BP_\ast BP\ar[d] \\
\BP_{\ast\ast}[t_i] \ar[r]^{\cong}& \BP_{\ast\ast}\BP.}\]
The proof is exactly the same as in the case of $MGL$. One simply has to replace each FGL by the corresponding $p$-typical FGL.

For (3), note that, by its definition, $\psi$ is the coaction map that comes naturally with any flat Hopf algebroid $(E_{\ast\ast},E_{\ast\ast}E)$ (as in \cite[Corollary 5.2(i)]{MLE}), meaning in particular that the diagram in (3) commutes.
\end{proof}

\begin{defn}\label{exp seq}
For an exponent sequence $E=(e_1,e_2,\cdots)$ as in Lemma \ref{ABPABP}(1), we set $|E|=\sum_i{e_i(2p^i-2)}$. Let $I_m=(p,v_1,\cdots,v_{m-1})\subset \BP_{\ast\ast}$ be the usual prime ideal.
\end{defn}

\begin{cor}\label{modulo}
Consider $s_E:\BP_{\ast\ast}\rightarrow \BP_{\ast\ast}$ as in Definition \ref{s_E} and assume that $|E|\geq 2kp^s(p^n-p^m)$ for $n\geq m$, $s\geq 0$ and $k\geq 1$. Then 
\[s_E(v_n^{kp^s})=\begin{cases}v_m^{kp^s} &\mod I^{s+1}_m \textup{ if } e_{n-m}=kp^{s+m}\textup{ and } e_i=0 \textup{ for } i\neq n-m\\ 0&\mod I^{s+1}_m \textup{ otherwise.}\end{cases}\]
\end{cor}

\begin{proof}
By the above Lemma, the following diagram commutes:
\[\xymatrix{\ar @{} [dr] |{}
BP_\ast \ar[d] \ar[rrr]^{\psi^{\Top}} &&& BP_\ast BP\otimes_{BP_\ast}BP_\ast\ar[d] \\
\BP_{\ast\ast}  \ar[rrr]^{\psi} &&& \BP_{\ast\ast} \BP\otimes_{\BP_{\ast\ast}}\BP_{\ast\ast },}\]
where the vertical arrows send $v_n^{\Top}$ to $v_n$ and $t^E_{\Top}$ to $t^E$. We consider the element $(v_n^{\Top})^{kp^s}\in BP_{\ast}$. It is mapped horizontally to 
\[\psi^{\Top}((v_n^{\Top})^{kp^s})=\sum_E{c^{\Top}(t_{\Top}^E)\otimes s_E^{\Top}((v_n^{\Top})^{kp^s})}.\] 
By \cite[Lemma 2.1]{JY}, $s_E^{\Top}((v_n^{\Top})^{kp^s})$ satisfies the formula we want to prove. Since all the elements from the topological case map to the corresponding elements in the lower row, the formula has to hold there, too.
\end{proof}

The rest of the proof of Theorem \ref{torsion} is exactly the same as \cite[Lemmas 2.2 and 2.3]{JY}, relying mainly on the above lemma and the Cartan formula. Theorem \ref{torsion} implies the following corollary. The analogous topological statement can be found in the proof of \cite[Theorem 2.1(d)]{RavLoc}.

\begin{cor}\label{bousfieldE}
Let $k\subseteq \C$. If $AE(n)_{\ast\ast} X=0$, then also $AE(i)_{\ast\ast}X=0$ for all $i\leq n$. In terms of Bousfield classes in $\SH(k)$:
\[\left\langle AE(n) \right\rangle \geq \left\langle AE(i) \right\rangle \textup{ for all } n\geq i.\]
\end{cor}

\begin{proof}
Since $E(n)$ is Landweber exact (see \cite{Landweber} or \cite[Section 4.2]{green}), the $ABP$-version of \cite[Theorem 8.7]{MLE} applies to $AE(n)_{\ast\ast}(X)$, yielding
\[AE(n)_{\ast\ast}(X)\cong ABP_{\ast\ast}(X)\otimes_{BP_{\ast}}E(n)_\ast,\]
which is an $ABP_{\ast\ast}ABP$-comodule via the map $\psi_X$ from Definition \ref{s_E}. 

As $E(n)=(v_n^{\Top})^{-1}BP/(v_{n+1}^{\Top},v_{n+2}^{\Top},\cdots)$, the condition $AE(n)_{\ast\ast}(X)=0$ is equivalent to $ABP_{\ast\ast}(X)\otimes_{BP_{\ast}}BP/(v^{\Top}_{n+1},v^{\Top}_{n+2},\cdots)$ being $v_n$-torsion. By Theorem \ref{torsion}, it follows that $ABP_{\ast\ast}(X)\otimes_{BP_{\ast}}BP/(v^{\Top}_{n+1},v^{\Top}_{n+2},\cdots)$ is $v_i$-torsion for any $i\leq n$. This implies that also $ABP_{\ast\ast}(X)\otimes_{BP_{\ast}}BP/(v^{\Top}_{i+1},v^{\Top}_{i+2},\cdots)$ is $v_i$-torsion, which is equivalent to $AE(i)_{\ast\ast}(X)=0$.
\end{proof}

\section{Properties of Bousfield classes}

Ravenel has shown the following properties of Bousfield classes \cite[Section 1]{RavLoc}. They hold in any tensor triangulated category $(\mathcal T,\wedge)$.

\begin{lemma}\label{bousfield}
\begin{compactenum}[(1)]
\item In an exact triangle, each Bousfield class is less or equal to the wedge of the other two \cite[Proposition 1.23]{RavLoc}.
\item If $M$ is a module spectrum over the ring spectrum $E$, then $\left\langle M \right\rangle\leq \left\langle E \right\rangle$ \cite[Proposition 1.24]{RavLoc}.
\item Let $\Sigma$ be an auto-equivalence in $\mathcal T$. If $Y$ is the homotopy cofiber of $\Sigma^d X\overset{f}\rightarrow X$ and $\hat X=\underset{f}\colim (\Sigma^{-k d} X)$, then 
$\left\langle X \right\rangle=\langle \hat X \rangle\vee \left\langle Y \right\rangle$ \cite[Lemma 1.34]{RavLoc}.
\end{compactenum}
\end{lemma}

Furthermore, the following relations from \cite[Section 2]{RavLoc} also hold in $\SH(k)$:

\begin{lemma}\label{bousfield2}
\begin{compactenum}[(1)]
\item $\langle AE(n)\rangle \geq \langle AK(n)\rangle$,
\item $\langle AE(n)\rangle \wedge\langle AP(n+1)\rangle =\langle v_n^{-1} \BP\rangle\wedge\langle AP(n+1)\rangle=\langle 0\rangle$,
\item $\langle AP(n)\rangle = \langle AB(n)\rangle \vee \langle AP(n+1)\rangle$.
\end{compactenum}
\end{lemma}

\begin{proof}
Constructing $AK(n)$ from $AE(n)$, (1) follows from Lemma \ref{bousfield}(1). (3) is a direct application of Lemma \ref{bousfield}(3) (see also \cite[Theorem 2.1(c)]{RavLoc}). For the first part of (2), note that $\langle AE(n)\rangle\leq\langle v_n^{-1}\BP\rangle$ since $AE(n)=v_n^{-1}\BP/(v_{n+1},v_{n+2},\cdots)$. When we prove the second equation in (2), the first one will follow from this inequality because $\langle 0\rangle$ is the empty set.

It remains to show that $AP(n+1)\wedge v_n^{-1}\BP\cong 0$, as proven in the topological setting in \cite[Lemma 2.3]{RavLoc}. This spectrum is the homotopy cofiber of the map
\[v_n\wedge 1:AP(n)\wedge v_n^{-1}\BP\rightarrow AP(n)\wedge v_n^{-1}\BP.\] 
We claim that $(v_n\wedge 1)_\ast=(1\wedge v_n)_\ast$ on $\pi_{\ast\ast}(AP(n)\wedge \BP)$. Since $(1\wedge v_n)_\ast$ is an isomorphism on $\pi_{\ast\ast}(AP(n)\wedge v_n^{-1}\BP)$, this will imply that the homotopy cofiber is contractible.

To prove this claim, note that $(v_n\wedge 1)_\ast$ and $(1\wedge v_n)_\ast$ are induced by the respective maps on $\pi_{\ast\ast}(\BP\wedge \BP)$, where they are given by applying the left respectively right unit $\BP_{\ast\ast}\rightarrow \BP_{\ast\ast}\BP$ to $v_n$. In the topological case, the left and right units applied to $v_n$ are the same modulo $I_n$ by \cite[II.16.1 (ii)]{Adams}. By Lemma \ref{ABPABP}(1) and the inclusion of $BP_\ast$ in $\BP_{\ast\ast}$, this also holds motivically. 
Hence, $(v_n\wedge 1)_\ast$ and $(1\wedge v_n)_\ast$ are the same modulo $I_n$. It remains to show that $I_n\subseteq \BP_{\ast\ast}\BP$ maps to $0$ under $\BP_{\ast\ast}\BP\rightarrow AP(n)_{\ast\ast}\BP$. For $n=0$, there is nothing to show. Assume that $I_n$ is mapped to zero in $AP(n)_{\ast\ast}\BP$ for some $n$. Consider the map $AP(n)_{\ast\ast}\BP\overset {i_\ast}\rightarrow AP(n+1)_{\ast\ast}\BP$ induced by the map to the cofiber in $AP(n)\overset{v_n}\rightarrow AP(n)\overset i\rightarrow AP(n+1)$. The inductive assumption implies that $I_n$ is still zero in $AP(n+1)_{\ast\ast}\BP$. Recall that $I_{n+1}\subseteq \BP_{\ast\ast}\BP$ is the ideal generated by $I_n$ and $v_n$. Since $i_\ast\circ (v_n)_\ast=0$ in the long exact sequence 
\[\cdots \rightarrow AP(n)_{\ast\ast}\BP \overset{(v_n)_\ast}\rightarrow AP(n)_{\ast\ast}\BP\overset{i_\ast}\rightarrow AP(n+1)_{\ast\ast}\BP\rightarrow\cdots,\]
$i_\ast$ maps $v_n=(v_n)_\ast(1)$ to $0$. Hence, $I_{n+1}=0$ in $AP(n+1)_{\ast\ast}\BP$.
\end{proof}

\section{The action of $v_i$ on $AP(n)$}\label{section action}

In \cite[Appendix]{JW}, a geometric proof using the Baas-Sullivan construction of $P(n)$ shows that the action of $v_i^{\Top}$ on $P(n)_\ast(X)$ is zero for any $0\leq i<n$. A non-geometric proof of this result is given by \cite[Satz 1.3.4]{Nassau}, which was motivated by \cite{W-products}. Nassau's proof can be simplified using the language of triangulated categories of modules, which is basically done in \cite[Lemma V.2.4]{EKMM}, as well as in \cite[Lemma 3.2]{Strickland-modules}. 
These proofs rely on the fact that the $v_i$ are non-zero divisors of $MU_\ast$ and that $BP_{\ast}$ vanishes in certain degrees, which is not known in the motivic case. In the following, we will use ideas from \cite{Nassau} and \cite{EKMM} to give a proof which also works in $\SH(\C)$. The main difference is that we only know coefficients after passing to $MGL_{(p)}/(p)$ (see Lemma \ref{Ah}), which is why we have to work with $R/(x,y)$, while \cite[Chapter V]{EKMM} only works with $R/x$ for some ring spectrum $R$. 

In this section, we will prove that $v_i$ acts trivially on $AP(n)$ if $k=\C$. Furthermore, we will show that if $p$ is odd, then $AP(n)_{\ast\ast}(X)$ and $AP(n)^{\ast\ast}(X)$ are $AP(n)_{\ast\ast}$-modules for any $X\in \SH(\C)$.\\

Recall that $MGL$ can be constructed as an $E_\infty$-ring spectrum \cite[Theorem 14.2]{Hu}, which is equivalent to a strictly commutative ring spectrum by the motivic version of \cite[Corollary II.3.6]{EKMM}.

Let $R\in\SH(k)$ be a strictly commutative ring spectrum with multiplication $m:R\wedge R\rightarrow R$ and unit $i:S^{0,0}\rightarrow R$. Let $x:S^{k,l}\rightarrow R$ for some $k,l\in\Z$. In our application, we will have $R=MGL_{(p)}$. Note that $MGL_{(p)}$ is the homotopy colimit of the diagram of maps $MGL\overset{n}\rightarrow MGL$ for all positive integers relatively prime to $p$ (see \cite[end of Section 14]{Hu}). As these are maps of strictly commutative ring spectra, $MGL_{(p)}$ is also a strictly commutative ring spectrum (\cite[Theorem 4.1(3)]{SS} applied to $S^{0,0}$-algebras implies that the category of strictly commutative ring spectra is cocomplete).  

Let $M$ be an $R$-module with action map $\nu_M:R\wedge M\rightarrow M$. Let 
\[\phi=\nu_M\circ(x\wedge 1_M):S^{k,l}\wedge M\rightarrow R\wedge M\rightarrow M.\] 
The map $\phi$ is the action of $x$ on $M$.

The $R$-module structure on $S^{k,l}\wedge M$ is given by 
\[\nu_{S^{k,l}\wedge M}:R\wedge S^{k,l}\wedge M\overset{\tau\wedge 1_M}\longrightarrow S^{k,l}\wedge R\wedge M\overset{1_{S^{k,l}}\wedge \nu_M}\longrightarrow S^{k,l}\wedge M.\]

\begin{lemma}\label{module map}
The map $\phi$ is an $R$-module map.
\end{lemma}

\begin{proof}
We have to check the commutativity of the following diagram:
\[\xymatrix{\ar @{} [dr] |{}
R\wedge S^{k,l}\wedge M \ar[rr]^{1\wedge x\wedge 1}\ar[d]_{\tau\wedge 1} && R\wedge R\wedge M \ar[r]^{1\wedge \nu_M} & R\wedge M \ar[d]^{\nu_M}\\
S^{k,l}\wedge R\wedge M \ar[r]^{1\wedge \nu_M} & S^{k,l}\wedge M \ar[r]^{x\wedge 1} & R\wedge M \ar[r]^{\nu_M} & M.}\]
In this diagram, we can replace $(x\wedge 1)\circ (1\wedge \nu_M)$ by 
\[S^{k,l}\wedge R\wedge M \overset{x\wedge 1\wedge 1}\longrightarrow R\wedge R\wedge M \overset{1\wedge \nu_M}\longrightarrow R\wedge M\]
and we can fill in a diagonal across the upper left corner,
\[S^{k,l}\wedge R\wedge M \overset{(\tau\wedge 1)\circ (x\wedge 1\wedge 1)}\longrightarrow R\wedge R\wedge M.\]
It follows that the above diagram commutes if and only if the following diagram commutes:
\[\xymatrix{\ar @{} [dr] |{}
S^{k,l}\wedge R\wedge M \ar[r]^{x\wedge 1\wedge 1}\ar[d]_{x\wedge 1\wedge 1} & R\wedge R\wedge M \ar[r]^{\tau\wedge 1} & R\wedge R\wedge M \ar[r]^{1\wedge \nu_M} & R\wedge M\ar[d]^{\nu_M}\\
R\wedge R\wedge M \ar[rr]^{1\wedge \nu_M} && R\wedge M \ar[r]^{\nu_M} & M.}\]
Since $R$ is commutative, we have $m\circ\tau = m$ and, hence,
\[\nu_M\circ(m\wedge 1)\circ(\tau\wedge 1)= \nu_M\circ(m\wedge 1).\]
Since $M$ is an $R$-module, this is the same as 
\[\nu_M\circ(1\wedge \nu_M)\circ (\tau \wedge 1)=\nu_M\circ(1\wedge \nu_M),\]
proving the commutativity of the above diagram.
\end{proof}

In the following, we denote the homotopy category of $R$-modules by $\RMod$. A stable model structure on $R$-modules is given by \cite[Theorem 4.1]{SS} applied to the motivic stable model structure from \cite{J}, so that $\RMod$ is a triangulated category (compare \cite[page 554]{MLE}).  
Since $\phi$ from above is a map of $R$-modules (Lemma \ref{module map}), there is an exact triangle in $\RMod$,
\[S^{k,l}\wedge M\overset\phi\longrightarrow M \overset\eta\longrightarrow N\overset\partial \longrightarrow S^{k+1,l}\wedge M.\]
The cofiber $N$ is also denoted $M/x$. Application of $[-,N]_{\RMod}$ to this exact triangle yields a long exact sequence
\[\cdots \rightarrow [S^{2k+1,2l}\wedge M,N]_{\RMod} \overset{\partial^\ast}\rightarrow [S^{k,l}\wedge N,N]_{\RMod} \overset{\eta^\ast}\rightarrow [S^{k,l}\wedge M,N]_{\RMod} \rightarrow \cdots.\]
Let $\psi=\nu_N\circ(x\wedge 1_N):S^{k,l}\wedge N\rightarrow R\wedge N\rightarrow N$. This map is the action of $x$ on $M/x$, and a map of $R$-modules by Lemma \ref{module map}. We want to show that, under certain assumptions, $\psi=0$, meaning that $x$ acts trivially on $M/x$.

First, we consider 
\[\eta^\ast\psi=\psi\circ(1_{S^{k,l}}\wedge \eta): S^{k,l}\wedge M\overset{1\wedge\eta}\longrightarrow S^{k,l}\wedge N\overset{\psi}\rightarrow N.\]
By the definition of $\psi$, this is the map $\nu_N\circ(x\wedge 1_N)\circ(1_{S^{k,l}}\wedge \eta)=\nu_N\circ(1_R\wedge \eta)\circ (x\wedge 1_M)$. Since $\eta$ is a map of $R$-modules, this is the same as $\eta\circ\nu_M\circ(x\wedge 1_M)$, which, by definition of $\phi$, is the map $\eta\circ\phi$. By the above exact triangle, it follows that $\eta^\ast\psi=\eta\circ\phi=0$. The long exact sequence implies that there is a map in $\RMod$, 
\[\overline{\psi}:S^{2k+1,2l}\wedge M\rightarrow N,\]
such that $\psi=\partial^\ast \overline{\psi}$.\\

Now, we assume that either $M=R$ or that $M=R/y$ for some $y\in R_{\ast\ast}$. Furthermore, we assume that $\pi_{2k+1,2l}N=0$.

Note that Case 1 is a special instance of Case 2 (with $y=0$), so the reader may skip the following paragraph and continue reading at Case 2. However, Case 1 is easier, for which reason it might still be a good idea to read it, anyway. \\

{\bf Case 1.} $M=R$.

We have $\overline{\psi}:S^{2k+1,2l}\wedge R\rightarrow N$. Since $R$ is a ring spectrum, the unit $i:S^{0,0}\rightarrow R$ satisfies $1_R=m\circ (1_R\wedge i):R\wedge S^{0,0}\rightarrow R\wedge R\rightarrow R$. Hence,
\[\overline{\psi}=\overline{\psi}\circ(1_{S^{2k+1,2l}}\wedge m)\circ (1_{S^{2k+1,2l}\wedge R}\wedge i).\]
By the definition of the $R$-module structure $\nu_{S^{2k+1,2l}\wedge R}$ on $S^{2k+1,2l}\wedge R$,
\[(1_{S^{2k+1,2l}}\wedge m)=\nu_{S^{2k+1,2l}\wedge R}\circ(\tau\wedge 1_R):S^{2k+1,2l}\wedge R\wedge R\rightarrow R\wedge S^{2k+1,2l}\wedge R\rightarrow S^{2k+1,2l}\wedge R.\]
Hence, 
\[\overline{\psi}=\overline{\psi}\circ\nu_{S^{2k+1,2l}\wedge R}(\tau\wedge 1_R)(1_{S^{2k+1,2l}\wedge R}\wedge i)\]
\[= \overline{\psi}\circ\nu_{S^{2k+1,2l}\wedge R}(1_R\wedge 1_{S^{2k+1,2l}}\wedge i)(\tau\wedge 1_R).\]
Since $\overline{\psi}$ is an $R$-module map, $\overline{\psi}\nu_{S^{2k+1,2l}\wedge R}=\nu_N(1_R\wedge\overline{\psi})$, and therefore
\[\overline{\psi}=\nu_N(1_R\wedge \overline{\psi})(1_R\wedge 1_{S^{2k+1,2l}}\wedge i)(\tau\wedge 1_R)\]
\[=\nu_N(1_R\wedge (\overline{\psi}(1_{S^{2k+1,2l}}\wedge i)))(\tau\wedge 1_R).\]
Now, $\overline{\psi}(1_{S^{2k+1,2l}}\wedge i):S^{2k+1,2l}\wedge S^{0,0}\rightarrow N$ is in $\pi_{2k+1,2l}N$, which we assumed to be zero. Thus, $\overline{\psi}=0$ and it follows that also $\psi=\partial^\ast\overline{\psi}$, which is the action of $x$ on $M/x$, is zero, as we wanted to show.

Thus, we have shown:

\begin{prop}\label{trivial action 1}
Let $R\in\SH(k)$ be a strictly commutative ring spectrum and $x\in\pi_{k,l}R$. Assume that $\pi_{2k+1,2l}(R/x)=0$. Then $x$ acts trivially on $R/x$, i.e., the map $\psi$ from above is zero. The same holds if $x\in\pi_k R$ for a strictly commutative ring spectrum $R\in\SH$ such that $\pi_{2k+1}(R/x)=0$.
\end{prop}

Note that part of the above argument can be formulated more generally:

\begin{lemma}\label{ring to module}
Let $R$ be a (homotopy) ring spectrum, $M$ a left $R$-module, and $\pi_{k,l}M=0$. Then any $R$-module map $\psi:S^{k,l}\wedge R\rightarrow M$ is homotopically trivial.
\end{lemma}

\begin{proof}
Let $i:S^{0,0}\rightarrow R$ be the unit of $R$. It satisfies $1_R=m(1_R\wedge i)$. Thus, 
\[\psi=\psi\circ m(1_R\wedge i)=\nu_M(1_R\wedge\psi)(1_R\wedge i)=\nu_M(1_R\wedge \psi i),\]
with $\psi i\in\pi_{k,l}M=0$. It follows $\psi=\nu_M(1_R\wedge 0)=0$.
\end{proof}

Now we pass on to case 2.\\

{\bf Case 2.} $M=R/y$.\\

Let $y:S^{k',l'}\rightarrow R$ and let $\phi'=m(y\wedge 1_R):S^{k',l'}\wedge R\rightarrow R\wedge R\rightarrow R$ be the action of $y$ on $R$.
We have an exact triangle in $\RMod$,
\[S^{k',l'}\wedge R\overset{\phi'}\rightarrow R\overset{\eta'}\rightarrow M\overset{\partial'}\rightarrow S^{k'+1,l'}\wedge R,\]
and, again, an exact sequence
\[\rightarrow [S^{2k+k'+2,2l+l'}\wedge R,N]_{\RMod}\overset{\partial'^\ast}\rightarrow [S^{2k+1,2l}\wedge M,N]_{\RMod}\overset{\eta'^\ast}\rightarrow [S^{2k+1,2l}\wedge R,N]_{\RMod}\rightarrow \cdots\]
We consider 
\[\eta'^\ast\overline{\psi}=\overline{\psi}\circ(1_{S^{2k+1,2l}}\wedge \eta'):S^{2k+1,2l}\wedge R\rightarrow S^{2k+1,2l}\wedge M\rightarrow N.\]
Let $i:S^{0,0}\rightarrow R$ be the unit of $R$, as before. Since $\eta':R\rightarrow M$ is a map of $R$-modules (using Lemma \ref{module map}), $\eta'\circ m=\nu_M(1_R\wedge \eta')$,
and, hence,
\[\nu_M(1_R\wedge \eta')(1_R\wedge i)=\eta'\circ m(1_R\wedge i)=\eta'.\]
Thus,
\[\overline{\psi}(1_{S^{2k+1,2l}}\wedge \eta')=\overline{\psi}(1_{S^{2k+1,2l}}\wedge \nu_M(1_R\wedge \eta')(1_R\wedge i))\]
\[=\overline{\psi}(1_{S^{2k+1,2l}}\wedge \nu_M(1_R\wedge \eta'i)).\]
Since $\overline{\psi}$ is a map of $R$-modules, $\overline{\psi}\circ\nu_{S^{2k+1,2l}\wedge M}=\nu_N\circ(1_R\wedge \overline{\psi})$, where, by definition, $\nu_{S^{2k+1,2l}\wedge M}=(1_{S^{2k+1,2l}}\wedge \nu_M)(\tau\wedge 1_M)$. Hence,
$\overline{\psi}(1_{S^{2k+1,2l}}\wedge\nu_M)=\nu_N(1_R\wedge \overline{\psi})(\tau\wedge 1_M)$, and, therefore,
\[\eta'^\ast\overline{\psi}=\nu_N(1_R\wedge\overline{\psi})(\tau\wedge 1_M)(1_{S^{2k+1,2l}}\wedge 1_R\wedge \eta'i)\]
\[=\nu_N(1_R\wedge\overline{\psi})(1_R\wedge 1_{S^{2k+1,2l}}\wedge \eta'i)(\tau\wedge 1_{S^{0,0}})\]
\[=\nu_N(1_R\wedge(\overline{\psi}(1_{S^{2k+1,2l}}\wedge \eta'i)))(\tau\wedge 1_{S^{0,0}}).\]
Now, $\overline{\psi}(1_{S^{2k+1,2l}}\wedge \eta'i):S^{2k+1,2l}\wedge S^{0,0}\rightarrow N$ lies in $\pi_{2k+1,2l}N$, which we assumed to be zero. Hence, $\eta'^\ast\overline{\psi}=0$. By the long exact sequence from above, it follows that $\overline{\psi}=\partial'^\ast\overline{\overline{\psi}}$ for some $R$-module map $\overline{\overline{\psi}}:S^{2k+k'+2,2l+l'}\wedge R\rightarrow N$. Thus, $\psi=\partial^\ast\overline{\psi}=\partial^\ast\partial'^\ast\overline{\overline{\psi}}$. 

Consider the following commutative diagram. The map $\psi$ is the precomposition of $\overline{\overline{\psi}}$ with the diagonal of the righthand square.
\[\xymatrix{\ar @{} [dr] |{}
S^{2k,2l}\wedge M \ar[r]^{\phi} \ar[dr]\ar[d]_{\partial'} & S^{k,l}\wedge M \ar@{-->}[d]_{\zeta} \ar[r]^{\eta} & S^{k,l}\wedge N \ar@{-->}[dl]_{\xi}\ar[dr] \ar[r]^{\partial}\ar[d]^{\partial} & S^{2k+1,2l}\wedge M\ar[d]^{\partial'}\\
S^{2k+k'+1,2l+l'}\wedge R \ar[r]^{\phi'} & S^{2k+1,2l}\wedge R \ar[r]^{\eta'} & S^{2k+1,2l}\wedge M \ar[r]^{\partial'} & S^{2k+k'+2,2l+l'}\wedge R.}\]
Since both rows are exact triangles, we can fill in a map $\zeta:S^{k,l}\wedge M\rightarrow S^{2k+1,2l}\wedge R$. We have $\phi'\circ\partial'=0$, as both of these are maps in the lower triangle. Thus, the diagonal in the first square is zero and the map $\zeta$ lifts to a map $\xi:S^{k,l}\wedge N\rightarrow S^{2k+1,2l}\wedge R$. It follows that $\partial'\circ\partial=\partial'\circ\eta'\circ\xi=0$, and, hence, $\psi=\overline{\overline{\psi}}\circ\partial'\circ\partial=0$.

We have proven:

\begin{prop}\label{trivial action 2}
Let $R\in\SH(k)$ be a commutative ring spectrum, $x\in\pi_{k,l}R$ and $y\in\pi_{k',l'}R$. Assume that $\pi_{2k+1,2l}(R/(y,x))=0$. Then $x$ acts trivially on $R/(y,x)$. 

The same holds if $x\in\pi_k R$ and $y\in\pi_{k'}R$ for a commutative ring spectrum $R\in\SH$ such that $\pi_{2k+1}(R/(y,x))=0$.
\end{prop}

This result can be applied to the action of $v_i$ on $AP(n)$ for $0\leq i<n$, at least for $k=\C$.

\begin{cor}\label{trivial action}
Let $k=\C$ and $n\geq 1$. Then $v_i$ acts trivially on $AP(n)$ for any $0\leq i < n$ and $v_i$ acts trivially on $Ak(n)$ for any $i\neq n$. 
\end{cor}

\begin{proof}
First, we consider $MGL_{(p)}/(p,v_i)$ for some $0< i<n$ (thus, $n\geq 2$). Since $MU_\ast$ is concentrated in even degrees and Lemma \ref{Ah} holds for $k=\C$ and quotients of $MGL_{(p)}/p$, we get $\pi_{2k+1,2l}(MGL_{(p)}/(p,v_i))=0$ for any $k,l$. By Proposition \ref{trivial action 2} it follows that $v_0=p$ and $v_i$ act trivially on $MGL_{(p)}/(p,v_i)$. 

By \cite[Lemma V.1.10]{EKMM}, 
\[MGL_{(p)}/(p,v_i)\cong MGL_{(p)}/p\wedge_{MGL_{(p)}} MGL_{(p)}/v_i,\]
and, by \cite[Remark 6.20]{Hoy}, $AP(n)\cong MGL_{(p)}/J$, where $J$ contains $a_i\in MGL_{\ast\ast}$, $i\neq 2p^i-2$, as well as $v_i$, $0\leq i\leq n-1$. From \cite[Lemma V.1.10]{EKMM}, it follows that
\[AP(n)\cong MGL_{(p)}/(p,v_i)\wedge_{MGL_{(p)}} MGL_{(p)}/(J\setminus \{p,v_i\}).\] 
Now $v_i$ acts trivially on $MGL_{(p)}/(p,v_i)$, i.e., the respective map $\phi_i$ on $MGL_{(p)}/(p,v_i)$ is zero. It follows that also the map $\phi_i\wedge_{MGL_{(p)}} 1_{MGL_{(p)}/(J\setminus \{p,v_i\})}$ is zero, meaning that $v_i$ acts trivially on $AP(n)$. Similarly, $p$ acts trivially on $AP(n)$. This proves that all $v_i$, $0\leq i<n$, act trivially on $AP(n)$ if $n\geq 2$.

If $n=1$, one has to replace $MGL_{(p)}/(p,v_i)$ by $MGL_{(p)}/(p)$ in the above argument.

Furthermore,
\[Ak(n)\cong   AP(n)\wedge_{MGL_{(p)}} MGL_{(p)}/(p,v_{n+1},v_{n+2},\cdots) ,\]
so $v_i$ acts trivially on $Ak(n)$, too, for $0\leq i<n$. For $i>n$, the claim follows analogously to the above argument.
\end{proof}

\begin{rk}
Working with modules over $MGL_{(p)}$ has the advantage that the $E_\infty$-structure allows us to use the isomorphism from \cite[Lemma V.1.10]{EKMM}, as well as the results from \cite{SS} (below Lemma \ref{module map}). For $ABP$, we only know of a commutative ring structure in the weak sense (see \cite[Definition 5.3]{Vez}). Note that for $BP$, an $E_4$-structure is constructed in \cite{BasterraMandell}. 

The $ABP$-module structure on $AP(n)$ is the action of $ABP$ on itself in $ABP\wedge_{MGL_{(p)}}MGL_{(p)}/(v_0,\cdots,v_{n-1})\cong AP(n)$ and it is (by its construction) compatible with the $MGL_{(p)}$-action on $AP(n)$.
\end{rk}

Recall that $\nu:BP\wedge P(n)\rightarrow P(n)$ induces a $BP_\ast$-module structure on $P(n)_\ast(X)$ and $P(n)^\ast(X)$ for any $X\in\SH$ and that $P(n)_\ast=BP_\ast/(v_0,\cdots,v_{n-1})$. Therefore, the classical version of the above corollary immediately implies that the $BP_\ast$-module structure on $P(n)_\ast(X)$ and $P(n)^\ast(X)$ induces a $P(n)_\ast$-module structure, as also concluded in \cite[Remark 2.5(a)]{JW}. \\

Our next aim is to show that also for $X\in\SH(\C)$, the $ABP_{\ast\ast}$-module structure on $AP(n)_{\ast\ast}(X)$ and $AP(n)^{\ast\ast}(X)$ induces a structure of $AP(n)_{\ast\ast}$-modules, where the ring structure on $AP(n)_{\ast\ast}$ is defined via the isomorphism $AP(n)_{\ast\ast}\cong \h_{\ast\ast}\otimes_{\F_p} P(n)_\ast$ (Lemma \ref{Ah}). We will show in Lemma \ref{ring iso} that this is the right choice of ring structure on $AP(n)_{\ast\ast}$.\\

Let $R$ be a strictly commutative ring spectrum, let $M=R/y$ satisfy $\pi_{2k'+1,2l'}M=0$ (where $(k',l')$ is the degree of $y$, as in Case 2 above), and let $N=M/x$ satisfy $\pi_{2k+1,2l}N=0$. \\

In the commutative diagram,
\[\xymatrix{\ar @{} [dr] |{}
S^{k',l'}\wedge R\wedge M \ar[d]_{1\wedge\nu_M}\ar[r]^{y\wedge 1\wedge 1} & R\wedge R\wedge M \ar[d]_{1\wedge \nu_M}\ar[r]^{m\wedge 1} & R\wedge M \ar[d]^{\nu_M}\\
S^{k',l'}\wedge M \ar[r]_{y\wedge 1} & R\wedge M \ar[r]_{\nu_M} & M, }\]
the composition $\nu_M(y\wedge 1_M)$ is zero by Proposition \ref{trivial action 1}. Furthermore, $(m\wedge 1_M)\circ (y\wedge 1_R\wedge 1_M)=\phi'\wedge 1_M$, where $\phi'$ is, as before, the map whose cofiber is $M$. Thus, we have
\[\xymatrix{\ar @{} [dr] |{}
S^{k',l'}\wedge R\wedge M \ar[drr]_{0}\ar[rr]^{\phi'\wedge 1} && R\wedge M \ar[d]_{\nu_M}\ar[r]^{\eta'\wedge 1} & M\wedge M \ar@{-->}[dl]^{\mu_M} \\
 && M & ,}\]
and there exists a map $\mu_{M}:M\wedge M\rightarrow M$ in the homotopy category $\RMod$ such that $\mu_M\circ (\eta'\wedge 1_M)=\nu_M$.

Next, we define a map $\nu_{M,N}:M\wedge N\rightarrow N$ by $\nu_{M,N}=\mu_M\wedge_R 1_{R/x}$, using $N\cong M\wedge_R R/x$ \cite[Lemma V.1.10]{EKMM}. It satisfies $\nu_{M,N}(\eta'\wedge 1_N)=\nu_N$ (by applying $-\wedge_R R/x$ to the analogous equation for $\mu_M$) and $\nu_{M,N}(1\wedge \eta)=\eta\circ\mu_M$ (because $\eta:M\rightarrow N$ is the canonical map $M\wedge_R R\rightarrow M\wedge_R R/x $).

In the commutative diagram (where the right square commutes because $\nu_{M,N}$ is a map of $R$-modules)
\[\xymatrix{\ar @{} [dr] |{}
S^{k,l}\wedge M\wedge N \ar[d]_{1\wedge\nu_{M,N}}\ar[rr]^{x\wedge 1\wedge 1} && R\wedge M\wedge N \ar[d]_{1\wedge \nu_{M,N}}\ar[rr]^{\nu_M\wedge 1} && M\wedge N \ar[d]^{\nu_{M,N}}\\
S^{k,l}\wedge N \ar[rr]_{x\wedge 1} && R\wedge N \ar[rr]_{\nu_N} && N, }\]
the lower composition is the action of $x$ on $N$, which is trivial by Proposition \ref{trivial action 2}. Thus, 
\[\nu_{M,N}(\phi\wedge 1_N)=\nu_{M,N}(\nu_M\wedge 1_N)(x\wedge 1_M\wedge 1_N)=0.\]
Hence, there exists a map $\mu_N:N\wedge N\rightarrow N$ in $\RMod$ making the following diagram commutative.
\[\xymatrix{\ar @{} [dr] |{}
S^{k,l}\wedge M\wedge N \ar[drr]_{0}\ar[rr]^{\phi\wedge 1} && M\wedge N \ar[d]_{\nu_{M,N}}\ar[r]^{\eta\wedge 1} & N\wedge N \ar@{-->}[dl]^{\mu_N} \\
 && N & .}\]

In particular, this applies to $N=MGL_{(p)}/(p,x)$ as in Corollary \ref{trivial action}, yielding an $MGL_{(p)}$-module map 
\[\mu_x:MGL_{(p)}/(p,x)\wedge MGL_{(p)}/(p,x)\rightarrow MGL_{(p)}/(p,x).\] 

\begin{lemma}\label{p,x}
$AP(n)$ is isomorphic as an $MGL_{(p)}$-module to the $\wedge_{MGL_{(p)}}$-product of all $MGL_{(p)}/(p,x)$, $x\in J$,
where $J$ is as in the proof of Corollary \ref{trivial action}.
\end{lemma}

\begin{proof}
By Proposition \ref{trivial action 1}, $p$ acts trivially on $MGL_{(p)}/p$, proving $MGL_{(p)}/(p,p)=MGL_{(p)}/p$. By \cite[Lemma V.1.10]{EKMM}, it follows that 
\[MGL_{(p)}/(p,x)\wedge_{MGL_{(p)}}MGL_{(p)}/(p,y)\cong MGL_{(p)}/(p,x)\wedge_{MGL_{(p)}}MGL_{(p)}/y\] 
for any $x,y\in J$. With $p=v_0\in J$, this implies that the $\wedge_{MGL_{(p)}}$-product of all $MGL_{(p)}/(p,x)$ is isomorphic to the $\wedge_{MGL_{(p)}}$-product of all $MGL_{(p)}/x$, which is the quotient $MGL_{(p)}/J\cong AP(n)$ (as in Corollary \ref{trivial action}).
\end{proof}

We can, therefore, define a map of $MGL_{(p)}$-modules,
\[\mu_{AP(n)}:AP(n)\wedge AP(n)\rightarrow AP(n)\]
by applying the maps $\mu_x$ on each factor $MGL_{(p)}/(p,x)$. 

\begin{lemma}
If, in the above setting, $\pi_{k'+1,l'}M=\pi_{2k'+2,2l'}M=\pi_{3k'+3,3l'}M=0$, then $\mu_M$ is homotopy associative.

If, furthermore, $\pi_{k+1,l}N=\pi_{2k+2,2l}N=\pi_{3k+3,3l}N=0$, then $\mu_N$ is also homotopy associative.
\end{lemma}

\begin{proof}
Let $\mu=\mu_M$ and $\nu=\nu_N$.
We have to show that 
\[\delta=\mu(\mu\wedge 1_M-1_M\wedge\mu):M\wedge M\wedge M\rightarrow M\] is zero.
Let $\delta'=\delta(\eta'\wedge 1_M\wedge 1_M): R\wedge M\wedge M\rightarrow M$, $\delta''=\delta'(1_R\wedge \eta'\wedge 1_M):R\wedge R\wedge M\rightarrow M$ and $\delta'''=\delta''(1_R\wedge 1_R\wedge \eta')=\delta\circ (\eta')^{\wedge 3}:R\wedge R\wedge R\rightarrow M$.
Consider
\[\xymatrix{\ar @{} [dr] |{}
R\wedge R \ar[dd]_{m} \ar[dr]_{1\wedge\eta'} \ar[rr]^{\eta'\wedge\eta'} && M\wedge M \ar[dd]^{\mu} \\
 & R\wedge M \ar[ur]^{\eta'\wedge 1}\ar[dr]_{\nu} & \\
R \ar[rr]_{\eta'} && M.}\]
The top triangle obviously commutes, the right triangle commutes (up to homotopy) by the definition of $\mu$, and the large triangle commutes because $\eta'$ is an $R$-module map. Thus, $\mu\circ (\eta')^{\wedge 2}=\eta'\circ m$, and it follows 
\[\mu(\mu\wedge 1-1\wedge \mu)(\eta')^{\wedge 3}=\mu(\eta' m\wedge \eta'-\eta'\wedge\eta' m)\]
\[=\mu\circ(\eta')^{\wedge 2}(m\wedge 1-1\wedge m)=\eta' m(m\wedge 1-1\wedge m),\]
which vanishes by the associativity of $m$. Thus, $\delta'''=\delta''(1\wedge 1\wedge \eta')=0$. 

This implies that $\delta''$ factors through an $R$-module map $\zeta:S^{k'+1,l'}\wedge R\wedge R\wedge R\rightarrow M$, as in the following diagram:
\[\xymatrix{\ar @{} [dr] |{}
R\wedge R\wedge R \ar[r]^{1\wedge 1\wedge \eta'}\ar[dr]_{\delta'''=0} & R\wedge R\wedge M \ar[rr]^{1\wedge 1\wedge \partial'} \ar[d]_{\delta''} && S^{k'+1,l'}\wedge R\wedge R\wedge R \ar@{-->}[dll]^{\zeta} \\
 & M &&. }\]
 Now, $R\wedge R\wedge R$ is a ring spectrum and $\zeta$ can be considered as a map of $R\wedge R\wedge R$-modules. By Lemma \ref{ring to module} and the assumption on $\pi_{k'+1,l'}M$, $\zeta$ must be trivial. Thus, $\delta''=0$. Again, this implies $\delta'=\zeta'(1_R\wedge \partial'\wedge 1_M)$ for some $\zeta':S^{k'+1,l'}\wedge R\wedge R\wedge M\rightarrow M$. The $R\wedge R\wedge R$-module map $\zeta'(1_{S^{k'+1,l'}}\wedge 1_R\wedge 1_R\wedge \eta')$ is a degree $(k'+1,l')$-map from the ring spectrum $R\wedge R\wedge R$ to $M$, and therefore vanishes by Lemma \ref{ring to module}. It follows that $\zeta'=\zeta''(1_{S^{k'+1,l'}}\wedge 1_R\wedge 1_R\wedge \partial' )$ for some map $\zeta'':S^{2k'+2,2l'}\wedge R\wedge R\wedge R\rightarrow M$. 
\[\xymatrix{\ar @{} [dr] |{}
S^{k'+1,l'}\wedge R\wedge R\wedge R \ar[r]^{1\wedge \eta'}\ar[dr]_{0} & S^{k'+1,l'}\wedge R\wedge R\wedge M \ar[r]^{1\wedge \partial'} \ar[d]_{\zeta'} & S^{2k'+2,2l'}\wedge R\wedge R\wedge R \ar@{-->}[dl]^{\zeta''} \\
 & M &. }\] 
By the second assumption on $\pi_{\ast\ast}M$, $\zeta''=0$. It follows that $\delta'=\zeta'(1\wedge\partial'\wedge 1)=\zeta''(1\wedge\partial'\wedge\partial')=0$. That is, $\delta(\eta'\wedge 1_M\wedge 1_M)=\delta'=0$, which implies $\delta=\zeta'''(\partial'\wedge 1_M\wedge 1_M)$ for some $\zeta''':S^{k'+1,l'}\wedge R\wedge M\wedge M\rightarrow M$.
 
Now, $\zeta'''(1\wedge\eta'\wedge 1)(1\wedge 1\wedge\eta')$ is a map from $R\wedge R\wedge R$ to $M$ of degree $(k'+1,l')$, and therefore trivial. Thus, $\zeta'''(1\wedge\eta'\wedge 1)=\zeta^{(4)}(1\wedge 1\wedge\partial')$ with $\zeta^{(4)}:S^{2k'+2,2l'}\wedge R\wedge R\wedge R\rightarrow M$, which is also trivial. It follows that $\zeta'''=\zeta^{(5)}(1\wedge\partial'\wedge 1)$ for some $\zeta^{(5)}:S^{2k'+2,2l'}\wedge R\wedge R\wedge M\rightarrow M$, and $\delta=\zeta'''(\partial'\wedge 1\wedge 1)=\zeta^{(5)}(\partial'\wedge\partial'\wedge 1)$. The map $\zeta^{(5)}(1\wedge 1\wedge\eta):S^{2k'+2,2l'}\wedge R\wedge R\wedge R\rightarrow M$ is again zero, which implies $\zeta^{(5)}=\zeta^{(6)}(1\wedge 1\wedge\partial')$ for some $\zeta^{(6)}:S^{3k'+3,3l'}\wedge R\wedge R\wedge R\rightarrow M$. By the third condition on $\pi_{\ast\ast}M$, this $\zeta^{(6)}$ is zero. Finally, we have  
\[\delta=\zeta^{(6)}(\partial'\wedge\partial'\wedge\partial')=0.\]

The same line of proof can be used to derive the associativity of $\mu_N$ from that of $\mu_M$. It only needs to be checked that
\[\xymatrix{\ar @{} [dr] |{}
M\wedge M \ar[dd]_{\mu_M} \ar[dr]_{1\wedge\eta} \ar[rr]^{\eta\wedge\eta} && N\wedge N \ar[dd]^{\mu_N} \\
 & M\wedge N \ar[ur]^{\eta\wedge 1}\ar[dr]_{\nu_{M,N}} & \\
M \ar[rr]_{\eta} && N}\]
is commutative, but this was part of the definitions of $\nu_{M,N}$ and $\mu_N$. 
\end{proof}

\begin{prop}\label{associativity}
If $p>2$, then $\mu_{AP(n)}$ is homotopy associative.
\end{prop}

\begin{proof}
Note that this does not follow immediately from the above lemma, as $\pi_{2k+2,2l}N\neq 0$ for $N=MGL_{(p)}/(p,x)$. However, we can use a trick from \cite[Theorem V.3.1]{EKMM}, where the topological analogue of this statement is proven. 

Let $N=R/(x,y)$ and $A=R/(J-\{x,y\})$ for some set $J$ of elements in $\pi_{\ast\ast}R$, and assume that we already know that $A$ is equipped with a (homotopy) associative map $\mu_A:A\wedge A\rightarrow A$ as above. The product on $R/J\cong A/(x,y)\cong N\wedge_R A$ is given by 
\[(N\wedge_R A)\wedge (N\wedge_R A)\overset{\tau}\longrightarrow (N\wedge N)\wedge_R (A\wedge A)\overset{1\wedge_R \mu_A}\longrightarrow (N\wedge N)\wedge_R A\overset{\mu_N\wedge_R 1}\longrightarrow N\wedge_R A.\]
Therefore, to prove associativity for $\mu_{N\wedge_R A}$, it suffices to prove that the associativity diagram for $\mu_N$ commutes after applying $-\wedge_R 1_A$ to it. Applying $-\wedge_R 1_A$ to all diagrams appearing in the above lemma yields the following result: If $\pi_{i,j}(M\wedge_R A)=0$ for $(i,j)\in\{(k'+1,l'),(2k'+2,2l'),(3k'+3,3l')\}$ and $\pi_{i,j}(N\wedge_R A)=0$ for $(i,j)\in\{(k+1,l),(2k+2,2l),(3k+3,3l)\}$, then $\mu_{N\wedge_R A}$ is associative.

Now, let $R=MGL_{(p)}$, $A=ABP$, $M=MGL_{(p)}/p$ and $N=MGL_{(p)}/(p,x)$. From \cite[Definition 5.3]{Vez}, we know that $\mu_{ABP}$ is associative. The assumptions on the homotopy groups are satisfied by Lemma \ref{Ah}, by which $\pi_{\ast\ast}(ABP/p)\cong \h_{\ast\ast}\otimes \pi_\ast(BP/p)$ and $\pi_{\ast\ast}(ABP/(p,x))\cong \h_{\ast\ast}\otimes \pi_\ast(BP/(p,x))$. Note that these homotopy groups vanish in degrees $(k+1,l)$ and $(3k+3,3l)$ for any $p$ because $\pi_\ast BP$ is concentrated in even degrees and $k$ is even. However, for $\pi_{(2k+2,2l)}$ to vanish, we need to assume that $p$ is odd.

This proves that $\mu_{ABP/(p,x)}$ is associative. Inductively, we can apply this argument to $A=ABP/J'$ for some $(p)\subset J'\subset J$, where $J$ is as in the proof of Corollary \ref{trivial action}, using $ABP/(J'\cup\{y\})=MGL_{(p)}/(p,y)\wedge_{MGL_{(p)}}ABP/J'$ (compare Lemma \ref{p,x}).
\end{proof}

Recall that, for $k=\C$ and $n>0$, $AP(n)_{\ast\ast}\cong \h_{\ast\ast}\otimes_{\F_p}P(n)_\ast $ (Lemma \ref{Ah}), which is a ring. Hence, we can speak of $AP(n)_{\ast\ast}$-modules. Note that $AP(0)_{\ast\ast}$ is a ring, anyway, since $AP(0)=ABP$ is a ring spectrum.

\begin{lemma}\label{ring iso}
Let $k=\C$. On coefficients, the map
\[\mu_{AP(n)}:AP(n)\wedge AP(n)\rightarrow AP(n)\]
induces the multiplication on $AP(n)_{\ast\ast}$ given by $AP(n)_{\ast\ast}\cong \h_{\ast\ast}\otimes_{\F_p}P(n)_\ast$.
\end{lemma}

\begin{proof}
By \cite[Theorem 3.6.16]{Pelaez}, the motivic Atiyah Hirzebruch spectral sequence in Lemma \ref{Ah} is multiplicative, yielding an isomorphism of rings between $AP(n)_{\ast\ast}$ with multiplication induced from $\mu_{AP(n)}$ and $E_2=\h_{\ast\ast}\otimes_{\F_p}P(n)_\ast$ with multiplication induced from the ring structure of $\h_{\ast\ast}$ and from $R_\C(\mu_{AP(n)}):P(n)\wedge P(n)\rightarrow P(n)$. It therefore suffices to show that $R_\C(\mu_{AP(n)})$ induces the ring structure on $P(n)_\ast$. Now, $R_\C$ carries all the above diagrams to the analogous topological diagrams and, therefore, $BP\wedge P(n)\rightarrow P(n)$ (inducing the action of $BP_\ast$ on $P(n)_\ast$) factors through $R_\C(\mu_{AP(n)})$, which, hence, induces the induced action of $P(n)_\ast$ on $P(n)_\ast$. 
\end{proof}

\begin{cor}\label{AP(n) module}
In $ \SH(\C)$, the action $\nu_n:ABP\wedge AP(n)\rightarrow AP(n)$ factors through a map $\mu_{AP(n)}:AP(n)\wedge AP(n)\rightarrow AP(n)$. If $p>2$, the induced action of $ABP_{\ast\ast}$ on $AP(n)_{\ast\ast}(X)$ and $AP(n)^{\ast\ast}(X)$ gives $AP(n)_{\ast\ast}(X)$ and $AP(n)^{\ast\ast}(X)$ the structure of left $AP(n)_{\ast\ast}$-modules for any $n\geq 0$ and $X\in\SH(\C)$.
\end{cor}

\begin{proof}
We have to check that the action of $AP(n)_{\ast\ast}$ on $AP(n)_{\ast\ast}(X)$ and $AP(n)^{\ast\ast}(X)$ is unital and associative. This is equivalent to $\mu_{AP(n)}$ being homotopy left unital and associative. Unitality follows from the maps $\nu_M:R\wedge M\rightarrow M$ being unital, since, by definition of $\mu_M$, $\nu_M=\mu_M\circ (\eta'\wedge 1):R\wedge M\rightarrow M\wedge M\rightarrow M$. Associativity is proven in Proposition \ref{associativity} for $p>2$.
\end{proof}

\section{Bousfield classes of $AK(n)$ and $AB(n)$}\label{Section Wurgler}

Recall $AB(n)=v_n^{-1}AP(n)$.

We want to use methods of \cite{JW} and \cite{W} to prove that 
\[AK(n)_{\ast\ast}X=0 \Leftrightarrow AB(n)_{\ast\ast}X=0\]
for $X\in\SH(\C)$. For this, we need the following two results, which hold in $\SH(k)$ for any $k\subseteq\C$ and are analogous to \cite[Formula (2.8)]{W}. The reader may skip the first proposition, as it is a special case of the second one. The proof of the first one is maybe more illustrative.

Let $\mu_n:ABP^{\ast\ast}(-)\rightarrow AP(n)^{\ast\ast}(-)$ be induced from 
\[ABP\overset{1\wedge i}\rightarrow ABP\wedge AP(n)\overset{\nu_n}\rightarrow AP(n),\]
where $i:S\rightarrow AP(n)$ is induced from the unit map $S\rightarrow ABP$ and $\nu_n$ is the structure map of the $ABP$-module $AP(n)$. For $s^E\in ABP^{\ast\ast}ABP$ as in Lemma \ref{ABPABP}(2), we have $\mu_n(s^E)\in AP(n)^{\ast\ast}(ABP)$.

\begin{prop}\label{W}
For any $n\geq 0$, the map 
\[h_n:AP(n)^{\ast\ast}[[\mu_n(s^E)]]\rightarrow AP(n)^{\ast\ast}ABP,\]
given by $h_n(\underset{E}\sum x_E\cdot\mu_n(s^E))=\underset{E}\sum \nu_n(s^E\wedge x_E)$, is an isomorphism of $ABP^{\ast\ast}$-modules. If $k=\C$ and $p>2$, it is an isomorphism of $AP(n)^{\ast\ast}$-modules by Corollary \ref{AP(n) module}.
\end{prop}

\begin{proof}
We proceed by induction. For $AP(0)=ABP$, $\mu_0$ is the identity and the claim holds by Lemma \ref{ABPABP}(2). Now, assume $h_n$ is an isomorphism for some $n\geq 0$. Consider the following diagram, consisting of two exact sequences induced by $AP(n)\overset{v_n}\rightarrow AP(n)\rightarrow AP(n+1)$.
\begin{tiny}
\[\xymatrix{\ar @{} [dr] |{}
 AP(n)^{\ast\ast}ABP  \ar[r]^{v_n^{\ast}} & AP(n)^{\ast\ast}ABP \ar[r]^{\lambda_n^{\ast}} & AP(n+1)^{\ast\ast}ABP \ar[r]^{\delta^{\ast}} & AP(n)^{\ast\ast}ABP  \ar[r]^{v_n^{\ast}} & AP(n)^{\ast\ast}ABP  \\
 AP(n)^{\ast\ast}[[\mu_n s^E]] \ar[u]_{h_n}\ar[r] &  AP(n)^{\ast\ast}[[\mu_n s^E]] \ar[u]_{h_n}\ar[r] &  AP(n+1)^{\ast\ast}[[\mu_{n+1} s^E]] \ar[r]\ar[u]^{h_{n+1}} & AP(n)^{\ast\ast}[[\mu_n s^E]] \ar[u]_{h_n}\ar[r] &  AP(n)^{\ast\ast}[[\mu_n s^E]] \ar[u]_{h_n} }\]
\end{tiny}
The lower sequence is the exact sequence 
\[\cdots\rightarrow AP(n)^{\ast\ast}\rightarrow AP(n)^{\ast\ast}\rightarrow AP(n+1)^{\ast\ast}\rightarrow AP(n)^{\ast\ast}\rightarrow AP(n)^{\ast\ast}\cdots,\]
tensored over $\F_p$ with $\F_p[[s^E]]$. We show that the diagram commutes.

In the first square, the upper composition takes $x\cdot \mu_n(s^E)\in AP(n)^{\ast\ast}[[\mu_n s^E]]$ to 
\[v_n^\ast(h_n(x\mu_n s^E)): ABP\overset{s^E\wedge x}\longrightarrow ABP\wedge AP(n)\overset{\nu_n}\longrightarrow AP(n)\overset{v_n}\longrightarrow AP(n),\]
and the lower composition takes the same element to 
\[h_n((v_n\cdot x)\cdot \mu_n(s^E)):ABP\overset{s^E\wedge x}\longrightarrow ABP\wedge AP(n)\overset{1\wedge v_n}\longrightarrow ABP\wedge AP(n)\overset{\nu_n}\longrightarrow AP(n).\]
Therefore, the commutativity of the first square is equivalent to the commutativity of the square
\[\xymatrix{\ar @{} [dr] |{}
ABP\wedge AP(n) \ar[r]^{1\wedge v_n} \ar[d]_{\nu_n} & ABP\wedge AP(n) \ar[d]^{\nu_n} \\
AP(n) \ar[r]_{v_n} & AP(n).}\]
This square commutes because $v_n$ is a map of $ABP$-modules (compare Lemma \ref{module map}).

In the second square, the upper composition takes $x\cdot\mu_n(s^E)\in AP(n)^{\ast\ast}[[\mu_n s^E]]$ to
\[(\lambda^\ast\circ h_n)(x\mu_n s^E):ABP\overset{s^E\wedge x}\longrightarrow ABP\wedge AP(n)\overset{\nu_n}\longrightarrow AP(n)\overset{\lambda_n}\longrightarrow AP(n+1)\]
and the lower composition takes $x\cdot \mu_n(s^E)$ to 
\[h_{n+1}(\lambda_n x\cdot\mu_n s^E):ABP\overset{s^E\wedge x}\longrightarrow ABP\wedge AP(n)\overset{1\wedge \lambda_n}\longrightarrow ABP\wedge AP(n+1)\overset{\nu_{n+1}}\longrightarrow AP(n+1).\]
Thus, the commutativity of the second square is equivalent to the commutativity of
\[\xymatrix{\ar @{} [dr] |{}
ABP\wedge AP(n) \ar[r]^{1\wedge \lambda_n} \ar[d]_{\nu_n} & ABP\wedge AP(n+1) \ar[d]^{\nu_{n+1}} \\
AP(n) \ar[r]_{\lambda_n} & AP(n+1),}\]
which holds because $\lambda_n$ is, by definition, a map in the category of $ABP$-modules (see Definition \ref{AK(n)}).

In the third square, the upper composition takes $x\cdot \mu_{n+1}(s^E)\in AP(n+1)^{\ast\ast}[[\mu_{n+1}s^E]]$ to $(\delta\circ h_{n+1})(x\mu_{n+1}s^E)=\delta(\nu_{n+1}(s^E\wedge x))$ and the lower composition takes $x\cdot \mu_{n+1}(s^E)$ to $\nu_n(s^E\wedge \delta(x))$. Thus, the commutativity of the third square follows from the commutativity of 
\[\xymatrix{\ar @{} [dr] |{}
ABP\wedge AP(n+1) \ar[r]^{1\wedge \delta} \ar[d]_{\nu_{n+1}} & ABP\wedge AP(n) \ar[d]^{\nu_{n}} \\
AP(n+1) \ar[r]_{\delta} & AP(n).}\]
Finally, the five lemma implies that $h_{n+1}$ is an isomorphism.
\end{proof}

The above proposition holds more generally:

\begin{prop}\label{W2}
Let $h\in\SH(k)^{cell}$ be any cellular $ABP$-module spectrum. Then 
\[h^{\ast\ast}ABP\cong h^{\ast\ast}[[s^E]]\]
as $ABP^{\ast\ast}$-modules. In particular, this also holds for $h=Ak(n)$.
\end{prop}

\begin{proof}
We apply the universal coefficient spectral sequence from \cite[Proposition 7.7]{DI} (see also Proposition \ref{ucss}) to $E=ABP$ (which is a ring spectrum by \cite[Definition 5.3]{Vez} and is cellular by Remark \ref{AMU module}(1)), $M=ABP\wedge ABP$ and $N=h$.
\[\Ext_{ABP_{\ast\ast}}^{\ast\ast\ast}(ABP_{\ast\ast}ABP, h_{\ast\ast})\Rightarrow \pi_{\ast\ast}F_{ABP}(ABP\wedge ABP,h),\]
converging conditionally to $\pi_{\ast\ast}F_{ABP}(ABP\wedge ABP,h)\cong h^{\ast\ast}(ABP)$. As $ABP_{\ast\ast}ABP\cong ABP_{\ast\ast}\{t^E\}$ is free over $ABP_{\ast\ast}$ (Lemma \ref{ABPABP}(1)), the higher Ext-groups vanish and the sequence collapses to 
\[h^{\ast\ast}ABP\cong \Hom_{ABP_{\ast\ast}}(ABP_{\ast\ast}\{t^E\}, h_{\ast\ast}),\]
which is isomorphic to $h^{\ast\ast}[[s^E]]$, as in Lemma \ref{ABPABP}(2). 
\end{proof}

W\"urgler constructs operations 
\[s_Q^E:MUQ^i(-)\rightarrow MUQ^{i+|E|}(-)\] 
for regular sequences $Q$ \cite[Theorem 5.1]{W}. For $MUQ=P(n)$, these operations specify a choice of the operations $(r_E)_n:P(n)^\ast(-)\rightarrow P(n)^{\ast+|E|}(-)$ considered in \cite[Section 4]{JW}. 
They are needed in the proof of the isomorphism $B(n)_\ast(X)\cong K(n)_\ast(X)\otimes \F_p[v_{n+1},v_{n+2},\cdots]$ in \cite[Proposition 4.14]{JW}. We will now state motivic analogues of some of W\"urgler's lemmas for $MUQ=P(n)$.\\

In $\SH(k)^{cell}$, let $AP(n)[{\bf t}]$ represent the cohomology theory 
\[AP(n)^{\ast\ast}(-)[{\bf t}]=ABP^{\ast\ast}[t_1,t_2,\cdots]\otimes_{ABP^{\ast\ast}}AP(n)^{\ast\ast}(-)\] 
with $\deg(t_i)=-(2(p^i-1),p^i-1)$. 

\begin{rk}
W\"urgler defines $h^\ast(-)[{\bf t}]=h^\ast(-)\otimes_{h^\ast}h^\ast[t_1,t_2,\cdots]$, where the $t_i$ take any even degree, i.e., $\deg(t_i)=-2i$ in \cite[Section 5]{W}. He also considers operations $s^E$ of any degree $\sum_i 2ie_i$. However, Johnson and Wilson \cite[Sections 1 and 4]{JW} are only interested in operations of degree $|E|=\sum_i 2(p^i -1)e_i$. We will see that it suffices to restrict to those degrees. Most of W\"urgler's results we refer to in the following are formulated in a much greater generality and do not depend on any degrees of particular elements. The only point where the degrees of the $t_i$ are important is in \cite[Theorem 5.1]{W} and we will comment on that below Theorem \ref{5.1}.
\end{rk}

\begin{lemma}
For $X\in\SH(k)^{fin}$,
\[(ABP\wedge AP(n))^{\ast\ast}(X)\cong AP(n)^{\ast\ast}(X)[{\bf t}].\]
As any cohomology theory on $\SH(k)^{fin}$ extends uniquely to $\SH(k)^{cell}$ \cite[Lemma 4.10]{MLE}, it follows that we can take 
\[AP(n)[{\bf t}]=ABP\wedge AP(n).\]
\end{lemma}

\begin{proof}
We consider the map 
\[(\BP\wedge \BP)^{\ast\ast}\otimes_{\BP^{\ast\ast}}AP(n)^{\ast\ast}(X)\rightarrow (ABP\wedge AP(n))^{\ast\ast}(X),\]
induced by the $\BP$-module structure on $AP(n)$. First, we see that 
\[(\BP\wedge \BP)^{\ast\ast}\cong (\BP\wedge \BP)_{-\ast,-\ast}\cong \BP_{-\ast,-\ast}[{\bf t}]\cong \BP^{\ast\ast}[{\bf t}]\] 
by Lemma \ref{ABPABP}(1). Note that the $t_i$ appearing here are dual to the $t_i$ in Lemma \ref{ABPABP}(1), hence $\deg(t_i)=-(2(p^i-1),p^i-1)$. To prove the claim, it suffices to show that the above map is an isomorphism.

For $n=0$ and $X=S^0$, this is clear. Induction on $n$ shows 
\[(\BP\wedge \BP)^{\ast\ast}\otimes_{\BP^{\ast\ast}}AP(n)^{\ast\ast}\cong (\BP\wedge AP(n))^{\ast\ast}\] 
for any $n$, since $AP(n)\rightarrow AP(n)\rightarrow AP(n+1)$ induces long exact sequences on both sides and the diagram commutes because $AP(n)\rightarrow AP(n)\rightarrow AP(n+1)$ are maps of $ABP$-modules. As also any cofiber sequence $X\rightarrow Y\rightarrow Z$ induces long exact sequences on both sides, cellular induction proves the claim for any finite spectrum $X$.
\end{proof}

We state an analogue of \cite[Lemma 3.14]{W}.

\begin{lemma}\label{3.14.2}
Let $h\in\SH(k)^{cell}$, $k\subseteq\C$, be an $ABP$-module. The multiplication $ABP\wedge h\rightarrow h$ induces an isomorphism of $ABP^{\ast\ast}$-modules,
\[ ABP^{\ast\ast}ABP\otimes_{ABP^{\ast\ast}}h^{\ast\ast}AP(n)\cong h^{\ast\ast}(ABP\wedge AP(n)).\]
\end{lemma}

\begin{proof}
We apply the spectral sequence from \cite[Proposition 7.7]{DI} to $E=ABP$, $M=ABP\wedge ABP\wedge ABP$ and $N=h$:
\[\Ext_{ABP_{\ast\ast}}^{\ast\ast\ast}(ABP_{\ast\ast}(ABP\wedge ABP),h_{\ast\ast})\Rightarrow \pi_{\ast\ast}F_{ABP}(ABP\wedge ABP\wedge ABP,h),\]
converging conditionally to $ h^{\ast\ast}(ABP\wedge ABP)$.
By \cite[Lemma 5.1(i)]{MLE}, 
\[ABP_{\ast\ast}(ABP\wedge ABP)\cong ABP_{\ast\ast}ABP\otimes_{ABP_{\ast\ast}}ABP_{\ast\ast}ABP,\]
which is free over $ABP_{\ast\ast}$. Thus, the spectral sequence collapses and
\[h^{\ast\ast}(ABP\wedge ABP)\cong \Hom_{ABP_{\ast\ast}}(ABP_{\ast\ast}ABP\otimes_{ABP_{\ast\ast}}ABP_{\ast\ast}ABP, h_{\ast\ast}).\]
Now, $ABP_{\ast\ast}ABP\otimes_{ABP_{\ast\ast}}ABP_{\ast\ast}ABP\cong ABP_{\ast\ast}\{t_1^E t_2^F\}$, and, therefore, 
\[\Hom_{ABP_{\ast\ast}}(ABP_{\ast\ast}ABP\otimes_{ABP_{\ast\ast}}ABP_{\ast\ast}ABP, h_{\ast\ast})\cong h^{\ast\ast}[[s_1^E s_2^F]]\]
\[\cong ABP^{\ast\ast}[[s_1^E]]\otimes_{ABP_{\ast\ast}}h^{\ast\ast}[[s_2^F]]\cong ABP^{\ast\ast}ABP\otimes_{ABP_{\ast\ast}}h^{\ast\ast}ABP,\]
using Proposition \ref{W2}.
This proves the claim for $n=0$.

Assume we have shown that $ABP\wedge h\rightarrow h$ induces an isomorphism
\[ ABP^{\ast\ast}ABP\otimes_{ABP^{\ast\ast}}h^{\ast\ast}AP(n)\cong h^{\ast\ast}(ABP\wedge AP(n))\]
for some $n\geq 0$. The cofiber sequence $AP(n)\rightarrow AP(n)\rightarrow AP(n+1)$ induces long exact sequences on both sides of this isomorphism, which form a commutative diagram because the maps $AP(n)\rightarrow AP(n)\rightarrow AP(n+1)$ are maps of $ABP$-modules. Thus, the five lemma implies the claim for $n+1$.
\end{proof}

Combining the above isomorphism with the one from Proposition \ref{W2}, we get:

\begin{lemma}\label{3.14}
Let $k=\C$ and $p>2$.
Let $X=AP(n)$ or $X=AP(n)[{\bf t}]$.
The action of $ABP$ on $AP(m)$, $\BP\wedge AP(m)\rightarrow AP(m)$, induces an isomorphism of $AP(m)^{\ast\ast}$-modules
\[AP(m)^{\ast\ast}(\BP)\otimes_{AP(m)^{\ast\ast}}AP(m)^{\ast\ast}(X)\overset\cong\rightarrow AP(m)^{\ast\ast}(\BP\wedge X).\]
\end{lemma}

\begin{proof}
From Proposition \ref{W}, we know that 
\[AP(m)^{\ast\ast}(ABP)\cong AP(m)^{\ast\ast}\otimes_{ABP^{\ast\ast}}ABP^{\ast\ast}ABP\] as $AP(m)^{\ast\ast}$-modules. With $h=AP(m)$, the above lemma immediately implies 
\[AP(m)^{\ast\ast}(\BP)\otimes_{AP(m)^{\ast\ast}}AP(m)^{\ast\ast}(AP(n))\overset\cong\rightarrow AP(m)^{\ast\ast}(\BP\wedge AP(n))\]
as $ABP^{\ast\ast}$-modules and, by Corollary \ref{AP(n) module}, also as $AP(m)^{\ast\ast}$-modules.

For $X=AP(n)[{\bf t}]=ABP\wedge AP(n)$, we set $M=ABP^{\wedge 4}$ in the proof of the previous lemma. Using $\pi_{\ast\ast}ABP^{\wedge 4}\cong ABP_{\ast\ast}ABP\otimes_{ABP_{\ast\ast}}\pi_{\ast\ast}ABP^{\wedge 3}$ \cite[Proposition 5.1(i)]{MLE}, the proof proceeds with exactly the same arguments as the proof of Lemma \ref{3.14.2}.
\end{proof}

As a consequence of Corollary \ref{trivial action}, we show the following (see \cite[Lemma 2.8(b)]{JW}):

\begin{cor}\label{ses}
Let $k=\C$ and $p>2$. The cofibration $S^{2p^n-2,p^n-1}\wedge AP(n)\overset{v_n}\rightarrow AP(n)\rightarrow AP(n+1)$ induces short exact sequences
\[0\rightarrow Ak(j)^{\ast\ast}(S^{2p^n-1,p^n-1}\wedge AP(n))\rightarrow Ak(j)^{\ast\ast}(AP(n+1))\rightarrow Ak(j)^{\ast\ast}(AP(n))\rightarrow 0\]
for every $j>n$.
\end{cor}

\begin{proof}
We have to show that, in the $Ak(j)^{\ast\ast}$-long exact sequence, the map induced by $v_n$ is zero. This map is defined as the composition of the two left arrows in the following commutative diagram:
\[\xymatrix{\ar @{} [dr] |{}
Ak(j)^{\ast\ast}(S^{2p^n-2,p^n-1}\wedge AP(n)) & \BP^{\ast\ast}(S^{2p^n-2,p^n-1})\otimes_{\BP^{\ast\ast}}Ak(j)^{\ast\ast}(AP(n)) \ar[l]_{\cong\hspace{30pt}} \\
Ak(j)^{\ast\ast}(\BP\wedge AP(n))  \ar[u]^{Ak(j)^{\ast\ast}(v_n\wedge 1)}& \BP^{\ast\ast}(\BP)\otimes_{\BP^{\ast\ast}}Ak(j)^{\ast\ast}(AP(n)) \ar[u]^{\BP^{\ast\ast}(v_n)\otimes 1}\ar[l]_{\cong\hspace{20pt}}\\
Ak(j)^{\ast\ast}(AP(n))\ar[u]^{Ak(j)^{\ast\ast}(\nu_n)}}\]
The horizontal maps are induced by the $ABP$-module structure on $Ak(j)$. The lower horizontal map is an isomorphism by the previous lemma. The upper isomorphism is proven similarly, setting $M=ABP\wedge ABP$ in the above proof.
We show that the right hand map is zero. Let $\sum s^E\otimes x_E$ be an element of $ABP^{\ast\ast}(ABP)\otimes_{ABP^{\ast\ast}}Ak(j)^{\ast\ast}(AP(n))$. It maps to $\sum s^E(v_n)\otimes x_E$. In Corollary \ref{modulo}, we set $k=1$, $s=0$ and $m=n$, to see that for all $|E|\geq 0$, either $s^E(v_n) \equiv v_n$ mod $I_n$ or $s^E(v_n) \equiv 0$ mod $I_n$. Thus, $s^E(v_n)\in I_{n+1}=(v_0,\cdots,v_n)$. By Corollary \ref{trivial action}, $I_{n+1}$ acts trivially on $Ak(j)^{\ast\ast}(AP(n))$ for $j>n$, hence, the right hand map is zero.

It follows that the left map factors through zero, proving the claim.
\end{proof}

The isomorphism from Lemma \ref{3.14} is now used to define an $AP(m)^{\ast\ast}(ABP)$-comodule structure on $AP(m)^{\ast\ast}(AP(n))$ by
\[AP(m)^{\ast\ast}AP(n)\overset{\nu_n^\ast}\rightarrow AP(m)^{\ast\ast}(ABP\wedge AP(n))\overset{\cong}\leftarrow AP(m)^{\ast\ast}\BP\otimes_{AP(m)^{\ast\ast}}AP(m)^{\ast\ast}AP(n),\]
where $\nu_n$ is the $ABP$-module structure map on $AP(n)$. Note that, for $k\neq\C$ or $p=2$, we do not know if these groups are $AP(m)^{\ast\ast}$-modules (see Corollary \ref{AP(n) module}). In this case, we might only get an $ABP^{\ast\ast}ABP$-comodule structure on $AP(m)^{\ast\ast}AP(n)$.

\begin{lemma}\label{mu se vn}
Let $k=\C$, $m> n\geq 0$ and $E$ be an exponent sequence. Let $s^E:ABP\rightarrow ABP$ be as in Lemma \ref{ABPABP} and $\mu_m:ABP\rightarrow AP(m)$ be as defined in the beginning of this section. Then 
\[\mu_m\circ s^E\circ v_n: S\rightarrow ABP\rightarrow ABP\rightarrow AP(m)\]
is the zero map.
\end{lemma}

\begin{proof}
The realisation functor $R_\C$ takes the composition $\mu_m\circ s^E\circ v_n \in AP(m)_{\ast\ast}$ to $\mu^{\Top}_m\circ s_{\Top}^E\circ v^{\Top}_n \in P(m)_\ast$. By the $BP$-version of \cite[Lemma 2.2]{W}, $\mu_m^{\Top}: BP\rightarrow P(m)$ is the canonical projection. By the invariant prime ideal theorem, $I_{n+1}^{\Top}=(v^{\Top}_0,v^{\Top}_1,\cdots,v^{\Top}_n)\subset BP_\ast$
is invariant under the action of $s^E_{\Top}\in BP^\ast BP$, hence, $s_{\Top}^E(v^{\Top}_n)\in I_{n+1}^{\Top}$. Thus, $\mu^{\Top}_m(s_{\Top}^E\circ v^{\Top}_n) \in P(m)_\ast$ lies in the image of $I_{n+1}$ under the projection $BP_\ast\rightarrow P(m)_\ast=BP_\ast/(v^{\Top}_0,\cdots,v_{m-1}^{\Top})$, which is zero, since $m>n$. This implies $R_\C(\mu_m\circ s^E\circ v_n)=0\in P(m)_\ast$.

For $k=\C$, $AP(m)_{\ast\ast}\cong P(m)_\ast[\tau]$ by Lemma \ref{Ah}, and any homogeneous element in $AP(m)_{\ast\ast}$ that realises to zero in $P(m)_\ast$ already has to be zero in $AP(m)_{\ast\ast}$. Hence, $\mu_m\circ s^E\circ v_n=0$ in $AP(m)_{\ast\ast}$.
\end{proof}

Now we present an analogue of a special case of \cite[Proposition 4.12]{W}.

\begin{lemma}\label{4.12}
Let $k=\C$, $p>2$ and $m\geq n$. As $AP(m)^{\ast\ast}\BP$-comodules,
\[AP(m)^{\ast\ast}AP(n)\cong AP(m)^{\ast\ast}\BP\otimes_{AP(n)^{\ast\ast}}\Lambda_{AP(n)^{\ast\ast}}[[\beta_0,\cdots,\beta_{n-1}]]\]
with $\deg(\beta_i)=(2p^i-1, p^i-1)$.
\end{lemma}

\begin{proof}
For $n=0$, the statement is trivial.
Assume the proposition holds for some pair $(m,n)$ with $m>n$. We show that it also holds for $(m,n+1)$.

Consider $AP(n)\overset{v_n}\rightarrow AP(n)$. We show that the induced map
\[v_n^\ast:AP(m)^{\ast\ast}AP(n)\rightarrow AP(m)^{\ast\ast}AP(n)\] 
is trivial. This works similar to \cite[Lemma 3.15]{W}. Let $\phi:S\rightarrow ABP$ represent $v_n\in ABP_{\ast\ast}$ and let $\phi^\ast: AP(m)^{\ast\ast}ABP\rightarrow AP(m)^{\ast\ast}$ be the map induced on $AP(m)^{\ast\ast}(-)$. For $x\in AP(m)^{\ast\ast}\BP=AP(m)^{\ast\ast}[[\mu_m s^E]]$ (see Proposition \ref{W}), we have $x=\sum_E \lambda_E \mu_m s^E$ and
\[\phi^\ast(x)=\phi^\ast(\sum_E \lambda_E \mu_m s^E)=\sum_E \lambda_E\mu_m (s^E(v_n)),\]
because $\phi^\ast$ is precomposition with $v_n\in ABP_{\ast\ast}$. By Lemma \ref{mu se vn}, $\mu_m (s^E(v_n))=0$ in $ AP(m)^{\ast\ast}$, hence, $\phi^\ast=0$.
Now consider the following commutative square:
\[\xymatrix{\ar @{} [dr] |{}
AP(m)^{\ast\ast}(\BP\wedge AP(n))  \ar[r]^{(\phi\wedge 1)^\ast} & AP(m)^{\ast\ast}(AP(n))  \\
AP(m)^{\ast\ast}\BP\otimes_{AP(m)^{\ast\ast}}AP(m)^{\ast\ast}AP(n) \ar[u]^{\cong}\ar[r]_{\phi^\ast\otimes 1} & AP(m)^{\ast\ast}\otimes_{AP(m)^{\ast\ast}}AP(m)^{\ast\ast}AP(n)\ar[u]_{\cong}
}\]
The left map is an isomorphism by Lemma \ref{3.14}. Since $\phi^{\ast}=0$, it follows that $(\phi\wedge 1)^\ast=0$. Since $v_n^\ast$ is, by definition, the precomposition of $(\phi\wedge 1)^\ast$ with the map $AP(m)^{\ast\ast}AP(n)\rightarrow AP(m)^{\ast\ast}(\BP\wedge AP(n))$, $v_n^\ast=0$, as claimed.

It follows that the long exact $AP(m)^{\ast\ast}(-)$-sequence induced by 
\[S^{2p^n-2,p^n-1}\wedge AP(n)\overset{v_n}\rightarrow AP(n)\rightarrow AP(n+1)\] 
splits into short exact sequences of $AP(m)^{\ast\ast}ABP$-comodules
\[0\rightarrow AP(m)^{\ast\ast}(S^{2p^n-1,p^n-1}\wedge AP(n))\rightarrow AP(m)^{\ast\ast}AP(n+1)\rightarrow AP(m)^{\ast\ast}AP(n)\rightarrow 0.\]
Analogously to \cite[Proposition 4.12]{W}, it follows inductively that 
\[AP(m)^{\ast\ast}AP(n+1)\cong AP(m)^{\ast\ast}AP(n)\otimes_{AP(m)^{\ast\ast}} \Lambda_{AP(m)^{\ast\ast}}(\beta_n)\] 
and the degree of $\beta_n$ is determined by the degrees appearing in the exact sequence.
\end{proof}

If $M$ is an $AP(n)^{\ast\ast}(\BP)[{\bf t}]$-comodule with structure map $\psi$, an element $a\in M$ is called primitive if $\psi(a)=1\otimes a$. Similarly to \cite[Lemma 4.13]{W}, the following holds:

\begin{lemma}\label{4.13}
Let $k=\C$ and $p>2$. Let $g:AP(n)\rightarrow AP(n)[{\bf t}]$ be a map of spectra. Then $g$ is a map of $\BP$-module spectra if and only if it is a primitive element of the $AP(n)^{\ast\ast}(\BP)[{\bf t}]$-comodule $AP(n)^{\ast\ast}(AP(n))[{\bf t}]$.
\end{lemma}

\begin{proof}
This follows from Lemma \ref{3.14} in the same manner as \cite[Lemma 4.13]{W} follows from \cite[Lemma 3.14]{W} for $X=P(n)$ and $X=P(n)[{\bf t}]$.
\end{proof}

Now we state a result which is directly used for the construction of the operation we are aiming at. This corresponds to \cite[Theorem 4.17]{W}. 

\begin{prop}\label{4.17}
Let $k=\C$ and $p>2$. There is a degree-preserving group isomorphism 
\[\Hom^{\ast\ast}_{\BP}\left(AP(n),AP(n)[{\bf t}]\right)\cong \Lambda_{AP(n)^{\ast\ast}[{\bf t}]}[[\beta_0,\beta_1,\cdots]],\]
where the left hand side is the bigraded abelian group of maps of $\BP$-module spectra and the right hand side is an exterior algebra over $AP(n)^{\ast\ast}[{\bf t}]$.
\end{prop}

\begin{proof}
W\"urgler derives this from \cite[Proposition 4.12]{W} and \cite[Lemma 4.13]{W} using that inverse limits of primitive elements are primitive \cite[Lemma 4.16]{W}. The same line of proof proves this proposition using Proposition \ref{4.12} and Lemma \ref{4.13}. Denoting the set of primitive elements in $M$ by $\Pr\{M\}$, the proof can be summarised by the following sequence of isomorphisms:
\[\Hom_{\BP}^{\ast\ast}(AP(n),AP(n)[{\bf t}])\cong \Pr\left\{AP(n)^{\ast\ast}(AP(n))[{\bf t}]\right\}\]
\[\cong \Pr\left\{AP(n)^{\ast\ast}(\BP)[{\bf t}]\otimes_{AP(n)^{\ast\ast}[{\bf t}]}\Lambda_{AP(n)^{\ast\ast}[{\bf t}]}[[\beta_0,\beta_1,\cdots]]\right\}\]
\[\cong \Pr\left\{AP(n)^{\ast\ast}(\BP)[{\bf t}]\right\}\otimes_{AP(n)^{\ast\ast}[{\bf t}]}\Lambda_{AP(n)^{\ast\ast}[{\bf t}]}[[\beta_0,\beta_1,\cdots]]\]
\[\cong AP(n)^{\ast\ast}[{\bf t}]\otimes_{AP(n)^{\ast\ast}[{\bf t}]}\Lambda_{AP(n)^{\ast\ast}[{\bf t}]}[[\beta_0,\beta_1,\cdots]]\cong \Lambda_{AP(n)^{\ast\ast}[{\bf t}]}[[\beta_0,\beta_1,\cdots]].\]
\end{proof}

This completes the preparation W\"urgler needs for \cite[Theorem 5.1]{W}. We state our version of this theorem, now using the notation from \cite{JW}.

\begin{theorem}\label{5.1}
Let $p>2$. In $\SH(\C)^{cell}$, there exists a family ${(r_E)_n}$, $n\geq 0$, of natural stable operations
\[(r_E)_n:AP(n)^{\ast\ast}(-)\rightarrow AP(n)^{\ast\ast}(-)\]
of degree $(|E|,|E|/2)$, such that 
\[(r_E)_n(ux)=\sum_{F+G=E}s_F(u)(r_G)_n(x)\] 
for $u\in \BP^{\ast\ast}(X)$ and $x\in AP(n)^{\ast\ast}(X)$ (see Definition \ref{s_E} for the definition of $s_F$) and such that 
\[\xymatrix{\ar @{} [dr] |{}
\BP^{\ast\ast}(-)\ar[r]^{s_E}\ar[d] & \BP^{\ast\ast}(-) \ar[d]\\
AP(n)^{\ast\ast}(-)\ar[r]_{(r_E)_n} & AP(n)^{\ast\ast}(-)}\]
commutes.
\end{theorem}

\begin{proof}
Let $s_t:\BP^{\ast\ast}(-)\rightarrow \BP^{\ast\ast}(-)[{\bf t}]$ be given by $s_t(x)=\sum_E{s_E(x)\otimes t^E}$. 
As in \cite[Theorem 5.1]{W}, the square
\[\xymatrix{\ar @{} [dr] |{}
\BP^{\ast\ast}(-)\ar[r]^{s_t}\ar[d] & \BP^{\ast\ast}(-)[{\bf t}] \ar[d]\\
AP(n)^{\ast\ast}(-)\ar@{-->}[r]^{s_{t,n}} & AP(n)^{\ast\ast}(-)[{\bf t}]}\]
can be completed with the help of Proposition \ref{4.17} and we define $(r_E)_n$ by $s_{t,n}(x)=\sum_E{(r_E)_n(x)\otimes t^E}$.

Since Proposition \ref{4.17} gives us a map of $\BP$-module spectra, the operation $s_{t,n}$ satisfies $s_{t,n}(ux)=s_t(u)s_{t,n}(x)$. It follows that \[\sum_E (r_E)_n(ux)\otimes t^E=\underset{F+G=E}\sum s_F(u)(r_G)_n(x)\otimes t^F t^G\] 
and, hence, $(r_E)_n(ux)=\sum_{F+G=E}s_F(u)(r_G)_n(x)$, which proves the first property claimed. The second property holds because the above commutative square has to commute on the level of each $t^E$.
\end{proof}

\begin{rk}
Originally (see e.g. \cite[Formula (2.4)]{QuElem}), one first defines the operation $s_t:BP^{\ast}(-)\rightarrow BP^\ast(-)[t_1,t_2,\cdots]$, $\deg(t_i)=-2i$, and then uses it to construct operations $s_E$ of degree $\sum_i 2ie_i$. The $s_t$ used in the above proof contains only those summands $s_E \otimes t^E$ with $|E|=\sum_i 2(p^i-1)$, because we work with $ABP$-modules instead of $MGL$-modules. Hence, the $s_{t,n}$ also consists of less summands than the corresponding operation in \cite[Theorem 5.1]{W}, but this suffices to define exactly those $(r_E)_n$ with the degrees needed in \cite[Section 4]{JW}.
\end{rk}

In \cite[Proposition 4.14]{JW}, Johnson and Wilson show that, for any $X\in\SH^{fin}$, there is a natural isomorphism 
\[B(n)_\ast(X)\cong K(n)_\ast(X)\otimes \F_p[v_{n+1},v_{n+2},\cdots]\]
for $n<2p-2$. Johnson and Wilson needed the condition on $n$ because they were only able to canonically define the operations $(r_E)_n$ in this range. As stated in \cite[Remark 6.19]{W}, the condition becomes redundant if W\"urgler's operations are used.

With the above operations at hand, we can proceed exactly as in \cite[Section 4]{JW}:

\begin{defn}\label{a,b}
Let $n$ be fixed and let $\mathcal E$ be the set of all exponent sequences of the form $E=(0,\cdots,0,e_{n+1},e_{n+2},\cdots)$. Recall the definition of $|E|$ from Definition \ref{exp seq}. For $E\in\mathcal E$, let $q=|E|=2(p^n-1)b+a$ with $0\leq a<2(p^n-1)$. With $\sigma^n E=(p^n e_{n+1},p^n e_{n+2},\cdots)$ and $c=b-(e_{n+1}+e_{n+2}+\cdots)$, it follows $|\sigma^n E|=c2(p^n-1)+a$, exactly as in \cite{JW}. Note that $q$ and, hence, $a$ are even. We define $\bar s_E\in Ak(n)^{a,a/2}(AP(n))$ by
\[\bar s_E:AP(n)\overset{(r_{\sigma^n E})_n}\longrightarrow \Sigma^{c2(p^n-1)+a, c(p^n-1)+a/2}AP(n)\overset{v_n^c}\rightarrow \Sigma^{a,a/2}AP(n)\overset{\lambda_n}\rightarrow \Sigma^{a,a/2}Ak(n),\]  
where $\lambda_n$ is the quotient map from $AP(n)$ to $Ak(n)=AP(n)/(v_{n+1},\cdots)$.

Furthermore, we assume that the set $\{E\in\mathcal E\;|\; |E|=q\}$ is ordered as $\{E_1,\cdots,E_v\}$, where $v$ is the $\F_p$-dimension of $(\F_p[v_{n+1},v_{n+2},\cdots])_q$, and we denote $\bar s_{E_u}$ by $s_u$.
\end{defn}

\begin{lemma}\label{tau}
Let $k=\C$, $p>2$ and $\bar s_E:AP(n)_{\ast\ast}\rightarrow Ak(n)_{\ast\ast}$ be induced by the above map. Assume $n>0$ and let $\tau\in AP(n)_{\ast\ast}$ be as in Lemma \ref{Ah}. Then $\bar s_E(\tau)= 0$ for $E\neq 0$ and $\bar s_0(\tau)=\tau$. 
If, furthermore, $x\in AP(n)_{\ast\ast}X$ with $X\in\SH(\C)^{cell}$, then 
$\bar s_E(x\tau )=\bar s_E(x)\tau$.
\end{lemma}

\begin{proof}
Since $\bar s_E(\tau)=(\lambda_n\circ v_n^c\circ (r_{\sigma^n E})_n) (\tau)$ with $\lambda_n(\tau)=\tau$, and $c=0$ for $E=0$, the first claim is equivalent to $(r_F)_n(\tau)= 0$ for $F\neq 0$ and $(r_0)_n(\tau)=\tau$. 
The realisation functor $R_\C$ maps $\tau$ to $1\in P(n)^0$ and it takes $(r_F)_n:AP(n)^{\ast\ast}\rightarrow AP(n)^{\ast\ast}$ to $(r_F^{\Top})_n:P(n)^\ast\rightarrow P(n)^\ast$, which, by \cite[Theorem 5.1]{W}, is compatible with $s_F^{\Top}$ in the sense that the following diagram commutes:
\[\xymatrix{\ar @{} [dr] |{}
BP^\ast \ar[r]^{s_F^{\Top}} \ar[d] & BP^\ast \ar[d] \\
P(n)^\ast \ar[r]^{(r_F^{\Top})_n} & P(n)^\ast. }\] 
The map $s_F^{\Top}$ is defined via the coaction map
\[\psi^{\Top}:BP_\ast\overset{(1\wedge i)_\ast}\longrightarrow BP_\ast BP \overset{m_\ast^{-1}}\longrightarrow BP_\ast BP\otimes_{BP_\ast} BP_\ast,\]
which clearly takes $1$ to $1\otimes 1$. Therefore, in the formula 
\[\psi^{\Top}(1)=\underset{F}\sum c(t^F_{\Top})\otimes s_F^{\Top}(1),\]
all $s_F^{\Top}(1)$ have to be $0$, except for $s_0^{\Top}(1)$, which is $1$. It follows that also $(r_F^{\Top})_n(1)=0$ for $F\neq 0$ and $(r_0^{\Top})_n(1)=1$. Hence, $R_\C((r_F)_n(\tau))$ is $0$ for $F\neq 0$ and it is $1$ for $F=0$.

The preimage of $x\in P(n)^{\ast}$ under $R_\C:AP(n)^{\ast\ast}\rightarrow P(n)^\ast$ is 
\[\left\{\underset{k\geq 0}\sum x_k\cdot \tau^{k}\,|\,\sum_{k\geq 0} x_k=x\right\}.\] 
Since $(r_F)_n$ is a map of degree $(|F|,|F|/2)$, the only possible preimage of $R_\C((r_F)_n(\tau))$ is $(r_F)_n(\tau)=0$ for $F\neq 0$ and $(r_0)_n(\tau)=\tau$.

The claim on $\bar s_E( x\tau)$ now follows from the Cartan formula, Lemma \ref{Cartan}.
\end{proof}

The following is a motivic analogue to \cite[Proposition 4.14]{JW}.

\begin{theorem}\label{ABAK}
For any $X\in\SH(\C)^{fin}$ and any $n> 0$, $p>2$, there is a natural isomorphism
\[AB(n)_{\ast\ast}(X)\rightarrow AK(n)_{\ast\ast}(X)\otimes\F_p[v_{n+1},v_{n+2},\cdots].\]
Furthermore, for any $X\in\SH(\C)$,
$AK(n)_{\ast\ast}(X)=0$ if and only if $AB(n)_{\ast\ast}(X)=0$.
\end{theorem}

\begin{proof}
For a given exponent sequence $E$, consider the composition
\[AP(n)\overset{\overline{s}_E}\rightarrow \Sigma^{a,a/2}Ak(n)\rightarrow \Sigma^{a,a/2}AK(n)\overset{v_n^{-b}}\rightarrow \Sigma^{q,q/2}AK(n),\]
where $a=a_E$ and $b=b_E$ are as in Definition \ref{a,b}. These induce a natural homomorphism
\[\hat\Lambda: AP(n)_{\ast\ast}(X)\rightarrow AK(n)_{\ast\ast}(X)\otimes \F_p[v_{n+1},v_{n+2},\cdots]\]
\[\textup{by }\; \hat\Lambda(y)=\sum_{E\in\mathcal E}v_n^{-b_E}\bar{s}_E(y)\otimes v^E.\]
By Corollary \ref{modulo} and the Cartan formula, $r_E(yv_n)\equiv r_E(y)v_n \textup{ mod } I_n$. As $I_n=(v_0,\cdots,v_{n-1})$ acts trivially on $AP(n)$ and $Ak(n)$ (Proposition \ref{trivial action}), it follows that $\bar{s}_E:AP(n)_{\ast\ast}(X)\rightarrow Ak(n)_{\ast\ast}(X)$ is an $\F_p[v_n]$-homomorphism. Therefore, $\hat\Lambda$ can be extended to $AB(n)_{\ast\ast}(X)=v_n^{-1}AP(n)_{\ast\ast}X$. This yields a map
\[\Lambda:AB(n)_{\ast\ast}(X)\rightarrow AK(n)_{\ast\ast}(X)\otimes \F_p[v_{n+1},v_{n+2},\cdots].\]
We show that $\Lambda$ is an isomorphism for $X=S^0$, and, thus, for all $X=S^{p,q}$. By Lemma \ref{Ah}, we are considering a map $B(n)_{\ast}[\tau]\rightarrow K(n)_{\ast}[\tau]\otimes \F_p[v_{n+1},v_{n+2},\cdots]$ which, under $R_\C$, realises to the isomorphism $B(n)_{\ast}\overset{\cong}\rightarrow K(n)_{\ast}\otimes \F_p[v_{n+1},v_{n+2},\cdots]$ in \cite[Proposition 4.14]{JW}. Thus, it suffices to calculate $\Lambda(\tau)$, or, equivalently, $\hat\Lambda(\tau)=\sum_{E\in\mathcal E}v_n^{-b_E}\bar{s}_E(\tau)\otimes v^E$. By the previous lemma, this is equal to $v_n^{-b_0}\bar{s}_0(\tau)\otimes v^0=\tau$. Thus, $\Lambda$ is an isomorphism if $X$ is a sphere.

Cellular induction via the five lemma shows that $\Lambda$ is an isomorphism for all $X\in\SH(\C)^{fin}$.

By Definition \ref{fin}(3), any cell spectrum is the colimit of a diagram of finite cell spectra. As in \cite[Theorem 2.1(a)]{RavLoc}, it follows that, for any $X\in\SH(\C)^{cell}$, $AB(n)_{\ast\ast}X=0$ if and only if $AK(n)_{\ast\ast}X=0$. Here, we use that $\pi_{\ast\ast}(-)$ commutes with filtered colimits by \cite[Proposition 9.3]{DI}. 

Furthermore, by \cite[Proposition 7.3]{DI}, any $X\in\SH(\C)$ has a cellular approximation $X'\in\SH(\C)^{cell}$, $f:X'\rightarrow X^{fib}$ ($X^{fib}$ a fibrant replacement of $X$), such that $\pi_{\ast\ast}(f):\pi_{\ast\ast}X'\overset\cong\rightarrow\pi_{\ast\ast}X$. Hence, also $\pi_{\ast\ast}(1_E\wedge f):E_{\ast\ast}X'\rightarrow E_{\ast\ast}X$ is an isomorphism for any $E\in\SH(\C)$. In particular, 
\[ AB(n)_{\ast\ast}X=0\Leftrightarrow AB(n)_{\ast\ast}X'=0\Leftrightarrow AK(n)_{\ast\ast}X'=0\Leftrightarrow AK(n)_{\ast\ast}X=0.\]
\end{proof}

\begin{rk}
In addition to the above theorem, I also worked on a motivic version of \cite[Theorem 4.8]{JW}, which I was almost able to prove, except for one detail which is necessary for part (c), where \cite[Lemma 4.10]{JW} enters the argument. Large parts of the proof of this lemma can be translated motivically, using motivic cellular approximation as introduced in \cite[Proposition 7.3]{DI} and the motivic Hurewicz theorem from \cite[Theorem 5.57]{MorelA1algtop}. The only problematic point is part (e) in the proof of \cite[Lemma 4.10]{JW}, where a spectrum $Y^k$ with finitely generated homology is approximated by a finite spectrum $F$. A possible construction of $F\rightarrow Y^k$ is described in \cite[Proposition 4C.1]{Hatcher}. A similar construction can be performed in the motivic setting. However, it is not clear that the resulting spectrum $F$ is finite. This would be the case if $(H\Z_(p)){\ast\ast}$ is a coherent (or Noetherian) ring, because then, $\h_{\ast\ast}(X,\Z_{(p)})$ is finitely generated over $(H\Z_(p)){\ast\ast}$ for finite cell spectra $X$, making a finite induction step possible.
\end{rk}

\begin{cor}\label{AK=AB}
For $p>2$ and $n>0$,
\[\langle AK(n)\rangle =\langle AB(n)\rangle \textup{ in } \SH(\C).\]
\end{cor}

This also holds for $n=0$, in which case we do not need to assume $k=\C$ or $p>2$.

\begin{prop}\label{AK=AB2}
For any $k\subseteq \C$ and any prime $p$,
\[\langle AK(0)\rangle =\langle AB(0)\rangle \textup{ in } \SH(k).\]
\end{prop}

\begin{proof}
By definition, $AB(0)=p^{-1}ABP$ and 
\[AK(0)=p^{-1}ABP/(v_1,v_2,\cdots)=p^{-1}MGL_{(p)}/(a_1,a_2,\cdots)=MGL_\Q/(a_1,a_2,\cdots).\] 
By the main result in \cite{Hoy}, $MGL/(a_1,a_2,\cdots)\cong H\Z$, which implies 
\[p^{-1}MGL_{(p)}/(a_1,a_2,\cdots)\cong p^{-1}H\Z_{(p)}= H\Q.\] 
Hence, $AK(0)\cong H\Q$.
From Lemma \ref{bousfield}(1), we already know 
\[\langle AK(0)\rangle \leq \langle AB(0)\rangle.\]
To prove 
\[\langle AK(0)\rangle \geq \langle AB(0)\rangle,\]
let $X\in\SH(k)^{fin}$ satisfy $AK(0)_{\ast\ast}(X)=0$. We have to show $AB(0)_{\ast\ast}(X)=0$.
We use \cite[Lemma 7.10]{Hoy}, in which we rationalise $H\Z$ and $MGL$ and set $X=MGL_\Q$. Here, rationalising a spectrum $E$ means forming the homotopy colimit $E_\Q=p^{-1}E_{(p)}$. The lemma then states the following. 

If $F\in\SH(k)$ satisfies $H\Q\wedge F= 0$, then $[F,MGL_\Q]=0$.

When rationalising the proof of \cite[Lemma 7.10]{Hoy}, one uses that $\kappa_0(MGL_\Q)\cong (MGL_\Q)_{\leq 0}$ and $\kappa_0(H\Q)\cong H\Q_{\leq 0}$, which hold because $MGL_\Q$ and $H\Q$ are connective (which follows from the connectivity of $MGL$ \cite[Corollary 3.9]{Hoy} and of $H\Z$ \cite[Lemma 7.3]{Hoy}). Furthermore, one uses that 
\[(MGL_\Q)_{\leq 0}\cong (MGL_{\leq 0})_\Q\cong (H\Z_{\leq 0})_\Q\cong H\Q_{\leq 0},\]
where the first and third isomorphisms follow from the fact that $(-)_{\leq d}$ preserves filtered homotopy colimits \cite[Lemma 2.1]{Hoy} and the middle isomorphism is by \cite[Lemma 7.5]{Hoy}.

We apply this version of \cite[Lemma 7.10]{Hoy} to see that $(MGL_\Q)^{\ast\ast}X=0$ for $X$ as above, which is equivalent to $(MGL_\Q)_{\ast\ast}X=0$, since $X$ is finite (Proposition \ref{C_E}). Using Lemma \ref{bousfield}(1) again, it follows that $(p^{-1}ABP)_{\ast\ast}(X)=0$, as we wanted to show. 

As in the previous proof, the equivalence 
\[AK(0)_{\ast\ast}X=0 \Leftrightarrow AB(0)_{\ast\ast}X=0\]
passes from finite spectra to cellular spectra and then to arbitrary spectra $X\in \SH(k)$.
\end{proof}

As a corollary of the above results, we can prove the following analogue of \cite[Theorem 2.1(i)]{RavLoc}.

\begin{cor}
In $\SH(\C)$,
\[\langle AK(n)\rangle \wedge \langle AK(m)\rangle =0 \]
for any $m\neq n$.
\end{cor}

\begin{proof}
Assume $m>n$. By (1) and (2) of Lemma \ref{bousfield2}, 
\[\langle AK(n)\rangle \wedge \langle AP(n+1)\rangle \leq \langle AE(n)\rangle \wedge\langle AP(n+1)\rangle =\langle 0\rangle.\]
Furthermore, by Lemma \ref{bousfield2}(3) and the above result, 
\[\langle AP(m)\rangle=\langle AB(m)\rangle\vee\langle AP(m+1)\rangle=\langle AK(m)\rangle\vee\langle AP(m+1)\rangle.\]
This implies $\langle AK(m)\rangle\leq \langle AP(m)\rangle$.
Since $m>n$, $\langle AP(m)\rangle\leq\langle AP(n+1)\rangle$ by Lemma \ref{bousfield}(1). Hence,
\[\langle AK(n)\rangle\wedge \langle AK(m)\rangle \leq \langle AK(n)\rangle \wedge \langle AP(n+1)\rangle=\langle 0\rangle.\] 
\end{proof}

\section{Decomposition of $\langle AE(n)\rangle$}

Recall from Definition \ref{AK(n)} that
\[AE(n)=v_n^{-1} ABP/(v_{n+1},v_{n+2},\cdots).\]
With the above preparations, we are ready to prove an analogue of the decomposition of Bousfield classes given in \cite[Theorem 2.1(d)]{RavLoc}. This answers a special case of \cite[Question 2.17]{HornLoc}.

\begin{theorem}\label{decomposition}
For $p>2$,
\[\langle AE(n)\rangle =\underset{0\leq i\leq n}\bigvee\langle AK(i)\rangle  \textup{ in } \SH(\C).\]
\end{theorem}

\begin{proof}
These are the same arguments as for \cite[Theorem 2.1(d)]{RavLoc}:
By Lemma \ref{bousfield2}(3) and Corollary \ref{AK=AB}, $\langle AP(n)\rangle=\langle AB(n)\rangle \vee \langle AP(n+1)\rangle=\langle AK(n)\rangle\vee\langle AP(n+1)\rangle$. Since $AP(0)=\BP$, it follows inductively: 
\[\langle \BP\rangle=\langle AK(0)\rangle\vee\langle AK(1)\rangle\vee\cdots\vee \langle AK(n)\rangle\vee\langle AP(n+1)\rangle.\]
Since $AE(n)$ is an $\BP$-module spectrum, $\langle AE(n)\rangle\leq \langle \BP\rangle$ by Lemma \ref{bousfield}(2).
By Lemma \ref{bousfield2}(2), $\langle AE(n)\rangle\wedge\langle AP(n+1)\rangle=\langle 0\rangle$. It follows that 
\[\langle AE(n)\rangle\leq\langle AK(0)\rangle\vee\langle AK(1)\rangle\vee\cdots\langle AK(n)\rangle.\]
By Corollary \ref{bousfieldE} and Lemma \ref{bousfield2}(1), $\langle AE(n)\rangle\geq\langle AE(i)\rangle\geq\langle AK(i)\rangle$ for $i\leq n$ and hence also $\langle AE(n)\rangle\geq\underset{i\leq n}\bigvee\langle AK(i)\rangle$.
\end{proof}

\begin{rk}
The restriction $p>2$ originates from Section \ref{section action}, where it was needed to prove homotopy associativity for the map $\mu_{AP(n)}:AP(n)\wedge AP(n)\rightarrow AP(n)$. There might be a different way to show that the $ABP_{\ast\ast}$-action on $AP(n)_{\ast\ast}(X)$ induces an $AP(n)_{\ast\ast}$-action, in which case the condition $p>2$ could be removed in the previous section and in the theorem. 
\end{rk}

\section{$AK(n)$ and $AK(n+1)$}\label{section n+1}

\begin{lemma}\label{coherent}
Let $p$ be any fixed prime.
For $n\geq 1$ and $X\in\SH(\C)^{fin}$, $AP(n)_{\ast\ast}(X)$ is an $AP(1)_{\ast\ast}AP(1)$-comodule and a coherent module over $AP(1)_{\ast\ast}$. 
\end{lemma}

\begin{proof}
Note that $AP(1)=ABP/p$ is a ring spectrum because $p:ABP\rightarrow ABP$ is a map of ring spectra and the category of ring spectra is cocomplete. The coaction on $AP(n)_{\ast\ast}(X)$ is defined via
\[AP(n)_{\ast\ast}X\rightarrow AP(n)_{\ast\ast}(AP(1)\wedge X)\leftarrow AP(1)_{\ast\ast}AP(1)\otimes_{AP(1)_{\ast\ast}}AP(n)_{\ast\ast}X,\] 
where the left map is induced by the unit of $AP(1)$ and we need to show that the right map (induced by $AP(1)\wedge AP(n)\rightarrow AP(n)$, compare Corollary \ref{AP(n) module}) is an isomorphism. To prove this isomorphism, it suffices to show that, in the spectral sequence from \cite[Proposition 7.7]{DI},
\[\Tor^{AP(1)_{\ast\ast}}(AP(1)_{\ast\ast}AP(1),AP(n)_{\ast\ast}X)\Rightarrow AP(n)_{\ast\ast}(AP(1)\wedge X),\]
$AP(1)_{\ast\ast}AP(1)$ is free over $AP(1)_{\ast\ast}$, so that the spectral sequence collapses immediately. 
We want to apply Lemma \ref{Ah} to $AP(1)\wedge AP(1)$.
Recall that the slice spectral sequence considered in Lemma \ref{Ah} converges for quotients of Landweber exact spectra by \cite[Theorem 8.12 and Example 8.13]{Hoy}. Now, $ABP\wedge ABP$ is a product of Landweber exact spectra and is therefore Landweber exact, see e.g. \cite[Remark 9.2]{MLE}, and $AP(1)\wedge AP(1)$ is a quotient of $ABP\wedge ABP$. Hence, the slice spectral sequence converges strongly, and it collapses for the same degree reasons as in Lemma \ref{Ah}. Thus, $AP(1)_{\ast\ast}AP(1)\cong P(1)_\ast P(1)[\tau]$. By \cite[Section 1]{JY}, this is isomorphic to $P(1)_\ast[\tau, z^{E,A}]$ for certain $z^{E,A}$. In particular, it is free over $AP(1)_{\ast\ast}\cong P(1)_\ast[\tau]$, as we wanted to show.

Now we show that $AP(n)_{\ast\ast}(X)$ is coherent over $AP(1)_{\ast\ast}$. Recall that
$P(1)_\ast\cong \F_p[v_1^{\Top},v_{2}^{\Top},\cdots]$ is a coherent ring (see \cite[Section 1]{CS}). The same holds for $AP(1)_{\ast\ast}\cong \F_p[\tau, v_1,v_{1},\cdots]$. By \cite[Proposition 1.2]{CS}, coherence of modules satisfies the two out of three property for exact triangles of graded modules (i.e., long exact sequences). It follows that $AP(n)_{\ast\ast}$ is a coherent $AP(1)_{\ast\ast}$-module, and
cellular induction implies that $AP(n)_{\ast\ast}(X)$ is a coherent $AP(1)_{\ast\ast}$-module, too.
\end{proof}

Setting $n=1$ and $X=S^0$, the above lemma tells us that $AP(1)_{\ast\ast}$ is a coherent $AP(1)_{\ast\ast}AP(1)$-comodule, where coherent means coherent as an $AP(1)_{\ast\ast}$-module.

\begin{lemma}{\bf (Invariant prime ideals)}

For $k=\C$ and $p$ any prime, the invariant prime ideals of $AP(1)_{\ast\ast}$ (that is, prime ideals which are also sub-comodules) are given by $I_m=(v_1,\cdots,v_{m-1})$ and $\overline{I}_m=(\tau,v_1,\cdots,v_{m-1})$.
\end{lemma}

\begin{proof}
We have $AP(1)_{\ast\ast}\cong P(1)_\ast[\tau]$ and $AP(1)_{\ast\ast}AP(1)\cong P(1)_\ast P(1)[\tau]$, as in the proof of the previous lemma. Under the functor $R_\C$, the coaction
\[AP(1)_{\ast\ast}\rightarrow AP(1)_{\ast\ast}AP(1)\otimes_{AP(1)_{\ast\ast}}AP(1)_{\ast\ast}\]
realises to
\[P(1)_\ast\rightarrow P(1)_\ast P(1)\otimes_{P(1)_\ast}P(1)_\ast\]
(similarly as in the proof of Corollary \ref{modulo}). By the classical invariant prime ideal theorem (see \cite[Theorem 2.7]{LandweberAnni} or compare \cite[Theorem 1.16]{JY}), the invariant prime ideals of $P(1)_\ast$ are given by $I_m^{\Top}=(v_1^{\Top},\cdots,v_{m-1}^{\Top})$. The isomorphism \[AP(1)_{\ast\ast}AP(1)\otimes_{AP(1)_{\ast\ast}}AP(1)_{\ast\ast}\cong P(1)_\ast P(1)\otimes_{P(1)_\ast}AP(1)_{\ast\ast}\] 
implies that $I_m$ is an invariant prime ideal in $AP(1)_{\ast\ast}$. By $AP(1)_{\ast\ast}\cong P(1)_\ast[\tau]$, it follows also that $\overline{I}_m$ is an invariant prime ideal of $AP(1)_{\ast\ast}$, too.

It remains to show that there cannot be any further invariant prime ideals. As in \cite{LandweberAnni}, this follows from the fact that the only primitive elements in $AP(1)_{\ast\ast}/\overline{I}_m\cong P(1)_\ast/I_m^{\Top}$ are multiples of powers of $v_m$ (compare \cite[Proposition 2.11]{LandweberAnni}).
\end{proof}

\begin{cor}\label{Landweber}{\bf (Motivic Landweber filtration theorem)}

Let $p$ be any prime and $n\geq 1$. For $X\in\SH(\C)^{fin}$, $AP(n)_{\ast\ast}(X)$ can be filtered by $AP(1)_{\ast\ast}$-modules
\[AP(n)_{\ast\ast}(X)=M_0\supset\cdots\supset M_k=0\]
such that $M_i/M_{i+1}\cong  AP(1)_{\ast\ast}/I_m$ or $AP(1)_{\ast\ast}/\overline{I}_m$ for some $I_m$ and $\overline{I}_m$ ($m\geq n$) as above.
\end{cor}

\begin{proof}
By Lemma \ref{coherent}, $AP(n)_{\ast\ast}(X)$ is a coherent $AP(1)_{\ast\ast}AP(1)$-comodule. Landweber's filtration theorem \cite[Theorem 3.3]{LandweberFiltr} (see also \cite[Theorem 1.16]{JY}) implies that $AP(n)_{\ast\ast}(X)$ has a filtration 
\[AP(n)_{\ast\ast}(X)=M_0\supset\cdots\supset M_k=0\]
such that $M_i/M_{i+1}\cong AP(1)_{\ast\ast}/I$ for some $I$ which is invariant under the comodule action. Thus, the claim follows from the previous lemma.
\end{proof}

The following is a motivic version of one statement in \cite[Theorem 2.11]{RavLoc}. In the proof, we use ideas from Ravenel's proof.

\begin{theorem}\label{AK(n+1)}
Let $p>2$.
If $X\in\SH(\C)^{fin}$ satisfies $AK(n+1)_{\ast\ast}(X)=0$, then also $AK(n)_{\ast\ast}(X)=0$. That is, $\langle AK(n+1)\rangle\geq\langle AK(n)\rangle$ in $\SH(\C)^{fin}$.
\end{theorem}

\begin{proof}
Assume $n>0$. The case $n=0$ will be considered at the end of the proof.
Let $E_{\ast\ast}(-)$ be defined by 
\[E_{\ast\ast}X=AE(n+1)_{\ast\ast}\otimes_{\BP_{\ast\ast}}AP(n)_{\ast\ast}(X).\] 
As in \cite[Lemma 3.5]{JY}, the above Landweber filtration theorem and the fact that $AE(n+1)$ is Landweber exact (see \cite[Theorem 8.7]{MLE}) yield 
\[\Tor_1^{AP(n)_{\ast\ast}}(AE(n+1)_{\ast\ast}\otimes_{ABP_{\ast\ast}}AP(n)_{\ast\ast},AP(n)_{\ast\ast}/I)=0\]
for all invariant prime ideals $I\subseteq AP(n)_{\ast\ast}$ as above. (Alternatively, this can be derived from the topological analogue and Lemma \ref{Ah}, using \cite[Theorem 8.7]{MLE}.) As in \cite[Theorem 2.11]{RavLoc}, this implies that $E_{\ast\ast}(-)$ is an exact functor. 

In analogy to \cite{RavLoc}, we show that there is an injective pairing \[AK(n+1)_{\ast\ast}\otimes_{E_{\ast\ast}}E_{\ast\ast}X\hookrightarrow AK(n+1)_{\ast\ast}X.\] 
To construct this pairing, note that the $ABP$-action on $AK(n+1)$ factors through a map $AP(n)\wedge AK(n+1)\rightarrow AK(n+1)$ by methods from Section \ref{section action}, since $v_i$, $i< n+1$, acts trivially on $AK(n+1)$ by Corollary \ref{trivial action}. This induces a map 
\[AK(n+1)_{\ast\ast}\otimes_{ABP_{\ast\ast}}AP(n)_{\ast\ast}X\rightarrow AK(n+1)_{\ast\ast}X.\]
As in \cite[Theorem 2.11]{RavLoc}, this map factors through a pairing 
\[AK(n+1)_{\ast\ast}\otimes_{E_{\ast\ast}}E_{\ast\ast}X\rightarrow AK(n+1)_{\ast\ast}X,\]
the reason being again that the relevant elements act trivially. Such a pairing induces a universal coefficient spectral sequence, whose motivic version is constructed in \cite[Propositions 7.7 and 7.10]{DI},
\[\Tor_i^{E_{\ast\ast}}(E_{\ast\ast}(X),AK(n+1)_{\ast\ast})\Rightarrow AK(n+1)_{\ast\ast}(X).\]
As in \cite{RavLoc}, to prove the injectivity of the pairing, it suffices to prove the vanishing of
\[\Tor_i^{ABP_{\ast\ast}}(AP(n)_{\ast\ast}(X),AK(n+1)_{\ast\ast})\]
for $i>1$.  Since $n\geq 1$, $p$ acts trivially on both of these modules, and we can replace $ABP$ by $AP(1)=ABP/p$. Hence, we have to show the vanishing of
\[\Tor_i^{AP(1)_{\ast\ast}}(AP(n)_{\ast\ast}(X),AK(n+1)_{\ast\ast})\]
and, by Corollary \ref{Landweber}, the question reduces to the vanishing of
\[\Tor_i^{AP(1)_{\ast\ast}}(AP(1)_{\ast\ast}/I,AK(n+1)_{\ast\ast})\]
for $i>1$ and $I=I_m$ or $\overline{I}_m$, $m\geq n$. By Lemma \ref{Ah}, this equals
\[\Tor_i^{P(1)_{\ast}[\tau]}(P(1)_{\ast}/{I}_m[\tau],K(n+1)_{\ast}[\tau])\; \textup{ or}\]
\[\Tor_i^{P(1)_{\ast}[\tau]}(P(1)_{\ast}/{I}_m,K(n+1)_{\ast}[\tau]), \; \textup{ respectively.}\]
A projective resolution of $K(n+1)_\ast$ over $P(1)_\ast$ yields a projective resolution of $K(n+1)_\ast[\tau]$ over $P(1)_\ast[\tau]$ by applying $-\otimes\F_p[\tau]$. It follows that both of the above torsion terms vanish if 
\[ \Tor_i^{P(1)_{\ast}}(P(1)_{\ast}/{I}_m,K(n+1)_{\ast})=0.\]
For $i>1$, this follows from 
\[ \Tor_i^{BP_{\ast}}(BP_{\ast}/{I}_m,K(n+1)_{\ast})=0,\]
as in the proof of \cite[Theorem 2.11]{RavLoc}. This proves the injectivity of the above pairing.

Now, assume that $AK(n+1)_{\ast\ast}X=0$. The injectivity of the pairing implies that $AK(n+1)_{\ast\ast}\otimes_{E_{\ast\ast}}E_{\ast\ast}X=0$. 
Recall $AK(n+1)_{\ast\ast}\cong\h_{\ast\ast}[v_{n+1}^{\pm 1}]$ and note that 
\[E_{\ast\ast}\cong AE(n+1)_{\ast\ast}\otimes_{ABP_{\ast\ast}}AP(n)_{\ast\ast}\cong E(n+1)_\ast\otimes_{BP_{\ast}}P(n)_\ast[\tau]\cong\h_{\ast\ast}[v_n,v_{n+1}^{\pm 1}]\] 
by \cite[Theorem 8.7]{MLE} and Lemma \ref{Ah}. Therefore, $AK(n+1)_{\ast\ast}\otimes_{E_{\ast\ast}}E_{\ast\ast}X=0$ implies $v_n^{-1}E_{\ast\ast}X=0$. Now,
\[0=v_n^{-1}E_{\ast\ast}X=AE(n+1)_{\ast\ast}\otimes_{\BP_{\ast\ast}}AB(n)_{\ast\ast}X,\]
by the definitions of $E_{\ast\ast}(-)$ and of $AB(n)$.
By Theorem \ref{ABAK}, since $X$ is finite, this is equal to 
\[AE(n+1)_{\ast\ast}\otimes_{\BP_{\ast\ast}}AK(n)_{\ast\ast}(X)\otimes\F_p[v_{n+1},v_{n+2},\cdots].\]
By \cite[Theorem 8.7]{MLE}, $AE(n+1)_{\ast\ast}\cong E(n+1)_\ast\otimes_{BP_\ast} ABP_{\ast\ast}$. It follows
\[0=\Z_{(p)}[v_1,\cdots,v_n,v_{n+1}^{\pm 1}]\otimes_{\Z_{(p)}[v_1,\cdots]}AK(n)_{\ast\ast}X\otimes\F_p[v_{n+1},\cdots].\]
As $v_m$ acts trivially on $AK(n)_{\ast\ast}X$ for all $m\neq n$ by Corollary \ref{trivial action}, this implies that $AK(n)_{\ast\ast}X=0$.\\

It remains to show that $AK(1)_{\ast\ast}X=0$ implies $AK(0)_{\ast\ast}X=0$. Recall the definitions $AE(1)=v_1^{-1}ABP/(v_2,\cdots)$ and $AK(1)=v_1^{-1}ABP/(p,v_2,\cdots)=AE(1)/p$. Since multiplication by $p$ commutes with $\pi_{\ast\ast}(-)$, it follows that $AK(1)_{\ast\ast}X=0$ implies $p^{-1}AE(1)_{\ast\ast}X=0$. 
(Note that $p^{-1}(AE(1)_{\ast\ast}X)\cong (p^{-1}AE(1))_{\ast\ast}X$ since $\pi_{\ast\ast}(-)$ commutes with filtered colimits by \cite[Proposition 9.3]{DI}.) The following argument is closely related to Corollary \ref{bousfieldE}.
Since $p^{-1}AE(1)_{\ast\ast}X=0$, $p^{-1}(ABP/(v_2,\cdots))_{\ast\ast}X$ is $v_1$-torsion. By Theorem \ref{torsion}, this implies that $p^{-1}(ABP/(v_2,\cdots))_{\ast\ast}X$ is $p$-torsion. But this can only be the case if $p^{-1}(ABP/(v_2,\cdots))_{\ast\ast}X=0$. It follows that 
\[AE(0)_{\ast\ast}X=p^{-1}(ABP/(v_1,v_2,\cdots))_{\ast\ast}X=0.\] 
Now, by Theorem \ref{decomposition}, $\langle AE(0)\rangle = \langle AK(0)\rangle$. Hence, $AK(0)_{\ast\ast}X=0$.
\end{proof}

In terms of thick ideals in $\SH(\C)_{(p)}^{fin}$, $p>2$, we have proven that the motivic Morava K-theory spectra $AK(n)$ indeed describe a descending chain of thick ideals, similarly to the chain of thick subcategories in $\SH_{(p)}^{fin}$. The inclusions are proper by Chapter \ref{motivic type n}.

\begin{cor}
For $p>2$,
\[\SH(\C)_{(p)}^{fin}\supsetneq \mathcal C_{AK(0)}\supsetneq \mathcal C_{AK(1)}\supsetneq\cdots.\]
\end{cor}

To sum up, we have identified three sequences of thick ideals which may or may not be equal:
\[\xymatrix@1{
\SH(\C)_{(p)}^{fin}\ar @{} [d] |-*[@]{{\veq}} \ar @{} [r] |-*[@]{\supset} &  R^{-1}(\mathcal C_{1})\ar @{} [d] |-*[@]{{\vsupseteq}} \ar @{} [r] |-*[@]{\supset} &  R^{-1}(\mathcal C_{2})\ar @{}[d] |-*[@]{{\vsupseteq}} \ar @{} [r] |-*[@]{\supset}  &\cdots\\
\SH(\C)_{(p)}^{fin}\ar @{} [d] |-*[@]{{\veq}} \ar @{} [r] |-*[@]{\supset} & \mathcal C_{AK(0)}\ar @{} [d] |-*[@]{{\vsupseteq}} 
\ar @{} [r] |-*[@]{\supset} 
&  \mathcal C_{AK(1)}\ar @{}[d] |-*[@]{{\vsupseteq}} \ar @{} [r] |-*[@]{\supset}  &\cdots\\
\SH(\C)_{(p)}^{fin} \ar @{} [r] |-*[@]{\supset} &  \thickid(c\mathcal C_{1}) \ar @{} [r] |-*[@]{\supset} & \thickid(c\mathcal C_{2}) \ar @{} [r] |-*[@]{\supset} &\cdots.}\]

\printbibliography

\vspace{50pt}

\emph{Ruth Joachimi}

\emph{Fachbereich Mathematik und Naturwissenschaften} 

\emph{Bergische Universit\"at Wuppertal }

\emph{42119 Wuppertal }

\emph{Germany}

ruth.joachimi@d-fine.de

\end{document}